\documentclass{article}
\usepackage{preamble}

\hypersetup{
    pdfauthor={Shachar Carmeli, Guy Kapon, Noam Nissan}
}
\title{Six functor formalisms via internal higher algebra}
\author{%
    Shachar Carmeli \quad Guy Kapon \quad Noam Nissan\\[0.5em]
    {\small Department of Mathematics}\\
    {\small Weizmann Institute of Science}\\
    {\small Rehovot, Israel}
}
\date{September 29, 2026}

\begin{document}

\maketitle

\begin{abstract}
    We extend a six-functor formalism $D\colon\Span(C,E)\to\catofcats$ to a lax symmetric monoidal functor of $(\infty,2)$-categories $\spantwo(C,E)^P_I\to\twocatofcats$, where $P$ and $I$ are the classes of $D$-proper and $D$-\'etale morphisms, respectively. This proves a conjecture of Mann and generalizes a special case of a theorem of Cnossen--Lenz--Linskens \cite{CLLuniv}.

    To prove this result, we develop a theory of internal $\subuniverse$-monoidal categories and $\subuniverse$-operads, where $\subuniverse$ is a local class of morphisms in a topos. These notions generalize the internal symmetric monoidal categories and operads developed by Martini--Wolf.
\end{abstract}

\begin{center}
    \includegraphics[width=0.55\textwidth]{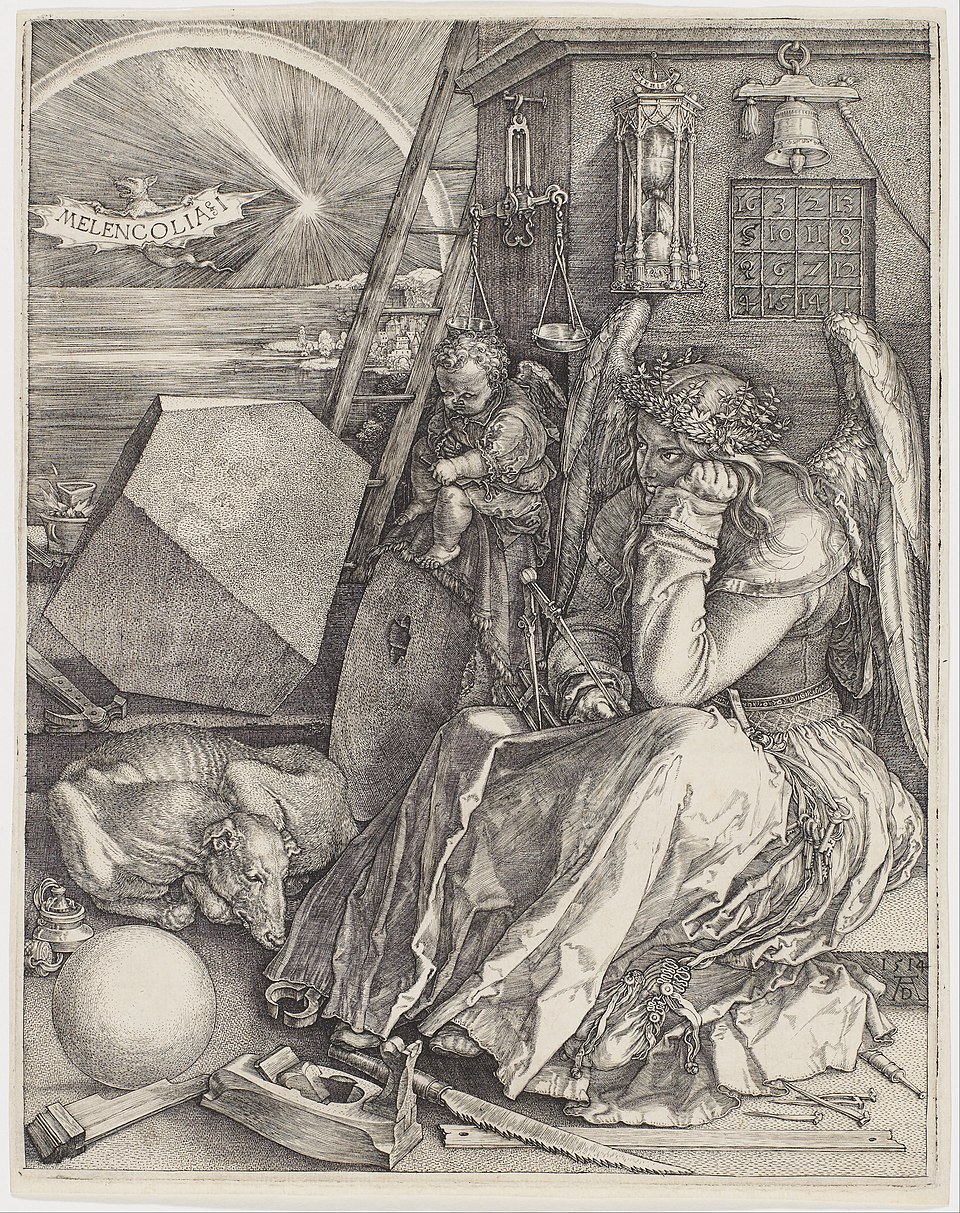}

    Albrecht Dürer, {\bf Melencolia I}, 1514. Engraving, 24 × 18.8 cm.
\end{center}

\tableofcontents
\section{Introduction.}

\subsection{Background}

Six-functor formalisms were first introduced by Grothendieck in the context of \'etale cohomology (see \cite{SGA4XVII}). Since then, they have appeared in many other contexts, including algebraic geometry, representation theory (see \cite{just_6ff}), and analytic geometry (see \cite{Mann6Functors}).

In $\infty$-categorical language, six-functor formalisms are encoded by functors from $\infty$-categories of spans\footnote{Mann and Scholze use the term ``correspondences'' for spans.}. This approach was first considered by Liu--Zheng \cite{LZscary} and Gaitsgory--Rozenblyum \cite{GaitsgoryRozenblyum2017} and was later developed by Mann (see Scholze's notes \cite{Scho6FFpdf}). Let $\varcat{C}$ be an $\infty$-category with finite products, and let $\varcat{E}\subset\varcat{C}$ be a wide subcategory stable under base change in $\varcat{C}$. Then there is a symmetric monoidal $\infty$-category $\Span(\varcat{C},\varcat{E})$ with the same objects as $\varcat{C}$ but with morphisms from $A$ to $B$ given by diagrams of the form
\[
\begin{tikzcd}
    &F\ar["f"',dl]\ar["g",dr,tail]\\
    A&&B
\end{tikzcd}
\]
where the tailed arrow belongs to $\varcat{E}$.

Mann defines a $3$-functor formalism as a lax symmetric monoidal functor into the $\infty$-category of $\infty$-categories
\[
\Span(\varcat{C},\varcat{E})\to\catofcats.
\]
A $3$-functor formalism encodes the operations $(f^*,g_!,\otimes)$. A $6$-functor formalism is a $3$-functor formalism such that the three operations $f^*$, $g_!$, and $X\otimes -$ admit right adjoints $f_*$, $g^!$, and $\underline{\operatorname{Hom}}(X,-)$, respectively.

Motivated by Grothendieck's construction of the six operations, Mann developed a general mechanism for constructing certain $3$-functor formalisms from adjoint functor data. The key input is a \emph{suitable decomposition} of $\varcat{E}$, that is, wide subcategories $\varcat{I},\varcat{P}\subset\varcat{E}$ such that every morphism in $\varcat{E}$ factors as a morphism in $\varcat{I}$ followed by a morphism in $\varcat{P}$. Building on work of Liu--Zheng, Mann showed in his thesis \cite{Mann6Functors} that if a contravariant lax monoidal functor
\[
D\colon\varcat{C}^{\op}\to\catofcats
\]
admits left adjoints along morphisms in $\varcat{I}$ and right adjoints along morphisms in $\varcat{P}$, satisfying suitable base-change and projection formulas, then these adjoints can be assembled into an extension
\[
D\colon\Span(\varcat{C},\varcat{E})\to\catofcats.
\]

In recent work \cite{CLLuniv}, Cnossen, Lenz, and Linskens show that, under the same assumptions, $D$ extends uniquely to a lax symmetric monoidal $2$-functor into the $(\infty,2)$-category of $\infty$-categories
\[
D\colon\spantwo(\varcat{C},\varcat{E})^{\varcat{P}}_{\varcat{I}}\to\twocatofcats.
\]
The source is the $(\infty,2)$-category with the same objects and $1$-morphisms as $\Span(\varcat{C},\varcat{E})$ but with $2$-morphisms given by diagrams of the form
\[
\begin{tikzcd}
    &F\ar["f"',dl]\ar["g",dr,tail]\\
    A&T\ar["p",u]\ar["i"',d]&B\\
    &F'\ar["{f'}",ul]\ar["{g'}"',ur,tail]
\end{tikzcd}
\]
with $i\in\varcat{I}$ and $p\in\varcat{P}$. The $2$-functoriality encodes the counits and units of the adjunctions $i_!\dashv i^*$ and $p^*\dashv p_!$.

In \cite{just_6ff},  Heyer and Mann introduced a procedure for ``sheafifying'' certain six-functor formalisms, producing more general formalisms which do not appear to arise from suitable decompositions.

Even without a suitable decomposition, Mann and Heyer define notions of $D$-\'etale and $D$-proper morphisms for any $3$-functor formalism $D\colon\Span(\varcat{C},\varcat{E})\to\catofcats$. These have properties analogous to those of morphisms in $\varcat{I}$ and $\varcat{P}$, respectively. In particular, if $i$ is $D$-\'etale, then there is a canonical adjunction $i_!\dashv i^*$, and if $p$ is $D$-proper, then there is a canonical adjunction $p^*\dashv p_!$.

$D$-\'etale morphisms are defined by an inductive procedure. Let $i\colon B\to A$ be a morphism in $\varcat{E}$. If $i$ is an isomorphism, then $i$ is automatically $D$-\'etale. If $i^*$ has a left adjoint $i_\sharp$ and the diagonal $\Delta_i\colon B\to B\times_A B$ is $D$-\'etale, then there is a canonically defined norm map $i_\sharp\to i_!$. In this case, $i$ is $D$-\'etale if the norm map is an isomorphism.

In private communications, Mann shared with us the following conjecture.

\begin{conjecture}[Mann]\label{Mann_conjecture}
    Let $D\colon\Span(\varcat{C},\varcat{E})\to\catofcats$ be a $3$-functor formalism. Let $\varcat{I},\varcat{P}\subset\varcat{E}$ denote the wide subcategories of $D$-\'etale and $D$-proper morphisms in $\varcat{E}$, respectively. Then $D$ extends uniquely to a lax symmetric monoidal $2$-functor of $(\infty,2)$-categories
    \[
    \spantwo(\varcat{C},\varcat{E})^{\varcat{P}}_{\varcat{I}}\to\twocatofcats.
    \]
\end{conjecture}

\subsection{Main results}

The goal of this paper is to prove \Cref{Mann_conjecture}. We first prove a version without the monoidal structure. Let
\[
D\colon\Span(\varcat{C},\varcat{E})\to\vartwocat{K}
\]
be a functor into an $(\infty,2)$-category $\vartwocat{K}$.
Let $\varcat{I}\subset\varcat{E}$ be a wide subcategory of truncated morphisms that is stable under base change in $\varcat{C}$ and left cancellable. We say that $D$ is $\varcat{I}$-left adjointed if, for every $i\in\varcat{I}$, there is a canonical adjunction $i_!\dashv i^*$, defined similarly to the notion of a $D$-\'etale morphism described above. Likewise, for a wide subcategory $\varcat{P}\subset\varcat{E}$ satisfying the same properties, we say that $D$ is $\varcat{P}$-right adjointed if the dual construction gives an adjunction $p^*\dashv p_!$ for each $p\in\varcat{P}$. The precise definitions are given in \Cref{def:adjointed_functors}.

Our main theorem is:

\begin{main}\label{main_theorem:free_adjointed_2_cat}
    Let $\varcat{C}$ be an $\infty$-category, and let $\varcat{P},\varcat{I}\subset\varcat{E}\subset\varcat{C}$ be wide subcategories stable under base change in $\varcat{C}$. Assume that $\varcat{P}$ and $\varcat{I}$ are left cancellable and that their morphisms are truncated. Then the canonical functor
    \[
    F\colon\Span(\varcat{C},\varcat{E})\to\spantwo(\varcat{C},\varcat{E})^{\varcat{P}}_{\varcat{I}}
    \]
    is the free $\varcat{I}$-left adjointed and $\varcat{P}$-right adjointed functor. That is, for any $(\infty,2)$-category $\vartwocat{K}$, the restriction functor
    \[
    \Fun(\spantwo(\varcat{C},\varcat{E})^{\varcat{P}}_{\varcat{I}},\vartwocat{K})
    \to
    \Fun(\Span(\varcat{C},\varcat{E}),\vartwocat{K})
    \]
    is fully faithful, and its essential image is spanned by the $\varcat{I}$-left adjointed and $\varcat{P}$-right adjointed functors.
\end{main}

In \Cref{subsec:lax_monoidal}, we deduce a lax symmetric monoidal version of this theorem:

\begin{main}\label{main_theorem:lax_monoidal_free_adjointed_2_cat}
    Let $\varcat{C}$ be an $\infty$-category with finite products, and let $\varcat{P},\varcat{I}\subset\varcat{E}\subset\varcat{C}$ be wide subcategories stable under base change in $\varcat{C}$. Assume that $\varcat{P}$ and $\varcat{I}$ are left cancellable and that their morphisms are truncated. Then the canonical functor
    \[
    F\colon\Span(\varcat{C},\varcat{E})\to\spantwo(\varcat{C},\varcat{E})^{\varcat{P}}_{\varcat{I}}
    \]
    is the free lax symmetric monoidal $\varcat{I}$-left adjointed and $\varcat{P}$-right adjointed functor. That is, for any symmetric monoidal $(\infty,2)$-category $\vartwocat{K}$, the restriction functor
    \[
    \Fun^{\lax-\otimes}(\spantwo(\varcat{C},\varcat{E})^{\varcat{P}}_{\varcat{I}},\vartwocat{K})
    \to
    \Fun^{\lax-\otimes}(\Span(\varcat{C},\varcat{E}),\vartwocat{K})
    \]
    is fully faithful, and its essential image is spanned by the $\varcat{I}$-left adjointed and $\varcat{P}$-right adjointed lax symmetric monoidal functors.
\end{main}

\Cref{main_theorem:lax_monoidal_free_adjointed_2_cat} in particular implies \Cref{Mann_conjecture} (see \Cref{cor:Mann_conj}).

\subsection{Outline of the proof}

Let $\varcat{C}$ be an $\infty$-category, and let $\varcat{E}\subset\varcat{C}$ be a wide subcategory stable under base change in $\varcat{C}$. In the rest of this introduction, we will assume for simplicity that $\varcat{E}$ is also left cancellable, but we do not need this assumption in the body of the paper. This simplifies the definitions considerably.

\subsubsection*{The role of internal higher category theory}

Let us first briefly recall the proof strategy of \cite{CLLuniv}. Let $\varcat{E}=\varcat{P}\circ\varcat{I}$ be a suitable decomposition. That is, we assume that $\varcat{P},\varcat{I}\subset\varcat{E}$ are wide subcategories stable under base change in $\varcat{C}$ and left cancellable, that every morphism in $\varcat{P}\cap\varcat{I}$ is truncated, and that $\varcat{E}=\varcat{P}\circ\varcat{I}$.

In \cite{CLLuniv}, the first step is the ``recognition principle'' (see \cite[Theorem 2.18]{CLLuniv}), which reduces their main theorem \cite[Theorem A]{CLLuniv} to a universal property of the functors
\[
h_A\colon\varcat{C}^{\op}\to\catofcats,
\qquad
X\mapsto\Fun_{\spantwo(\varcat{C},\varcat{E})^{\varcat{P}}_{\varcat{I}}}(A,X)
\]
for $A\in\varcat{C}$. We now want to explain this universal property.

Functors in $\Fun(\varcat{C}^{\op},\catofcats)$ are also called ``$\varcat{C}$-parametrized categories'' (see \cite{barwick2016parametrizedhighercategorytheory}). These are a special case of the notion of \emph{internal} categories developed by Martini--Wolf in their series of papers \cite{MInternalYoneda}, \cite{MInternalStraightenning}, \cite{MWcocomplete}, and \cite{MWPresentabilityAndTopoi}. Let $\vartopos[C]=\presh(\varcat{C})$. Then there is an equivalence of categories
\begin{equation}\label{eq:equivalence_between_parametrized_and_internal}
    \Fun(\varcat{C}^{\op},\catofcats)
    \simeq
    \Fun^R(\vartopos[C]^{\op},\catofcats)
\end{equation}
between $\varcat{C}$-parametrized categories and $\vartopos[C]$-categories in the sense of internal category theory.

Let $\vartopos$ be an $\infty$-topos. Martini--Wolf define $\vartopos$-categories and notions of $\vartopos$-limits and colimits, the $\vartopos$-Yoneda embedding, left and right fibrations, and cartesian and cocartesian fibrations. They also establish the $\vartopos$-straightening and unstraightening equivalences and develop an internal theory of $\vartopos$-presentable $\vartopos$-categories.

Let $\varintcat{D}$ be a $\vartopos$-category. An object $A\in\vartopos$ is called a context, and an object of $\varintcat{D}(A)$ is called an object of $\varintcat{D}$ defined in the context $A$.

Every $A\in\vartopos[C]$ defines a $\vartopos[C]$-category (more specifically, a $\vartopos[C]$-groupoid) by the composition
\[
\vartopos[C]^{\op}\xto{\yo(A)}\catofanima\into\catofcats.
\]
We will denote this $\vartopos[C]$-category by $A$ as well. For $A\in\varcat{C}$, we also write $A$ for its Yoneda image in $\vartopos[C]$.

Let $\spanEPI[A]$ denote the $\vartopos[C]$-category corresponding to $h_A$ under the equivalence \eqref{eq:equivalence_between_parametrized_and_internal}. By \cite[Theorem 3.7 and the proof of Theorem 4.20]{CLLuniv}, $\spanEPI[A]$ is the free $\vartopos[C]$-category on $A$ with $\varcat{P}$-limits and $\varcat{I}$-colimits such that $\varcat{P}$-limits and $\varcat{I}$-colimits commute. That is, if $\varintcat{D}$ is a $\vartopos[C]$-category with $\varcat{P}$-limits and $\varcat{I}$-colimits that commute, then there is an equivalence between the category of $\vartopos[C]$-functors
\[
A\to\varintcat{D}
\]
and the category of $\vartopos[C]$-functors
\[
\spanEPI[A]\to\varintcat{D}
\]
that preserve $\varcat{P}$-limits and $\varcat{I}$-colimits.

\subsubsection*{Internal higher algebra}

Whereas the goal in \cite{CLLuniv} was to endow $\spantwo(\varcat{C},\varcat{E})^{\varcat{P}}_{\varcat{I}}$ with a universal property under $\varcat{C}^{\op}$, our goal is to endow it with a universal property under $\Span(\varcat{C},\varcat{E})$. To do so, we interpret functors
\[
\Span(\varcat{C},\varcat{E})\to\catofcats,
\]
also known as $2$-functor formalisms, as $\vartopos[C]$-categories endowed with a certain \emph{algebraic structure}, which we now explain.

In \cite{HA}, Lurie develops the modern theory of higher algebra, including the theory of symmetric monoidal categories and operads.

In modern accounts of higher algebra (see, for example, \cite{barkan2022envelopes}), the categories $\Fin$ and $\Span(\Fin)$ play pivotal roles. For example, a symmetric monoidal category can be defined as a finite-product-preserving functor
\[
\Span(\Fin)\to\catofcats.
\]

Let $\vartopos$ be an $\infty$-topos. Let $\universe$ be the $\vartopos$-category defined by the formula
\[
\universe(A):=\vartopos_{/A}.
\]
$\universe$ is called the universe, and it plays the role of $\catofanima$ in the theory of $\vartopos$-categories.

An input to our theory is a certain subcategory $\subuniverse\subset\universe$ playing the role of $\Fin$ (see \Cref{def:storngly_regular_universe}). We call such an $\subuniverse$ a context free subuniverse.

$\subuniverse$ is determined by a wide subcategory $\varcat{\mathcal{E}}\subset\vartopos$ that is stable under base change in $\vartopos$ and local on the target.
For simplicity, in this introduction we also assume that $\varcat{\mathcal{E}}$ is left cancellable, but we do not need or assume this in the body of the paper.
Given such an $\varcat{\mathcal{E}}$, we define
\[
\subuniverse(A):=\varcat{\mathcal{E}}_{/A}\subset\vartopos_{/A}.
\]

To define monoidal categories in this setup, we rely on the notion of $\subuniverse$-monoids from \cite{CLLambi}.
Let $\intspan(\subuniverse)$ be defined by applying $\Span$ contextwise to $\subuniverse$. We define an $\subuniverse$-monoidal category to be a functor of $\vartopos$-categories
\[
\intspan(\subuniverse)\to\internalcatofcats
\]
preserving $\subuniverse$-limits, mimicking the above definition of symmetric monoidal categories.

This captures the notion of $2$-functor formalisms in the following sense. Let $\vartopos[C]=\presh(\varcat{C})$ as above. Then we can take $\varcat{\mathcal{E}}\subset\vartopos[C]$ to be the wide subcategory of morphisms in $\vartopos[C]$ that are representable in $\varcat{E}$. We take $\subuniverse\subset\universe$ to be the corresponding context free subuniverse.

\begin{theorem}[\Cref{thm:equiv_between_2ff_and_E_monoidal}]
\label{thm:equiv_between_2ff_and_E_monoidal_intro}
    There is an equivalence of $\infty$-categories
    \[
    \Fun(\Span(\varcat{C},\varcat{E}),\catofcats)
    \simeq
    \intsmon(\internalcatofcats_{\vartopos[C]})(*)
    \]
    between the $\infty$-category of $2$-functor formalisms and the $\infty$-category of $\subuniverse$-monoidal $\vartopos[C]$-categories.
\end{theorem}

\subsubsection*{Cartesian and cocartesian $\subuniverse$-monoidal categories}

Let $\varcat{C}$ be a symmetric monoidal category with unit $1\in\varcat{C}$. Suppose that $\varcat{C}$ admits finite coproducts. Then we have a canonical morphism $\emptyset\to 1$ from the initial object to the unit. If this morphism is an isomorphism, then we can construct a canonical morphism $x\coprod y\to x\otimes y$ for every $x,y\in\varcat{C}$. If this morphism is an equivalence for every $x,y$, then the symmetric monoidal structure on $\varcat{C}$ is cocartesian.

In \cite{HarpazAmbi}, Yonatan Harpaz considered a related procedure. Let $\varcat{C}$ be a category with limits and colimits indexed by $\pi$-finite anima. Then there is a sequence of inductively defined norm maps which measures the difference between the colimits and limits indexed by these anima. A category is called $\infty$-semiadditive if all these norm maps are isomorphisms.
The induction is on the truncation level: the first map is a morphism $\emptyset\to *$ from the initial to the terminal object, and the second compares finite coproducts with cartesian products, similar to what we discussed earlier.
Harpaz's result was later generalized to the internal setting in \cite{CLLambi}, from which we drew inspiration.

Let $\vartopos$ be an $\infty$-topos, let $\subuniverse$ be a context free subuniverse, and let $\varintcat{C}$ be a $\vartopos$-category. If $\varintcat{C}$ admits $\subuniverse$-colimits, we can define a canonical \emph{cocartesian} $\subuniverse$-monoidal structure on $\varintcat{C}$. If the corresponding wide subcategory $\varcat{\mathcal{E}}\subset\vartopos$ is left cancellable and every morphism in $\varcat{\mathcal{E}}$ is locally truncated on the target, then we show (\Cref{thm:uniqueness_of_cartesian_structures}) that $\subuniverse$-cocomplete categories with $\subuniverse$-cocontinuous functors embed fully faithfully into $\subuniverse$-monoidal categories by forming the cocartesian structure.

Suppose that $\subuniverse[I],\subuniverse[P]\subset\subuniverse$ are two context free subuniverses corresponding to truncated, left cancellable wide subcategories $\varcat{\mathcal{I}},\varcat{\mathcal{P}}\subset\vartopos$ as above.
Then we can define the full subcategory of $\subuniverse$-monoidal categories which are $\subuniverse[I]$-cocartesian and $\subuniverse[P]$-cartesian (see \Cref{def:P_I_ambi_E_monoidal_cats}).

We are ready to explain the main idea of our proof.

Let $\varcat{C}$ be an $\infty$-category, and let $\varcat{E}\subset\varcat{C}$ be a wide subcategory stable under base change in $\varcat{C}$.

Let $\varcat{P},\varcat{I}\subset\varcat{E}$ be two smaller wide subcategories, both stable under base change in $\varcat{C}$.
We further assume that every morphism in $\varcat{P}$ and $\varcat{I}$ is truncated and that $\varcat{P}$ and $\varcat{I}$ are left cancellable.
Let $\vartopos=\vartopos[C]=\presh(\varcat{C})$, and let $\subuniverse$ be as in \Cref{thm:equiv_between_2ff_and_E_monoidal_intro}. Let $\subuniverse[P],\subuniverse[I]$ be the context free subuniverses corresponding to $\varcat{P},\varcat{I}$ in the same way that $\subuniverse$ corresponds to $\varcat{E}$.

For $A\in\varcat{C}$, consider
\[
h_A^{\varcat{E}}\colon\Span(\varcat{C},\varcat{E})\to\catofcats,
\qquad
X\mapsto\Fun_{\spantwo(\varcat{C},\varcat{E})^{\varcat{P}}_{\varcat{I}}}(A,X).
\]
Then $h_A^{\varcat{E}}$ corresponds to an $\subuniverse$-monoidal category $\spanEPI[A]$ under the equivalence of \Cref{thm:equiv_between_2ff_and_E_monoidal_intro}.

We prove the following theorem under the extra assumption that there is a single integer $n\geq -2$ such that every morphism in $\varcat{I}$ and $\varcat{P}$ is $n$-truncated.

\begin{main}[\Cref{thm:generalized_main_C}]
\label{main_theorem:free_E_monoidal_semiadditive}
    $\spanEPI[A]$ is the free $\subuniverse$-monoidal $\subuniverse[I]$-cocartesian $\subuniverse[P]$-cartesian category generated by $A\in\vartopos$.

    That is, for any $\subuniverse[P]$-cartesian, $\subuniverse[I]$-cocartesian $\subuniverse$-monoidal category $\varintcat{D}$, the canonical functor
    \[
    \intfun^{\subuniverse-\otimes}(\spanEPI[A],\varintcat{D})
    \xto{\sim}
    \varintcat{D}^A
    \]
    is an equivalence of $\vartopos$-categories.
\end{main}

Then, by the equivalence between $\subuniverse$-monoidal categories and $2$-functor formalisms, we get \Cref{main_theorem:free_adjointed_2_cat} under the extra assumption that every morphism in $\varcat{I}$ and $\varcat{P}$ is $n$-truncated for a fixed $n$. We deduce \Cref{main_theorem:free_adjointed_2_cat} in the general case by taking an appropriate colimit.

\subsection{Linear overview}

In \Cref{sec:recognition_principle}, we show that a $2$-category satisfying the universal property described in \Cref{main_theorem:free_adjointed_2_cat} exists. Then we prove a recognition principle which reduces \Cref{main_theorem:free_adjointed_2_cat} to showing that the functors $h_A^{\varcat{E}}$ above have the required universal property as $2$-functor formalisms.

In \Cref{sec:E_monoidal}, we develop the theory of $\subuniverse$-monoidal categories and $\subuniverse$-operads.

In \Cref{sec:2_cats}, we develop some internal $2$-category theory to build the machinery needed in the next section.

In \Cref{sec:proof_of_C}, we prove \Cref{main_theorem:free_E_monoidal_semiadditive}, which is the technical heart of this article.

In \Cref{sec:proof_of_main}, we prove \Cref{thm:equiv_between_2ff_and_E_monoidal_intro} and deduce \Cref{main_theorem:free_adjointed_2_cat} from it and from \Cref{main_theorem:free_E_monoidal_semiadditive}. Namely, we explain why, under the equivalence between $2$-functor formalisms and $\subuniverse$-monoidal categories, the universal property of $\spanEPI[A]$ gives the required universal property of $h_A^{\varcat{E}}$ as a $2$-functor formalism. We then deduce \Cref{main_theorem:lax_monoidal_free_adjointed_2_cat} from \Cref{main_theorem:free_adjointed_2_cat} by an argument very similar to the deduction of \cite[Theorem B]{CLLuniv} from \cite[Theorem A]{CLLuniv}.

\subsection{Further results}

\subsubsection*{$\subuniverse$-operads}

In addition to $\subuniverse$-monoidal categories, in \Cref{subsec:E_operads} we define the $\vartopos$-category $\intoperads$ of $\subuniverse$-operads, generalizing the classical notion from \cite{HA}. We also prove the existence of an adjunction of $\vartopos$-categories
\[
\intenv\colon\intoperads\fromto\intsmon(\internalcatofcats)\colon(-)^\otimes
\]
between the $\vartopos$-category of $\subuniverse$-operads and the $\vartopos$-category of $\subuniverse$-monoidal categories (see \Cref{thm:env_associated_operad_adjunction}).

The counit of this adjunction, $\intenv(\varintcat{C}^\otimes)\to\varintcat{C}$, is used in the construction of the $\subuniverse$-monoidal functor
\[
\spanEPI[A]\to\varintcat{C}
\]
which exhibits the universal property of $\spanEPI[A]$.

\subsubsection*{The symmetric monoidal category of $\subuniverse$-monoidal categories}

In classical (non-internal) higher algebra, the category $\cmon(\catofcats)$ carries an auxiliary symmetric monoidal structure defined by a localization of the Day convolution structure on $\Fun(\Span(\Fin),\catofcats)$.
This symmetric monoidal structure is neither cartesian nor cocartesian. It was defined in greater generality by Ben-Moshe in \cite[\S 4]{Ben_Moshe_ambi_K}. Commutative algebras in $\cmon(\catofcats)$ are $2$-semirings. That is, they correspond to categories $\varcat{C}$ with two symmetric monoidal structures $\otimes,\oplus$ such that $\otimes$ distributes over $\oplus$.

In \Cref{subsec:symmetric_monoidal_cat_of_E_monoidal_cats}, we construct a similar symmetric monoidal structure on the $\vartopos$-category $\intsmon(\internalcatofcats)$. For the comparison with $2$-functor formalisms, assume that $\varcat{C}$ has finite products. In fact, we show that the equivalence of \Cref{thm:equiv_between_2ff_and_E_monoidal_intro} can be upgraded to a symmetric monoidal equivalence between $\Fun(\Span(\varcat{C},\varcat{E}),\catofcats)$ equipped with Day convolution and $\intsmon(\internalcatofcats)(*)$ equipped with a localization of internal Day convolution.

Taking $\operatorname{CAlg}$ and using that commutative algebras for Day convolution are equivalent to lax symmetric monoidal functors, we obtain an equivalence
\[
\Fun^{\lax-\otimes}(\Span(\varcat{C},\varcat{E}),\catofcats)
\simeq
\operatorname{CAlg}(\intsmon(\internalcatofcats)(*)),
\]
giving an equivalent and completely internal definition of $3$-functor formalisms.

In fact, we can say more. Suppose we are in the setting of \Cref{main_theorem:lax_monoidal_free_adjointed_2_cat}.
Combining \Cref{cor:P_I_ambi_is_a_smashing_localization} with the identification
\[
L_{(\subuniverse[P],\subuniverse[I])-\oplus}
(\subuniverse^{\simeq})
\simeq\spanEPI
\]
established at the end of \Cref{subsec:lax_monoidal}, we obtain the following theorem.

\begin{theorem}
    The category $\spanEPI$ is an idempotent algebra in $\intsmon(\internalcatofcats)(*)$. Moreover, the modules over $\spanEPI$ are the $\subuniverse[P]$-cartesian and $\subuniverse[I]$-cocartesian $\subuniverse$-monoidal categories.
\end{theorem}

Let $*\in\varcat{C}$ denote the terminal object. By the equivalence of symmetric monoidal categories between $\Fun(\Span(\varcat{C},\varcat{E}),\catofcats)$ and $\intsmon(\internalcatofcats)(*)$, we get that the functor
\[
h_*\colon\Span(\varcat{C},\varcat{E})\to\catofcats,
\qquad
B\mapsto\Fun_{\spantwo(\varcat{C},\varcat{E})^{\varcat{P}}_{\varcat{I}}}(*,B),
\]
which corresponds to $\spanEPI$, is an idempotent algebra whose modules are the $\varcat{P}$-right adjointed and $\varcat{I}$-left adjointed $2$-functor formalisms.

\subsection{Future work and applications}

\subsubsection*{Application to the derived mod $p$ Hecke algebra}

This work emerged from thinking about the derived mod $p$ Hecke algebra. Mann and Heyer \cite{just_6ff} construct a six-functor formalism on the category of condensed anima $\operatorname{Cond}(\catofanima)$
\[
D\colon\Span(\operatorname{Cond}(\catofanima),\varcat{\mathcal{E}})\to\catofcats,
\]
where $\varcat{\mathcal{E}}$ is a wide subcategory defined by an inductive sheafification procedure.
For a suitable $p$-adic reductive group $G$, there is a condensed anima $BG$ such that $D(BG)$ is equivalent to the category of smooth $\bar{\mathbb{F}}_p$-representations of $G$. Mann and Heyer \cite[proof of Proposition 5.5.4]{just_6ff} describe the derived mod $p$ Hecke algebra $\mathcal{H}$ in terms of the six-functor formalism $D$ and use this description to lift the anti-involution of $\pi_*(\mathcal{H})$ to a coherent anti-involution of $\mathcal{H}$.

In work in progress, we apply \Cref{main_theorem:lax_monoidal_free_adjointed_2_cat} to $D$ to construct a coherent filtration on $\mathcal{H}$ lifting the length filtration on $\pi_*(\mathcal{H})$. The fact that $D$ is not constructed from a suitable decomposition naturally led us to consider these questions.

\subsubsection*{Norms on $\operatorname{MS}$}

Apart from applications to six-functor formalisms, the theory of $\subuniverse$-monoidal categories is a natural framework for normed structures.

Bachmann--Hoyois \cite{BHNormsInMotivic} define a norm structure on Voevodsky's stable homotopy category $\operatorname{SH}$. In work in progress, the third author uses our theory of internal higher algebra to lift this norm structure to the category of motivic spectra $\operatorname{MS}$ developed by Annala--Iwasa \cite{annala2025motivicspectrauniversalityktheory}.

\subsection{Conventions}

Throughout this paper, ``category'' means $\infty$-category, and ``$2$-category'' means $(\infty,2)$-category.

\begin{itemize}
    \item We use calligraphic capital letters to denote $\infty$-topoi and wide subcategories thereof, for example
    \[
    \vartopos,\varcat{\mathcal{E}}.
    \]

    \item We use bold serif fonts to denote named $(\infty,2)$-categories, for example
    \[
    \spantwo,\spanhalf,\twocatofcats.
    \]

    \item We use blackboard bold capital letters to denote general $(\infty,2)$-categories, for example
    \[
    \vartwocat{K},\vartwocat{C}.
    \]

    \item We use upright serif fonts to denote named $(\infty,1)$-categories, for example
    \[
    \Span,\catofcats.
    \]

    \item We use ordinary italic capital letters to denote general $(\infty,1)$-categories, for example
    \[
    \varcat{C},\varcat{D}.
    \]

    \item We use sans-serif fonts to denote both named and general $\vartopos$-categories, for example
    \[
    \intspan,\internalcatofcats,\universe,\varintcat{C},\varintcat{D}.
    \]

    \item We use bold sans-serif fonts to denote named $\vartopos$-sheaves of $(\infty,2)$-categories, for example
    \[
    \intspantwo,\intspanhalf.
    \]
\end{itemize}

\subsection{Acknowledgments}

We would like to thank Lucas Mann for suggesting his conjecture and for interesting correspondence.
We would like to thank Sebastian Wolf for helpful conversations.
We would like to thank Eitan Sayag and Avraham Aizenbud for useful suggestions.
We would like to thank Bastiaan Cnossen for valuable correspondence and for reading and commenting on an earlier draft of this work.
We would like to thank Yuval Lotenberg, Amos Kaminski, Avital Binyamin and all other members of our research group at the Weizmann Institute of Science for camaraderie and interesting conversations.
Some elements of this work were inspired by discussions with Lior Yanovsky and Beckham Myers regarding idempotent algebras in the category of symmetric monoidal categories.

This research is partially supported by ISF Beresheet grant 4093/25, BSF grant 2024766, Minerva grant 715294, and the Azrieli Foundation.

\subsection{Statement on the use of generative AI}

We used ChatGPT Astra (OpenAI) to assist with checking mathematical arguments and proofreading this work. In particular, it provided substantial assistance in correcting the proofs of \Cref{prop:calg_has_E_colimits}, \Cref{prop:description_of_oplax_slice}, and \Cref{prop:preoperad_premonoidal_adjunction}. The authors take full responsibility for the final statements and proofs.

\section{Recognition principle over \texorpdfstring{$\Span(\varcat{C},\varcat{E})$}{Span(C,E)}.}\label{sec:recognition_principle}

In this section, we prove a recognition principle which reduces our main theorem (\Cref{main_theorem:free_adjointed_2_cat}) to a universal property of the functors
\[
h_A\colon\Span(\varcat{C},\varcat{E})\longrightarrow\catofcats,
\qquad
X\longmapsto
\Fun_{\spantwo(\varcat{C},\varcat{E})^{\varcat{P}}_{\varcat{I}}}(A,X),
\]
for $A\in\varcat{C}$, in the $\infty$-category $\Fun(\Span(\varcat{C},\varcat{E}),\catofcats)$.

In \Cref{subsec:adjointability_and_adjointedness}, we define $(\varcat{P},\varcat{I})$-adjointedness for functors from $\Span(\varcat{C},\varcat{E})$ to an $(\infty,2)$-category (\Cref{def:adjointed_functors}). The definition uses norm maps constructed inductively on the truncation level.

In \Cref{subsec:existence_of_universal_adjointed_functor}, we show that a free $(\varcat{P},\varcat{I})$-adjointed functor exists for abstract reasons (\Cref{cor:existence_of_universal_adjointed_functor}). We then show that its universal property also holds at the $(\infty,2)$-categorical level, i.e., for mapping $\infty$-\emph{categories}  (\Cref{prop:enriched_universal_property}).

In \Cref{subsec:recognition_principle}, we prove the recognition principle (\Cref{recognition_principle}), following the argument of Cnossen--Linskens--Lenz \cite{CLLuniv}: A $(\varcat{P},\varcat{I})$-adjointed functor $h\colon\Span(\varcat{C},\varcat{E})\to\vartwocat{V}$, with $\vartwocat{V}$ an $(\infty,2)$-category, is universal precisely when it is essentially surjective and each functor $\Fun_{\vartwocat{V}}(h(a),h(-))$ is free on the identity of $h(a)$. In \Cref{subsec:reduction_of_main_thm}, we apply this criterion to the canonical functor into $\spantwo(\varcat{C},\varcat{E})^{\varcat{P}}_{\varcat{I}}$ (\Cref{prop:reduction_of_main_theorem}). We describe its mapping $\infty$-categories as $\infty$-categories of spans and formulate the universal property that remains to be proved (\Cref{thm:red_of_main}). This universal property will be established in \Cref{subsec:2_FF_are_just_E_monoidal}.

\subsection{Adjointability and adjointedness}\label{subsec:adjointability_and_adjointedness}

In this subsection, we define $(\varcat{P},\varcat{I})$-adjointedness for functors from $\Span(\varcat{C},\varcat{E})$ to an $(\infty,2)$-category (\Cref{def:adjointed_functors}). We begin with naive adjointability and then define adjointedness using norm maps constructed inductively on the truncation level. We introduce the $\infty$-category of $(\varcat{P},\varcat{I})$-adjointed functors and show that adjointedness is preserved under postcomposition by $2$-functors (\Cref{cor:upward_close_lemma}).

\subsubsection*{Naive adjointability}

First, we need to describe the situation in which we can define $(P,I)$-adjointedness.

\begin{para}{\bf Setup}
    \begin{mylist}
        \item Let $\varcat{C}$ be an $\infty$-category.
        \item Let $\varcat{I},\varcat{P},\varcat{E}$ be wide sub-$\infty$-categories of $\varcat{C}$ that are stable under base change in $\varcat{C}$. Here, a wide sub-$\infty$-category $\varcat{D}\subset\varcat{C}$ is stable under base change in $\varcat{C}$ if pullbacks of morphisms in $\varcat{D}$ along arbitrary morphisms in $\varcat{C}$ exist and again belong to $\varcat{D}$.
        \item Assume that $\varcat{I},\varcat{P}$ are left cancellable and that every morphism in $\varcat{I}$ and $\varcat{P}$ is truncated.
        \item Assume that $\varcat{I},\varcat{P}\subset\varcat{E}$.
        \item Let $\Span(\varcat{C},\varcat{E})$ be the span $\infty$-category.
        \item For a functor $F:\Span(\varcat{C},\varcat{E})\to\catofkernels$ to an $(\infty,2)$-category and a map $f$ in $\varcat{C}$, we denote its left-pointing image by $f^*$ and, if $f\in\varcat{E}$, its right-pointing image by $f_!$.
    \end{mylist}
\end{para}

For a functor $\Span(\varcat{C},\varcat{E})\to\catofkernels$ to an $(\infty,2)$-category to be $(\varcat{P},\varcat{I})$-adjointed, we need to assume that the relevant morphisms $f^*$ admit adjoints. The following definition captures this.

\begin{definition}
    Let $F:\Span(\varcat{C},\varcat{E})\to \catofkernels$ be a functor to an $(\infty,2)$-category. It is called $\varcat{I}$-naively left adjointable if $i^*$ has a left adjoint $i_\sharp$ for every map $i$ in $\varcat{I}$. It is called $\varcat{P}$-naively right adjointable if $p^*$ has a right adjoint $p_*$ for every map $p$ in $\varcat{P}$.
\end{definition}

The word naively is explained by the following remark.

\begin{remark}
    One can compare this with the definition in \cite{CLLuniv}, where a functor $\varcat{C}^{\op}\to \catofkernels$ is defined to be adjointable if it is naively adjointable and also satisfies some Beck--Chevalley conditions. We circumvent the need to consider the Beck--Chevalley conditions by considering functors from $\Span(\varcat{C},\varcat{E})$ together with an adjointedness property that will be defined below.
\end{remark}

\subsubsection*{Adjointed functors}

We are now ready to define adjointed functors.

\begin{definition}{\bf Adjointed functors}\label{def:adjointed_functors}
    \begin{mylist}
        \item Let $F:\Span(\varcat{C},\varcat{E})\to \catofkernels$ be a functor to an $(\infty,2)$-category $\catofkernels$.
        \item We say that this functor is $\varcat{P}^{(-2)}$-right adjointed if it is $\varcat{P}$-naively right adjointable.
        \item Let $F$ be a $\varcat{P}^{(n)}$-right adjointed functor. Let $p$ be an $(n+1)$-truncated map in $\varcat{P}$. We define a morphism $N_p:p_{!}\to p_*$ from the right-pointing image of $p$ to the right adjoint of the left-pointing image of $p$. We then define a $\varcat{P}^{(n+1)}$-right adjointed functor to be a $\varcat{P}^{(n)}$-right adjointed functor for which these maps are isomorphisms.
        \item If $F$ is a $\varcat{P}^{(n)}$-right adjointed functor and $p$ is in $\varcat{P}^{(n)}$, then we have an adjunction $p^*\dashv p_!$ coming from the identification $N_p:p_!\simeq p_*$.
        \item Let $F$ be a $\varcat{P}^{(n)}$-right adjointed functor, and let $p:A\to B$ be an $(n+1)$-truncated morphism in $\varcat{P}$. Consider the diagram:
        \[
        \begin{tikzcd}
            A \\
            & {A\times_BA} & A \\
            & A & B.
            \arrow["\Delta"', from=1-1, to=2-2]
            \arrow[curve={height=-6pt}, equal, from=1-1, to=2-3]
            \arrow[curve={height=6pt}, equal, from=1-1, to=3-2]
            \arrow["{\pi_2}", from=2-2, to=2-3]
            \arrow["{\pi_1}"', from=2-2, to=3-2]
            \arrow["p", from=2-3, to=3-3]
            \arrow["p"', from=3-2, to=3-3]
        \end{tikzcd}
        \]
        Notice that $\Delta$ is an $n$-truncated map in $\varcat{P}$, and so by induction $\Delta_{!}$ is the right adjoint of $\Delta^*$.

        Consider now the following map:
        \[
        p^*p_!\simeq\pi_{1!}\pi_2^*
        \to\pi_{1!}\Delta_!\Delta^*\pi_2^*
        \simeq\id_!\id^*=\id.
        \]
        The first map is part of the functoriality of $F$.
        The second map is given by the unit of the adjunction $\Delta^*\dashv \Delta_!$. This gives the mate $N_p^F:p_{!}\to p_{*}$ as desired. We call this map the norm map of $p$. If $F$ is clear from the context, we write $N_p:=N_p^F$.
        \item We say that $F$ is $\varcat{P}$-right adjointed if it is $\varcat{P}^{(n)}$-right adjointed for all $n$.
        \item We say that $F$ is $\varcat{I}^{(n)}$-left adjointed (resp. $\varcat{I}$-left adjointed) if the corresponding functor $F^{\co}:\Span(\varcat{C},\varcat{E})\to\catofkernels^{\co}$ is $\varcat{I}^{(n)}$-right adjointed (resp. $\varcat{I}$-right adjointed).
    \end{mylist}
\end{definition}

We are interested in the $\infty$-category of $(\infty,2)$-categories equipped with a functor from $\Span(\varcat{C},\varcat{E})$.

\begin{notation}
    By $\catoftwocats$ we denote the $\infty$-category of $(\infty,2)$-categories. We let $\catoftwocats_{\Span(\varcat{C},\varcat{E})/}$ be the slice in which we view the $\infty$-category $\Span(\varcat{C},\varcat{E})$ as an $(\infty,2)$-category via the canonical embedding $\catofcats\into\catoftwocats$.
\end{notation}

The $(\varcat{P},\varcat{I})$-adjointed functors naturally form a particularly nice sub-$\infty$-category.

\begin{definition}
    We define the $\infty$-category of $(\varcat{P},\varcat{I})$-adjointed functors to be the full sub-$\infty$-category of $\catoftwocats_{\Span(\varcat{C},\varcat{E})/}$ spanned by functors which are $\varcat{I}$-left adjointed and $\varcat{P}$-right adjointed, and denote it by $\catoftwocats_{\Span(\varcat{C},\varcat{E})/}^{(\varcat{P},\varcat{I})\adjtd}$.
\end{definition}

\begin{remark}
    In \cite[Definition 2.4]{CLLuniv}, natural transformations between functors $\varcat{C}^{\op}\to\catofkernels$ are called $(\varcat{P},\varcat{I})$-adjointable if they satisfy a suitable Beck--Chevalley condition.
    The commuting squares arising from those conditions are subsumed in the fact that we care about functors from $\Span(\varcat{C},\varcat{E})$.
    In fact, one can show that the commutative squares arising from the functoriality of $F$ agree with the ones provided by the Beck--Chevalley transformation.
    However, we do not need this in our paper.
\end{remark}

\subsubsection*{Closure under postcomposition}

The following lemma shows that the norm maps are natural.

\begin{lemma}\label{lem:nat_of_norms}
    Let $F:\Span(\varcat{C},\varcat{E})\to\catofkernels$ be a functor, and let $G:\catofkernels\to\catofkernels'$ be a $2$-functor of $(\infty,2)$-categories. Suppose that both $F$ and $G\circ F$ are $\varcat{P}^{(n)}$-right adjointed. Let $p$ be an $(n+1)$-truncated map in $\varcat{P}$.

    Then we have a canonical homotopy $G(N_p^F)\simeq N_p^{G\circ F}$ of $2$-morphisms in $\catofkernels'$.
\end{lemma}

\begin{proof}
    We argue by induction on $n$. For $n=-2$, there is nothing to prove.

    Suppose that we know the result for $n$. Let $p_!^F,p_F^*$ denote the $1$-morphisms arising as the images of morphisms in $\Span(\varcat{C},\varcat{E})$ under the functor $F$.
    Let $p_*^F$ denote the right adjoint of $p_F^*$.
    Use analogous notation for $G\circ F$.

    $2$-functors preserve the formation of mates.
    Thus, to show that $G$ sends $N_p^F$ to $N_p^{G\circ F}$, it is enough to show that it sends the mate of $N_p^F$,
    \[
    p_F^*p_!^F\to\id,
    \]
    to the mate of $N_p^{G\circ F}$,
    \[
    p_{G\circ F}^*p_!^{G\circ F}\to\id.
    \]

    The mate $p_F^*p_!^F\to\id$ is defined as the following composition of $2$-morphisms:
    \[
    p_F^*p_!^F
    \simeq(\pi_1)_!^F(\pi_2)_F^*
    \xto{\eta}(\pi_1)_!^F\Delta_*^F\Delta_F^*(\pi_2)_F^*
    \xto{(N_{\Delta_p}^F)^{-1}}(\pi_1)_!^F\Delta_!^F\Delta_F^*(\pi_2)_F^*
    \simeq\id.
    \]
    The first isomorphism comes from a $2$-isomorphism in $\Span(\varcat{C},\varcat{E})$. The second morphism is a unit of an adjunction in $\catofkernels$. The third is the inverse of $N_{\Delta_p}^F$. The fourth comes from a $2$-isomorphism in $\Span(\varcat{C},\varcat{E})$.

    Any $2$-isomorphism from $\Span(\varcat{C},\varcat{E})$ is sent by $G$ to a $2$-isomorphism defined in the same way in $\catofkernels'$. Units of adjunctions are preserved by $2$-functors. By the induction hypothesis, $(N_{\Delta_p}^F)^{-1}$ is sent to $(N_{\Delta_p}^{G\circ F})^{-1}$. Thus $G$ sends the mate of $N_p^F$ to the mate of $N_p^{G\circ F}$.
\end{proof}

Hence, if $F$ is adjointed, so is $G\circ F$:

\begin{corollary}\label{cor:upward_close_lemma}
    Let $F:\Span(\varcat{C},\varcat{E})\to\catofkernels$ be a $(\varcat{P},\varcat{I})$-adjointed functor, and let $G:\catofkernels\to\catofkernels'$ be a $2$-functor of $(\infty,2)$-categories. Then $G\circ F$ is also $(\varcat{P},\varcat{I})$-adjointed.
\end{corollary}

\begin{proof}
    This follows from the naturality of the norm maps and the fact that $2$-functors send invertible $2$-morphisms to invertible $2$-morphisms.
\end{proof}

\subsection{Existence of the universal adjointed functor}
\label{subsec:existence_of_universal_adjointed_functor}

In this subsection, we show that a free $(\varcat{P},\varcat{I})$-adjointed functor exists (\Cref{cor:existence_of_universal_adjointed_functor}). We first prove that adjointed functors are closed under nonempty colimits and limits. Using these closure properties and the reflection theorem of Ragimov and Schlank \cite{shauly_reflection}, we show that the inclusion of the $\infty$-category of adjointed functors admits a left adjoint (\Cref{prop:adjointed_reflection}). We then construct the universal adjointed functor and show that its universal property also holds at the level of mapping $\infty$-categories (\Cref{prop:enriched_universal_property}).

Let $\varcat{C}$ be an $\infty$-category, and let $\varcat{I},\varcat{P}\subset\varcat{E}\subset\varcat{C}$ be wide sub-$\infty$-categories that are stable under base change in $\varcat{C}$. Assume that $\varcat{P},\varcat{I}$ are truncated and left cancellable.

\subsubsection*{Closure under nonempty colimits}

We first show that adjointed functors are closed under nonempty colimits.

\begin{proposition}\label{prop:adjtd_preserves_nonempty_colimits}
    The inclusion
    \[
    \catoftwocats_{\Span(\varcat{C},\varcat{E})/}^{(\varcat{P},\varcat{I})\adjtd}
    \subset
    \catoftwocats_{\Span(\varcat{C},\varcat{E})/}
    \]
    preserves nonempty colimits.
\end{proposition}

\begin{proof}
    Let $D:\varcat{K}\to\catoftwocats_{\Span(\varcat{C},\varcat{E})/}^{(\varcat{P},\varcat{I})\adjtd}$ be a diagram indexed by a nonempty $\infty$-category $\varcat{K}$, and let $\vartwocat{M}$ be its colimit in $\catoftwocats_{\Span(\varcat{C},\varcat{E})/}$. For some $k\in\varcat{K}$, which exists by assumption, we have a factorization
    \[
    \Span(\varcat{C},\varcat{E})\to D(k)\to\vartwocat{M}.
    \]
    Since $D(k)$ is $(\varcat{P},\varcat{I})$-adjointed, \Cref{cor:upward_close_lemma} shows that $\vartwocat{M}$ is also $(\varcat{P},\varcat{I})$-adjointed.
\end{proof}

\begin{remark}
    Notice that the condition that the colimit be nonempty is vital. The inclusion does not preserve the empty colimit unless $\varcat{P},\varcat{I}$ consist only of isomorphisms. Indeed, preservation would imply, by \Cref{cor:upward_close_lemma}, that every $(\infty,2)$-category equipped with a functor from $\Span(\varcat{C},\varcat{E})$ is adjointed, which is false.
\end{remark}

\subsubsection*{Closure under limits}

The following lemma shows that adjoint morphisms in objects of $\catoftwocats$ pass to limits.

\begin{lemma}\label{lem:adj_closed_under_limits}
    Let $D:\varcat{K}\to\catoftwocats$ be a diagram indexed by an $\infty$-category $\varcat{K}$. Suppose that we have a cone $D':\varcat{K}^{\triangleleft}\to\catoftwocats$ such that $D'(-\infty)=\Delta^1$ and, for each $k\in\varcat{K}$, the functor $\Delta^1\to D'(k)$ picks a left adjoint in $D'(k)$. Then the functor $\Delta^1\to\lim_{\varcat{K}}D$ picks a left adjoint in $\lim_{\varcat{K}}D$.
\end{lemma}

\begin{proof}
    Let $\vartwocat{D}$ be an $(\infty,2)$-category.
    Let $\mathbf{Adj}$ be the initial $(\infty,2)$-category with adjoint $1$-morphisms $L\dashv R$. There is a canonical functor $\Delta^1\to\mathbf{Adj}$ picking $L$.

    Recall that a $1$-morphism corresponding to a functor $\Delta^1\to\vartwocat{D}$ is a left adjoint in $\vartwocat{D}$ if and only if there exists a necessarily unique lift
    \[
    \begin{tikzcd}
        \Delta^1\ar[r]\ar[d]
        &\vartwocat{D}\\
        \mathbf{Adj}.\ar[ur, dotted]
    \end{tikzcd}
    \]
    Thus, since $2$-functors preserve adjunctions, the cone from $\Delta^1$ to $D$ factors through a cone from $\mathbf{Adj}$.
\end{proof}

Let us denote by $\catoftwocats_{\Span(\varcat{C},\varcat{E})/}^{(\varcat{P},\varcat{I})^{(n)}\adjtd}$ the full sub-$\infty$-category of $\catoftwocats_{\Span(\varcat{C},\varcat{E})/}$ spanned by functors which are $\varcat{I}^{(n)}$-left adjointed and $\varcat{P}^{(n)}$-right adjointed.

Consider $\catoftwocats_{\Span(\varcat{C},\varcat{E})/}^{(\varcat{P},\varcat{I})^{(-2)}\adjtd}$. This is the full sub-$\infty$-category of $(\varcat{P},\varcat{I})$-naively adjointable functors.

\begin{corollary}\label{cor:naively_adjbl_functors_are_closed_under_limits}
    $\catoftwocats_{\Span(\varcat{C},\varcat{E})/}^{(\varcat{P},\varcat{I})^{(-2)}\adjtd}$ is closed under limits.
\end{corollary}

\begin{proof}
    This is immediate from \Cref{lem:adj_closed_under_limits}.
\end{proof}

Next we show that adjointed functors are closed under all limits. We first need another lemma.

Let $\mathbf{O}^{\lax}$ be the $(\infty,2)$-category given by the following diagram:
\[
\begin{tikzcd}
    \bullet && \bullet.
    \arrow[""{name=0, anchor=center, inner sep=0}, shift right=2, curve={height=12pt}, from=1-1, to=1-3]
    \arrow[""{name=1, anchor=center, inner sep=0}, shift left=2, curve={height=-12pt}, from=1-1, to=1-3]
    \arrow[between={0.2}{0.8}, Rightarrow, nfold, from=1, to=0]
\end{tikzcd}
\]
This $(\infty,2)$-category classifies a $2$-morphism in the sense that a functor $\mathbf{O}^{\lax}\to\vartwocat{K}$ to an $(\infty,2)$-category amounts to choosing a $2$-morphism.
There is a canonical functor $\mathbf{O}^{\lax}\to\Delta^1$ choosing the identity $2$-morphism of the nonidentity arrow of $\Delta^1$. A functor $\mathbf{O}^{\lax}\to\vartwocat{K}$ classifies a $2$-isomorphism in $\vartwocat{K}$ if and only if there exists a necessarily unique factorization
\[
\begin{tikzcd}
    \mathbf{O}^{\lax}\ar[r]\ar[d]
    &\vartwocat{K}\\
    \Delta^1.\ar[ur, dotted]
\end{tikzcd}
\]

\begin{lemma}\label{lem:2_isos_closed_under_limits}
    Let $D:\varcat{K}\to\catoftwocats$ be a diagram indexed by an $\infty$-category $\varcat{K}$. Suppose that we have a cone $D':\varcat{K}^{\triangleleft}\to\catoftwocats$ such that $D'(-\infty)=\mathbf{O}^{\lax}$ and, for each $k\in\varcat{K}$, the functor $\mathbf{O}^{\lax}\to D'(k)$ picks a $2$-isomorphism in $D'(k)$. Then the corresponding functor $\mathbf{O}^{\lax}\to\lim_{\varcat{K}}D$ picks a $2$-isomorphism in $\lim_{\varcat{K}}D$.
\end{lemma}

\begin{proof}
    The proof is identical to the proof of \Cref{lem:adj_closed_under_limits}.
\end{proof}

\begin{proposition}\label{prop:adjtd_preserves_limits}
    The inclusion
    \[
    \catoftwocats_{\Span(\varcat{C},\varcat{E})/}^{(\varcat{P},\varcat{I})\adjtd}
    \subset
    \catoftwocats_{\Span(\varcat{C},\varcat{E})/}
    \]
    preserves all limits.
\end{proposition}

\begin{proof}
    We will prove by induction that each full sub-$\infty$-category
    \[
    \catoftwocats_{\Span(\varcat{C},\varcat{E})/}^{(\varcat{P},\varcat{I})^{(n)}\adjtd}
    \subset
    \catoftwocats_{\Span(\varcat{C},\varcat{E})/}^{(\varcat{P},\varcat{I})^{(n-1)}\adjtd}
    \]
    is closed under limits, and that $\catoftwocats_{\Span(\varcat{C},\varcat{E})/}^{(\varcat{P},\varcat{I})^{(-2)}\adjtd}$ is closed under limits in $\catoftwocats_{\Span(\varcat{C},\varcat{E})/}$.

    \Cref{cor:naively_adjbl_functors_are_closed_under_limits} provides the base of the induction.

    Now consider a diagram $D:\varcat{K}\to\catoftwocats_{\Span(\varcat{C},\varcat{E})/}^{(\varcat{P},\varcat{I})^{(n)}\adjtd}$. By induction, its limit in $\catoftwocats_{\Span(\varcat{C},\varcat{E})/}$ belongs to $\catoftwocats_{\Span(\varcat{C},\varcat{E})/}^{(\varcat{P},\varcat{I})^{(n-1)}\adjtd}$. For all objects of this full sub-$\infty$-category, we have naturally defined norm maps $i_\sharp\to i_!$ and $p_!\to p_*$ (see \Cref{lem:nat_of_norms}). Thus, for each $i\in\varcat{I}^{(n)}$ or $p\in\varcat{P}^{(n)}$, we get a cone $D':\varcat{K}^{\triangleleft}\to\catoftwocats$ such that $D'(-\infty)=\mathbf{O}^{\lax}$ and the components $\mathbf{O}^{\lax}\to D'(k)$ for $k\in\varcat{K}$ classify the norm maps of $i$ or $p$.

    By \Cref{lem:2_isos_closed_under_limits}, the limit $\lim_{\varcat{K}}D$ belongs to $\catoftwocats_{\Span(\varcat{C},\varcat{E})/}^{(\varcat{P},\varcat{I})^{(n)}\adjtd}$.

    We have proved that a limit in $\catoftwocats_{\Span(\varcat{C},\varcat{E})/}^{(\varcat{P},\varcat{I})^{(n-1)}\adjtd}$ of objects in $\catoftwocats_{\Span(\varcat{C},\varcat{E})/}^{(\varcat{P},\varcat{I})^{(n)}\adjtd}$ belongs to $\catoftwocats_{\Span(\varcat{C},\varcat{E})/}^{(\varcat{P},\varcat{I})^{(n)}\adjtd}$, and we are done.
\end{proof}

\subsubsection*{Abstract existence}

Hence, there is a left adjoint to the inclusion of adjointed functors.

\begin{proposition}\label{prop:adjointed_reflection}
    The inclusion
    \[
    \catoftwocats_{\Span(\varcat{C},\varcat{E})/}^{(\varcat{P},\varcat{I})\adjtd}
    \into
    \catoftwocats_{\Span(\varcat{C},\varcat{E})/}
    \]
    has a left adjoint $L_{\operatorname{adjointed}}$.
\end{proposition}

\begin{proof}
    By the reflection theorem of Ragimov and Schlank \cite{shauly_reflection} and the adjoint functor theorem, it suffices to show that the inclusion preserves all limits and is accessible.

    Note that filtered $\infty$-categories are nonempty, for example because their geometric realizations are contractible. Hence we are done by \Cref{prop:adjtd_preserves_nonempty_colimits,prop:adjtd_preserves_limits}.
\end{proof}

In particular, we get the following:

\begin{corollary}\label{cor:existence_of_universal_adjointed_functor}
    There exists a $(\varcat{P},\varcat{I})$-adjointed functor $F:\Span(\varcat{C},\varcat{E})\to\vartwocat{U}$ such that precomposition with $F$ gives an equivalence of $\infty$-categories
    \[
    (-)\circ F\colon
    \catoftwocats_{\vartwocat{U}/}
    \to
    \catoftwocats_{\Span(\varcat{C},\varcat{E})/}^{(\varcat{P},\varcat{I})\adjtd}.
    \]

    Equivalently, $F$ is a universal $(\varcat{P},\varcat{I})$-adjointed functor, in the sense that for any $(\varcat{P},\varcat{I})$-adjointed functor $G:\Span(\varcat{C},\varcat{E})\to\catofkernels$, there is a unique functor $H:\vartwocat{U}\to\catofkernels$ factoring $G$.
\end{corollary}

\begin{proof}
    We take $\vartwocat{U}:=L_{\operatorname{adjointed}}(\Span(\varcat{C},\varcat{E}))$.
\end{proof}

\subsubsection*{The enriched universal property}

Before finishing the subsection, let us explain that the universality of $\vartwocat{U}$ in the $\infty$-category $\catoftwocats_{\Span(\varcat{C},\varcat{E})/}^{(\varcat{P},\varcat{I})\adjtd}$ bootstraps to a universal property in the $(\infty,2)$-category of adjointed $(\infty,2)$-categories.

Let $\twocatoftwocats_{\Span(\varcat{C},\varcat{E})/}^{(\varcat{P},\varcat{I})\adjtd}$ denote the full sub-$(\infty,2)$-category of $\twocatoftwocats_{\Span(\varcat{C},\varcat{E})/}$ spanned by the $(\varcat{P},\varcat{I})$-adjointed functors.

\begin{proposition}\label{prop:enriched_universal_property}
    For any object $F:\Span(\varcat{C},\varcat{E})\to\vartwocat{G}$ of $\twocatoftwocats_{\Span(\varcat{C},\varcat{E})/}^{(\varcat{P},\varcat{I})\adjtd}$, we have an equivalence of $\infty$-categories
    \[
    \Fun(\vartwocat{U},\vartwocat{G})\simeq *,
    \]
    where $\Fun(\vartwocat{U},\vartwocat{G})$ denotes the mapping $\infty$-category in $\twocatoftwocats_{\Span(\varcat{C},\varcat{E})/}^{(\varcat{P},\varcat{I})\adjtd}$.
\end{proposition}

\begin{proof}
    The $\infty$-category $\catoftwocats_{\Span(\varcat{C},\varcat{E})/}^{(\varcat{P},\varcat{I})\adjtd}$ can be endowed with the structure of a $\catofcats$-module in $\prl$. Thus it is both enriched and cotensored over $\catofcats$.

    It is enough to show that, for every $\infty$-category $\varcat{T}$, the anima
    \[
    \operatorname{Map}(\varcat{T},\Fun(\vartwocat{U},\vartwocat{G}))
    \]
    is contractible.

    Let $\funtwocat(\varcat{T},\vartwocat{G})$ denote the $(\infty,2)$-category of functors from $\varcat{T}$ to $\vartwocat{G}$.

    Consider the functor $\Span(\varcat{C},\varcat{E})\to\funtwocat(\varcat{T},\vartwocat{G})$ given by the composition of $F$ with the diagonal functor $\vartwocat{G}\to\funtwocat(\varcat{T},\vartwocat{G})$.

    We have
    \[
    \operatorname{Map}(\varcat{T},\Fun(\vartwocat{U},\vartwocat{G}))
    \simeq
    \operatorname{Map}(\vartwocat{U},\funtwocat(\varcat{T},\vartwocat{G})).
    \]
    The right-hand side is contractible by the universality of $\vartwocat{U}$ in the $\infty$-category $\catoftwocats_{\Span(\varcat{C},\varcat{E})/}^{(\varcat{P},\varcat{I})\adjtd}$.
\end{proof}

\subsection{Characterization of the universal adjointed functor}
\label{subsec:recognition_principle}

In this subsection, we characterize the universal $(\varcat{P},\varcat{I})$-adjointed functor, following \cite[\S 2]{CLLuniv}. We first define what it means for a $(\varcat{P},\varcat{I})$-adjointed functor to $\catofcats$ to be free on an object $a\in\varcat{C}$ (\Cref{def:free_adjointed_functor_on_an_object}). The recognition principle (\Cref{recognition_principle}) states that a $(\varcat{P},\varcat{I})$-adjointed functor $h\colon\Span(\varcat{C},\varcat{E})\to\vartwocat{V}$, with $\vartwocat{V}$ an $(\infty,2)$-category, is universal if and only if it is essentially surjective and each functor $\Fun_{\vartwocat{V}}(h(a),h(-))$ is free on $\id_{h(a)}$. The proof verifies these conditions for the universal adjointed functor (\Cref{lem:universal_adjointed_representables_are_free,lem:universal_adjointed_essentially_surjective}). It then shows that any adjointed functor satisfying them is equivalent to the universal one. 

\begin{definition}\label{def:free_adjointed_functor_on_an_object}
    Let $a\in\varcat{C}$ be an object. A pair consisting of a $(\varcat{P},\varcat{I})$-adjointed functor $F\colon\Span(\varcat{C},\varcat{E})\to\catofcats$ and an object $1_a\in F(a)$ is called the free $(\varcat{P},\varcat{I})$-adjointed functor on $a$ if, for every $(\varcat{P},\varcat{I})$-adjointed functor $G\colon\Span(\varcat{C},\varcat{E})\to\catofcats$, the functor given by evaluation at $1_a$ is an equivalence of $\infty$-categories:
    \[
    \Fun_{\funtwocat(\Span(\varcat{C},\varcat{E}),\twocatofcats)}(F,G)\to G(a).
    \]

    When $a$ is clear, we might say that $F$ is a free $(\varcat{P},\varcat{I})$-adjointed functor on $1_a$ instead.
\end{definition}

Let us recall the $(\infty,2)$-categorical Yoneda embedding:

\begin{notation}
    For an $(\infty,2)$-category $\vartwocat{V}$, we denote by $\yo_{\vartwocat{V}}$ the $\catofcats$-enriched Yoneda functor, i.e. the $2$-functor
    \[
    \yo_{\vartwocat{V}}\colon\vartwocat{V}^{\op}\to\funtwocat(\vartwocat{V},\twocatofcats).
    \]
    In particular, for $a\in\vartwocat{V}$, $\yo_{\vartwocat{V}}(a)$ is a $2$-functor $\vartwocat{V}\to\twocatofcats$ sending $b$ to $\Fun_{\vartwocat{V}}(a,b)$.
\end{notation}

The following theorem gives a characterization of the free $(\varcat{P},\varcat{I})$-adjointed functor.

\begin{theorem}\label{recognition_principle}
    Let $h\colon\Span(\varcat{C},\varcat{E})\to\vartwocat{V}$ be a $(\varcat{P},\varcat{I})$-adjointed functor to an $(\infty,2)$-category $\vartwocat{V}$. Then it is the free $(\varcat{P},\varcat{I})$-adjointed functor if and only if the following two conditions hold.
    \begin{itemize}
        \item $h$ is essentially surjective.
        \item For every object $a\in\varcat{C}$, $\yo_{\vartwocat{V}}(h(a))\circ h\colon\Span(\varcat{C},\varcat{E})\to\catofcats$ is a free $(\varcat{P},\varcat{I})$-adjointed functor on $\id_{h(a)}$.
    \end{itemize}
\end{theorem}

The proof follows almost exactly the analogous proof in \cite{CLLuniv}.

First we show that the conditions hold for the universal $(\varcat{P},\varcat{I})$-adjointed functor.

\begin{lemma}\label{lem:universal_adjointed_representables_are_free}
    If $F\colon\Span(\varcat{C},\varcat{E})\to\vartwocat{U}$ is the free $(\varcat{P},\varcat{I})$-adjointed functor, then, for every $a\in\Span(\varcat{C},\varcat{E})$, $\yo_{\vartwocat{U}}(F(a))\circ F$ is a free $(\varcat{P},\varcat{I})$-adjointed functor on $\id_{F(a)}$.
\end{lemma}

\begin{proof}
    By the universal property of $F$, any $(\varcat{P},\varcat{I})$-adjointed functor $G\colon\Span(\varcat{C},\varcat{E})\to\catofcats$ can be lifted uniquely to a $2$-functor $\tilde{G}\colon\vartwocat{U}\to\twocatofcats$, and we have:
    \begin{align*}
        \Fun_{\funtwocat(\Span(\varcat{C},\varcat{E}),\twocatofcats)}(\yo_{\vartwocat{U}}(F(a))\circ F,G)
        &\simeq\Fun_{\funtwocat(\vartwocat{U},\twocatofcats)}(\yo_{\vartwocat{U}}(F(a)),\tilde{G})\\
        &\simeq\tilde{G}(F(a))\\
        &\simeq G(a).
    \end{align*}

    It is standard to check that this equivalence comes from evaluation at $\id_{F(a)}$, and so we are done.
\end{proof}

\begin{lemma}\label{lem:universal_adjointed_essentially_surjective}
    If $F\colon\Span(\varcat{C},\varcat{E})\to\vartwocat{U}$ is the free $(\varcat{P},\varcat{I})$-adjointed functor, then it is essentially surjective.
\end{lemma}

\begin{proof}
    Let $\vartwocat{U}'\subseteq\vartwocat{U}$ be the full sub-$(\infty,2)$-category given by the essential image of $F$. It is clear that $F$, regarded as a functor to $\vartwocat{U}'$, is still $(\varcat{P},\varcat{I})$-adjointed. Therefore there is a $2$-functor $\phi\colon\vartwocat{U}\to\vartwocat{U}'$ over $\Span(\varcat{C},\varcat{E})$. Composing with the inclusion, we get a self-map of $\vartwocat{U}$. Since $\vartwocat{U}$ is universal, this self-map must be equivalent to the identity. This implies that the inclusion is essentially surjective, so $\vartwocat{U}'\simeq\vartwocat{U}$ and $F$ is essentially surjective.
\end{proof}

We now show that any functor satisfying the conditions of \Cref{recognition_principle} must be the universal $(\varcat{P},\varcat{I})$-adjointed functor.

\begin{proof}[Proof of \Cref{recognition_principle}]
    In one direction, we already showed in \Cref{lem:universal_adjointed_representables_are_free,lem:universal_adjointed_essentially_surjective} that the free $(\varcat{P},\varcat{I})$-adjointed functor satisfies the conditions.

    To see the other direction, let $h\colon\Span(\varcat{C},\varcat{E})\to\vartwocat{V}$ be a $(\varcat{P},\varcat{I})$-adjointed functor satisfying the conditions. That is, $h$ is essentially surjective and, for every $a\in\varcat{C}$, $\yo_{\vartwocat{V}}(h(a))\circ h$ is the free $(\varcat{P},\varcat{I})$-adjointed functor on $\id_{h(a)}$. Let $F\colon\Span(\varcat{C},\varcat{E})\to\vartwocat{U}$ be the free $(\varcat{P},\varcat{I})$-adjointed functor.

    By the universal property of $\vartwocat{U}$, we have a $2$-functor $\bar{F}\colon\vartwocat{U}\to\vartwocat{V}$ satisfying $\bar{F}\circ F\simeq h$. Since $h$ is essentially surjective, so is $\bar{F}$. It remains to show that it is fully faithful.

    We wish to show that, for every pair of objects $X,Y\in\vartwocat{U}$, we have an equivalence of $\infty$-categories:
    \[
    \Fun_{\vartwocat{U}}(X,Y)\simeq\Fun_{\vartwocat{V}}(\bar{F}(X),\bar{F}(Y)).
    \]
    Since $F$ is essentially surjective, this is the same as showing that, for every $x\in\varcat{C}$,
    \[
    \Fun_{\vartwocat{U}}(F(x),F(-))\simeq\Fun_{\vartwocat{V}}(h(x),h(-))
    \]
    as functors $\Span(\varcat{C},\varcat{E})\to\catofcats$. However, these are simply $\yo_{\vartwocat{U}}(F(x))\circ F$ and $\yo_{\vartwocat{V}}(h(x))\circ h$, respectively. By \Cref{lem:universal_adjointed_representables_are_free} and the assumption, they are free on the respective identities. Since the map between them takes $\id_{F(x)}$ to $\id_{h(x)}$, it must be an equivalence, and we are done.
\end{proof}

\subsection{Reduction of the main theorem}
\label{subsec:reduction_of_main_thm}

In this subsection, we apply the recognition principle to the $(\infty,2)$-category of spans $\spantwo(C,E)^P_I$ appearing in our main theorem (\Cref{main_theorem:free_adjointed_2_cat}). We describe its mapping $\infty$-categories and reduce the main theorem to a universal property of them (\Cref{prop:reduction_of_main_theorem}). We formulate this universal property in \Cref{thm:red_of_main} and defer its proof to \Cref{subsec:2_FF_are_just_E_monoidal}.

\begin{definition}
    For $a,b\in\varcat{C}$, the $\infty$-category $\varcat{E}[a](b)$ is the mapping $\infty$-category from $a$ to $b$ in the $(\infty,2)$-category $\spanhalf(\varcat{C},\varcat{E})^{\varcat{E}}$. Objects of $\varcat{E}[a](b)$ are of the form $a\leftarrow x\to b$, where the right arrow is in $\varcat{E}$. A morphism from $a\leftarrow x\to b$ to $a\leftarrow y\to b$ is a morphism $x\to y$ in $\varcat{E}$ together with the commutation data over $a$ and $b$.

    Let $\varcat{I}\subset\varcat{E}$ be a wide sub-$\infty$-category. We let $\varcat{E}^{\varcat{I}}[a](b)$ denote the wide sub-$\infty$-category of $\varcat{E}[a](b)$ where we allow morphisms between $x$ and $y$ which are in $\varcat{I}$.
\end{definition}

\begin{remark}
    Since $\varcat{E}[a](b)$ was defined as a mapping $\infty$-category, it can be regarded as a functor in both $a$ and $b$ (contravariantly in $a$ and covariantly in $b$).
\end{remark}

\begin{proposition}
    Let $\varcat{P},\varcat{I}$ be as before. Then $(\varcat{E}[a](b),\varcat{E}^{\varcat{P}}[a](b),\varcat{E}^{\varcat{I}}[a](b))$ is an adequate triplet which is functorial in $a$ and $b$.
\end{proposition}

\begin{proof}
    It is clear that pullbacks in $\varcat{E}[a](b)$ are computed in $\varcat{C}$. Thus $(\varcat{E}[a](b),\varcat{E}^{\varcat{P}}[a](b),\varcat{E}^{\varcat{I}}[a](b))$ is an adequate triplet. Functoriality follows immediately from the functoriality of $\varcat{E}[a](b)$.
\end{proof}

To reduce clutter of notation, we will denote $(\varcat{E}[a](b),\varcat{E}^{\varcat{P}}[a](b),\varcat{E}^{\varcat{I}}[a](b))$ by $(\varcat{E}[a](b),\varcat{P},\varcat{I})$.

We will prove the following theorem in \Cref{subsec:2_FF_are_just_E_monoidal}.

\begin{theorem}\label{thm:red_of_main}
    Let $\varcat{C}$ be an $\infty$-category, and let $\varcat{E},\varcat{I},\varcat{P}\subset\varcat{C}$ be wide sub-$\infty$-categories that are stable under base change in $\varcat{C}$. Assume that $\varcat{I},\varcat{P}\subset\varcat{E}$ are left cancellable and that every morphism in $\varcat{I}$ and $\varcat{P}$ is truncated. Let $a\in\varcat{C}$. Let $F\colon\Span(\varcat{C},\varcat{E})\to\catofcats$ denote the functor given by
    \[
    F(b):=\Span(\varcat{E}[a](b),\varcat{P},\varcat{I}),
    \]
    obtained by composing the functor $b\mapsto(\varcat{E}[a](b),\varcat{P},\varcat{I})$ with $\Span$. Then $F$ is the free $(\varcat{P},\varcat{I})$-adjointed functor on $\id_a\in F(a)$.
\end{theorem}

What we can say now is the following:

\begin{proposition}\label{prop:reduction_of_main_theorem}
    \Cref{thm:red_of_main} implies \Cref{main_theorem:free_adjointed_2_cat}.
\end{proposition}

\begin{proof}
    Let
    \[
    h\colon\Span(\varcat{C},\varcat{E})\to\spantwo(\varcat{C},\varcat{E})^{\varcat{P}}_{\varcat{I}}
    \]
    denote the canonical functor. Since $h$ is clearly essentially surjective, \Cref{recognition_principle} shows that \Cref{main_theorem:free_adjointed_2_cat} is equivalent to the restricted Yoneda functor $\yo_{\spantwo(\varcat{C},\varcat{E})^{\varcat{P}}_{\varcat{I}}}(a)\circ h$ being the free $(\varcat{P},\varcat{I})$-adjointed functor on $\id_a$ for every $a\in\varcat{C}$. This functor sends
    \[
    b\mapsto\Fun_{\spantwo(\varcat{C},\varcat{E})^{\varcat{P}}_{\varcat{I}}}(a,b).
    \]
    By the description of the mapping $\infty$-categories in \cite[proof of Theorem 4.20]{CLLuniv}, we have a natural equivalence
    \[
    \Fun_{\spantwo(\varcat{C},\varcat{E})^{\varcat{P}}_{\varcat{I}}}(a,b)
    \simeq\Span(\varcat{E}[a](b),\varcat{P},\varcat{I}).
    \]
\end{proof}
\newpage

\section{\texorpdfstring{$\subuniverse$}{E}-monoidal categories.}\label{sec:E_monoidal}

Let $\vartopos$ be an $\infty$-topos fixed throughout this section. 

In this section, we develop the theory of $\subuniverse$-monoidal $\vartopos$-categories. These generalize symmetric monoidal $\infty$-categories by allowing tensor products indexed by a chosen class $\subuniverse$ of morphisms in $\vartopos$. This will provide the internal language for proving the universal property formulated in \Cref{sec:recognition_principle}.

We begin in \Cref{subsec:internal_spans_and_unfurling} by recalling internal spans and unfurling, including the uniqueness of unfurling under truncation assumptions (\Cref{thm:internal_uniqueness_of_unfurling}). In \Cref{subsec:E_monoids,subsec:E_monoidal_categories}, we define $\subuniverse$-monoids and $\subuniverse$-monoidal $\vartopos$-categories. We construct free examples and prove an analogue of the symmetric monoidal localization theorem (\Cref{thm:monoidal_localization}).

In \Cref{subsec:cartesian_and_cocartesian_structures}, we construct cartesian and cocartesian monoidal structures from internal limits and colimits. In \Cref{subsec:inductible_subuniverses}, we study these structures when the indexing subuniverse is inductible. In this case, cocartesian structures are unique (\Cref{thm:uniqueness_of_cartesian_structures}), and the construction of commutative algebras is right adjoint to equipping a $\vartopos$-category with its cocartesian structure (\Cref{thm:calg_is_right_adjoint}). We also show that formation of cocartesian structures admits \emph{left} adjoint (\Cref{thm:cocart_cats_are_localization_of_monoidal}).

In \Cref{subsec:E_operads}, we introduce $\subuniverse$-operads and their monoidal envelopes. We prove the envelope adjunction (\Cref{thm:env_associated_operad_adjunction}) and use operads to define lax $\subuniverse$-monoidal functors.

In \Cref{sec:partially_cartesian_structures}, we consider monoidal structures that are cartesian or cocartesian along specified smaller classes of morphisms. This gives the notion of ambidexterity used in the rest of the paper. Finally, in \Cref{subsec:symmetric_monoidal_cat_of_E_monoidal_cats}, we construct a tensor product of $\subuniverse$-monoidal $\vartopos$-categories and show that the localizations imposing cocartesian and ambidextrous structures are given by tensoring with idempotent algebras (\Cref{thm:cocart_is_a_mode,cor:P_I_ambi_is_a_smashing_localization}).

\subsubsection*{Conventions for internal category theory}

We use the theory of $\vartopos$-categories developed by Martini and Martini--Wolf in their works on Yoneda's lemma, internal straightening, colimits, and presentability \cite{MInternalYoneda,MInternalStraightenning,MWcocomplete,MWPresentabilityAndTopoi}. For an overview, we refer to \cite[introduction and \S 1.2.2]{MWPresentabilityAndTopoi}. We freely identify a $\vartopos$-category $\varintcat{C}$ with its associated sheaf of $\infty$-categories, writing $\varintcat{C}(A)$ for its $\infty$-category of sections in context $A$. We follow their terminology with the conventions described below.

Our basic notation translates as follows. Subscripts indicating the ambient $\infty$-topos are generally omitted when it is clear from context.

\begin{center}
\small
\begin{tabular}{@{}p{0.28\linewidth}p{0.23\linewidth}p{0.39\linewidth}@{}}
\hline
\textbf{Martini--Wolf} & \textbf{This paper} & \textbf{Meaning} \\[4pt]
\hline
$\mathcal{S}$ & $\catofanima$ & The $\infty$-category of animae \\[4pt]
$\catofcats_\infty,\ \prl_\infty$ & $\catofcats,\ \prl$ & The $\infty$-categories \\[4pt]
$\mathcal{C},\mathcal{D}$ & $\varcat{C},\varcat{D}$ & General $\infty$-categories \\[4pt]
\shortstack[l]{$\Omega_{\vartopos},\ \internalcatofcats_{\vartopos},$\\$\intprl_{\vartopos}$}
& \shortstack[l]{$\universe,\ \internalcatofcats,$\\$\intprl$}
& The corresponding $\vartopos$-categories \\[4pt]
$\intfun_{\vartopos},\ \intpresh_{\vartopos}$ & $\intfun,\ \intpresh$ & Internal functor and presheaf $\vartopos$-categories \\[4pt]
$\map_{\varintcat{C}}$ & $\intmap_{\varintcat{C}}$ & Internal mapping objects \\[4pt]
$h_{\varintcat{C}}$ & $\yo_{\varintcat{C}}$ & The Yoneda embedding \\[4pt]
$1$ & $*$ & A terminal object \\[4pt]
\hline
\end{tabular}
\end{center}

We also use ``terminal'' for what they call a ``final'' object. The superscript $\simeq$ continues to denote the core.

For a morphism $p:B\to A$ in $\vartopos$, we write $p_\sharp$ for the left adjoint of the restriction functor $p^*: \varintcat{C}(A)\to\varintcat{C}(B)$, when it exists. Thus our $p_{\sharp}$ is their $p_!$. This distinguishes indexed colimits from the exceptional pushforward notation used for six functor formalisms. We retain $p_*$ for the right adjoint and $F_{\sharp}$ for left Kan extension along a functor $F$. 

For an internal class $\varintcat{U}\subset\internalcatofcats$, we write
\[
\internalcatofcats^{\varintcat{U}-\sqcap},
\qquad
\internalcatofcats^{\varintcat{U}-\sqcup}
\]
for the $\vartopos$-categories of $\varintcat{U}$-complete and $\varintcat{U}$-cocomplete $\vartopos$-categories, with $\varintcat{U}$-continuous and $\varintcat{U}$-cocontinuous functors, respectively. Likewise,
\[
\intfun^{\varintcat{U}-\sqcap}(\varintcat{C},\varintcat{D}),
\qquad
\intfun^{\varintcat{U}-\sqcup}(\varintcat{C},\varintcat{D})
\]
denote the full $\vartopos$-subcategories of $\varintcat{U}$ (co)continuous functors. In particular, our superscript $\varintcat{U}-\sqcup$ replaces their $\varintcat{U}\text{-}\mathrm{cc}$. The meanings of completeness and continuity are unchanged \cite[\S 1.2.13]{MWPresentabilityAndTopoi}.

Our use of \emph{subuniverse} requires more care. Martini--Wolf associate to a local class $E$ the full subcategory $\Omega_E\subset\Omega$ \cite[\S 1.2.6]{MWPresentabilityAndTopoi}. Our context-free subuniverse $\subuniverse$, defined in \Cref{def:storngly_regular_universe}, also restricts the morphisms: in context $A$, its objects are maps $B\to A$ in $E$, and its morphisms are maps over $A$ that themselves belong to $E$. We write $\subuniverse^{\full}$ for the full subcategory on the same objects; this is the subuniverse denoted $\Omega_E$ by Martini--Wolf. Our context-free subuniverses additionally satisfy the smallness, pointedness, and independence-of-context conditions specified in \Cref{def:storngly_regular_universe}. Expressions such as ``$\subuniverse$-limits'' and ``$\subuniverse$-cocomplete'' refer to the indexing shapes in $\subuniverse^{\full}$.

When we call a full subuniverse $\varintcat{U}\subset\universe$ \emph{right regular}, we mean that it is closed under $\varintcat{U}$-indexed colimits in $\universe$. This is a relative use of the terminology: Martini--Wolf's definition for internal classes in $\internalcatofcats$ requires closure there and also requires the class to contain the simplices \cite[Definition 2.1.3.1]{MWPresentabilityAndTopoi}.

We use a few further abbreviations for constructions recalled later. The free $\subuniverse^{\full}$-cocompletion of $A\in\vartopos$, denoted $\intpresh_{\vartopos}^{\subuniverse^{\full}}(A)$ in Martini--Wolf's notation, is written $\subuniverse^{\full}[A]$ (\Cref{def:free_on_A_with_E_colimits}). For $B\to A$, the notation $B^{A\triangleleft}$ denotes their left cone on $B$, formed internally to $\vartopos_{/A}$; the extra $A$ records the context \cite[Remark 4.1.4]{MWcocomplete}. We also write $\varintcat{C}^{B}$ for the internal functor $\vartopos_{/A}$-category from $B$ to the restriction of $\varintcat{C}$, and $\intunst$ for internal unstraightening.

Finally, ``symmetric monoidal $\vartopos$-category'' retains its meaning as a sheaf of symmetric monoidal $\infty$-categories. The $\subuniverse$-monoidal structures introduced here allow tensor products indexed by $\subuniverse$. We refer to the tensor product on presentable $\vartopos$-categories constructed by Martini--Wolf as the \emph{multilinear tensor product} \cite[\S\S 2.5--2.6]{MWPresentabilityAndTopoi}.

\subsection{Internal categories of spans and unfurling.}\label{subsec:internal_spans_and_unfurling}

In this subsection, we recall how the unfurling construction uses adjoints to extend functors to $\infty$-categories of spans. We begin with adequate triplets and the span construction, then explain how unfurling works internally to $\vartopos$ (\Cref{const:internal_unfurling}). These constructions will later allow us to obtain monoidal structures from internal limits and colimits.

We then establish the uniqueness of extensions of a functor $\varintcat{C}^{\op}\to \internalcatofcats$, under suitable left-cancellation and truncation assumptions. Under these assumptions, unfurling identifies adjointable functors $\varintcat{C}^{\op}\to \internalcatofcats$ with adjointed functors $\intspan(\varintcat{C},\varintcat{C}_R)\to \internalcatofcats$ (\Cref{thm:internal_uniqueness_of_unfurling}). We conclude with the dual of the construction, in which the additional morphisms act by right adjoints.

\subsubsection*{Adequate triplets and unfurling}

We begin by recalling some basic facts about adequate triplets, which provide the input for the construction of $\infty$-categories of spans.

\begin{definition}
    An adequate triplet $(\varcat{C},\varcat{C}_L,\varcat{C}_R)$ is an $\infty$-category $\varcat{C}$ with two wide subcategories $\varcat{C}_L,\varcat{C}_R\subset\varcat{C}$, called the left-pointing and right-pointing maps, respectively, such that for every $Y\to W$ in $\varcat{C}_R$ and $Z\to W$ in $\varcat{C}_L$, the pullback
    \begin{equation}\label{eq:pullback_square_of_left_along_right}
    \begin{tikzcd}
        Z\times_W Y \ar[r]\ar[d]\arrow[rd, "\lrcorner"{anchor=center, pos=0.125}, draw=none] & Z\ar[d]\\
        Y\ar[r] & W
    \end{tikzcd}
    \end{equation}
    exists, and the upper and left morphisms belong to $\varcat{C}_R$ and $\varcat{C}_L$, respectively.
\end{definition}

Adequate triplets form an $\infty$-category with the following notion of functor.

\begin{definition}
    A functor of adequate triplets $(\varcat{C},\varcat{C}_L,\varcat{C}_R)\to(\varcat{D},\varcat{D}_L,\varcat{D}_R)$ is a functor $F:\varcat{C}\to\varcat{D}$ satisfying the following conditions.
    \begin{mylist}
        \item If $f$ belongs to $\varcat{C}_L$, then $F(f)$ belongs to $\varcat{D}_L$. If $g$ belongs to $\varcat{C}_R$, then $F(g)$ belongs to $\varcat{D}_R$.
        \item $F$ preserves pullbacks of right-pointing maps along left-pointing maps, as in \Cref{eq:pullback_square_of_left_along_right}.
    \end{mylist}
    We denote by $\adtrip$ the $\infty$-category of adequate triplets and functors of adequate triplets.
\end{definition}

There are also two $(\infty,2)$-categories of adequate triplets, corresponding to two notions of natural transformation.

\begin{definition}
    Let $F,G:(\varcat{C},\varcat{C}_L,\varcat{C}_R)\to(\varcat{D},\varcat{D}_L,\varcat{D}_R)$ be functors of adequate triplets.

    A right-pointing natural transformation of functors of adequate triplets is a natural transformation $\theta:F\Rightarrow G$ such that each component $\theta_X:F(X)\to G(X)$ belongs to $\varcat{D}_R$ and, for every morphism $f:X\to Y$ in $\varcat{C}_L$, the naturality square
    \[
    \begin{tikzcd}
        F(X)\arrow[rd, "\lrcorner"{anchor=center, pos=0.125}, draw=none]\ar[r, "F(f)"]\ar[d,"\theta_X"'] & F(Y)\ar[d,"\theta_Y"]\\
        G(X)\ar[r, "G(f)"'] & G(Y)
    \end{tikzcd}
    \]
    is cartesian in $\varcat{D}$.

    A left-pointing natural transformation of functors of adequate triplets is a natural transformation $\theta:F\Rightarrow G$ such that each component $\theta_X:F(X)\to G(X)$ belongs to $\varcat{D}_L$ and, for every morphism $f:X\to Y$ in $\varcat{C}_R$, the naturality square
    \[
    \begin{tikzcd}
        F(X)\arrow[rd, "\lrcorner"{anchor=center, pos=0.125}, draw=none]\ar[r, "F(f)"]\ar[d,"\theta_X"'] & F(Y)\ar[d,"\theta_Y"]\\
        G(X)\ar[r, "G(f)"'] & G(Y)
    \end{tikzcd}
    \]
    is cartesian in $\varcat{D}$.

    We denote by $\adtriptwo^R$ the $(\infty,2)$-category of adequate triplets, functors of adequate triplets, and right-pointing natural transformations. Likewise, we denote by $\adtriptwo^L$ the $(\infty,2)$-category defined using left-pointing natural transformations.
\end{definition}

We will frequently use the following special case of an adequate triplet.

\begin{definition}
    A span pair is an adequate triplet with $\varcat{C}_L=\varcat{C}$. Equivalently, it is an $\infty$-category $\varcat{C}$ with a wide subcategory $\varcat{C}_R\subset\varcat{C}$ that is stable under base change in $\varcat{C}$.

    We denote by $\spanpairs\subset\adtrip$ the full subcategory spanned by span pairs.
\end{definition}

Adequate triplets give rise to $\infty$-categories of spans as follows.

\begin{construction}[categories of spans]
    \begin{mylist}
        \item For an adequate triplet $(\varcat{C},\varcat{C}_L,\varcat{C}_R)$, there is an $\infty$-category $\Span(\varcat{C},\varcat{C}_L,\varcat{C}_R)$ with the same space of objects as $\varcat{C}$, morphisms given by correspondences, and composition given by pullbacks. We refer to \cite[\S 2]{HHLNSpansDef} for a precise treatment.
        \item The span construction assembles into a limit-preserving functor
        \[
        \Span:\adtrip\to\catofcats.
        \]
        This functor admits two enhancements to a $2$-functor of $(\infty,2)$-categories:
        \[
        \Span:\adtriptwo^R\to\twocatofcats,
        \qquad
        \Span:(\adtriptwo^L)^{\co}\to\twocatofcats
        \]
        (see \cite[Proposition C.20]{BHNormsInMotivic}).
        \item\label{def:catsofspans:adjuncions}
        Adjunctions in $\adtriptwo^R$ are adjunctions $F:\varcat{C}\fromto\varcat{D}:G$ in which $F$ and $G$ are functors of adequate triplets and the unit and counit are right-pointing natural transformations. Since $\Span:\adtriptwo^R\to\twocatofcats$ is a $2$-functor, it sends such adjunctions to adjunctions of $\infty$-categories.

        Adjunctions in $\adtriptwo^L$ are defined similarly using left-pointing natural transformations. The $2$-functor $\Span:(\adtriptwo^L)^{\co}\to\twocatofcats$ also sends these to adjunctions of $\infty$-categories, with the direction of the adjunction reversed because of the $\co$.
    \end{mylist}
\end{construction}

Unfurling is a construction of functors out of $\infty$-categories of spans, introduced by Barwick and given a more conceptual treatment in \cite{CLRUniversalityOfUnferling}. We recall the construction in the form presented there.

\begin{construction}[unfurling]\label{const:unfurling}
    Following \cite[Definition 3.12(2) and Theorem 3.13(2)]{CLRUniversalityOfUnferling}, let $(\varcat{C},\varcat{C}_R)$ be a span pair.
    \begin{mylist}
        \item Let $F:\varcat{C}^{\op}\to\catofcats$ be a functor. We say that $F$ is \emph{left $R$-adjointable} if it sends morphisms in $\varcat{C}_R^{\op}$ to functors admitting left adjoints and sends every base change square
        \[
        \begin{tikzcd}
            Z\times_W Y\ar[r,"\varcat{C}_R"]\ar[d]\arrow[rd, "\lrcorner"{anchor=center, pos=0.125}, draw=none] & Z\ar[d]\\
            Y\ar[r,"\varcat{C}_R"'] & W
        \end{tikzcd}
        \]
        of a morphism in $\varcat{C}_R$ to a horizontally left adjointable square
        \[
        \begin{tikzcd}
            F(W)\ar[r]\ar[d] & F(Y)\ar[d]\\
            F(Z)\ar[r] & F(Z\times_W Y).
        \end{tikzcd}
        \]
        That is, the horizontal functors admit left adjoints and the square satisfies the Beck--Chevalley condition.
        \item A \emph{morphism of left $R$-adjointable functors} is a natural transformation $\eta:F\Rightarrow G$ such that, for every morphism $f:X\to Y$ in $\varcat{C}_R$, the naturality square
        \[
        \begin{tikzcd}
            F(Y)\ar[r,"F(f)"]\ar[d,"\eta_Y"'] & F(X)\ar[d,"\eta_X"]\\
            G(Y)\ar[r,"G(f)"'] & G(X)
        \end{tikzcd}
        \]
        satisfies the Beck--Chevalley condition. We denote by
        \[
        \Fun^{R\operatorname{-adjointable}}(\varcat{C}^{\op},\catofcats)
        \subset\Fun(\varcat{C}^{\op},\catofcats)
        \]
        the $\infty$-subcategory of left $R$-adjointable functors and their morphisms.
        \item Let $F:\varcat{C}^{\op}\to\catofcats$ be left $R$-adjointable, and let $p:\varcat{E}\to\varcat{C}$ be its cartesian unstraightening. Let $\varcat{E}^{p\dashcart}$ be the wide subcategory of $\varcat{E}$ spanned by $p$-cartesian morphisms, and let $\varcat{E}_R\subset\varcat{E}$ be the wide subcategory of morphisms lying over $\varcat{C}_R$. Then $(\varcat{E},\varcat{E}^{p\dashcart},\varcat{E}_R)$ is an adequate triplet, and $p$ induces a functor of adequate triplets
        \[
        (\varcat{E},\varcat{E}^{p\dashcart},\varcat{E}_R)
        \to(\varcat{C},\varcat{C},\varcat{C}_R).
        \]
        Applying the span construction gives a functor
        \[
        \Span(p):\Span(\varcat{E},\varcat{E}^{p\dashcart},\varcat{E}_R)
        \to\Span(\varcat{C},\varcat{C},\varcat{C}_R).
        \]
        \item The functor $\Span(p)$ is a cocartesian fibration classifying a functor
        \[
        \widetilde{F}:\Span(\varcat{C},\varcat{C}_R)\to\catofcats,
        \]
        called the \emph{unfurling} of $F$. It extends $F$, and for every $f:Y\to W$ in $\varcat{C}_R$, the functor $\widetilde{F}(Y=Y\xto{f}W)$ is left adjoint to $F(f)=\widetilde{F}(W\xleftarrow{f}Y=Y)$.
    \end{mylist}
\end{construction}

\subsubsection*{Internal unfurling}

We now recall the notion of an adequate triplet of $\vartopos$-categories from \cite{CLLambi}.

\begin{definition}\label{def:internal_ad_triplet}
    Following \cite[Definition 4.3]{CLLambi}, an adequate triplet of $\vartopos$-categories is a triplet $(\varintcat{C},\varintcat{C}_L,\varintcat{C}_R)$ consisting of a $\vartopos$-category $\varintcat{C}$ and two wide $\vartopos$-subcategories $\varintcat{C}_L,\varintcat{C}_R\subset\varintcat{C}$ such that, for every $A\in\vartopos$, the triplet $(\varintcat{C}(A),\varintcat{C}_L(A),\varintcat{C}_R(A))$ is adequate and, for every morphism $p:B\to A$ in $\vartopos$, the restriction functor $p^*:\varintcat{C}(A)\to\varintcat{C}(B)$ is a functor of adequate triplets.
\end{definition}

Equivalently, an adequate triplet of $\vartopos$-categories is a limit-preserving functor $\vartopos^{\op}\to\adtrip$.

A $\vartopos$-span pair $(\varintcat{C},\varintcat{C}_R)$ is an internal adequate triplet of the form $(\varintcat{C},\varintcat{C},\varintcat{C}_R)$.

We denote by $\intadtrip$ the $\vartopos$-category of internal adequate triplets, defined by
\[
\intadtrip(A):=\Fun^R((\vartopos_{/A})^{\op},\adtrip).
\]

Since the span construction preserves limits, it also defines internal $\infty$-categories of spans.

\begin{definition}
    The span $\vartopos$-category $\intspan(\varintcat{C},\varintcat{C}_L,\varintcat{C}_R)$ is the composite
    \[
    \vartopos^{\op}\to\adtrip\xto{\Span}\catofcats.
    \]
    This defines a $\vartopos$-functor
    \[
    \intspan:\intadtrip\to\internalcatofcats.
    \]
\end{definition}

Before discussing internal unfurling, we recall internal adjunctions. The $\infty$-category $\catofcats(\vartopos)$ has a natural enhancement to the $(\infty,2)$-category $\twocatofcats(\vartopos)$ of $\vartopos$-categories, in which we can consider adjunctions of $\vartopos$-functors.

Let $\varintcat{C},\varintcat{D}:\vartopos^{\op}\to\catofcats$ be $\vartopos$-categories, and let $F:\varintcat{C}\to\varintcat{D}$ be a $\vartopos$-functor. Then $F$ is a left adjoint in $\twocatofcats(\vartopos)$ if and only if every $F(A):\varintcat{C}(A)\to\varintcat{D}(A)$ is a left adjoint and, for every morphism $p:B\to A$ in $\vartopos$, the square
\[
\begin{tikzcd}
    \varintcat{C}(A)\ar[r,"p^*"]\ar[d,"F(A)"'] & \varintcat{C}(B)\ar[d,"F(B)"]\\
    \varintcat{D}(A)\ar[r,"p^*"'] & \varintcat{D}(B)
\end{tikzcd}
\]
is vertically right adjointable.

Adjunctions are local in the following sense.

\begin{proposition}[{\cite[Remark 3.3.6]{MWcocomplete}}]
    Let $F:\varintcat{C}\to\varintcat{D}$ be a $\vartopos$-functor. The condition that $F|_{\vartopos_{/A}}$ is a left adjoint is local in $A$: for every effective epimorphism $A'\to A$, the functor $F|_{\vartopos_{/A}}$ is a left adjoint if and only if $F|_{\vartopos_{/A'}}$ is a left adjoint.
\end{proposition}

We can now describe the internal unfurling construction.

\begin{construction}[internal unfurling]\label{const:internal_unfurling}
    Let $(\varintcat{C},\varintcat{C}_R)$ be a $\vartopos$-span pair.
    \begin{mylist}
        \item Let $F:\varintcat{C}^{\op}\to\internalcatofcats$ be a $\vartopos$-functor. We say that $F$ is \emph{left $R$-adjointable} if it sends morphisms in $\varintcat{C}_R^{\op}$ to internal right adjoints and, for every context $A$, sends each base change square in $\varintcat{C}(A)$
        \[
        \begin{tikzcd}
            Z\times_W Y\ar[r,"\varintcat{C}_R"]\ar[d]\arrow[rd, "\lrcorner"{anchor=center, pos=0.125}, draw=none] & Z\ar[d]\\
            Y\ar[r,"\varintcat{C}_R"'] & W
        \end{tikzcd}
        \]
        of a morphism in $\varintcat{C}_R(A)$ to a horizontally left adjointable square
        \[
        \begin{tikzcd}
            F(W)\ar[r]\ar[d] & F(Y)\ar[d]\\
            F(Z)\ar[r] & F(Z\times_W Y)
        \end{tikzcd}
        \]
        in the $(\infty,2)$-category $\twocatofcats(\vartopos_{/A})$. That is, the horizontal functors admit internal left adjoints and the square satisfies the Beck--Chevalley condition in every context.
        \item A \emph{morphism of internally left $R$-adjointable functors} is a natural transformation $\eta:F\Rightarrow G$ such that, for every context $A$ and every morphism $f:X\to Y$ in $\varintcat{C}_R(A)$, the naturality square
        \[
        \begin{tikzcd}
            F(Y)\ar[r,"F(f)"]\ar[d,"\eta_Y"'] & F(X)\ar[d,"\eta_X"]\\
            G(Y)\ar[r,"G(f)"'] & G(X)
        \end{tikzcd}
        \]
        satisfies the Beck--Chevalley condition in $\twocatofcats(\vartopos_{/A})$. We denote by
        \[
        \intfun^{R\operatorname{-adjointable}}(\varintcat{C}^{\op},\internalcatofcats)
        \subset\intfun(\varintcat{C}^{\op},\internalcatofcats)
        \]
        the $\vartopos$-subcategory of left $R$-adjointable functors and their morphisms. This is a $\vartopos$-subcategory by locality of adjunctions.
        \item\label{const:unst_of_unf}
        Let $F:\varintcat{C}^{\op}\to\internalcatofcats$ be left $R$-adjointable, and let $p:\varintcat{P}\to\varintcat{C}$ be its cartesian unstraightening. Let $\varintcat{P}^{p\dashcart}$ be the wide $\vartopos$-subcategory of $\varintcat{P}$ spanned by $p$-cartesian morphisms, and let $\varintcat{P}_R$ be the wide $\vartopos$-subcategory of morphisms lying over $\varintcat{C}_R$. Then $(\varintcat{P},\varintcat{P}^{p\dashcart},\varintcat{P}_R)$ is an internal adequate triplet, and $p$ induces a functor of internal adequate triplets
        \[
        (\varintcat{P},\varintcat{P}^{p\dashcart},\varintcat{P}_R)
        \to(\varintcat{C},\varintcat{C},\varintcat{C}_R).
        \]
        Applying the internal span construction gives a $\vartopos$-functor
        \[
        \intspan(p):\intspan(\varintcat{P},\varintcat{P}^{p\dashcart},\varintcat{P}_R)
        \to\intspan(\varintcat{C},\varintcat{C}_R).
        \]
    \end{mylist}
\end{construction}

The internal construction has the same properties as unfurling for $\infty$-categories.

\begin{proposition}
    Let $F:\varintcat{C}^{\op}\to\internalcatofcats$ be a left $R$-adjointable $\vartopos$-functor. The $\vartopos$-functor
    \[
    \intspan(p):\intspan(\varintcat{P},\varintcat{P}^{p\dashcart},\varintcat{P}_R)
    \to\intspan(\varintcat{C},\varintcat{C}_R)
    \]
    of \Cref{const:unst_of_unf} is a cocartesian fibration classifying a $\vartopos$-functor
    \[
    \widetilde{F}:\intspan(\varintcat{C},\varintcat{C}_R)\to\internalcatofcats,
    \]
    called the \emph{unfurling} of $F$. It extends $F$, and for every context $A$ and every $f:Y\to W$ in $\varintcat{C}_R(A)$, the functor $\widetilde{F}(Y=Y\xto{f}W)$ is internally left adjoint to $F(f)=\widetilde{F}(W\xleftarrow{f}Y=Y)$.
\end{proposition}
\begin{proof}
    The only nontrivial assertion is that $\intspan(p)$ is a cocartesian fibration. By \cite[Proposition 3.1.7]{MInternalStraightenning}, it suffices to check that $\intspan(p)(A)$ is a cocartesian fibration of $\infty$-categories for every $A\in\vartopos$ and that, for every morphism $q:B\to A$ in $\vartopos$, the square
    \[
    \begin{tikzcd}
        \intspan(\varintcat{P},\varintcat{P}^{p\dashcart},\varintcat{P}_R)(A)\ar[d,"q^*"]\ar[rr,"\intspan(p)(A)"]
        &&\intspan(\varintcat{C},\varintcat{C}_R)(A)\ar[d,"q^*"]\\
        \intspan(\varintcat{P},\varintcat{P}^{p\dashcart},\varintcat{P}_R)(B)\ar[rr,"\intspan(p)(B)"]
        &&\intspan(\varintcat{C},\varintcat{C}_R)(B)
    \end{tikzcd}
    \]
    is a morphism of cocartesian fibrations.

    In each context $A$, the functor $\intspan(p)(A)$ is obtained by applying \Cref{const:unfurling} to the composite
    \[
    \varintcat{C}^{\op}(A)\xto{F(A)}\internalcatofcats(A)\xto{\Gamma_A}\catofcats.
    \]
    Thus it is a cocartesian fibration by the construction for $\infty$-categories.

    Functoriality of unfurling shows that the square induced by $q^*$ is a morphism of cocartesian fibrations. Here we use the Beck--Chevalley conditions for adjunctions of $\vartopos$-categories.
\end{proof}

\subsubsection*{Uniqueness in the truncated case}

Let $(\varcat{C},\varcat{C}_R)$ be a span pair. Suppose that $\varcat{C}_R$ is left cancellable and that, for every morphism $f$ in $\varcat{C}_R$, there exists $n$ such that $f$ is $n$-truncated. We say that a functor
\[
F:\Span(\varcat{C},\varcat{C}_R)\to\catofcats
\]
is \emph{$R$-adjointed} if it is $\varcat{C}_R$-left adjointed in the sense of \Cref{def:adjointed_functors}. We denote by
\[
\Fun^{R\adjtd}(\Span(\varcat{C},\varcat{C}_R),\catofcats)
\subset\Fun(\Span(\varcat{C},\varcat{C}_R),\catofcats)
\]
the full $\infty$-subcategory of $R$-adjointed functors.

Let $\spanpairs'\subset\spanpairs$ denote the full subcategory spanned by span pairs $(\varcat{C},\varcat{C}_R)$ such that $\varcat{C}_R$ is left cancellable and every morphism in $\varcat{C}_R$ is truncated.

\begin{theorem}\label{thm:uniqueness_of_unfurling}
    Restriction along the natural inclusion $\varcat{C}^{\op}\into\Span(\varcat{C},\varcat{C}_R)$ defines a natural transformation between functors $(\spanpairs')^{\op}\to\catoflargecats$:
    \[
    \theta:\Fun^{R\adjtd}(\Span(\varcat{C},\varcat{C}_R),\catofcats)
    \to\Fun^{R\operatorname{-adjointable}}(\varcat{C}^{\op},\catofcats).
    \]
    The unfurling construction assembles into a section of $\theta$. For every $(\varcat{C},\varcat{C}_R)\in\spanpairs'$, this section is an equivalence of $\infty$-categories.
\end{theorem}
\begin{proof}
    The assertion that $\theta$ is an equivalence under the truncation and left cancellation assumptions is a special case of \cite[Theorem 3.3]{uniqueness_of_6ff}. Since unfurling gives a section of $\theta$, it also gives its inverse.
\end{proof}

To obtain an internal uniqueness theorem, we first need an internal analogue of the condition that every morphism in $\varcat{C}_R$ is truncated.

\begin{definition}
    Let $\varintcat{D}$ be a $\vartopos$-category, and let $f:x\to y$ be a morphism of $\varintcat{D}$ in context $A$. We say that $f$ is \emph{locally truncated} if there exists a cover $(p_i:U_i\to A)_i$ such that each $p_i^*(f)$ is truncated in $\varintcat{D}(U_i)$.

    We say that $\varintcat{D}$ is \emph{truncated} if every morphism of $\varintcat{D}$ is locally truncated.
\end{definition}

To describe the essential image of internal unfurling under this assumption, we use the following subcategories.

\begin{definition}
    Let $(\varintcat{C},\varintcat{C}_R)$ be a $\vartopos$-span pair, with $\varintcat{C}_R$ left cancellable and truncated. For each $A\in\vartopos$, we write
    \[
    \varintcat{C}_R(A)^t\subset\varintcat{C}_R(A)
    \]
    for the wide subcategory spanned by truncated morphisms.
\end{definition}

\begin{proposition}
    Let $(\varintcat{C},\varintcat{C}_R)$ be a $\vartopos$-span pair, with $\varintcat{C}_R$ left cancellable and truncated. For every context $A\in\vartopos$, the pair $(\varintcat{C}(A),\varintcat{C}_R(A)^t)$ is a span pair. Moreover, the inclusion
    \[
    (\varintcat{C}(A),\varintcat{C}_R(A)^t)
    \to(\varintcat{C}(A),\varintcat{C}_R(A))
    \]
    is a morphism of span pairs.
\end{proposition}
\begin{proof}
    This follows because truncated morphisms are stable under base change.
\end{proof}

We can now define internal adjointedness.

\begin{definition}
    Let $(\varintcat{C},\varintcat{C}_R)$ be a $\vartopos$-span pair, with $\varintcat{C}_R$ left cancellable and truncated. A $\vartopos$-functor
    \[
    F:\intspan(\varintcat{C},\varintcat{C}_R)\to\internalcatofcats
    \]
    is \emph{$R$-adjointed} if, for every context $A\in\vartopos$, the composite
    \[
    \Span(\varintcat{C}(A),\varintcat{C}_R(A)^t)
    \into\intspan(\varintcat{C},\varintcat{C}_R)(A)
    \xto{F(A)}\internalcatofcats(A)
    \simeq\catofcats(\vartopos_{/A}),
    \]
    viewed as a functor into $\twocatofcats(\vartopos_{/A})$, is $\varintcat{C}_R(A)^t$-left adjointed in the sense of \Cref{def:adjointed_functors}.
\end{definition}

We denote by
\[
\intfun^{R\adjtd}(\intspan(\varintcat{C},\varintcat{C}_R),\internalcatofcats)
\subset\intfun(\intspan(\varintcat{C},\varintcat{C}_R),\internalcatofcats)
\]
the full $\vartopos$-subcategory of $R$-adjointed functors.

Our goal is to identify left $R$-adjointable $\vartopos$-functors $\varintcat{C}^{\op}\to\internalcatofcats$ with $R$-adjointed $\vartopos$-functors $\intspan(\varintcat{C},\varintcat{C}_R)\to\internalcatofcats$. The key step is to recover the internal span category by sheafifying the span categories formed using truncated morphisms in each context.

Let $\vartopos[P]=\presh(S)$ be a presheaf $\infty$-topos, and let $i^*:\vartopos[P]\fromto\vartopos:i_*$ be a geometric embedding, with $i_*$ fully faithful. For example, we may present $\vartopos$ as the $\infty$-topos of sheaves on $S$ for a suitable topology. We use the same notation for the induced functors on internal categories. In expressions for $\vartopos[P]$-categories, we write $\varintcat{C}$ for $i_*\varintcat{C}$.

Let $\varintcat{C}_R^t$ denote the $\vartopos[P]$-category corresponding to the presheaf of $\infty$-categories
\[
S^{\op}\to\vartopos^{\op}\to\catofcats,
\qquad A\mapsto\varintcat{C}_R(A)^t.
\]
The left exact localization $i^*:\vartopos[P]\to\vartopos$ induces
\[
i^*:\catofcats(\vartopos[P])\to\catofcats(\vartopos).
\]

\begin{lemma}\label{lem:truncated_sheafify_to_locally_trancated}
    The inclusion of $\vartopos[P]$-categories
    \[
    \intspan(\varintcat{C},\varintcat{C}_R^t)
    \into i_*\intspan(\varintcat{C},\varintcat{C}_R)
    \]
    is adjoint to an equivalence of $\vartopos$-categories
    \[
    i^*\intspan(\varintcat{C},\varintcat{C}_R^t)
    \simeq\intspan(\varintcat{C},\varintcat{C}_R).
    \]
\end{lemma}
\begin{proof}
    Recall that an internal category $\varintcat{C}$ is a simplicial object in the corresponding $\infty$-topos:
    \[
    \varintcat{C}_0
    \;\substack{
    \longleftarrow\\[-0.2em]
    \longrightarrow\\[-0.2em]
    \longleftarrow
    }\;
    \varintcat{C}_1
    \;\substack{
    \longleftarrow\\[-0.2em]
    \longrightarrow\\[-0.2em]
    \longleftarrow\\[-0.2em]
    \longrightarrow\\[-0.2em]
    \longleftarrow
    }\;
    \varintcat{C}_2
    \;\cdots.
    \]
    Both $i^*$ and $i_*$ are computed levelwise on this simplicial object.

    The morphism
    \[
    \intspan(\varintcat{C},\varintcat{C}_R^t)_0
    \to i_*\intspan(\varintcat{C},\varintcat{C}_R)_0
    \]
    is an equivalence because the inclusion in the statement is wide. The morphism
    \[
    \intspan(\varintcat{C},\varintcat{C}_R^t)_1
    \to i_*\intspan(\varintcat{C},\varintcat{C}_R)_1
    \]
    is a monomorphism in $\vartopos[P]$. Since $i^*$ is left exact, it preserves monomorphisms. By local truncation, the induced morphism
    \[
    i^*\intspan(\varintcat{C},\varintcat{C}_R^t)_1
    \to\intspan(\varintcat{C},\varintcat{C}_R)_1
    \]
    is also an effective epimorphism, hence an equivalence. The Segal conditions then give an equivalence in every degree.
\end{proof}

\begin{theorem}\label{thm:internal_uniqueness_of_unfurling}
    Let $(\varintcat{C},\varintcat{C}_R)$ be a $\vartopos$-span pair, with $\varintcat{C}_R$ left cancellable and truncated. The internal unfurling construction defines an equivalence of $\vartopos$-categories
    \[
    \intfun^{R\operatorname{-adjointable}}(\varintcat{C}^{\op},\internalcatofcats)
    \simeq\intfun^{R\adjtd}(\intspan(\varintcat{C},\varintcat{C}_R),\internalcatofcats).
    \]
\end{theorem}
\begin{proof}
    Let $x:\varcat{X}\to\spanpairs'$ be the cartesian unstraightening of the functor $(\spanpairs')^{\op}\to\catoflargecats$ given by
    \[
    (\varcat{C},\varcat{C}_R)
    \mapsto\Fun^{R\operatorname{-adjointable}}(\varcat{C}^{\op},\catofcats).
    \]
    Similarly, let $y:\varcat{Y}\to\spanpairs'$ be the cartesian unstraightening of the functor $(\spanpairs')^{\op}\to\catoflargecats$ given by
    \[
    (\varcat{C},\varcat{C}_R)
    \mapsto\Fun^{R\adjtd}(\Span(\varcat{C},\varcat{C}_R),\catofcats).
    \]

    Left $R$-adjointable $\vartopos$-functors $\varintcat{C}^{\op}\to\internalcatofcats$ correspond to $\vartopos$-sheaves in $\varcat{X}$, while $R$-adjointed $\vartopos$-functors $\intspan(\varintcat{C},\varintcat{C}_R)\to\internalcatofcats$ correspond to $\vartopos$-sheaves in $\varcat{Y}$. By \Cref{thm:uniqueness_of_unfurling}, restriction defines an equivalence $\varcat{Y}\simeq\varcat{X}$ over $\spanpairs'$.

    By \Cref{lem:truncated_sheafify_to_locally_trancated}, it suffices to consider the case in which $\vartopos=\vartopos[P]=\presh(S)$ is a presheaf $\infty$-topos and the corresponding functor $S^{\op}\to\spanpairs$ takes values in $\spanpairs'$. In this case, the result follows from the equivalence $\varcat{Y}\simeq\varcat{X}$.
\end{proof}

\subsubsection*{Covariant unfurling}

We conclude with the covariant version of the unfurling construction.

\begin{proposition}\label{prop:cov_unfurling}
    Let $(\varintcat{C},\varintcat{C}_L)$ be a $\vartopos$-span pair, and let $F:\varintcat{C}\to\internalcatofcats$ be a right $L$-adjointable $\vartopos$-functor. Let $p:\varintcat{P}\to\varintcat{C}$ be its cocartesian unstraightening, and let $\varintcat{P}^{L\dashcart}$ be the wide $\vartopos$-subcategory of $\varintcat{P}$ spanned by cartesian lifts of morphisms in $\varintcat{C}_L$. Then
    \[
    \intspan(\varintcat{P},\varintcat{P}^{L\dashcart},\varintcat{P})
    \to\intspan(\varintcat{C},\varintcat{C}_L,\varintcat{C})
    \]
    is a cocartesian fibration classifying an extension
    \[
    \widetilde{F}:\intspan(\varintcat{C},\varintcat{C}_L,\varintcat{C})
    \to\internalcatofcats
    \]
    of $F$. This extension sends left-pointing maps to the right adjoints of the corresponding right-pointing maps.

    If $\varintcat{C}_L$ is left cancellable and truncated, there is also a canonical equivalence of $\vartopos$-categories
    \[
    \intfun^{L\operatorname{-adjointable}}(\varintcat{C},\internalcatofcats)
    \simeq\intfun^{L\adjtd}(\intspan(\varintcat{C},\varintcat{C}_L,\varintcat{C}),\internalcatofcats).
    \]
\end{proposition}
\begin{proof}
    It suffices to verify the corresponding statement for $\infty$-categories and then internalize the construction as above. The covariant unfurling construction is given in \cite[Theorem 3.13(1)]{CLRUniversalityOfUnferling}. The uniqueness statement follows by the same argument as in the contravariant case.
\end{proof}

\subsection{\texorpdfstring{$\subuniverse$}{E}-monoids.}\label{subsec:E_monoids}

A commutative monoid encodes a coherent way of adding finite families of elements. In an $\infty$-category $\varcat{C}$ with finite products, this structure can be described by a finite-product-preserving functor
\[
M:\Span(\Fin)\to\varcat{C}.
\]
We generalize this description to $\vartopos$-categories, allowing sums indexed by more general families. The permitted families and summation maps are specified by a \emph{context-free subuniverse} $\subuniverse$.

We begin by introducing context-free subuniverses (\Cref{def:storngly_regular_universe}) and explaining their relation to local classes of morphisms in $\vartopos$ (\Cref{prop:from_local_class_to_subuniverse}). We then construct the corresponding internal categories of spans and use them to define $\subuniverse$-monoids (\Cref{def:E_monoids}). This recovers ordinary commutative monoids when the indexing families are finite sets (\Cref{example:Fin_monoids}).

We show that the $\vartopos$-category of $\subuniverse$-monoids in $\universe$ is presentable (\Cref{prop:cmon_is_presentable}). We also give another formula for $\subuniverse$-monoids in complete $\vartopos$-categories (\Cref{prop:extra_back_functoriality_for_E_monoidal_cats}). Finally, we describe free $\subuniverse$-monoids in terms of free cocompletions (\Cref{prop:free_E_monoid_on_A}), generalizing the description of the free commutative monoid on a space as the space of finite families of its elements.

\subsubsection*{Context-free subuniverses}

\begin{definition}\label{def:storngly_regular_universe}
    A context-free subuniverse is a $\vartopos$-subcategory $\subuniverse\subset\universe$ satisfying the following conditions.
    \begin{mylist}
        \item\label[condition]{cond:srs_is_kappa_small} (Smallness) $\subuniverse$ is a small $\vartopos$-category.
        \item\label[condition]{cond:srs_has_iso} (Pointedness) $\subuniverse$ is closed under empty limits in $\universe$. That is, in every context $A\in\vartopos$, the terminal object $[\id:A\to A]\in\universe(A)$ belongs to $\subuniverse(A)$.
        \item\label[condition]{cond:idependent_of_context} (Independence of context) Let $B\in\subuniverse(A)$. A morphism $p:C\to B$ in $\universe(A)$ belongs to $\subuniverse(A)$ if and only if the corresponding object $[p:C\to B]\in\universe(B)$ belongs to $\subuniverse(B)$.
    \end{mylist}

    A context-free subuniverse $\subuniverse$ is \emph{full} if $\subuniverse\subset\universe$ is a full $\vartopos$-subcategory.

    It is \emph{truncated} if, for every context $A$ and every object $[B\to A]\in\subuniverse(A)$, there exists a cover $(U_\alpha\to A)_\alpha$ such that $[B\times_A U_\alpha\to U_\alpha]$ is a truncated object of $\universe(U_\alpha)$ for every $\alpha$.

    A context-free subuniverse $\subuniverse[I]$ is \emph{inductible} if it is full and truncated.
\end{definition}

Recall that a \emph{local class} of morphisms in $\vartopos$ is a class $E$ that is stable under base change in $\vartopos$ and local on the target. The latter condition means that, for every effective epimorphism $A'\to A$, a morphism $B\to A$ belongs to $E$ whenever its pullback $B\times_A A'\to A'$ belongs to $E$.

Given a local class $E$ that is closed under composition and contains equivalences, define $\subuniverse\subset\universe$ as follows. In context $A$, the objects of $\subuniverse(A)$ are the morphisms $B\to A$ in $E$, and its morphisms are the commutative triangles
\[
\begin{tikzcd}
    C\ar[rd]\ar[rr,"f"]&&D\ar[ld]\\
    &A
\end{tikzcd}
\]
for which $f\in E$.

\begin{warning}
    One can instead take the full $\vartopos$-subcategory of $\universe$ spanned by the morphisms in $E$, as in \cite[beginning of page 9]{MWPresentabilityAndTopoi}. Our construction also restricts the morphisms between these objects. We discuss the full subcategory below.
\end{warning}

\begin{proposition}\label{prop:from_local_class_to_subuniverse}
    Let $E$ be a local class of morphisms in $\vartopos$ that is closed under composition and contains equivalences. If the associated $\vartopos$-subcategory $\subuniverse$ is small, then it is a context-free subuniverse.

    Moreover, $E$ is left cancellable if and only if $\subuniverse$ is full, and every morphism in $E$ is locally truncated if and only if $\subuniverse$ is truncated.
\end{proposition}
\begin{proof}
    Independence of context (\Cref{cond:idependent_of_context}) follows directly from the construction. Smallness (\Cref{cond:srs_is_kappa_small}) holds by assumption, and pointedness (\Cref{cond:srs_has_iso}) follows because $E$ contains equivalences.

    Suppose that $E$ is left cancellable, and consider a commutative triangle
    \[
    \begin{tikzcd}
        C\ar[rd,"g"']\ar[rr,"f"]&&D\ar[ld,"h"]\\
        &A.
    \end{tikzcd}
    \]
    If $g,h\in E$, then $f\in E$ by left cancellation. Thus every morphism in $\universe(A)$ between objects of $\subuniverse(A)$ belongs to $\subuniverse(A)$, so $\subuniverse$ is full. Conversely, fullness applied to such a triangle shows that $E$ is left cancellable.

    Finally, a morphism $f:B\to A$ in $\vartopos$ is truncated if and only if $[f:B\to A]$ is a truncated object of $\vartopos_{/A}\simeq\universe(A)$. Locality on the target corresponds to locality in the context, proving the last assertion.
\end{proof}

\begin{remark}
    This construction gives a one-to-one correspondence between context-free subuniverses and local classes that are closed under composition, contain equivalences, and have small associated $\vartopos$-subcategories of $\universe$.
\end{remark}

\begin{example}\label{example:Fin_in_Ani}
    Let $\vartopos=\catofanima$ be the $\infty$-category of anima, and let $\varintcat{Fin}\subset\universe$ correspond to the local class of finite covering maps. By \Cref{prop:from_local_class_to_subuniverse}, $\varintcat{Fin}$ is inductible: finite covering maps are $0$-truncated and satisfy left cancellation.
\end{example}

\begin{example}
    For any $\infty$-topos $\vartopos$, the smallest context-free subuniverse is $\explicitset{*}\subset\universe$, corresponding to the local class of equivalences.
\end{example}

For a context-free subuniverse $\subuniverse$, let $\subuniverse^{\full}$ denote the full $\vartopos$-subcategory of $\universe$ spanned by the objects of $\subuniverse$.

\begin{remark}
    A context-free subuniverse $\subuniverse$ is full if and only if $\subuniverse=\subuniverse^{\full}$.
\end{remark}

\begin{remark}
    The full $\vartopos$-subcategory $\subuniverse^{\full}$ need not be a context-free subuniverse. It is context-free precisely when $\subuniverse=\subuniverse^{\full}$.
\end{remark}

Let $[B\to A]\in\universe(A)$. A morphism $C\to B$ determines a $B$-shaped diagram in $\universe$ in context $A$, whose colimit is the composite $C\to B\to A$. Consequently, a full context-free subuniverse $\subuniverse$ is right regular: it is closed under $\subuniverse$-colimits in $\universe$. More generally, we have the following.

\begin{proposition}
    Let $\subuniverse$ be a context-free subuniverse. Then $\subuniverse^{\full}$ is right regular. That is, $\subuniverse^{\full}$ is $\subuniverse^{\full}$ cocomplete and the inclusion  $\subuniverse^{\full}\into \universe$  is $\subuniverse^{\full}$-cocontinuous.
\end{proposition}
\begin{proof}
    Let $[B\to A]\in\subuniverse(A)$, and let $c:B\to\pi_A^*\subuniverse^{\full}$ be a $B$-shaped diagram in context $A$. Such a diagram corresponds to an object $[C\to B]\in\subuniverse(B)$. Its colimit in $\universe(A)$ is the composite $C\to B\to A$, which belongs to $\subuniverse(A)$ by independence of context.
\end{proof}

Combining this with the formula for free cocompletions in \cite{MWcocomplete} gives the following universal property.

\begin{corollary}\label{prop:S_is_free_with_S_colmitis}
    Let $\subuniverse$ be a context-free subuniverse. Then $\subuniverse^{\full}$ is the free $\subuniverse$-cocompletion, equivalently the free $\subuniverse^{\full}$-cocompletion, of the point. Thus, for every $\subuniverse$-cocomplete $\vartopos$-category $\varintcat{C}$, restriction to the point induces an equivalence of $\vartopos$-categories
    \[
    \intfun^{\subuniverse\text{-}\sqcup}(\subuniverse^{\full},\varintcat{C})
    \simeq\intfun(*,\varintcat{C})\simeq\varintcat{C},
    \]
    where the left-hand side denotes the $\vartopos$-category of $\subuniverse$-cocontinuous functors.
\end{corollary}
\begin{proof}
    By \cite[Theorem 7.1.13]{MWcocomplete}, the free $\subuniverse$-cocompletion of $*$ is the full $\vartopos$-subcategory of $\intpresh(*)\simeq\universe$ generated by $*$ under $\subuniverse^{\full}$-colimits. This is $\subuniverse^{\full}$: right regularity gives closure under these colimits, and every object of $\subuniverse^{\full}$ is the colimit of the constant point indexed by itself.
\end{proof}

\subsubsection*{Span categories and $\subuniverse$-monoids}

We now define monoids that admit summation along morphisms in $\subuniverse$. Restriction will be allowed along every morphism in $\subuniverse^{\full}$, including when $\subuniverse$ is not full. The relevant span category is provided by the following proposition.

\begin{proposition}
    Let $\subuniverse$ be a context-free subuniverse. Then $(\subuniverse^{\full},\subuniverse)$ is a $\vartopos$-span pair.
\end{proposition}
\begin{proof}
    We must show that $\subuniverse(A)$ is stable under base change in $\subuniverse^{\full}(A)$ for every context $A$, and that restriction functors preserve these pullbacks.

    Let $f:C\to D$ belong to $\subuniverse(A)$, and let $g:B\to D$ belong to $\subuniverse^{\full}(A)$. By independence of context, $[f:C\to D]$ is an object of $\subuniverse(D)$. Hence
    \[
    [C\times_D B\to B]\simeq g^*([f:C\to D])\in\subuniverse(B).
    \]
    Applying independence of context again shows that $C\times_D B\to B$ belongs to $\subuniverse(A)$.

    For a morphism $p:B\to A$ in $\vartopos$, the restriction functor $p^*:\universe(A)\to\universe(B)$ preserves pullbacks. Its restriction to $\subuniverse^{\full}$ therefore preserves the pullbacks in question.
\end{proof}

We may thus form the span $\vartopos$-category $\spanEfE$. Recall that finite products in $\Span(\Fin)$ are given by finite coproducts in $\Fin$. The following is the corresponding internal statement.

\begin{proposition}\label{prop:limits_in_span_E}
    Let $\subuniverse$ be a context-free subuniverse. The $\vartopos$-category $\spanEfE$ is $\subuniverse^{\full}$-complete, and the inclusion
    \[
    (\subuniverse^{\full})^{\op}\into\spanEfE
    \]
    is $\subuniverse^{\full}$-continuous.
\end{proposition}
\begin{proof}
    We argue as in \cite[Proposition 4.5]{CLLambi}. For $[p:B\to A]\in\subuniverse(A)$, we verify that the adjunction
    \[
    p_\sharp:\subuniverse^{\full}(B)\fromto\subuniverse^{\full}(A):p^*
    \]
    is an adjunction in $\adtriptwo^L$ (see \Cref{def:catsofspans:adjuncions}). The restriction functor $p^*$ is already a functor of adequate triplets. We must check this for $p_\sharp$ and verify that the unit and counit are left-pointing natural transformations of functors of adequate triplets.

    By independence of context, $p_\sharp$ restricts to a functor $\subuniverse(B)\to\subuniverse(A)$. Moreover, $p_\sharp:\universe(B)\to\universe(A)$ preserves pullbacks, since it is the forgetful functor between slice $\infty$-categories. Thus $p_\sharp$ is a functor of adequate triplets.

    All morphisms in $\subuniverse^{\full}$ are left-pointing, so the components of the unit and counit are left-pointing. To check the remaining condition, we show that their naturality squares are cartesian. In fact, this holds for all morphisms $C\to C'$ in $\subuniverse^{\full}(A)$ and $D\to D'$ in $\subuniverse^{\full}(B)$:
    \[
    \begin{tikzcd}
        p_\sharp p^*C\ar[r]\ar[d] & C\ar[d]
        &D\ar[r]\ar[d] & p^*p_\sharp D\ar[d]\\
        p_\sharp p^*C'\ar[r] & C'
        &D'\ar[r] & p^*p_\sharp D'.
    \end{tikzcd}
    \]
    We check the counit square; the unit square is analogous. Under the inclusion into $\universe(A)\simeq\vartopos_{/A}$, the identification
    \[
    p_\sharp p^*[C\to A]\simeq[B\times_A C\to A]
    \]
    identifies it with the cartesian square
    \[
    \begin{tikzcd}
        B\times_A C\ar[r]\ar[d] & C\ar[d]\\
        B\times_A C'\ar[r] & C'.
    \end{tikzcd}
    \]
    For the unit square, one similarly uses $p^*p_\sharp[D\to B]\simeq[B\times_A D\to B]$.

    The span construction reverses this adjunction, making $\Span(p_\sharp)$ right adjoint to restriction. Its restriction to $(\subuniverse^{\full})^{\op}$ is the corresponding limit functor, proving both assertions.
\end{proof}

\begin{warning}
    If $\subuniverse$ is not full, $\spanEfE$ need not be $\subuniverse^{\full}$-cocomplete. The preceding adjunction need not lie in $\adtriptwo^R$: the unit maps $D\to p^*p_\sharp D$ need not belong to $\subuniverse$.
\end{warning}

Preservation of these limits can be checked after restriction to $(\subuniverse^{\full})^{\op}$.

\begin{proposition}\label{prop:limits_in_span_E_are_computed_in_E_op}
    Let $\subuniverse$ be a context-free subuniverse, and let $\varintcat{C}$ be an $\subuniverse^{\full}$-complete $\vartopos$-category. A $\vartopos$-functor $F:\spanEfE\to\varintcat{C}$ is $\subuniverse^{\full}$-continuous if and only if the composite
    \[
    (\subuniverse^{\full})^{\op}\into\spanEfE\xto{F}\varintcat{C}
    \]
    is $\subuniverse^{\full}$-continuous.
\end{proposition}
\begin{proof}
    The forward implication follows from \Cref{prop:limits_in_span_E}.

    Conversely, suppose that the composite is $\subuniverse^{\full}$-continuous. Let $B\in\subuniverse^{\full}(A)$, and let $d:B\to\pi_A^*\spanEfE$ be a diagram in context $A$. The inclusion $(\subuniverse^{\full})^{\op}\into\spanEfE$ is an equivalence on underlying $\vartopos$-groupoids, so $d$ lifts to a diagram
    \[
    \widetilde{d}:B\to\pi_A^*(\subuniverse^{\full})^{\op}.
    \]
    Both the inclusion and its composite with $F$ preserve the limit of $\widetilde{d}$. Hence $F$ preserves the limit of $d$.
\end{proof}

The following definition is due to \cite{CLLambi} when $\subuniverse$ is full.

\begin{definition}\label{def:E_monoids}
    Let $\varintcat{C}$ be an $\subuniverse^{\full}$-complete $\vartopos$-category. An \emph{$\subuniverse$-monoid} in $\varintcat{C}$ in context $A$ is an $\subuniverse^{\full}$-continuous $\vartopos$-functor
    \[
    \spanEfE\to\varintcat{C}
    \]
    defined in context $A$.

    The $\infty$-category of $\subuniverse$-monoids in $\varintcat{C}$ is
    \[
    \cmon^{\subuniverse}(\varintcat{C})
    :=\Fun^{\subuniverse\text{-}\sqcap}(\spanEfE,\varintcat{C}),
    \]
    and the $\vartopos$-category of $\subuniverse$-monoids in $\varintcat{C}$ is
    \[
    \intsmon(\varintcat{C})
    :=\intfun^{\subuniverse\text{-}\sqcap}(\spanEfE,\varintcat{C}).
    \]
    Thus an object of $\intsmon(\varintcat{C})(A)$ is an $\subuniverse$-monoid in context $A$.

    When $\varintcat{C}$ is omitted, it is understood to be $\universe$. When the context is omitted, it is understood to be $A=*$.
\end{definition}

\begin{example}\label{example:Fin_monoids}
    Let $\vartopos=\catofanima$, and let $\varintcat{Fin}\subset\universe$ be the context-free subuniverse of finite covering maps from \Cref{example:Fin_in_Ani}. A $\varintcat{Fin}$-monoid in $\varintcat{C}$ is the same as a commutative monoid in $\varintcat{C}(*)$. Indeed, global sections identify the $\infty$-category of $\catofanima$-categories with $\catofcats$.
\end{example}

We next generalize the familiar description of $\Fin^{\simeq}$ as the free commutative monoid on one generator. For an $\subuniverse$-monoid $F:\spanEfE\to\universe$, evaluation at the distinguished point of $\spanEfE$ gives its underlying object $F(*)$. This defines the forgetful $\vartopos$-functor
\[
\intsmon(\universe)\to\universe,
\qquad F\mapsto F(*).
\]

\begin{proposition}\label{example:free_E_monoid}
    The $\vartopos$-groupoid $\subuniverse^{\simeq}$ carries a natural $\subuniverse$-monoid structure. With this structure, it is the free $\subuniverse$-monoid on one generator: it corepresents the forgetful $\vartopos$-functor $\intsmon(\universe)\to\universe$.
\end{proposition}
\begin{proof}
    The representable $\vartopos$-functor
    \[
    \intmap(*,-):\spanEfE\to\universe
    \]
    preserves limits and hence defines an $\subuniverse$-monoid. Its underlying object is $\intmap(*,*)\simeq\subuniverse^{\simeq}$.

    By the internal Yoneda lemma, for every $\vartopos$-functor $F:\spanEfE\to\universe$, there is a natural equivalence of $\vartopos$-anima
    \[
    \intmap(\intmap(*,-),F)\simeq F(*).
    \]
    Since $\intsmon(\universe)$ is a full $\vartopos$-subcategory of $\intfun(\spanEfE,\universe)$, this proves the universal property.
\end{proof}

We will also use the following presentability result.

\begin{proposition}\label{prop:cmon_is_presentable}
    The $\vartopos$-category $\intsmon(\universe)$ is $\vartopos$-presentable.
\end{proposition}
\begin{proof}
    We show that
    \[
    \intsmon(\universe)\subset\intfun(\spanEfE,\universe)
    \]
    is a Bousfield localization at a small $\vartopos$-category of morphisms. This suffices by \cite[Corollary 2.4.2.9 and Theorem 2.4.2.5]{MWPresentabilityAndTopoi}.

    Let $F:\spanEfE\to\universe$ be a $\vartopos$-functor. Choose $[C\to B]\in\subuniverse(B)$ and $[p:B\to A]\in\subuniverse(A)$, and let $x$ denote $[C\to B]$, viewed as a $B$-shaped diagram in $\spanEfE$ in context $A$. By \Cref{prop:limits_in_span_E}, its limit $p_*x$ is $[C\to A]$.

    Write
    \[
    \yo(A):\spanEfE^{\op}(A)\into\intfun(\spanEfE,\universe)(A)
    \]
    for the internal Yoneda embedding in context $A$. Thus $\yo(A)(D)$ is the internally representable functor associated to $D\in\spanEfE(A)$. By the internal Yoneda lemma and the adjunction for $p_\sharp$, we have
    \[
    F(p_*x)\simeq\intmap(\yo(A)(p_*x),F),
    \qquad
    p_*F(x)\simeq\intmap(p_\sharp\yo(B)(x),F),
    \]
    where the mapping anima are formed in context $A$. The comparison map $F(p_*x)\to p_*F(x)$ is obtained by applying $\intmap(-,F)$ to the canonical morphism
    \[
    p_\sharp\yo(B)(x)\to\yo(A)(p_*x).
    \]

    Consequently, $\intsmon(\universe)$ is the full $\vartopos$-subcategory of objects that are Bousfield local with respect to these morphisms. They generate a small $\vartopos$-category: they lie in the small $\vartopos$-subcategory generated from representables under $\subuniverse^{\full}$-colimits.
\end{proof}

\subsubsection*{Extra functoriality in the complete case}

For complete $\vartopos$-categories, we can describe $\subuniverse$-monoids using a larger span category. We begin by extending the permitted summation maps to a wide $\vartopos$-subcategory of $\universe$.

\begin{definition}
    Let $\subuniverse$ be a context-free subuniverse. Define $\universe_{\subuniverse}\subset\universe$ to be the wide $\vartopos$-subcategory whose morphisms in context $A$ are maps $C\to B$ over $A$ such that $[C\to B]\in\subuniverse(B)$.
\end{definition}

The resulting span category has the same limit properties as the one considered above, now for all indexing objects in $\universe$.

\begin{lemma}\label{prop:groupoid_limits_span_universe_E}
    The pair $(\universe,\universe_{\subuniverse})$ is a $\vartopos$-span pair. The $\vartopos$-category $\intspan(\universe,\universe_{\subuniverse})$ is $\universe$-complete, and the inclusion
    \[
    \universe^{\op}\into\intspan(\universe,\universe_{\subuniverse})
    \]
    is $\universe$-continuous.

    Moreover, for a $\vartopos$-category $\varintcat{C}$, a $\vartopos$-functor $F:\intspan(\universe,\universe_{\subuniverse})\to\varintcat{C}$ is $\universe$-continuous if and only if the composite
    \[
    \universe^{\op}\into\intspan(\universe,\universe_{\subuniverse})\xto{F}\varintcat{C}
    \]
    is $\universe$-continuous.
\end{lemma}
\begin{proof}
    The argument is the same as in \Cref{prop:limits_in_span_E,prop:limits_in_span_E_are_computed_in_E_op}.
\end{proof}

We can now give the promised formula for $\subuniverse$-monoids in a complete $\vartopos$-category.

\begin{proposition}\label{prop:extra_back_functoriality_for_E_monoidal_cats}
    Let $\varintcat{C}$ be a complete $\vartopos$-category. Restriction induces an equivalence of $\vartopos$-categories
    \[
    \intfun^{\universe\text{-}\sqcap}(\intspan(\universe,\universe_{\subuniverse}),\varintcat{C})
    \simeq\intfun^{\subuniverse\text{-}\sqcap}(\spanEfE,\varintcat{C})
    =\intsmon(\varintcat{C}),
    \]
    with inverse given by right Kan extension.
\end{proposition}
\begin{proof}
    We follow the proof of \cite[Lemma 5.21]{CLLambi}. Restriction along the fully faithful inclusion $\spanEfE\into\intspan(\universe,\universe_{\subuniverse})$ gives a $\vartopos$-functor
    \[
    F:\intfun(\intspan(\universe,\universe_{\subuniverse}),\varintcat{C})
    \to\intfun(\spanEfE,\varintcat{C}).
    \]
    By \Cref{prop:limits_in_span_E_are_computed_in_E_op,prop:groupoid_limits_span_universe_E}, this restricts to
    \[
    G:\intfun^{\universe\text{-}\sqcap}(\intspan(\universe,\universe_{\subuniverse}),\varintcat{C})
    \to\intfun^{\subuniverse\text{-}\sqcap}(\spanEfE,\varintcat{C}).
    \]
    Indeed, a $\vartopos$-functor out of $\intspan(\universe,\universe_{\subuniverse})$ is $\universe$-continuous if and only if its restriction to $\universe^{\op}$ is right Kan extended from the point. Its restriction to $(\subuniverse^{\full})^{\op}$ is then also right Kan extended from the point.

    To show that $G$ is an equivalence, we apply \cite[Proposition 5.18]{CLLambi}. Both span categories have factorization systems given by backward-facing and forward-facing morphisms, as discussed after \cite[Convention 5.19]{CLLambi}. Their inclusion is compatible with these factorization systems. Moreover, if $X\in\spanEfE(A)$ and $Y\to X$ is a forward-facing morphism in $\intspan(\universe,\universe_{\subuniverse})(A)$, then $Y$ also belongs to $\spanEfE(A)$. Thus
    \[
    \spanEfE^{\op}\into\intspan(\universe,\universe_{\subuniverse})^{\op}
    \]
    is a good inclusion in the sense of \cite[Definition 5.10]{CLLambi}, by \cite[Example 5.8]{CLLambi}. Proposition 5.18 of that work therefore identifies $G$ as an equivalence, with inverse given by right Kan extension.
\end{proof}

\subsubsection*{Free $\subuniverse$-monoids}

Recall that the free $\infty$-category with finite coproducts on an anima $A$ is the $\infty$-category of finite sets equipped with a map to $A$. Under the straightening equivalence, representable presheaves on $A$ correspond to maps $X\to A$ with contractible source. Their finite coproducts are therefore precisely the maps to $A$ with finite discrete source.

The underlying groupoid of this $\infty$-category is the free commutative monoid on $A$. It is canonically equivalent to the familiar formula
\[
\coprod_{n\in\nats}(A^n)//\Sigma_n.
\]
We now give the corresponding constructions for $\subuniverse$-colimits and $\subuniverse$-monoids.

\begin{definition}\label{def:free_on_A_with_E_colimits}
    For $A\in\vartopos$, let $\subuniverse^{\full}[A]$ denote the full $\vartopos$-subcategory of $\intpresh(A)$ generated by the image of the Yoneda embedding under $\subuniverse^{\full}$-colimits.
\end{definition}

By straightening, $\intpresh(A)(B)\simeq\vartopos_{/A\times B}$. For a morphism $p:B\to B'$ in $\vartopos$, the colimit functor
\[
p_\sharp:\intpresh(A)(B)\to\intpresh(A)(B')
\]
sends $[X\to A\times B]$ to the composite $[X\to A\times B\to A\times B']$.

In context $B$, the Yoneda embedding $A\into\intpresh(A)$ sends a morphism $s:B\to A$ to its graph
\[
B\xto{\Delta}B\times B\xto{s\times\id}A\times B.
\]
It follows that $\subuniverse^{\full}[A](B)$ is the full subcategory of $\vartopos_{/A\times B}$ spanned by maps $C\to A\times B$ for which the composite $p:C\to A\times B\to B$ belongs to $\subuniverse(B)$. Indeed, such an object is the $p$-indexed colimit of the Yoneda image of the composite $C\to A\times B\to A$.

\begin{remark}
    When $\vartopos=\catofanima$ and $\subuniverse=\varintcat{Fin}$, the underlying groupoid of $\varintcat{Fin}[A](*)$ is the anima of maps from finite sets to $A$. Thus it is equivalent to
    \[
    \coprod_{n\in\nats}(A^n)//\Sigma_n,
    \]
    as above.
\end{remark}

\begin{proposition}\label{prop:free_on_A_with_E_colimits}
    Let $A\in\vartopos$, and let $\varintcat{C}$ be an $\subuniverse^{\full}$-cocomplete $\vartopos$-category. Restriction along the Yoneda embedding induces an equivalence of $\vartopos$-categories
    \[
    \intfun^{\subuniverse\text{-}\sqcup}(\subuniverse^{\full}[A],\varintcat{C})
    \simeq\intfun(A,\varintcat{C}).
    \]
\end{proposition}
\begin{proof}
    This is an application of \cite[Theorem 7.1.13]{MWcocomplete}.
\end{proof}

The underlying groupoid of this free cocompletion also admits a description by spans.

\begin{proposition}\label{prop:the_underlying_groupoid_of_the_free_cat_with_colims_on_E}
    For $A\in\vartopos$, there is a natural equivalence of $\vartopos$-anima
    \[
    \intmap_{\intspan(\universe,\universe_{\subuniverse})}(A,*)
    \simeq\subuniverse^{\full}[A]^{\simeq}.
    \]
    Thus $\subuniverse^{\full}[A]^{\simeq}$ is the underlying object of the $\subuniverse$-monoid obtained by restricting the representable functor $\intmap_{\intspan(\universe,\universe_{\subuniverse})}(A,-)$.
\end{proposition}
\begin{proof}
    In context $B$, a morphism from $A$ to $*$ in $\intspan(\universe,\universe_{\subuniverse})$ is a span over $B$
    \[
    A\times B\longleftarrow X\longrightarrow B
    \]
    such that $[X\to B]\in\subuniverse(B)$. The anima of these spans is the core of the full subcategory of $\vartopos_{/A\times B}$ on objects whose projection to $B$ belongs to $\subuniverse(B)$. By the description above, this is $\subuniverse^{\full}[A](B)^{\simeq}$. These equivalences are compatible with pullback in $B$, and hence give the asserted equivalence of $\vartopos$-anima.
\end{proof}

For the rest of the paper, we write $\subuniverse^{\simeq}[A]:=\subuniverse^{\full}[A]^{\simeq}$, with the $\subuniverse$-monoid structure just described. The following universal property identifies it as the free $\subuniverse$-monoid on $A$.

\begin{proposition}\label{prop:free_E_monoid_on_A}
    Let $A\in\vartopos$, and let $\varintcat{M}$ be an $\subuniverse$-monoid. There is a natural equivalence of $\vartopos$-anima
    \[
    \intmap_{\intsmon(\universe)}(\subuniverse^{\simeq}[A],\varintcat{M})
    \simeq\intmap(A,\varintcat{M}(*)),
    \]
    where $\varintcat{M}(*)$ is the underlying object of $\varintcat{M}$.
\end{proposition}
\begin{proof}
    By \Cref{prop:extra_back_functoriality_for_E_monoidal_cats}, $\varintcat{M}$ extends to a $\universe$-continuous $\vartopos$-functor
    \[
    \widetilde{\varintcat{M}}:\intspan(\universe,\universe_{\subuniverse})\to\universe.
    \]
    By \Cref{prop:the_underlying_groupoid_of_the_free_cat_with_colims_on_E} and the internal Yoneda lemma, the mapping $\vartopos$-anima on the left is equivalent to $\widetilde{\varintcat{M}}(A)$. Since $\widetilde{\varintcat{M}}$ is $\universe$-continuous, this is $\intmap(A,\varintcat{M}(*))$, as required. This is the same argument as in \Cref{example:free_E_monoid}.
\end{proof}

\subsection{\texorpdfstring{$\subuniverse$}{E}-monoidal categories.}\label{subsec:E_monoidal_categories}

We now apply the construction of $\subuniverse$-monoids to the $\vartopos$-category of $\vartopos$-categories. The resulting $\subuniverse$-monoidal structures allow tensor products indexed by $\subuniverse$, generalizing tensor products over finite sets in symmetric monoidal categories. We begin by describing the underlying category and these indexed tensor products (\Cref{prop:values_of_E_monoidal_categories,notation:exponent_and_lower_otimes}).

We then introduce the internal categories of $\subuniverse$-monoidal functors (\Cref{cor:cmon_is_tensored_and_powered_by_cat}) and describe free $\subuniverse$-monoidal categories on $\vartopos$-anima (\Cref{example:free_E_monoidal_cat}). Finally, we explain when subcategories inherit an $\subuniverse$-monoidal structure (\Cref{prop:monoidal_subcategories_are_monoidal}) and prove an analogue of the symmetric monoidal localization theorem (\Cref{thm:monoidal_localization}).

\subsubsection*{Indexed tensor products}

\begin{example}
    The $\vartopos$-category $\internalcatofcats$ is $\vartopos$-presentable and hence admits all $\vartopos$-limits and colimits. An \emph{$\subuniverse$-monoidal $\vartopos$-category} is an $\subuniverse$-monoid in $\internalcatofcats$. The $\vartopos$-category of these objects is $\intsmon(\internalcatofcats)$.
\end{example}

As with symmetric monoidal categories, we regard an $\subuniverse$-monoidal structure as additional structure on an underlying category.

\begin{definition}
    Let $F:\spanEfE\to\internalcatofcats$ be an $\subuniverse$-monoidal $\vartopos$-category. Its \emph{underlying $\vartopos$-category} is $\varintcat{C}:=F(*)$, obtained by restriction along $\explicitset{*}\into\spanEfE$. We will often say that $\varintcat{C}$ is $\subuniverse$-monoidal or that it is equipped with an $\subuniverse$-monoidal structure.
\end{definition}

For a symmetric monoidal $\infty$-category $F:\Span(\Fin)\to\catofcats$, one has $F(A)\simeq F(*)^A$. The following proposition gives the internal analogue.

\begin{proposition}\label{prop:values_of_E_monoidal_categories}
    Let $F:\spanEfE\to\internalcatofcats$ be an $\subuniverse$-monoidal $\vartopos$-category with underlying category $\varintcat{C}$. For every context $A\in\vartopos$ and every $[s:B\to A]\in\subuniverse(A)$, there are canonical equivalences of $\vartopos_{/A}$-categories
    \[
    F(B)\simeq s_*s^*\varintcat{C}
    \simeq\intfun_{\vartopos_{/A}}(B,\varintcat{C}),
    \]
    where $\varintcat{C}$ on the right is restricted to context $A$.
\end{proposition}
\begin{proof}
    The limit functor $s_*:\spanEfE(B)\to\spanEfE(A)$ sends $[\id:B\to B]$ to $[s:B\to A]$. Since $F$ is $\subuniverse^{\full}$-continuous, this gives the first equivalence. The second is the formula for this limit in $\internalcatofcats(A)$.
\end{proof}

This description motivates the following notation for indexed tensor products.

\begin{notation}\label{notation:exponent_and_lower_otimes}
    Let $F:\spanEfE\to\internalcatofcats$ be an $\subuniverse$-monoidal $\vartopos$-category with underlying category $\varintcat{C}$. For $[B\to A]\in\subuniverse(A)$, write
    \[
    \varintcat{C}^{B}:=F(B\to A)
    \simeq\intfun_{\vartopos_{/A}}(B,\varintcat{C})
    \in\internalcatofcats(A).
    \]
    For a morphism $s:C\to B$ in $\subuniverse(A)$, the span $C=C\xto{s}B$ induces a $\vartopos_{/A}$-functor
    \[
    s_\otimes:\varintcat{C}^{C}\to\varintcat{C}^{B}.
    \]
    Evaluating in context $A$ gives a functor of $\infty$-categories, denoted by the same symbol:
    \[
    s_\otimes:\varintcat{C}(C)
    \simeq(\varintcat{C}^{C})(A)
    \to(\varintcat{C}^{B})(A)
    \simeq\varintcat{C}(B).
    \]

    More generally, a span $B\xfrom{t}X\xto{s}C$ in $\spanEfE(A)$ induces the $\vartopos_{/A}$-functor
    \[
    s_\otimes t^*:\varintcat{C}^{B}\to\varintcat{C}^{C}.
    \]
    Its evaluation in context $A$ is the functor
    \[
    s_\otimes t^*:\varintcat{C}(B)\to\varintcat{C}(C).
    \]
\end{notation}

\subsubsection*{Monoidal functors and free examples}

We next equip $\intsmon(\internalcatofcats)$ with internal categories of $\subuniverse$-monoidal functors. These arise from the following module structure.

\begin{proposition}\label{prop:cmon_is_cat_module_in_prl}
    The $\vartopos$-category $\intsmon(\internalcatofcats)$ is a module over $\internalcatofcats$ in $\prl(\vartopos)$.
\end{proposition}
\begin{proof}
    There are equivalences
    \begin{align*}
        \intfun(\spanEfE,\internalcatofcats)
        &\simeq\intfun^R(\intpresh(\spanEfE^{\op})^{\op},\internalcatofcats)\\
        &\simeq\intpresh(\spanEfE^{\op})\otimes\internalcatofcats
    \end{align*}
    in $\prl(\vartopos)$, where $\otimes$ is the multilinear tensor product of presentable $\vartopos$-categories developed in \cite[\S 2.6]{MWPresentabilityAndTopoi}. By the universal property of the Bousfield localization from \Cref{prop:cmon_is_presentable}, the $\vartopos$-category
    \[
    \intfun^R(\intsmon(\universe)^{\op},\internalcatofcats)
    \]
    corresponds compatibly to the $\subuniverse$-continuous functors. Thus
    \[
    \intsmon(\internalcatofcats)
    \simeq\intsmon(\universe)\otimes\internalcatofcats,
    \]
    which gives the asserted module structure.
\end{proof}

This module structure supplies mapping categories, cotensors, and tensors.

\begin{corollary}\label{cor:cmon_is_tensored_and_powered_by_cat}
    There are $\vartopos$-functors
    \[
    \intfun^{\subuniverse\text{-}\otimes}(-,-):
    \intsmon(\internalcatofcats)^{\op}\times\intsmon(\internalcatofcats)
    \to\internalcatofcats
    \]
    and
    \[
    (-)^{(-)}:\intsmon(\internalcatofcats)\times\internalcatofcats^{\op}
    \to\intsmon(\internalcatofcats),
    \]
    each preserving limits separately in both variables. There is also a bilinear $\vartopos$-functor
    \[
    -\otimes-:\intsmon(\internalcatofcats)\times\internalcatofcats
    \to\intsmon(\internalcatofcats).
    \]
    For every $\vartopos$-category $\varintcat{C}$, there is a canonical equivalence of $\vartopos$-categories, natural in both remaining variables,
    \[
    \intfun^{\subuniverse\text{-}\otimes}(-\otimes\varintcat{C},-)
    \simeq\intfun^{\subuniverse\text{-}\otimes}(-,(-)^{\varintcat{C}}).
    \]
    Moreover, taking the underlying groupoid of the mapping category recovers the mapping $\vartopos$-anima:
    \[
    \intfun^{\subuniverse\text{-}\otimes}(-,-)^{\simeq}
    \simeq\intmap_{\intsmon(\internalcatofcats)}(-,-).
    \]
\end{corollary}

Recall that the free commutative monoid on an anima $A$ is also the free symmetric monoidal $\infty$-category on $A$. The corresponding statement for $\subuniverse$-monoidal $\vartopos$-categories is as follows.

\begin{proposition}\label{example:free_E_monoidal_cat}
    Let $A\in\vartopos$, and let $\varintcat{C}$ be an $\subuniverse$-monoidal $\vartopos$-category. There is a natural equivalence of $\vartopos$-categories
    \[
    \intfun^{\subuniverse\text{-}\otimes}(\subuniverse^{\simeq}[A],\varintcat{C})
    \simeq\intfun(A,\varintcat{C}).
    \]
\end{proposition}
\begin{proof}
    This follows from \Cref{prop:free_E_monoid_on_A,prop:cmon_is_cat_module_in_prl}.
\end{proof}

\subsubsection*{Monoidal subcategories and localizations}

A subcategory of a symmetric monoidal $\infty$-category inherits a symmetric monoidal structure if it contains the unit and is closed under tensor products. We now formulate the corresponding condition for $\subuniverse$-monoidal $\vartopos$-categories.

\begin{definition}
    Let $\varintcat{C}$ be an $\subuniverse$-monoidal $\vartopos$-category. A $\vartopos$-subcategory $\varintcat{D}\subset\varintcat{C}$ is an \emph{$\subuniverse$-monoidal subcategory} if, for every $[s:B\to A]\in\subuniverse(A)$ and every morphism $f:X\to Y$ in $\varintcat{D}(B)$, the morphism
    \[
    s_\otimes(f):s_\otimes(X)\to s_\otimes(Y)
    \]
    belongs to $\varintcat{D}(A)$.
\end{definition}

\begin{proposition}\label{prop:monoidal_subcategories_are_monoidal}
    Let $F:\spanEfE\to\internalcatofcats$ be an $\subuniverse$-monoidal $\vartopos$-category with underlying category $\varintcat{C}$, and let $\varintcat{D}\subset\varintcat{C}$ be an $\subuniverse$-monoidal subcategory. There is a subfunctor $G\into F$ that equips $\varintcat{D}$ with an $\subuniverse$-monoidal structure.

    Moreover, $G\to F$ is terminal among $\subuniverse$-monoidal functors $T\to F$ whose underlying functor $T(*)\to\varintcat{C}$ factors through $\varintcat{D}$.
\end{proposition}
\begin{proof}
    In context $A$, define
    \[
    G(B):=\varintcat{D}^B\subset\varintcat{C}^B\simeq F(B).
    \]
    To obtain a subfunctor, it suffices to check compatibility with restriction along morphisms in $\subuniverse^{\full}(A)$ and tensor products along morphisms in $\subuniverse(A)$. Restriction preserves $\varintcat{D}$ because it is a $\vartopos$-subcategory of $\varintcat{C}$. The tensor functors preserve $\varintcat{D}$ by assumption. The formula $G(B)=\varintcat{D}^B$ also shows that $G$ is $\subuniverse^{\full}$-continuous, so it defines an $\subuniverse$-monoidal structure on $\varintcat{D}$.

    The inclusion $G\into F$ is a monomorphism in $\intfun(\spanEfE,\internalcatofcats)$. For an $\subuniverse$-monoidal functor $T\to F$, the underlying functor $T(*)\to\varintcat{C}$ factors through $\varintcat{D}$ if and only if, in every context $A$ and for every $B\in\spanEfE(A)$, the functor $T(B)\to F(B)$ factors through $\varintcat{D}^B=G(B)$. This proves the universal property.
\end{proof}

We conclude with an analogue of the symmetric monoidal localization theorem.

\begin{theorem}\label{thm:monoidal_localization}
    Let $F:\spanEfE\to\internalcatofcats$ be an $\subuniverse$-monoidal $\vartopos$-category with underlying category $\varintcat{C}$. Let $i:\varintcat{D}\into\varintcat{C}$ be a reflective (respectively, coreflective) $\vartopos$-subcategory, with reflection (respectively, coreflection) $L:\varintcat{C}\to\varintcat{D}$. Let
    \[
    \varintcat{S}:=L^{-1}(\varintcat{D}^{\simeq})\subset\varintcat{C}
    \]
    be the wide $\vartopos$-subcategory of morphisms inverted by $L$. Suppose that $\varintcat{S}$ is an $\subuniverse$-monoidal subcategory. Then $\varintcat{D}$ carries a canonical $\subuniverse$-monoidal structure for which $L$ is $\subuniverse$-monoidal.

    Moreover, for every $\subuniverse$-monoidal $\vartopos$-category $\varintcat{T}$, precomposition with $L$ defines a fully faithful $\vartopos$-functor
    \[
    \intfun^{\subuniverse\text{-}\otimes}(\varintcat{D},\varintcat{T})
    \to\intfun^{\subuniverse\text{-}\otimes}(\varintcat{C},\varintcat{T}),
    \]
    whose essential image consists of the functors that send every morphism in $\varintcat{S}$ to an equivalence.
\end{theorem}
\begin{proof}
    Apply \Cref{prop:monoidal_subcategories_are_monoidal} to $\varintcat{S}\subset\varintcat{C}$ to obtain a subfunctor $G\into F$. Define $H:\spanEfE\to\internalcatofcats$ by the pointwise pushout
    \[
    H(-):=G(-)^{\grpd}\coprod_{G(-)}F(-).
    \]
    For $[B\to A]\in\spanEfE(A)$, we claim that
    \[
    H(B)\simeq(\varintcat{S}^B)^{\grpd}
    \coprod_{\varintcat{S}^B}\varintcat{C}^B
    \simeq\varintcat{D}^B.
    \]
    Indeed, $\varintcat{S}^B$ is the wide subcategory of morphisms inverted by $L^B:\varintcat{C}^B\to\varintcat{D}^B$, so this is the pushout description of localization in \cite[Proposition 3.4.6]{MWcocomplete}.

    The identification $H(B)\simeq\varintcat{D}^B$ shows that $H$ is $\subuniverse^{\full}$-continuous. Thus $H$ equips $\varintcat{D}$ with an $\subuniverse$-monoidal structure, and the natural transformation $F\to H$ equips $L$ with an $\subuniverse$-monoidal structure. The universal property follows from that of the pushout.
\end{proof}

\subsection{Cartesian and cocartesian \texorpdfstring{$\subuniverse$}{E}-monoidal structures.}\label{subsec:cartesian_and_cocartesian_structures}

Finite coproducts and finite products give rise to the cocartesian and cartesian symmetric monoidal structures. In this subsection, we construct their analogues for $\subuniverse$-monoidal $\vartopos$-categories.

Using internal unfurling (\Cref{const:internal_unfurling}), we equip every $\subuniverse^{\full}$-cocomplete $\vartopos$-category with a cocartesian $\subuniverse$-monoidal structure, whose $\subuniverse$-indexed tensor products are given by colimits (\Cref{prop:categories_with_E_colimits_have_E_monoidal_structure}). Dually, indexed limits define a cartesian $\subuniverse$-monoidal structure on every $\subuniverse^{\full}$-complete $\vartopos$-category (\Cref{prop:categories_with_E_limits_have_E_monoidal_structure}). These constructions are functorial with respect to functors preserving the corresponding colimits or limits.

\begin{proposition}\label{prop:categories_with_E_colimits_have_E_monoidal_structure}
    Let $\varintcat{C}$ be an $\subuniverse^{\full}$-cocomplete $\vartopos$-category. There is an $\subuniverse$-monoidal $\vartopos$-category $\equipcocart(\varintcat{C})$ with underlying category $\varintcat{C}$ such that, for every $[p:B\to A]\in\subuniverse(A)$,
    \[
    p_\otimes\simeq p_\sharp:\varintcat{C}(B)\to\varintcat{C}(A).
    \]
    This construction is functorial in $\subuniverse^{\full}$-cocontinuous $\vartopos$-functors.
\end{proposition}
\begin{proof}
    Recall that $\internalcatofcats^{\subuniverse\text{-}\sqcup}$ denotes the $\vartopos$-category of $\subuniverse^{\full}$-cocomplete $\vartopos$-categories and $\subuniverse^{\full}$-cocontinuous $\vartopos$-functors. It is closed under all $\vartopos$-limits in $\internalcatofcats$. By the dual of \Cref{prop:S_is_free_with_S_colmitis}, there are equivalences of $\vartopos$-categories
    \[
    \internalcatofcats^{\subuniverse\text{-}\sqcup}
    \simeq\intfun(*,\internalcatofcats^{\subuniverse\text{-}\sqcup})
    \simeq\intfun^{\subuniverse\text{-}\sqcap}
    ((\subuniverse^{\full})^{\op},\internalcatofcats^{\subuniverse\text{-}\sqcup}).
    \]

    Let $F:(\subuniverse^{\full})^{\op}\to\internalcatofcats^{\subuniverse\text{-}\sqcup}$ correspond to $\varintcat{C}$ under this equivalence. We claim that its underlying $\vartopos$-functor to $\internalcatofcats$ is left $\subuniverse$-adjointable in the sense of \Cref{const:internal_unfurling}. In each context, $F$ sends a morphism $p:B\to A$ in $\subuniverse$ to the restriction functor
    \[
    p^*:\varintcat{C}^{A}\to\varintcat{C}^{B}.
    \]
    Since $\varintcat{C}$ is $\subuniverse^{\full}$-cocomplete, this functor admits an internal left adjoint $p_\sharp$. The required Beck--Chevalley conditions are those for indexed colimits in $\varintcat{C}$. Moreover, $\subuniverse^{\full}$-cocontinuous functors commute with these indexed colimits and therefore induce morphisms of left $\subuniverse$-adjointable functors.

    Internal unfurling thus gives a $\vartopos$-functor
    \[
    \intfun^{\subuniverse\text{-}\sqcap}
    ((\subuniverse^{\full})^{\op},\internalcatofcats^{\subuniverse\text{-}\sqcup})
    \to\intfun(\spanEfE,\internalcatofcats).
    \]
    The unfurled functor still takes values in $\internalcatofcats^{\subuniverse\text{-}\sqcup}$: the original restriction functors preserve these colimits, and so do the added left adjoints $p_\sharp$. By \Cref{prop:limits_in_span_E_are_computed_in_E_op}, it is $\subuniverse^{\full}$-continuous because its restriction to $(\subuniverse^{\full})^{\op}$ is. It therefore defines an $\subuniverse$-monoidal structure on $\varintcat{C}$, with $p_\otimes\simeq p_\sharp$ by the construction of unfurling. Functoriality follows from functoriality of unfurling.
\end{proof}

We call this the \emph{cocartesian $\subuniverse$-monoidal structure} and write the resulting $\vartopos$-functor as
\[
\equipcocart:\internalcatofcats^{\subuniverse\text{-}\sqcup}
\to\intsmon(\internalcatofcats).
\]

\begin{remark}
    When the cocartesian structure is understood, we will sometimes omit $\equipcocart$ from the notation.
\end{remark}

Passing to opposite categories gives the cartesian construction.

\begin{proposition}\label{prop:categories_with_E_limits_have_E_monoidal_structure}
    Let $\varintcat{C}$ be an $\subuniverse^{\full}$-complete $\vartopos$-category. There is an $\subuniverse$-monoidal $\vartopos$-category $\equipcart(\varintcat{C})$ with underlying category $\varintcat{C}$ such that, for every $[p:B\to A]\in\subuniverse(A)$,
    \[
    p_\otimes\simeq p_*:\varintcat{C}(B)\to\varintcat{C}(A).
    \]
    This construction is functorial in $\subuniverse^{\full}$-continuous $\vartopos$-functors.
\end{proposition}
\begin{proof}
    Apply \Cref{prop:categories_with_E_colimits_have_E_monoidal_structure} to $\varintcat{C}^{\op}$ and then take opposites of the resulting categories. The indexed colimit functors for $\varintcat{C}^{\op}$ become the indexed limit functors for $\varintcat{C}$.
\end{proof}

We call this the \emph{cartesian $\subuniverse$-monoidal structure} and write the resulting $\vartopos$-functor as
\[
\equipcart:\internalcatofcats^{\subuniverse\text{-}\sqcap}
\to\intsmon(\internalcatofcats).
\]

\subsection{The case of an inductible subuniverse \texorpdfstring{$\subuniverse[I]$}{I}.}\label{subsec:inductible_subuniverses}

We now assume that the indexing subuniverse $\subuniverse[I]$ is inductible. Under this assumption, equipping an $\subuniverse[I]$-cocomplete $\vartopos$-category with its cocartesian structure defines a fully faithful functor (\Cref{thm:uniqueness_of_cartesian_structures})     
    \[
    \equipcocart:\internalcatofcats^{\subuniverse[I]\text{-}\sqcup}
    \into\intsmon[I](\internalcatofcats).
    \]We begin by proving this fully-faithfulness and characterizing the essential images of these embeddings.

We then recall the ambidexterity theorem of Cnossen--Lenz--Linskens (\Cref{Thm:CLL_amby}) and use it to construct pointwise $\subuniverse[I]$-monoidal structures on categories of $\subuniverse[I]$-monoidal functors (\Cref{cor:cmon_is_self_enriched_if_inductible}).

We also define $\subuniverse$-commutative algebras and coalgebras for a general context-free subuniverse $\subuniverse$ (\Cref{def:E_algebras_and_coalgebras}). When $\subuniverse[I]$ is inductible, categories of $\subuniverse[I]$-commutative algebras are $\subuniverse[I]$-cocomplete  (\Cref{prop:calg_has_E_colimits}). Moreover, the construction of commutative algebras is right adjoint to the formation of the cocartesian structure (\Cref{thm:calg_is_right_adjoint}).

Finally, we characterize cocartesian structures by the requirement that the forgetful functor from commutative algebras be an equivalence (\Cref{thm:in_cocart_everything_is_alg}). We also prove the existence of a localization that turns an $\subuniverse[I]$-monoidal category into a cocartesian one (\Cref{thm:cocart_cats_are_localization_of_monoidal}).

\subsubsection*{Uniqueness and recognition of cartesian and cocartesian structures}

Let $\subuniverse[I]$ be an inductible context-free subuniverse. To prove uniqueness of the cocartesian structure, we first check that $\subuniverse[I]$ is truncated as a $\vartopos$-category.

\begin{lemma}\label{lem:every_morphism_in_I_is_locally_truncated_on_the_target}
    Let $f:C\to B$ be a morphism in $\subuniverse[I](A)$. There is a cover $(s_i:A_i\to A)_i$ such that each $s_i^*f$ is truncated.
\end{lemma}
\begin{proof}
    By assumption, $B$ and $C$ are locally truncated. After taking a common refinement of the corresponding covers of $A$, the morphism $f$ is a morphism between truncated objects and hence is truncated.
\end{proof}

We can therefore apply the uniqueness theorem for internal unfurling.

\begin{theorem}\label{thm:uniqueness_of_cartesian_structures}
    Let $\subuniverse[I]$ be an inductible context-free subuniverse. The $\vartopos$-functor
    \[
    \equipcocart:\internalcatofcats^{\subuniverse[I]\text{-}\sqcup}
    \to\intsmon[I](\internalcatofcats)
    \]
    is fully faithful.
\end{theorem}
\begin{proof}
    By \Cref{lem:every_morphism_in_I_is_locally_truncated_on_the_target}, the hypotheses of \Cref{thm:internal_uniqueness_of_unfurling} hold. Thus internal unfurling gives an equivalence of $\vartopos$-categories
    \[
    \intfun^{\subuniverse[I]\text{-adjointable}}(\subuniverse[I]^{\op},\internalcatofcats)
    \simeq
    \intfun^{\subuniverse[I]\text{-adjointed}}(\intspan(\subuniverse[I]),\internalcatofcats).
    \]
    By \Cref{prop:limits_in_span_E_are_computed_in_E_op}, it restricts to an equivalence
    \[
    \intfun^{\subuniverse[I]\text{-}\sqcap,\subuniverse[I]\text{-adjointable}}
    (\subuniverse[I]^{\op},\internalcatofcats)
    \simeq
    \intfun^{\subuniverse[I]\text{-}\sqcap,\subuniverse[I]\text{-adjointed}}
    (\intspan(\subuniverse[I]),\internalcatofcats).
    \]
    The right-hand side is a full $\vartopos$-subcategory of $\intsmon[I](\internalcatofcats)$.

    On the left, the equivalence
    \[
    \intfun^{\subuniverse[I]\text{-}\sqcap}(\subuniverse[I]^{\op},\internalcatofcats)
    \simeq\internalcatofcats
    \]
    coming from \Cref{prop:S_is_free_with_S_colmitis} identifies adjointable functors with $\subuniverse[I]$-cocomplete $\vartopos$-categories and adjointable natural transformations with $\subuniverse[I]$-cocontinuous $\vartopos$-functors. It therefore restricts to
    \[
    \intfun^{\subuniverse[I]\text{-}\sqcap,\subuniverse[I]\text{-adjointable}}
    (\subuniverse[I]^{\op},\internalcatofcats)
    \simeq\internalcatofcats^{\subuniverse[I]\text{-}\sqcup}.
    \]
    Under these identifications, unfurling is precisely $\equipcocart$. Hence $\equipcocart$ identifies its source with the full $\vartopos$-subcategory of $\subuniverse[I]$-adjointed objects in $\intsmon[I](\internalcatofcats)$.
\end{proof}

We next describe cartesian structures by an inductive condition on norm maps. Passing to opposites gives the corresponding description of the essential image of $\equipcocart$.

\begin{definition}[$m$-cartesian structures]\label{def:cartesian_S_monoidal_category}
    Let $\subuniverse[I]$ be an inductible context-free subuniverse, and let $\varintcat{C}$ be an $\subuniverse[I]$-complete $\subuniverse[I]$-monoidal $\vartopos$-category.
    \begin{mylist}
        \item For a $(-2)$-truncated morphism $s:A\to B$ in $\subuniverse[I]$, that is, an equivalence, there is a canonical equivalence
        \[
        N_s:s_\otimes\to s_*.
        \]
        We call this the $s$-norm. Every such $\varintcat{C}$ is declared to be $(-2)$-cartesian.
        \item Suppose that $\varintcat{C}$ is $n$-cartesian. For an $(n+1)$-truncated morphism $s:A\to B$ in $\subuniverse[I]$, define the $s$-norm $N_s:s_\otimes\to s_*$ to be the mate of the composite
        \[
        \begin{aligned}
        s^*s_\otimes
        &\simeq(\pi_1)_\otimes\pi_2^*
        \to(\pi_1)_\otimes\Delta_*\Delta^*\pi_2^*\\
        &\xto{N_\Delta^{-1}}
        (\pi_1)_\otimes\Delta_\otimes\Delta^*\pi_2^*
        \simeq\id_\otimes\id^*=\id,
        \end{aligned}
        \]
        where the maps are given by the diagram
        \[
        \begin{tikzcd}
            A\\
            &{A\times_B A}\arrow[rd,"\lrcorner"{anchor=center,pos=0.125},draw=none]&A\\
            &A&B.
            \arrow["\Delta"',from=1-1,to=2-2]
            \arrow[curve={height=-6pt},equal,from=1-1,to=2-3]
            \arrow[curve={height=6pt},equal,from=1-1,to=3-2]
            \arrow["\pi_2",from=2-2,to=2-3]
            \arrow["\pi_1"',from=2-2,to=3-2]
            \arrow["s",from=2-3,to=3-3]
            \arrow["s"',from=3-2,to=3-3]
        \end{tikzcd}
        \]
        The inverse of $N_\Delta$ exists because $\Delta$ is $n$-truncated and $\varintcat{C}$ is $n$-cartesian.
        \item We say that $\varintcat{C}$ is $(n+1)$-cartesian if it is $n$-cartesian and $N_s$ is an equivalence for every $(n+1)$-truncated morphism $s$ in $\subuniverse[I]$.
    \end{mylist}
    We say that $\varintcat{C}$ is \emph{inductively cartesian} if it is $\subuniverse[I]$-complete and is $n$-cartesian for every $n\geq-2$.
\end{definition}

\begin{proposition}
    An $\subuniverse[I]$-monoidal $\vartopos$-category is cartesian if and only if it is inductively cartesian.
\end{proposition}
\begin{proof}
    This follows from \Cref{thm:internal_uniqueness_of_unfurling} applied to the opposite $\subuniverse[I]$-monoidal structure.
\end{proof}

The dual definitions describe cocartesian structures.

\begin{para}[$m$-cocartesian structures]
    \begin{mylist}
        \item Let $\subuniverse[I]$ be an inductible context-free subuniverse.
        \item\label{prop:op_is_also_monoidal} If $\varintcat{C}$ is an $\subuniverse[I]$-monoidal $\vartopos$-category, then $\varintcat{C}^{\op}$ carries the opposite $\subuniverse[I]$-monoidal structure, defined by the composite
        \[
        \intspan(\subuniverse[I])\to\internalcatofcats
        \xto{(-)^{\op}}\internalcatofcats.
        \]
        \item We say that $\varintcat{C}$ is $m$-cocartesian if $\varintcat{C}^{\op}$ is $m$-cartesian, and that $\varintcat{C}$ is cocartesian if $\varintcat{C}^{\op}$ is cartesian.
        \item By \Cref{thm:uniqueness_of_cartesian_structures}, forgetting the cocartesian structure induces an equivalence of $\vartopos$-categories
        \[
        \intsmon[I](\internalcatofcats)^{\text{cocart}}
        \simeq\internalcatofcats^{\subuniverse[I]\text{-}\sqcup}.
        \]
    \end{mylist}
\end{para}

There is also a criterion in terms of preservation of indexed limits or colimits.

\begin{proposition}\label{prop:C_is_cartesian_iff_tensor_preserves_limits}
    Let $\subuniverse[I]$ be an inductible context-free subuniverse, and let $\varintcat{C}$ be an $\subuniverse[I]$-monoidal $\vartopos$-category. Then $\varintcat{C}$ is $\subuniverse[I]$-(co)cartesian if and only if, for every context $A$ and every morphism $p:C\to B$ in $\subuniverse[I](A)$, the functor
    \[
    p_\otimes:\varintcat{C}^{C}\to\varintcat{C}^{B}
    \]
    is $\subuniverse[I]$-(co)continuous.
\end{proposition}
\begin{proof}
    The argument is the same as in \cite[Proposition 3.44]{CLLambi}.
\end{proof}

\subsubsection*{Ambidexterity}

Cnossen--Lenz--Linskens define $\subuniverse[I]$-ambidextrous categories in \cite{CLLambi}. In our terminology, these are $\subuniverse[I]$-monoidal $\vartopos$-categories that are both cartesian and cocartesian. We will use the following theorem from their work.

\begin{theorem}[{\cite[Theorem 7.5]{CLLambi}}]\label{Thm:CLL_amby}
    Let $\varintcat{C}$ be an $\subuniverse[I]$-cartesian $\vartopos$-category. The forgetful functor
    \[
    \intsmon[I](\varintcat{C})\to\varintcat{C},
    \qquad M\mapsto M(*),
    \]
    is terminal among $\subuniverse[I]$-cartesian $\vartopos$-categories over $\varintcat{C}$ that are also cocartesian.
\end{theorem}

\begin{corollary}[the doubling functor]\label{cor:doubling}
    The forgetful $\vartopos$-functor
    \[
    \intsmon[I](\intsmon[I](\internalcatofcats))
    \to\intsmon[I](\internalcatofcats)
    \]
    is an equivalence. Its inverse $\delta$ sends an $\subuniverse[I]$-monoidal $\vartopos$-category $F:\intspan(\subuniverse[I])\to\internalcatofcats$ to
    \[
    \delta(F):\intspan(\subuniverse[I])\to\intsmon[I](\internalcatofcats).
    \]
    The underlying functor with values in $\intfun(\intspan(\subuniverse[I]),\internalcatofcats)$ is obtained by currying the composite
    \[
    \intspan(\subuniverse[I])\times\intspan(\subuniverse[I])
    \xto{\times}\intspan(\subuniverse[I])\xto{F}\internalcatofcats.
    \]
\end{corollary}

\begin{definition}
    Following \cite{CLLambi}, an $\subuniverse[I]$-monoidal $\vartopos$-category is \emph{$\subuniverse[I]$-ambidextrous} if it is both $\subuniverse[I]$-cartesian and $\subuniverse[I]$-cocartesian. We write $\internalcatofcats^{\subuniverse[I]\text{-}\oplus}$ for the $\vartopos$-category of $\subuniverse[I]$-ambidextrous categories.
\end{definition}

\begin{remark}
    Equivalently, \Cref{Thm:CLL_amby} states that $\intsmon[I]$ is right adjoint to the inclusion
    \[
    \internalcatofcats^{\subuniverse[I]\text{-}\oplus}
    \into\internalcatofcats^{\subuniverse[I]\text{-}\sqcap}.
    \]
\end{remark}

The next corollary shows that an endofunctor of $\subuniverse[I]$-monoids that commutes with the forgetful functor is canonically equivalent to the identity.

\begin{corollary}\label{lem:no_endomorphisms_of_CMon}
    Let $\varintcat{C}$ be an $\subuniverse[I]$-cartesian $\vartopos$-category. There is an equivalence of $\vartopos$-categories
    \[
    \intfun_{/\varintcat{C}}(\intsmon[I](\varintcat{C}),\intsmon[I](\varintcat{C}))
    \simeq *.
    \]
\end{corollary}
\begin{proof}
    The forgetful functor $\intsmon[I](\varintcat{C})\to\varintcat{C}$ is $\subuniverse[I]$-continuous and conservative. Hence every endofunctor over $\varintcat{C}$ is $\subuniverse[I]$-continuous. The result now follows from the universal property in \Cref{Thm:CLL_amby}.
\end{proof}

\begin{question}\label{question:no_endomorphism_of_cmon_without_iductiblity}
    For $\varintcat{C}=\internalcatofcats$, does \Cref{lem:no_endomorphisms_of_CMon} hold without assuming that the context-free subuniverse $\subuniverse$ is inductible? We expect that it does, but do not know a proof or a counterexample.
\end{question}

\subsubsection*{Commutative algebras and coalgebras}

For the definitions in this part, let $\subuniverse$ be a general context-free subuniverse. Recall that, for a symmetric monoidal $\infty$-category $\varcat{C}$, one has $\operatorname{CAlg}(\varcat{C})\simeq\Fun^\otimes(\Fin,\varcat{C})$. To give the analogous definition, we first equip $\subuniverse$ with an $\subuniverse$-monoidal structure.

\begin{proposition}
    The $\vartopos$-subcategory $\subuniverse\subset\subuniverse^{\full}$ is an $\subuniverse$-monoidal subcategory of $\equipcocart(\subuniverse^{\full})$.
\end{proposition}
\begin{proof}
    We must check that morphisms in $\subuniverse$ are closed under the indexed tensor products in $\equipcocart(\subuniverse^{\full})$. These tensor products are $\subuniverse^{\full}$-indexed colimits, computed by forgetting context. Membership in $\subuniverse$ is independent of context, so the result follows.
\end{proof}

Equip $\subuniverse$ with this structure and $\subuniverse^{\op}$ with the opposite structure.

\begin{definition}\label{def:E_algebras_and_coalgebras}
    Let $\varintcat{C}$ be an $\subuniverse$-monoidal $\vartopos$-category. An \emph{$\subuniverse$-commutative coalgebra} in $\varintcat{C}$ is an $\subuniverse$-monoidal functor $\subuniverse^{\op}\to\varintcat{C}$. The $\vartopos$-category of these coalgebras is
    \[
    \intcoalg^{\subuniverse}(\varintcat{C})
    :=\intfun^{\subuniverse\text{-}\otimes}(\subuniverse^{\op},\varintcat{C}).
    \]
    Similarly, an \emph{$\subuniverse$-commutative algebra} in $\varintcat{C}$ is an $\subuniverse$-monoidal functor $\subuniverse\to\varintcat{C}$, and we write
    \[
    \intalg^{\subuniverse}(\varintcat{C})
    :=\intfun^{\subuniverse\text{-}\otimes}(\subuniverse,\varintcat{C}).
    \]
    The underlying object of an algebra or coalgebra $F$ is $F(*)$. We will sometimes refer to this object as the algebra or coalgebra itself.
\end{definition}

\subsubsection*{Pointwise monoidal structures}

We return to an inductible context-free subuniverse $\subuniverse[I]$. For symmetric monoidal $\infty$-categories $\varcat{C}$ and $\varcat{D}$, the $\infty$-category of symmetric monoidal functors $\varcat{C}\to\varcat{D}$ carries a pointwise symmetric monoidal structure. The following is the corresponding internal statement.

\begin{corollary}\label{cor:cmon_is_self_enriched_if_inductible}
    Let $\subuniverse[I]$ be an inductible context-free subuniverse. The $\vartopos$-functor
    \[
    \intfun^{\subuniverse[I]\text{-}\otimes}:
    \intsmon[I](\internalcatofcats)^{\op}\times\intsmon[I](\internalcatofcats)
    \to\internalcatofcats
    \]
    lifts canonically to a $\vartopos$-functor
    \[
    \intfun^{\subuniverse[I]\text{-}\otimes}:
    \intsmon[I](\internalcatofcats)^{\op}\times\intsmon[I](\internalcatofcats)
    \to\intsmon[I](\internalcatofcats).
    \]
\end{corollary}
\begin{proof}
    Apply \Cref{Thm:CLL_amby} in the second variable. The $\vartopos$-category $\intsmon[I](\internalcatofcats)$ is both $\subuniverse[I]$-cartesian and $\subuniverse[I]$-cocartesian, and $\intfun^{\subuniverse[I]\text{-}\otimes}(\varintcat{C},-)$ preserves $\subuniverse[I]$-limits. The canonical lift is also functorial in $\varintcat{C}$.
\end{proof}

This construction gives a coherent self-enrichment, discussed in \Cref{subsec:symmetric_monoidal_cat_of_E_monoidal_cats}. Here we describe its effect on internal functor categories, beginning with some representable endofunctors.

\begin{example}
    By \Cref{example:free_E_monoidal_cat}, there is a natural equivalence
    \[
    \intfun^{\subuniverse[I]\text{-}\otimes}(\subuniverse[I]^{\simeq},\varintcat{C})
    \simeq\varintcat{C}.
    \]
    More generally, for $A\in\vartopos$, both endofunctors
    \[
    (-)^A,\qquad
    \intfun^{\subuniverse[I]\text{-}\otimes}(\subuniverse[I]^{\simeq}[A],-)
    \]
    of $\intsmon[I](\internalcatofcats)$ lift the composite
    \[
    \intsmon[I](\internalcatofcats)\to\internalcatofcats
    \xto{(-)^A}\internalcatofcats.
    \]
    By \Cref{lem:no_endomorphisms_of_CMon}, these lifts are canonically equivalent.
\end{example}

\begin{proposition}\label{prop:evaluation_is_I_monoidal}
    Let $\varintcat{C}$ and $\varintcat{D}$ be $\subuniverse[I]$-monoidal $\vartopos$-categories, and let $x:A\to\varintcat{C}$ be an object in context $A$. The evaluation functor
    \[
    \operatorname{ev}_x:
    \intfun^{\subuniverse[I]\text{-}\otimes}(\varintcat{C},\varintcat{D})
    \to\varintcat{D}^A
    \]
    is canonically $\subuniverse[I]$-monoidal.
\end{proposition}
\begin{proof}
    Under the preceding identification, evaluation is precomposition with the $\subuniverse[I]$-monoidal functor $\subuniverse[I]^{\simeq}[A]\to\varintcat{C}$ corresponding to $x$.
\end{proof}

\begin{definition}
    In view of \Cref{prop:evaluation_is_I_monoidal}, we call the structure on $\intfun^{\subuniverse[I]\text{-}\otimes}(\varintcat{C},\varintcat{D})$ supplied by \Cref{cor:cmon_is_self_enriched_if_inductible} the \emph{pointwise $\subuniverse[I]$-monoidal structure}.
\end{definition}

In particular, this construction equips categories of commutative algebras with monoidal structures.

\begin{proposition}
    For an $\subuniverse[I]$-monoidal $\vartopos$-category $\varintcat{C}$, the pointwise structure makes
    \[
    \intalg^{\subuniverse[I]}(\varintcat{C})
    =\intfun^{\subuniverse[I]\text{-}\otimes}(\subuniverse[I],\varintcat{C})
    \]
    an $\subuniverse[I]$-monoidal $\vartopos$-category.
\end{proposition}
\begin{proof}
    Apply \Cref{cor:cmon_is_self_enriched_if_inductible} with first variable $\subuniverse[I]$.
\end{proof}

The same description identifies the forgetful functor.

\begin{proposition}\label{prop:forgetful_from_calg_is_internally_representable}
    The forgetful functor $\intalg^{\subuniverse[I]}(\varintcat{C})\to\varintcat{C}$ is internally represented in $\intsmon[I](\internalcatofcats)$ by the canonical $\subuniverse[I]$-monoidal inclusion
    \[
    \subuniverse[I]^{\simeq}\into\subuniverse[I].
    \]
    More explicitly, it is the functor induced by precomposition,
    \[
    \intfun^{\subuniverse[I]\text{-}\otimes}(\subuniverse[I],\varintcat{C})
    \to\intfun^{\subuniverse[I]\text{-}\otimes}(\subuniverse[I]^{\simeq},\varintcat{C})
    \simeq\varintcat{C}.
    \]
\end{proposition}
\begin{proof}
    The universal property of $\subuniverse[I]^{\simeq}$ identifies the last equivalence with evaluation at the point. Thus the displayed composite sends an algebra $F$ to $F(*)$, as required.
\end{proof}

\subsubsection*{The adjunction for commutative algebras}

Our next goal is to construct an adjunction
\[
\equipcocart:\internalcatofcats^{\subuniverse[I]\text{-}\sqcup}
\fromto\intsmon[I](\internalcatofcats):\intalg^{\subuniverse[I]}.
\]
We begin by showing that the pointwise structure on commutative algebras is cocartesian.

\begin{proposition}\label{prop:calg_has_E_colimits}
    Let $\subuniverse[I]$ be an inductible context-free subuniverse, and let $\varintcat{C}$ be an $\subuniverse[I]$-monoidal $\vartopos$-category. The pointwise $\subuniverse[I]$-monoidal structure on
    \[
    \intalg^{\subuniverse[I]}(\varintcat{C})
    \simeq\intfun^{\subuniverse[I]\text{-}\otimes}(\subuniverse[I],\varintcat{C})
    \]
    is cocartesian.
\end{proposition}
\begin{proof}
    Let $[s:B\to A]\in\subuniverse[I](A)$. We construct a unit and counit for $s_\otimes\dashv s^*$. Write $*_A=[\id:A\to A]\in\subuniverse[I](A)$ and similarly for $*_B$.

    Consider the counit morphism
    \[
    \mu_s:s_\sharp s^* *_A\to *_A
    \]
    in $\subuniverse[I](A)$, which we call the \emph{universal multiplication}. It is canonically a morphism of $\subuniverse[I]$-algebras, using the equivalences
    \[
    \begin{aligned}
    \intalg^{\subuniverse[I]}(\subuniverse[I])
    &=\intfun^{\subuniverse[I]\text{-}\otimes}(\subuniverse[I],\subuniverse[I])\\
    &\simeq\intfun^{\subuniverse[I]\text{-}\sqcup}(\subuniverse[I],\subuniverse[I])
    \simeq\subuniverse[I].
    \end{aligned}
    \]

    For $R\in\intalg^{\subuniverse[I]}(\varintcat{C})(A)$, there is an identification
    \[
    (s_\otimes s^*R)(*_A)\simeq R(s_\sharp s^* *_A).
    \]
    Applying $R$ to the universal multiplication therefore gives a natural morphism whose underlying map is
    \[
    (s_\otimes s^*R)(*_A)
    \simeq R(s_\sharp s^* *_A)
    \xto{R(\mu_s)}R(*_A).
    \]
    Since $\subuniverse[I]$-monoidal functors preserve $\subuniverse[I]$-algebras, this is a morphism of algebras
    \[
    \mu_s^R:s_\otimes s^*R\to R.
    \]

    We next construct the unit for an arbitrary algebra $Q\in\intalg^{\subuniverse[I]}(\varintcat{C})(B)$. Consider the diagram
    \[
    \begin{tikzcd}
        B\\
        &{B\times_A B}\arrow[rd,"\lrcorner"{anchor=center,pos=0.125},draw=none]&B\\
        &B&A.
        \arrow["\Delta"',from=1-1,to=2-2]
        \arrow[curve={height=-6pt},equal,from=1-1,to=2-3]
        \arrow[curve={height=6pt},equal,from=1-1,to=3-2]
        \arrow["\pi_2",from=2-2,to=2-3]
        \arrow["\pi_1"',from=2-2,to=3-2]
        \arrow["s",from=2-3,to=3-3]
        \arrow["s"',from=3-2,to=3-3]
    \end{tikzcd}
    \]
    It gives a morphism in $\intalg^{\subuniverse[I]}(\varintcat{C})(B)$:
    \[
    \begin{aligned}
    \upsilon_s^Q:Q
    &\simeq(\pi_1)_\otimes\Delta_\otimes\Delta^*\pi_2^*Q\\
    &\xto{(\pi_1)_\otimes\mu_\Delta^{\pi_2^*Q}}
    (\pi_1)_\otimes\pi_2^*Q
    \simeq s^*s_\otimes Q.
    \end{aligned}
    \]

    In $\subuniverse[I](B)$, write
    \[
    \upsilon_s:=\upsilon_s^{*_B}:*_B\to s^*s_\sharp *_B.
    \]
    We call this the \emph{universal split unit}. It identifies with the unit of the adjunction
    \[
    s_\sharp:\subuniverse[I](B)\fromto\subuniverse[I](A):s^*.
    \]
    We say that $Q$ \emph{splits} if $Q\simeq s^*R$ for some $R\in\intalg^{\subuniverse[I]}(\varintcat{C})(A)$.

    If $Q$ splits, then $\pi_1^*Q\simeq\pi_2^*Q$, and hence
    \[
    Q(\upsilon_s)
    \simeq(\pi_1)_\otimes\mu_\Delta^{\pi_1^*Q}
    \simeq(\pi_1)_\otimes\mu_\Delta^{\pi_2^*Q}
    =\upsilon_s^Q.
    \]
    The first equivalence follows from functoriality of $Q$ and naturality of $\upsilon$ in $\intspan(\subuniverse[I])$. The second uses the splitting, and the last equality is the definition of $\upsilon_s^Q$.

    For arbitrary $Q$, the corresponding comparison holds after applying $s_\otimes$:
    \[
    \begin{aligned}
    s_\otimes\upsilon_s^Q
    &=s_\otimes(\pi_1)_\otimes\mu_\Delta^{\pi_2^*Q}\\
    &\simeq s_\otimes(\pi_2)_\otimes\mu_\Delta^{\pi_2^*Q}
    \simeq s_\otimes Q(\upsilon_s).
    \end{aligned}
    \]
    Here the first equality and the last equivalence are as above, while the middle equivalence uses $s\pi_1=s\pi_2$.

    These natural maps satisfy the triangle identities: the preceding comparisons reduce them to the identities for the universal algebra $*_-$. The Beck--Chevalley conditions hold as well. Thus $s_\otimes$ computes the indexed colimit along $s$, proving the proposition.
\end{proof}

\begin{remark}
    We expect this proposition to hold for every full context-free subuniverse $\subuniverse$, but to fail in general when $\subuniverse$ is not full.
\end{remark}

Passing to opposites gives the corresponding statement for coalgebras.

\begin{proposition}\label{prop:cocalg_has_E_limits}
    Let $\subuniverse[I]$ be an inductible context-free subuniverse, and let $\varintcat{C}$ be an $\subuniverse[I]$-monoidal $\vartopos$-category. The pointwise $\subuniverse[I]$-monoidal structure on
    \[
    \intcoalg^{\subuniverse[I]}(\varintcat{C})
    \simeq\intfun^{\subuniverse[I]\text{-}\otimes}(\subuniverse[I]^{\op},\varintcat{C})
    \]
    is cartesian.
\end{proposition}
\begin{proof}
    Apply \Cref{prop:calg_has_E_colimits} to $\varintcat{C}^{\op}$ and take opposites.
\end{proof}

The doubling functor $\delta$ from \Cref{cor:doubling} gives another description of the cocartesian structure on algebras.

\begin{proposition}\label{prop:calg_has_E_colimits_coherence}
    Let $\subuniverse[I]$ be an inductible context-free subuniverse, and let $F:\intspan(\subuniverse[I])\to\internalcatofcats$ be an $\subuniverse[I]$-monoidal $\vartopos$-category with underlying category $\varintcat{C}=F(*)$. The composite
    \[
    D:\intspan(\subuniverse[I])\xto{\delta(F)}\intsmon[I](\internalcatofcats)
    \xto{\intalg^{\subuniverse[I]}}\internalcatofcats
    \]
    identifies with $\equipcocart(\intalg^{\subuniverse[I]}(\varintcat{C}))$.
\end{proposition}
\begin{proof}
    Both constructions lift the limit-preserving $\vartopos$-functor
    \[
    \intalg^{\subuniverse[I]}:\intsmon[I](\internalcatofcats)\to\internalcatofcats.
    \]
    They are therefore equivalent by \Cref{Thm:CLL_amby}.
\end{proof}

We now construct the unit and counit of the desired adjunction. This generalizes the fact that commutative algebras form the right adjoint to the inclusion of cocartesian symmetric monoidal categories.

\begin{proposition}\label{prop:canonical_Calg_structure_for_cocartesian_categories}
    Let $\varintcat{C}$ be an $\subuniverse[I]$-cocomplete $\vartopos$-category. The composite
    \[
    \begin{aligned}
    \varintcat{C}
    &\simeq\intfun(*,\varintcat{C})
    \simeq\intfun^{\subuniverse[I]\text{-}\sqcup}(\subuniverse[I],\varintcat{C})\\
    &\xto{\equipcocart}
    \intfun^{\subuniverse[I]\text{-}\otimes}
    (\equipcocart(\subuniverse[I]),\equipcocart(\varintcat{C}))
    =\intalg^{\subuniverse[I]}(\equipcocart(\varintcat{C}))
    \end{aligned}
    \]
    is $\subuniverse[I]$-cocontinuous and natural in $\varintcat{C}$.
\end{proposition}
\begin{proof}
    By \Cref{prop:calg_has_E_colimits}, colimits in the target are given by its pointwise tensor products. These are induced by the tensor products in $\equipcocart(\varintcat{C})$, which are the indexed colimits in $\varintcat{C}$ by \Cref{prop:categories_with_E_colimits_have_E_monoidal_structure}. Thus the displayed functor is $\subuniverse[I]$-cocontinuous.

    Naturality of the free cocontinuous extension follows from naturality of left Kan extension. The final functor is natural by functoriality of $\equipcocart$.
\end{proof}

The forgetful functor provides the counit.

\begin{proposition}\label{prop:forgetful_from_calg_is_monoidal}
    Let $\varintcat{C}$ be an $\subuniverse[I]$-monoidal $\vartopos$-category. The forgetful functor
    \[
    \intalg^{\subuniverse[I]}(\varintcat{C})\to\varintcat{C},
    \qquad F\mapsto F(*),
    \]
    lifts to a natural $\subuniverse[I]$-monoidal functor
    \[
    \equipcocart(\intalg^{\subuniverse[I]}(\varintcat{C}))\to\varintcat{C}.
    \]
\end{proposition}
\begin{proof}
    This follows from \Cref{prop:forgetful_from_calg_is_internally_representable}, together with the identification of the pointwise structure as cocartesian in \Cref{prop:calg_has_E_colimits}.
\end{proof}

\begin{theorem}\label{thm:calg_is_right_adjoint}
    Let $\subuniverse[I]$ be an inductible context-free subuniverse. There is an adjunction of $\vartopos$-categories
    \[
    \equipcocart:\internalcatofcats^{\subuniverse[I]\text{-}\sqcup}
    \fromto\intsmon[I](\internalcatofcats):\intalg^{\subuniverse[I]},
    \]
    with $\equipcocart$ as the left adjoint.
\end{theorem}
\begin{proof}
    The unit is given by \Cref{prop:canonical_Calg_structure_for_cocartesian_categories}, and the counit by \Cref{prop:forgetful_from_calg_is_monoidal}. We verify the triangle identities.

    For an $\subuniverse[I]$-cocomplete $\vartopos$-category $\varintcat{C}$, the composite
    \[
    \equipcocart(\varintcat{C})
    \to\equipcocart(\intalg^{\subuniverse[I]}(\equipcocart(\varintcat{C})))
    \to\equipcocart(\varintcat{C})
    \]
    first extends a functor $*\to\varintcat{C}$ cocontinuously to $\subuniverse[I]$ and then evaluates at the point. It is therefore canonically equivalent to the identity.

    Now let $\varintcat{C}$ be an $\subuniverse[I]$-monoidal $\vartopos$-category. We must identify
    \[
    \intalg^{\subuniverse[I]}(\varintcat{C})
    \to\intalg^{\subuniverse[I]}(\equipcocart(\intalg^{\subuniverse[I]}(\varintcat{C})))
    \to\intalg^{\subuniverse[I]}(\varintcat{C})
    \]
    with the identity. Let $F:\subuniverse[I]\to\varintcat{C}$ be an algebra in context $A$, with underlying object $F(*_A)$. The first arrow assigns to $F$ the $\subuniverse[I]$-cocontinuous extension
    \[
    G:\subuniverse[I]\to\intalg^{\subuniverse[I]}(\varintcat{C})
    \]
    of the point selecting $F$, equipped with its induced monoidal structure. The second arrow postcomposes with the forgetful functor. It remains to identify the composite
    \[
    \subuniverse[I]\xto{G}\intalg^{\subuniverse[I]}(\varintcat{C})\to\varintcat{C}
    \]
    with $F$.

    For $[p:B\to A]\in\subuniverse[I](A)$, the value of $G$ is $p_\sharp p^*F$. Its underlying object is $p_\otimes p^*F(*_A)\simeq F([p:B\to A])$. On morphisms, $G$ is given by the corresponding counits. By their construction in \Cref{prop:calg_has_E_colimits}, these are the multiplication morphisms of $F$. Thus the composite recovers $F$, naturally in $F$, proving the second triangle identity.
\end{proof}

\begin{remark}
    We expect this adjunction to exist for every full context-free subuniverse. We do not expect $\equipcocart$ to remain fully faithful in that generality, as it is in \Cref{thm:uniqueness_of_cartesian_structures}.
\end{remark}

\subsubsection*{The cocartesian localization}

A symmetric monoidal $\infty$-category is cocartesian precisely when every object carries a unique commutative algebra structure. We first establish the corresponding criterion for $\subuniverse[I]$-monoidal $\vartopos$-categories.

\begin{theorem}\label{thm:in_cocart_everything_is_alg}
    An $\subuniverse[I]$-monoidal $\vartopos$-category $\varintcat{C}$ is cocartesian if and only if the forgetful functor
    \[
    \intalg^{\subuniverse[I]}(\varintcat{C})\to\varintcat{C},
    \qquad R\mapsto R(*),
    \]
    is an equivalence of $\vartopos$-categories.
\end{theorem}
\begin{proof}
    This is the underlying functor of the counit in \Cref{thm:calg_is_right_adjoint}. The left adjoint is fully faithful by \Cref{thm:uniqueness_of_cartesian_structures}. Hence the counit is an equivalence precisely on its essential image, namely the cocartesian $\subuniverse[I]$-monoidal $\vartopos$-categories.
\end{proof}

This criterion identifies cocartesian categories as the objects internally local with respect to $\subuniverse[I]^{\simeq}\into\subuniverse[I]$.

\begin{theorem}\label{thm:cocart_cats_are_localization_of_monoidal}
    Let $\subuniverse[I]$ be an inductible context-free subuniverse. The fully faithful $\vartopos$-functor
    \[
    \equipcocart:\internalcatofcats^{\subuniverse[I]\text{-}\sqcup}
    \into\intsmon[I](\internalcatofcats)
    \]
    also admits a left adjoint $L_{\sqcup}$. Moreover, $\internalcatofcats^{\subuniverse[I]\text{-}\sqcup}$ is $\vartopos$-presentable.
\end{theorem}
\begin{proof}
    By \Cref{thm:in_cocart_everything_is_alg,prop:forgetful_from_calg_is_internally_representable}, an $\subuniverse[I]$-monoidal $\vartopos$-category is cocartesian if and only if it is internally local with respect to the $\subuniverse[I]$-monoidal inclusion $\subuniverse[I]^{\simeq}\into\subuniverse[I]$. By \Cref{cor:cmon_is_tensored_and_powered_by_cat}, this is equivalent to locality with respect to the functors
    \[
    \subuniverse[I]^{\simeq}\otimes\Delta^n
    \to\subuniverse[I]\otimes\Delta^n,
    \qquad n\in\nats.
    \]
    Since $\intsmon[I](\internalcatofcats)$ is $\vartopos$-presentable by \Cref{prop:cmon_is_cat_module_in_prl}, these local objects form an accessible localization. This gives $L_{\sqcup}$ and the asserted presentability.
\end{proof}

There is also a description in terms of colocality with respect to the forgetful functors from commutative algebras.

\begin{corollary}\label{cor:Icocart_iff_forgetful_colocal}
    An $\subuniverse[I]$-monoidal $\vartopos$-category $\varintcat{C}$ is cocartesian if and only if, for every $\subuniverse[I]$-monoidal $\vartopos$-category $\varintcat{T}$, the canonical morphism
    \[
    \intmap(\varintcat{C},\intalg^{\subuniverse[I]}(\varintcat{T}))
    \to\intmap(\varintcat{C},\varintcat{T})
    \]
    is an equivalence of $\vartopos$-anima. Here the mapping objects are taken in $\intsmon[I](\internalcatofcats)$, with the cocartesian structure on $\intalg^{\subuniverse[I]}(\varintcat{T})$.
\end{corollary}
\begin{proof}
    If $\varintcat{C}$ is cocartesian, the equivalence is the adjunction of \Cref{thm:calg_is_right_adjoint}.

    Conversely, the two adjunctions identify the mapping functor represented by $\varintcat{C}$ with that represented by $\equipcocart(L_{\sqcup}\varintcat{C})$. The internal Yoneda lemma gives $\varintcat{C}\simeq\equipcocart(L_{\sqcup}\varintcat{C})$, so $\varintcat{C}$ is cocartesian.
\end{proof}

\subsection{\texorpdfstring{$\subuniverse$}{E}-operads.}\label{subsec:E_operads}

We now introduce $\subuniverse$-operads and their monoidal envelopes. Our definitions are inspired by the work of Lenz--Linskens--P\"utzst\"uck \cite{lenz2026norms} and Barkan--Haugseng--Steinebrunner \cite{barkan2022envelopes}, although our formulation differs slightly.

We begin by defining $\subuniverse$-preoperads (\Cref{def:E_preoperads}), which may be viewed as a generalization of operads without the Segal condition. We construct the envelope adjunction at this level (\Cref{prop:preoperad_premonoidal_adjunction}). We then define $\subuniverse$-operads by requiring their envelopes to be $\subuniverse$-monoidal categories (\Cref{def:E_operads}). For finite covering maps of anima, this recovers the usual notion of an $\infty$-operad (\Cref{ex:classical_E_operads}).

Our main result identifies the monoidal envelope as the left adjoint to the associated operad construction (\Cref{thm:env_associated_operad_adjunction}).
\subsubsection*{Preoperads}

We first recall the notation for internal cocartesian fibrations.

\begin{definition}
    Let $\varintcat{C}$ be a $\vartopos$-category. We write
    \[
    \intcocart_{\varintcat{C}}\into\internalcatofcats_{/\varintcat{C}}
    \]
    for the $\vartopos$-subcategory whose objects are cocartesian fibrations over $\varintcat{C}$ and whose morphisms are functors over $\varintcat{C}$ that preserve cocartesian lifts.
\end{definition}

For preoperads, we require cocartesian lifts only for left-pointing arrows in the span category.

\begin{definition}\label{def:E_preoperads}
    The $\vartopos$-category $\intpreoperads$ of \emph{$\subuniverse$-preoperads} is the $\vartopos$-subcategory of $\internalcatofcats_{/\spanEfE}$ defined as follows. Its objects in context $A$ are functors of $\vartopos_{/A}$-categories
    \[
    \varintcat{O}^{\otimes}\to\pi_A^*\spanEfE
    \]
    admitting cocartesian lifts of left-pointing arrows. Its morphisms are functors over $\pi_A^*\spanEfE$ that preserve these cocartesian lifts.
\end{definition}

\begin{warning}
    This condition is stronger than requiring the restriction
    \[
    \varintcat{O}^{\otimes}\times_{\pi_A^*\spanEfE}
    \pi_A^*(\subuniverse^{\full})^{\op}
    \to\pi_A^*(\subuniverse^{\full})^{\op}
    \]
    to be a cocartesian fibration. A lift must be cocartesian over the entire span category, so its universal property must hold against all morphisms, including those that are not left-pointing.
\end{warning}

Unstraightening associates a preoperad to every functor $\spanEfE\to\internalcatofcats$.

\begin{definition}
    Consider the composite
    \[
    (-)^\otimes:\intfun(\spanEfE,\internalcatofcats)
    \xto{\sim}\intcocart_{\spanEfE}\into\intpreoperads.
    \]
    We call $F^\otimes$ the \emph{associated preoperad} of $F$. If $F$ defines an $\subuniverse$-monoidal $\vartopos$-category $\varintcat{C}$, we also write $\varintcat{C}^{\otimes}:=F^\otimes$.
\end{definition}

\subsubsection*{A span description of \texorpdfstring{$\subuniverse^\otimes$}{E tensor}}

To construct a left adjoint to $(-)^\otimes$, we need an explicit description of $\subuniverse^\otimes$ and its source and target functors. We begin with the corresponding description for $\subuniverse^{\full}$.

\begin{definition}\label{def:E_full_1}
    Let $\subuniverse^{\full,(1)}\subset(\subuniverse^{\full})^{\Delta^1}$ be the full $\vartopos$-subcategory spanned by arrows in $\subuniverse$. Thus a morphism in $\subuniverse^{\full,(1)}(A)$ is a commutative square in $\subuniverse^{\full}(A)$ of the form
    \[
    \begin{tikzcd}
        C'\ar[r]\ar[d,tail]&B'\ar[d,tail]\\
        C\ar[r]&B,
    \end{tikzcd}
    \]
    where the tailed arrows lie in $\subuniverse$.

    Let $\subuniverse^{(1)}\subset\subuniverse^{\full,(1)}$ be the wide $\vartopos$-subcategory whose morphisms have the form
    \[
    \begin{tikzcd}
        C'\ar[r]\ar[d,tail]&B'\ar[d,tail]\\
        C\ar[r,tail]&B,
    \end{tikzcd}
    \]
    again with the tailed arrows in $\subuniverse$. Finally, let $\subuniverse^{\full,(1)}_{\cart}$ be the wide $\vartopos$-subcategory whose morphisms are cartesian squares.
\end{definition}

\begin{proposition}
    The triplet
    \[
    (\subuniverse^{\full,(1)},\subuniverse^{\full,(1)}_{\cart},\subuniverse^{(1)})
    \]
    is adequate. The target functor $t:\subuniverse^{\full,(1)}\to\subuniverse^{\full}$ defines a morphism of adequate triplets
    \[
    (\subuniverse^{\full,(1)},\subuniverse^{\full,(1)}_{\cart},\subuniverse^{(1)})
    \to(\subuniverse^{\full},\subuniverse^{\full},\subuniverse).
    \]
\end{proposition}
\begin{proof}
    We check that the required pullbacks exist and are preserved by change of context. Consider a morphism in $\subuniverse^{(1)}(A)$ and a morphism in $\subuniverse^{\full,(1)}_{\cart}(A)$ with the same target:
    \[
    \begin{tikzcd}
        D'\ar[r]\ar[d,tail]&B'\ar[d,tail]\\
        D\ar[r,tail]&B
    \end{tikzcd}
    \qquad
    \begin{tikzcd}
        C'\ar[r]\ar[d,tail]\pullbackdr&B'\ar[d,tail]\\
        C\ar[r]&B.
    \end{tikzcd}
    \]
    Their pullback is computed at the source and target. At the target, $C\times_B D$ exists because $D\to B$ lies in $\subuniverse$. At the source, the pullback exists by the identification
    \[
    C'\times_{B'}D'\simeq(C\times_B D)\times_D D',
    \]
    where we use $C'\simeq C\times_B B'$. The induced arrow
    \[
    C'\times_{B'}D'\to C\times_B D
    \]
    lies in $\subuniverse$ by base change. The resulting square over $D'\to D$ is cartesian, and the morphism to $C'\to C$ lies in $\subuniverse^{(1)}$ because $C\times_B D\to C$ lies in $\subuniverse$.

    These pullbacks are preserved by change of context, since pullbacks in $\universe$ are preserved by change of context. The same description shows that $t$ preserves the required pullbacks and the two distinguished classes of morphisms.
\end{proof}

\begin{lemma}\label{lem:descript_of_E_tensor}
    There is a canonical equivalence
    \[
    \intspan(\subuniverse^{\full,(1)},\subuniverse^{\full,(1)}_{\cart},\subuniverse^{(1)})
    \simeq(\subuniverse^{\full})^{\otimes}.
    \]
\end{lemma}
\begin{proof}
    The target functor $t:\subuniverse^{\full,(1)}\to\subuniverse^{\full}$ is a cartesian fibration: arrows in $\subuniverse$ admit pullbacks and are stable under base change in $\subuniverse^{\full}$. For $B\in\subuniverse^{\full}(A)$, its fibre is the full $\vartopos_{/A}$-subcategory
    \[
    (\pi_A^*\subuniverse^{\full})_{\subuniverse/B}
    \subset(\pi_A^*\subuniverse^{\full})_{/B}
    \]
    spanned by objects whose structure morphism to $B$ lies in $\subuniverse$.

    Straightening identifies this fibre with $(\pi_A^*\subuniverse^{\full})^B$. Thus $t$ classifies the $\subuniverse^{\full}$-continuous functor
    \[
    (\subuniverse^{\full})^{\op}\to\internalcatofcats
    \]
    determined by $\subuniverse^{\full}$. The $t$-cartesian arrows are precisely the morphisms in $\subuniverse^{\full,(1)}_{\cart}$. The displayed span category is therefore the unfurling defining $(\subuniverse^{\full})^\otimes$.
\end{proof}

Morphisms in the wide subcategory $\subuniverse^{\Delta^1}$ are stable under base change in $\subuniverse^{\full,(1)}$ along morphisms in $\subuniverse^{\full,(1)}_{\cart}$. Restricting the preceding construction gives the following description of $\subuniverse^\otimes$.

\begin{corollary}\label{cor:description_of_the_env_of_triv}
    The triplet
    \[
    (\subuniverse^{\full,(1)},\subuniverse^{\full,(1)}_{\cart},\subuniverse^{\Delta^1})
    \]
    is adequate, and there is a canonical equivalence
    \[
    \intspan(\subuniverse^{\full,(1)},\subuniverse^{\full,(1)}_{\cart},\subuniverse^{\Delta^1})
    \simeq\subuniverse^{\otimes}.
    \]
\end{corollary}

Source and target induce functors $s,t:\subuniverse^\otimes\to\spanEfE$. Here $t$ is the cocartesian fibration associated with the $\subuniverse$-monoidal structure on $\subuniverse$. These functors fit into two adjunctions.

\begin{proposition}\label{prop:source_target_identity_adjunctions}
    The source and target functors
    \[
    s,t:\subuniverse^\otimes\to\spanEfE,
    \]
    induced by $\delta_1,\delta_0:\Delta^0\to\Delta^1$, fit into adjunctions
    \[
    t\dashv r\dashv s,
    \]
    where $r:\spanEfE\to\subuniverse^\otimes$ is the fully faithful functor induced by $\sigma:\Delta^1\to\Delta^0$.
\end{proposition}
\begin{proof}
    For $t\dashv r$, the counit is the canonical equivalence $tr\simeq\id$. The unit $\id\to rt$ at an object $C\to B$ is represented by
    \[
    \begin{tikzcd}
        C\ar[d,tail]&C\ar[l,equal]\ar[d,tail]\ar[r,tail]\arrow["\lrcorner"{anchor=center,pos=0.125,rotate=-90},draw=none,dl]&B\ar[d,equal]\\
        B&B\ar[r,equal]\ar[l,equal]&B.
    \end{tikzcd}
    \]
    This morphism is equivalent to the identity when $C\to B$ is an equivalence, and applying $t$ to it gives the identity. These are the two triangle identities.

    For $r\dashv s$, the unit is the canonical equivalence $\id\simeq sr$. The counit $rs\to\id$ at $C\to B$ is represented by
    \[
    \begin{tikzcd}
        C\ar[d,equal]&C\ar[l,equal]\ar[d,equal]\ar[r,equal]\arrow["\lrcorner"{anchor=center,pos=0.125,rotate=-90},draw=none,dl]&C\ar[d,tail]\\
        C&C\ar[r,tail]\ar[l,equal]&B.
    \end{tikzcd}
    \]
    The triangle identities follow in the same way. Either adjunction shows that $r$ is fully faithful.
\end{proof}

\subsubsection*{The envelope adjunction}

The envelope is obtained by pulling a preoperad back along $s$ and then projecting along $t$. We first show that this produces a cocartesian fibration over the entire span category.

\begin{definition}
    Let $\varintcat{O}^{\otimes}\to\spanEfE$ be an $\subuniverse$-preoperad. Define $\intenv(\varintcat{O})^{\otimes}$ by the pullback
    \[
    \begin{tikzcd}
        \intenv(\varintcat{O})^\otimes\ar[r]\ar[d]\pullbackdr
        &\varintcat{O}^{\otimes}\ar[d]\\
        \subuniverse^{\otimes}\ar[r,"s"]\ar[d,"t"]
        &\spanEfE\\
        \spanEfE.
    \end{tikzcd}
    \]
    We regard it as an object over $\spanEfE$ via the left vertical composite.
\end{definition}

\begin{proposition}
    The functor $\intenv(\varintcat{O})^{\otimes}\to\spanEfE$ is a cocartesian fibration.
\end{proposition}
\begin{proof}
    We construct cocartesian lifts in every context and check that they are preserved by change of context. The construction is compatible with slicing, so we describe it in the global context.

    Let
    \[
    A\xfrom{x}F\xto{y}B
    \]
    be a morphism in $\spanEfE(*)$. An object of $\intenv(\varintcat{O})^{\otimes}(*)$ over $A$ consists of an object $O\in\varintcat{O}^{\otimes}(*)$ over $A'$ and an arrow $A'\to A$ in $\subuniverse(*)$. Its image in $\subuniverse^\otimes$ has the $t$-cocartesian lift
    \[
    \begin{tikzcd}
        A'\ar[d]&F'\ar[l,"x'"']\ar[d]\ar[r,equal]\arrow["\lrcorner"{anchor=center,pos=0.125,rotate=-90},draw=none,dl]&F'\ar[d]\\
        A&F\ar[r,"y"]\ar[l,"x"]&B,
    \end{tikzcd}
    \]
    where $F'=A'\times_A F$. Applying $s$ gives a left-pointing arrow, which admits a cocartesian lift from $O$. Together these lifts give the required cocartesian lift in $\intenv(\varintcat{O})^{\otimes}$.

    Change of context preserves this construction, since it preserves the pullback defining $F'$ and the cocartesian lifts in the preoperad.
\end{proof}

\begin{definition}
    The \emph{envelope functor} is the composite
    \[
    \intenv:\intpreoperads\to\intcocart_{\spanEfE}
    \xto{\sim}\intfun(\spanEfE,\internalcatofcats),
    \]
    where the first functor sends $\varintcat{O}^{\otimes}$ to the cocartesian fibration $\intenv(\varintcat{O})^\otimes\to\spanEfE$.
\end{definition}

We use $\intenv(\varintcat{O})$ both for the classified functor and for its value at $*$. When this functor is $\subuniverse^{\full}$-continuous, the same notation denotes the resulting $\subuniverse$-monoidal category and its underlying $\vartopos$-category, as in the preceding subsections. The underlying category has a simple pullback description.

\begin{proposition}
    There is a pullback square of $\vartopos$-categories
    \[
    \begin{tikzcd}
        \intenv(\varintcat{O})\ar[r]\ar[d]\pullbackdr&\varintcat{O}^\otimes\ar[d]\\
        \subuniverse\ar[r]&\spanEfE.
    \end{tikzcd}
    \]
\end{proposition}
\begin{proof}
    In the diagram
    \[
    \begin{tikzcd}
        \intenv(\varintcat{O})\ar[r]\ar[d]\pullbackdr
        &\intenv(\varintcat{O})^\otimes\ar[r]\ar[d]\pullbackdr
        &\varintcat{O}^\otimes\ar[d]\\
        \subuniverse\ar[r]\ar[d]\pullbackdr
        &\subuniverse^\otimes\ar[r,"s"]\ar[d,"t"]
        &\spanEfE\\
        *\ar[r]&\spanEfE,
    \end{tikzcd}
    \]
    all three squares are pullbacks. Pasting the two upper squares gives the claim.
\end{proof}

\begin{remark}
    Unless $\subuniverse=\explicitset{*}$ is the context-free subuniverse of equivalences, $s$ is not a morphism of preoperads: it does not commute with the structure maps to $\spanEfE$. Indeed, $\spanEfE$ is the terminal preoperad, and the structure map of $\subuniverse^\otimes$ is $t$.

    In contrast, $r$ and $t$ are morphisms of preoperads. Once operads are defined below, they are also morphisms of $\subuniverse$-operads.
\end{remark}

The identity $sr\simeq\id$ gives the unit of the envelope adjunction.

\begin{proposition}\label{prop:eta_is_a_morphism_of_preoperads}
    Let $\varintcat{O}^\otimes\to\spanEfE$ be an $\subuniverse$-preoperad. The identity $sr\simeq\id$ gives a pasting of pullbacks
    \[
    \begin{tikzcd}
        \varintcat{O}^\otimes\ar[r,"\eta"]\ar[d]\pullbackdr
        &\intenv(\varintcat{O})^\otimes\ar[r]\ar[d]\pullbackdr
        &\varintcat{O}^\otimes\ar[d]\\
        \spanEfE\ar[r,"r"]
        &\subuniverse^\otimes\ar[r,"s"]
        &\spanEfE.
    \end{tikzcd}
    \]
    The induced functor $\eta$ is a morphism of $\subuniverse$-preoperads.
\end{proposition}
\begin{proof}
    The functor $r$ preserves cocartesian lifts of left-pointing arrows. Its pullback $\eta$ therefore preserves these lifts as well.
\end{proof}

To construct the counit, we use an internal version of \cite[Lemma 2.13]{glossner2025modelindependentuniversalproperty}.

\begin{lemma}\label{lem:right_adj_are_stable_under_BC}
    Let $p:\varintcat{X}\to\varintcat{D}$ be a cocartesian fibration of $\vartopos$-categories, and let $R:\varintcat{C}\to\varintcat{D}$ admit a left adjoint $L\dashv R$. Then the pullback square
    \[
    \begin{tikzcd}
        \varintcat{Y}\ar[r,"R'"]\ar[d,"q"']\pullbackdr&\varintcat{X}\ar[d,"p"]\\
        \varintcat{C}\ar[r,"R"]&\varintcat{D}
    \end{tikzcd}
    \]
    is horizontally left adjointable. Moreover, writing $L'\dashv R'$, the induced commutative square
    \[
    \begin{tikzcd}
        \varintcat{Y}\ar[d,"q"']&\varintcat{X}\ar[d,"p"]\ar[l,"L'"]\\
        \varintcat{C}&\varintcat{D}\ar[l,"L"]
    \end{tikzcd}
    \]
    is a morphism of cocartesian fibrations.
\end{lemma}
\begin{proof}
    For $\infty$-categories, the statement is \cite[Lemma 2.13]{glossner2025modelindependentuniversalproperty}. Its proof constructs the adjunction using cocartesian lifts of the unit and counit and verifies the triangle identities. The construction applies in every context and is preserved by change of context, since change-of-context functors preserve cocartesian lifts. It therefore gives the internal adjunction and the asserted adjointability. Preservation of cocartesian lifts by $L'$ can likewise be checked in every context.
\end{proof}

Now let $F:\spanEfE\to\internalcatofcats$ be a functor. Apply \Cref{lem:right_adj_are_stable_under_BC} to $t\dashv r$ and the pullback defining $\eta_{F^\otimes}$. We obtain the upper adjunction in
\[
\begin{tikzcd}
    F^\otimes&\intenv(F^\otimes)^\otimes\\
    \spanEfE&\subuniverse^\otimes.
    \arrow[""{name=0,anchor=center,inner sep=0},"\eta_{F^\otimes}"',from=1-1,to=1-2]
    \arrow[from=1-1,to=2-1]
    \arrow[""{name=1,anchor=center,inner sep=0},"\epsilon_F^\otimes"',shift right=3,dashed,from=1-2,to=1-1]
    \arrow[from=1-2,to=2-2]
    \arrow[""{name=2,anchor=center,inner sep=0},"r"',from=2-1,to=2-2]
    \arrow[""{name=3,anchor=center,inner sep=0},"t"',shift right=3,dashed,from=2-2,to=2-1]
    \arrow["\dashv"{anchor=center,rotate=-90},draw=none,from=1,to=0]
    \arrow["\dashv"{anchor=center,rotate=-90},draw=none,from=3,to=2]
\end{tikzcd}
\]
The left adjoint $\epsilon_F^\otimes$ preserves cocartesian lifts over $\spanEfE$. Its straightening defines the counit $\epsilon_F:\intenv(F^\otimes)\to F$.

\begin{proposition}\label{prop:preoperad_premonoidal_adjunction}
    The natural transformations $\eta$ and $\epsilon$ defined above are the unit and counit of an adjunction
    \[
    \intenv:\intpreoperads
    \fromto\intfun(\spanEfE,\internalcatofcats):(-)^\otimes.
    \]
\end{proposition}
\begin{proof}
    Since $r$ is fully faithful, its pullback $\eta_{F^\otimes}$ is fully faithful. The counit of $\epsilon_F^\otimes\dashv\eta_{F^\otimes}$ is therefore an equivalence
    \[
    \epsilon_F^\otimes\eta_{F^\otimes}\simeq\id_{F^\otimes}.
    \]
    This proves the first triangle identity.

    For the second, let $\varintcat{O}^\otimes$ be an $\subuniverse$-preoperad. We must show that the composite
    \[
    \intenv(\varintcat{O})
    \xto{\intenv(\eta)}\intenv(\intenv(\varintcat{O})^\otimes)
    \xto{\epsilon_{\intenv(\varintcat{O})}}\intenv(\varintcat{O})
    \]
    is equivalent to the identity. It suffices to check this after unstraightening, for the composite
    \[
    \intenv(\varintcat{O})^\otimes
    \xto{\intenv(\eta)^\otimes}\intenv(\intenv(\varintcat{O})^\otimes)^\otimes
    \xto{\epsilon_{\intenv(\varintcat{O})}^\otimes}\intenv(\varintcat{O})^\otimes.
    \]

    Consider the pasting
    \[
    \begin{tikzcd}
        \intenv(\varintcat{O})^\otimes\ar[r,"q"]\ar[d]\pullbackdr
        &\varintcat{O}^\otimes\ar[r,"\eta_{\varintcat{O}}"]\ar[d]\pullbackdr
        &\intenv(\varintcat{O})^\otimes\ar[d]\\
        \subuniverse^\otimes\ar[r,"s"]
        &\spanEfE\ar[r,"r"]
        &\subuniverse^\otimes.
    \end{tikzcd}
    \]
    Lifting the counit $rs\to\id$ gives a natural transformation
    \[
    \eta_{\varintcat{O}}q\to\id
    \]
    whose components are cocartesian. We will construct a second such transformation
    \[
    \eta_{\varintcat{O}}q
    \to\epsilon_{\intenv(\varintcat{O})}^\otimes\intenv(\eta)^\otimes
    \]
    over the same counit. Uniqueness of cocartesian lifts will then identify the composite with the identity.

    In the commutative diagram
    \[
    \begin{tikzcd}
        \intenv(\intenv(\varintcat{O})^\otimes)^\otimes\ar[r,"\epsilon_{\intenv(\varintcat{O})}^\otimes"]\ar[d]
        &\intenv(\varintcat{O})^\otimes\ar[r,"\eta_{\intenv(\varintcat{O})^\otimes}"]\ar[d]\pullbackdr
        &\intenv(\intenv(\varintcat{O})^\otimes)^\otimes\ar[r,"p"]\ar[d]\pullbackdr
        &\intenv(\varintcat{O})^\otimes\ar[d]\\
        \subuniverse^\otimes\ar[r,"t"]
        &\spanEfE\ar[r,"r"]
        &\subuniverse^\otimes\ar[r,"s"]
        &\spanEfE,
    \end{tikzcd}
    \]
    the middle and right squares are pullbacks; the left square need not be. The middle upper arrow is the unit at the preoperad $\intenv(\varintcat{O})^\otimes$, rather than $\intenv(\eta)^\otimes$.

    By construction, the unit
    \[
    \id\to
    \eta_{\intenv(\varintcat{O})^\otimes}\epsilon_{\intenv(\varintcat{O})}^\otimes
    \]
    is a cocartesian lift of $\id\to rt$. Postcomposing with $p$ gives a natural transformation
    \[
    p\to\epsilon_{\intenv(\varintcat{O})}^\otimes
    \]
    over $s\to t$. Precomposing with $\intenv(\eta)^\otimes$ then gives a transformation with cocartesian components
    \[
    p\intenv(\eta)^\otimes
    \to\epsilon_{\intenv(\varintcat{O})}^\otimes\intenv(\eta)^\otimes.
    \]
    Finally, the upper square in
    \[
    \begin{tikzcd}
        \intenv(\varintcat{O})^\otimes\ar[r,"q"]\ar[d,"\intenv(\eta)^\otimes"']\pullbackdr
        &\varintcat{O}^\otimes\ar[d,"\eta_{\varintcat{O}}"]\\
        \intenv(\intenv(\varintcat{O})^\otimes)^\otimes\ar[d]\ar[r,"p"]\pullbackdr
        &\intenv(\varintcat{O})^\otimes\ar[d]\\
        \subuniverse^\otimes\ar[r,"s"]
        &\spanEfE
    \end{tikzcd}
    \]
    identifies $p\intenv(\eta)^\otimes$ with $\eta_{\varintcat{O}}q$. Thus we have the required second cocartesian lift, proving the triangle identity.
\end{proof}

\subsubsection*{Operads and the classical case}

We now use the envelope to define operads.

\begin{definition}\label{def:E_operads}
    An $\subuniverse$-preoperad is an \emph{$\subuniverse$-operad} if its envelope is an $\subuniverse$-monoidal category. Thus the $\vartopos$-category of $\subuniverse$-operads is defined by the pullback
    \[
    \begin{tikzcd}
        \intoperads\ar[r]\ar[d]\pullbackdr
        &\intpreoperads\ar[d,"\intenv"]\\
        \intsmon(\internalcatofcats)\ar[r,hook]
        &\intfun(\spanEfE,\internalcatofcats).
    \end{tikzcd}
    \]
\end{definition}

\begin{example}\label{ex:classical_E_operads}
    Let $\vartopos=\catofanima$, and let $\subuniverse$ be the context-free subuniverse corresponding to finite covering maps of anima. Then an $\subuniverse$-operad in the global context is the same as an $\infty$-operad in the sense of \cite[Definition 2.1.1.10]{HA}.
\end{example}
\begin{proof}
    By \cite[Corollary 5.1.15]{barkan2022envelopes}, $\infty$-operads in the sense of \cite[Definition 2.1.1.10]{HA} are equivalent to functors
    \[
    p:\mathcal{O}^\otimes\to\Span(\Fin)
    \]
    satisfying the following conditions:
    \begin{mylist}
        \item The functor $p$ admits cocartesian lifts of inert, that is, left-pointing, morphisms.
        \item For every finite set $A$, the square
        \[
        \begin{tikzcd}
            \mathcal{O}^\otimes\times_{\Span(\Fin)}\Fin_{/A}\ar[r]\ar[d]
            &(\mathcal{O}^\otimes\times_{\Span(\Fin)}\Fin)^A\ar[d]\\
            \Fin_{/A}\ar[r]&\Fin^A
        \end{tikzcd}
        \]
        is a pullback.
    \end{mylist}
    The lower arrow is an equivalence, so the second condition is equivalent to requiring the upper arrow to be an equivalence. This upper arrow is the comparison map for preservation of finite products by the functor classified by
    \[
    \operatorname{env}(\mathcal{O})^\otimes\to\Span(\Fin).
    \]
    Thus a preoperad is an $\infty$-operad precisely when its envelope preserves finite products, as required.
\end{proof}

In this example, evaluation at $*$ identifies $\catofanima$-categories with $\infty$-categories. The comparison maps can therefore be described using the elements of finite sets. For a finite set $B$, each $b\in B$ gives a section $l_b:*\to B$. Applying a functor
\[
F:\Span(\Fin)\to\catofcats
\]
to the spans $B\xfrom{l_b}*=*$ gives the comparison map
\[
F(B)\to\prod_{b\in B}F(*).
\]
The functor $F$ defines a symmetric monoidal category precisely when these comparison maps are equivalences.

\begin{warning}
    For a general topos $\vartopos$, an object $[p:B\to A]\in\subuniverse(A)$ need not admit a section in context $A$. Sections after change of context play the role of the maps $l_b$ above.

    Nevertheless, a functor $F:\spanEfE\to\internalcatofcats$ has canonical comparison maps of $\vartopos_{/A}$-categories
    \[
    F(B)\to F(A)^B,
    \]
    where $A$ denotes the terminal object in context $A$. By definition, $F$ is an $\subuniverse$-monoidal category if and only if these comparison maps are equivalences in every context.
\end{warning}

\subsubsection*{The adjunction for operads}

To restrict the envelope adjunction to operads and monoidal categories, it remains to show that $(-)^\otimes$ sends $\subuniverse$-monoidal categories to operads. We will check this using explicit indexed limit cones in $\spanEfE$ and a criterion for recognizing limit cones from their unstraightenings.

\begin{notation}
    For $[B\to A]\in\spanEfE(A)$, let $B^{A\triangleleft}$ be the $\vartopos_{/A}$-category defined by the pushout
    \[
    \begin{tikzcd}
        B\sqcup B\ar[r]\ar[d]&B\times\Delta^1\ar[d]\\
        A\sqcup B\ar[r]&B^{A\triangleleft}\arrow["\lrcorner"{anchor=center,pos=0.125,rotate=180},draw=none,from=2-2,to=1-1].
    \end{tikzcd}
    \]
    Here the upper arrow includes the two endpoints, and the left arrow is the structure map on the first summand and the identity on the second. The projection $B\times\Delta^1\to\Delta^1$ induces a projection $B^{A\triangleleft}\to\Delta^1$, sending $A$ to the source and $B$ to the target.
\end{notation}

\begin{remark}
    This is the left cone $B^{\triangleleft}$ formed in $\vartopos_{/A}$, as in \cite[Remark 4.1.4]{MWcocomplete}. We include $A$ in the notation to make the context explicit.
\end{remark}

The following criterion expresses the limit condition in terms of cocartesian sections.

\begin{lemma}\label{lem:limit_cone_cocartesian_fib_critirion}
    Let
    \[
    \varintcat{Y}\xto{p}\varintcat{X}\xto{q}\varintcat{I}^{\triangleleft}
    \]
    be cocartesian fibrations of $\vartopos$-categories, and write $\intsec(q)$ for the $\vartopos$-category of cocartesian sections of $q$.
    \begin{mylist}
        \item The functor $q$ classifies a limit cone if and only if the restriction functor
        \[
        \intsec(q)\to\intsec(q|_{\varintcat{I}})
        \]
        is an equivalence.
        \item Suppose that $q$ classifies a limit cone. Then $qp$ classifies a limit cone if and only if, for every object $\sigma:A\to\intsec(q)$, the pullback $\sigma^*p$ in the diagram
        \[
        \begin{tikzcd}
            \sigma^*\varintcat{Y}\ar[r]\ar[d,"\sigma^*p"]\pullbackdr
            &\pi_A^*\varintcat{Y}\ar[d]\\
            \pi_A^*\varintcat{I}^{\triangleleft}\ar[dr,equal]\ar[r,"\sigma"]
            &\pi_A^*\varintcat{X}\ar[d]\\
            &\pi_A^*\varintcat{I}^{\triangleleft}
        \end{tikzcd}
        \]
        classifies a limit cone of $\vartopos_{/A}$-categories.
    \end{mylist}
\end{lemma}
\begin{proof}
    For (a), \cite[Proposition 7.1.2]{MInternalStraightenning} identifies cocartesian sections with the limit of the classified functor. Since $\varintcat{I}^{\triangleleft}$ has an initial object, $\intsec(q)$ identifies with the value at the cone point, while $\intsec(q|_{\varintcat{I}})$ identifies with the limit over $\varintcat{I}$. The restriction functor is the comparison map between them, proving (a).

    For (b), consider the commutative square
    \[
    \begin{tikzcd}
        \intsec(qp)\ar[r]\ar[d]
        &\intsec((qp)|_{\varintcat{I}})\ar[d]\\
        \intsec(q)\ar[r]
        &\intsec(q|_{\varintcat{I}}).
    \end{tikzcd}
    \]
    By (a), the lower arrow is an equivalence, and the upper arrow is an equivalence precisely when $qp$ classifies a limit cone. The vertical functors are cocartesian fibrations. It therefore suffices to check that the upper arrow induces an equivalence on the fibres over every object $\sigma:A\to\intsec(q)$. The induced functor on these fibres is
    \[
    \intsec(\sigma^*p)
    \to\intsec((\sigma^*p)|_{\pi_A^*\varintcat{I}}),
    \]
    which is an equivalence precisely when $\sigma^*p$ classifies a limit cone, again by (a).
\end{proof}

We now describe the indexed limit cones in $\spanEfE$. Recall that an object $B\in\vartopos$ defines a $\vartopos$-groupoid with $B(C)=\map_{\vartopos}(C,B)$; these maps correspond to sections of $C\times B\to C$. The same observation applies in every slice topos.

\begin{definition}
    Let $[B\to A]\in\spanEfE(A)$. Define a functor of $\vartopos_{/A}$-categories
    \[
    T:B\to\bigl(\pi_A^*(\subuniverse^{\full})^{\op}\bigr)^{\Delta^1}
    \into(\pi_A^*\spanEfE)^{\Delta^1}
    \]
    as follows. In context $[C\to A]$, a map $C\to B$ over $A$ determines a section $C\to C\times_A B$. We send it to the left-pointing span
    \[
    C\times_A B\from C=C.
    \]
    The source and target of $T$ are the constant diagrams on $[B\to A]$ and $[A=A]$, respectively. Since the corresponding functor $B\times\Delta^1\to\pi_A^*\spanEfE$ is constant on the source, it factors through a cone
    \[
    B\times\Delta^1\to B^{A\triangleleft}
    \xto{c_B}\pi_A^*\spanEfE.
    \]
\end{definition}

\begin{proposition}\label{prop:canonical_span_limit_cones}
    The cones $c_B:B^{A\triangleleft}\to\pi_A^*\spanEfE$ are limit cones.
\end{proposition}
\begin{proof}
    By \Cref{prop:limits_in_span_E_are_computed_in_E_op}, $\subuniverse^{\full}$-limits in $\spanEfE$ are computed in $(\subuniverse^{\full})^{\op}$. It therefore suffices to check the dual statement in $\subuniverse^{\full}$.

    The constant diagram $B\to\pi_A^*\subuniverse^{\full}$ on the terminal object $A$ has colimit $B$. In context $C$, its cocone sends a map $C\to B$ over $A$ to the corresponding section $C\to C\times_A B$. These are precisely the arrows defining $c_B$.
\end{proof}

We can now verify that the associated preoperad of a monoidal category is an operad.

\begin{proposition}\label{prop:otimes_sends_monoidal_to_operads}
    The functor
    \[
    (-)^\otimes:\intfun(\spanEfE,\internalcatofcats)\to\intpreoperads
    \]
    sends $\subuniverse$-monoidal categories to $\subuniverse$-operads.
\end{proposition}
\begin{proof}
    Let $F:\spanEfE\to\internalcatofcats$ define an $\subuniverse$-monoidal category. We must show that $\intenv(F^\otimes)$ is $\subuniverse^{\full}$-continuous. Consider the defining diagram
    \[
    \begin{tikzcd}
        \intenv(F^\otimes)^\otimes\ar[r]\ar[d,"p"]\pullbackdr
        &F^\otimes\ar[d,"q"]\\
        \subuniverse^\otimes\ar[r,"s"]\ar[d,"t"]
        &\spanEfE\\
        \spanEfE.
    \end{tikzcd}
    \]
    It suffices to show that, for every context $A$ and every $[B\to A]\in\subuniverse(A)$, the pullback $c_B^*(tp)$ classifies a limit cone.

    The functor $t$ classifies an $\subuniverse^{\full}$-continuous functor, so $c_B^*t$ already classifies a limit cone. By \Cref{lem:limit_cone_cocartesian_fib_critirion}(b), it remains to show that $\sigma^*p$ classifies a limit cone for every $\pi:A'\to A$ and every $t$-cocartesian section
    \[
    \sigma:\pi^*B^{A\triangleleft}\to\pi_{A'}^*\subuniverse^\otimes
    \]
    over $\pi^*c_B$. Since
    \[
    \sigma^*p\simeq(s\sigma)^*q
    \]
    and $q$ classifies the $\subuniverse^{\full}$-continuous functor $F$, it is enough to show that $s\sigma$ is an indexed limit cone in $\spanEfE$.

    After changing context from $A$ to $A'$, we may assume $A'=A$. By \Cref{lem:limit_cone_cocartesian_fib_critirion}(a), the section $\sigma$ is determined by its value at the cone point. This value is an arrow $B'\to B$ in $\subuniverse(A)$. For a map $C\to B$ over $A$, cocartesianness determines the corresponding arrow of $\sigma$ as
    \[
    \begin{tikzcd}
        C\times_A B'\ar[d]
        &P\ar[d]\ar[r,equal]\ar[l]\arrow["\lrcorner"{anchor=center,pos=0.125,rotate=-90},draw=none,from=1-2,to=2-1]
        &P\ar[d]\\
        C\times_A B
        &C\ar[l]\ar[r,equal]
        &C,
    \end{tikzcd}
    \]
    where $P=C\times_B B'$. Applying $s$ gives the cone whose dual cocone in $\subuniverse^{\full}$ exhibits $B'$ as the colimit of its fibres over $B$. Thus $s\sigma$ is a limit cone by the same argument as in \Cref{prop:canonical_span_limit_cones}. This proves the required continuity.
\end{proof}

The envelope adjunction therefore restricts to operads and monoidal categories.

\begin{theorem}\label{thm:env_associated_operad_adjunction}
    The natural transformations $\eta$ and $\epsilon$ defined above are the unit and counit of an adjunction
    \[
    \intenv:\intoperads
    \fromto\intsmon(\internalcatofcats):(-)^\otimes.
    \]
\end{theorem}
\begin{proof}
    Both $\intoperads\subset\intpreoperads$ and
    \[
    \intsmon(\internalcatofcats)\subset\intfun(\spanEfE,\internalcatofcats)
    \]
    are full $\vartopos$-subcategories. By definition, $\intenv$ sends operads to $\subuniverse$-monoidal categories. Conversely, \Cref{prop:otimes_sends_monoidal_to_operads} shows that $(-)^\otimes$ sends $\subuniverse$-monoidal categories to operads. The adjunction of \Cref{prop:preoperad_premonoidal_adjunction} therefore restricts as claimed.
\end{proof}

Finally, morphisms of associated operads give the notion of a lax monoidal functor.

\begin{definition}\label{def:lax_E_monoidal_functors}
    Let $\varintcat{C}$ and $\varintcat{D}$ be $\subuniverse$-monoidal $\vartopos$-categories. A \emph{lax $\subuniverse$-monoidal functor} $\varintcat{C}\to\varintcat{D}$ is a morphism of $\subuniverse$-operads
    \[
    \varintcat{C}^\otimes\to\varintcat{D}^\otimes.
    \]
\end{definition}

\subsection{Partially cartesian and cocartesian structures.}\label{sec:partially_cartesian_structures}

We now consider partial cartesian and cocartesian conditions on an $\subuniverse$-monoidal category. Given inductible subuniverses $\subuniverse[P],\subuniverse[I]\subseteq\subuniverse$, we require the structure to be cartesian with respect to $\subuniverse[P]$ and cocartesian with respect to $\subuniverse[I]$. We call such categories $(\subuniverse[P,I])$-ambidextrous (\Cref{def:P_I_ambi_E_monoidal_cats}).

These categories form an accessible localization of the category of $\subuniverse$-monoidal categories (\Cref{prop:localization_to_partial_ambidexterity}). This gives an abstract construction of free ambidextrous categories (\Cref{cor:existance_of_free_ambi_monoidal_category}), whose explicit description is one of the main goals of this paper. We conclude with criteria for recognizing ambidexterity (\Cref{prop:inductive_amby_iff_ambi,cor:critiria_for_amby}).

\subsubsection*{Restriction and partially cartesian structures}

Let $\subuniverse[Q]\subset\subuniverse\subset\universe$ be context-free subuniverses. The inclusion
\[
\intspan(\subuniverse[Q]^{\full},\subuniverse[Q])\into\spanEfE
\]
is $\subuniverse[Q]$-continuous. Restriction therefore gives a forgetful $\vartopos$-functor
\[
\intsmon(\internalcatofcats)\to\intsmon[Q](\internalcatofcats).
\]

\begin{proposition}\label{prop:forgetting_E_to_I_monoidal_in_prr}
    The forgetful $\vartopos$-functor
    \[
    \intsmon(\internalcatofcats)\to\intsmon[Q](\internalcatofcats)
    \]
    is accessible, conservative, and continuous. It admits a left adjoint.
\end{proposition}
\begin{proof}
    Restriction commutes with the forgetful functors to $\internalcatofcats$. These functors are accessible, conservative, and continuous, so the same properties hold for restriction. For accessibility, choose a $\vartopos$-regular cardinal $\kappa$ such that both forgetful functors preserve $\kappa$-filtered colimits; conservativity then detects preservation of these colimits by restriction. The same argument applies to limits.

    Both categories are $\vartopos$-presentable by \Cref{prop:cmon_is_cat_module_in_prl}. The internal adjoint functor theorem \cite[Proposition 2.4.3.3]{MWPresentabilityAndTopoi} therefore gives the left adjoint.
\end{proof}

We use restriction to impose cartesian or cocartesian conditions with respect to a smaller inductible subuniverse.

\begin{definition}
    Let $\subuniverse[I]\subset\subuniverse$ be context-free subuniverses, with $\subuniverse[I]$ inductible. An $\subuniverse$-monoidal $\vartopos$-category $\varintcat{C}$ is \emph{$\subuniverse[I]$-cartesian}, respectively \emph{$\subuniverse[I]$-cocartesian}, if its image under
    \[
    \intsmon(\internalcatofcats)\to\intsmon[I](\internalcatofcats)
    \]
    is cartesian, respectively cocartesian. We write
    \[
    \intmon^{\subuniverse,\subuniverse[I]\text{-}\sqcap}(\internalcatofcats)
    \quad\text{and}\quad
    \intmon^{\subuniverse,\subuniverse[I]\text{-}\sqcup}(\internalcatofcats)
    \]
    for the corresponding full $\vartopos$-subcategories of $\intsmon(\internalcatofcats)$.
\end{definition}

The partially cocartesian categories form an accessible localization, obtained from the cocartesian localization for $\subuniverse[I]$ by the following pushout.

\begin{proposition}\label{prop:pushout_in_prl_defining_E_mon_I_cocart}
    There is a pushout square in $\intprl$
    \[
    \begin{tikzcd}
        \intsmon[I](\internalcatofcats)\ar[r]\ar[d]
        &\internalcatofcats^{\subuniverse[I]\text{-}\sqcup}\ar[d]\\
        \intsmon(\internalcatofcats)\ar[r]
        &\intmon^{\subuniverse,\subuniverse[I]\text{-}\sqcup}(\internalcatofcats)
        \arrow[ul,"\lrcorner"{anchor=center,pos=0.125,rotate=180},draw=none].
    \end{tikzcd}
    \]
    The horizontal functors are Bousfield localizations. The left vertical functor is the left adjoint to restriction.
\end{proposition}
\begin{proof}
    By definition, the full $\vartopos$-subcategory of $\subuniverse[I]$-cocartesian $\subuniverse$-monoidal categories is the pullback of
    \[
    \equipcocart:\internalcatofcats^{\subuniverse[I]\text{-}\sqcup}
    \into\intsmon[I](\internalcatofcats)
    \]
    along restriction. By \Cref{prop:forgetting_E_to_I_monoidal_in_prr,thm:cocart_cats_are_localization_of_monoidal}, these are accessible right adjoints. Passing to left adjoints gives the displayed pushout. The horizontal right adjoints are fully faithful, so the horizontal left adjoints are Bousfield localizations.
\end{proof}

\subsubsection*{Ambidexterity and free objects}

We can impose the cartesian and cocartesian conditions for two different indexing subuniverses.

\begin{definition}\label{def:P_I_ambi_E_monoidal_cats}
    Let $\subuniverse[I],\subuniverse[P]\subset\subuniverse$ be inductible context-free subuniverses of a context-free subuniverse $\subuniverse$. An $\subuniverse$-monoidal $\vartopos$-category is \emph{$(\subuniverse[P,I])$-ambidextrous} if it is $\subuniverse[P]$-cartesian and $\subuniverse[I]$-cocartesian. We write
    \[
    \intmon^{\subuniverse,(\subuniverse[P,I])\text{-}\oplus}(\internalcatofcats)
    \]
    for the full $\vartopos$-subcategory of $\intsmon(\internalcatofcats)$ spanned by these categories.
\end{definition}

Imposing both conditions again gives an accessible localization.

\begin{proposition}\label{prop:localization_to_partial_ambidexterity}
    There is a Bousfield localization
    \[
    L_{\oplus}:\intsmon(\internalcatofcats)
    \onto\intmon^{\subuniverse,(\subuniverse[P,I])\text{-}\oplus}(\internalcatofcats).
    \]
\end{proposition}
\begin{proof}
    By \Cref{prop:pushout_in_prl_defining_E_mon_I_cocart} and its cartesian analogue obtained by taking opposites, there are Bousfield localizations onto
    \[
    \intmon^{\subuniverse,\subuniverse[I]\text{-}\sqcup}(\internalcatofcats)
    \quad\text{and}\quad
    \intmon^{\subuniverse,\subuniverse[P]\text{-}\sqcap}(\internalcatofcats).
    \]
    The pullback of their fully faithful right adjoints is the full $\vartopos$-subcategory of $(\subuniverse[P,I])$-ambidextrous categories. Passing to left adjoints identifies this category with the pushout of the two localizations in $\intprl$ and gives $L_{\oplus}$.
\end{proof}

Applying this localization to a free $\subuniverse$-monoidal category gives a free ambidextrous category. An explicit description of these objects is one of the main goals of this paper.

\begin{example}\label{cor:existance_of_free_ambi_monoidal_category}
    For $A\in\vartopos$, the $(\subuniverse[P,I])$-ambidextrous category $L_{\oplus}\subuniverse^{\simeq}[A]$ internally represents the composite
    \[
    \begin{aligned}
    \Phi:\intmon^{\subuniverse,(\subuniverse[P,I])\text{-}\oplus}(\internalcatofcats)
    &\into\intsmon(\internalcatofcats)\\
    &\to\internalcatofcats\xto{\intmap(A,-)}\universe.
    \end{aligned}
    \]
    This follows from the adjunction defining $L_{\oplus}$ and the universal property in \Cref{example:free_E_monoidal_cat}.
\end{example}

For $A=*$, several cases recover constructions already encountered.

\begin{remark}
    The following are special cases of the preceding construction. In the last two cases, assume that $\subuniverse$ is inductible.
    \begin{mylist}
        \item If $\subuniverse[I]=\subuniverse[P]=\explicitset{*}$ is the context-free subuniverse of equivalences, then
        \[
        L_{\oplus}\subuniverse^{\simeq}\simeq\subuniverse^{\simeq}.
        \]
        \item If $\subuniverse[I]=\subuniverse$ and $\subuniverse[P]=\explicitset{*}$, then
        \[
        L_{\oplus}\subuniverse^{\simeq}\simeq\subuniverse.
        \]
        This follows from the free cocompletion property in \Cref{prop:S_is_free_with_S_colmitis} and the full faithfulness of the cocartesian construction in \Cref{thm:uniqueness_of_cartesian_structures}.
        \item If $\subuniverse[I]=\subuniverse[P]=\subuniverse$, then
        \[
        L_{\oplus}\subuniverse^{\simeq}\simeq\intspan(\subuniverse).
        \]
        Its universal property among $\subuniverse$-ambidextrous $\subuniverse$-monoidal categories is established in \cite{CLLambi}.
    \end{mylist}
\end{remark}

\subsubsection*{Recognition criteria}

We finish by describing the local objects of $L_{\oplus}$ using norm maps and indexed tensor products. The first criterion applies the inductive constructions of \Cref{def:cartesian_S_monoidal_category} and its dual to the relevant classes of morphisms in $\subuniverse$.

\begin{definition}
    An $\subuniverse$-monoidal $\vartopos$-category $\varintcat{C}$ is \emph{inductively $(\subuniverse[P,I])$-ambidextrous} if the inductive norm constructions of \Cref{def:cartesian_S_monoidal_category} and its dual are defined and give natural equivalences
    \[
    s_\sharp\xto{\sim}s_\otimes
    \quad\text{for }s\in\subuniverse_{\subuniverse[I]},
    \qquad
    s_\otimes\xto{\sim}s_*
    \quad\text{for }s\in\subuniverse_{\subuniverse[P]},
    \]
    in every context.
\end{definition}

\begin{proposition}\label{prop:inductive_amby_iff_ambi}
    An $\subuniverse$-monoidal $\vartopos$-category is inductively $(\subuniverse[P,I])$-ambidextrous if and only if it is $(\subuniverse[P,I])$-ambidextrous.
\end{proposition}
\begin{proof}
    If the category is inductively $(\subuniverse[P,I])$-ambidextrous, restricting the norm conditions to $\subuniverse[I]$ and $\subuniverse[P]$ shows that the induced structures are cocartesian and cartesian, respectively.

    Conversely, the norm constructions and the condition that they be equivalences are independent of context. A morphism in $\subuniverse_{\subuniverse[I]}$, respectively $\subuniverse_{\subuniverse[P]}$, can be regarded in the context of its target as an object of $\subuniverse[I]$, respectively $\subuniverse[P]$. The required norm equivalences therefore follow from the cocartesian and cartesian structures on the restrictions.
\end{proof}

The preservation criterion of \Cref{prop:C_is_cartesian_iff_tensor_preserves_limits} likewise extends to partial structures.

\begin{corollary}\label{cor:critiria_for_amby}
    An $\subuniverse$-monoidal $\vartopos$-category $\varintcat{C}$ is $(\subuniverse[P,I])$-ambidextrous if and only if the following conditions hold in every context $A$:
    \begin{mylist}
        \item For every $p:C\to B$ in $\subuniverse_{\subuniverse[I]}(A)$, the functor
        \[
        p_\otimes:\varintcat{C}^{C}\to\varintcat{C}^{B}
        \]
        is $\subuniverse[I]$-cocontinuous.
        \item For every $p:C\to B$ in $\subuniverse_{\subuniverse[P]}(A)$, the functor
        \[
        p_\otimes:\varintcat{C}^{C}\to\varintcat{C}^{B}
        \]
        is $\subuniverse[P]$-continuous.
    \end{mylist}
\end{corollary}
\begin{proof}
    This follows from \Cref{prop:inductive_amby_iff_ambi,prop:C_is_cartesian_iff_tensor_preserves_limits}.
\end{proof}

\subsection{The symmetric monoidal category of \texorpdfstring{$\subuniverse$}{E}-monoidal categories.}\label{subsec:symmetric_monoidal_cat_of_E_monoidal_cats}

In this final subsection, we construct a symmetric monoidal structure on the $\vartopos$-category of $\subuniverse$-monoidal $\vartopos$-categories using internal Day convolution (\Cref{cor:E_monoidal_cats_are_symmetric_monoidal}) and its compatibility with some of the previous constructions.

Our main result is that the localization of $\intsmon(\internalcatofcats)$ imposing $(\subuniverse[P,I])$-ambidexterity is smashing: it is given by tensoring with the idempotent algebra $L_{\oplus}(\subuniverse^{\simeq})$ (\Cref{cor:P_I_ambi_is_a_smashing_localization}). We first prove this for the localization $L_{\subuniverse[I]\sqcup}$ imposing $\subuniverse[I]$-cocartesianness (\Cref{thm:cocart_is_a_mode}).

This subsection was heavily inspired by ideas of Yanovsky.

\subsubsection*{Day convolution and tensor products}

The following construction borrows from Ben-Moshe's work \cite[\S 4.3]{Ben_Moshe_ambi_K}.

In \cite[\S\S 2.5--2.6]{MWPresentabilityAndTopoi}, the authors define symmetric monoidal $\vartopos$-categories\footnote{Note that symmetric monoidal $\vartopos$-categories and $\subuniverse$-monoidal $\vartopos$-categories are different notions.} as sheaves of symmetric monoidal $\infty$-categories. They also show that $\intprl$ carries a natural tensor product, which in this subsection we will call the multilinear tensor product.

In \cite[Remark 2.6.2.6]{MWPresentabilityAndTopoi}, the authors also show that the functor
\[
\intpresh:\internalcatofcats\to\intprl
\]
is symmetric monoidal. Thus, we get a theory of Day convolution for $\vartopos$-categories; if $\varintcat{C}$ is a symmetric monoidal $\vartopos$-category, then there is a symmetric monoidal structure on $\intpresh(\varintcat{C})$ such that the Yoneda embedding
\[
\varintcat{C}\into\intpresh(\varintcat{C})
\]
is symmetric monoidal.

\begin{proposition}
    The $\vartopos$-category $\spanEfE$ carries a canonical symmetric monoidal structure induced by products in $\subuniverse^{\full}$.
\end{proposition}
\begin{proof}
    Observe that in every context $A\in\vartopos$, the $\infty$-category $\subuniverse^{\full}(A)$ admits finite products since morphisms in $\subuniverse$ are stable under base change in $\vartopos$. Moreover, for $p:B\to A$, the functor $p^*:\subuniverse^{\full}(A)\to\subuniverse^{\full}(B)$ preserves finite products.

    Thus, by \cite[Theorem 1.2(iv)]{Haugseng_2017}, the functor $\spanEfE:\vartopos^{\op}\to\catofcats$ lifts canonically to a limit-preserving functor $\spanEfE:\vartopos^{\op}\to\operatorname{CMon}(\catofcats)$. In every context $A\in\vartopos$, the symmetric monoidal structure can be described on objects in $\spanEfE(A)$ by binary products in $\subuniverse^{\full}(A)$.
\end{proof}

Hence, using Day convolution, we get a symmetric monoidal structure on $\intfun(\spanEfE,\universe)$.

\begin{corollary}
    The $\vartopos$-category $\intfun(\spanEfE,\universe)$ carries a symmetric monoidal structure such that the Yoneda embedding
    \[
    \spanEfE^{\op}\into\intfun(\spanEfE,\universe)
    \]
    is symmetric monoidal.
\end{corollary}
\begin{proof}
    As explained above, this follows from \cite[Remark 2.6.2.6]{MWPresentabilityAndTopoi}.
\end{proof}

This gives the promised symmetric monoidal structure on $\intsmon$.

\begin{corollary}\label{cor:E_cmon_is_symmetric_monoidal}
    The $\vartopos$-category $\intsmon$ carries a symmetric monoidal structure such that the localization
    \[
    \intfun(\spanEfE,\universe)\onto\intsmon
    \]
    is symmetric monoidal.
\end{corollary}
\begin{proof}
    The proof of \cite[Lemma 4.23]{Ben_Moshe_ambi_K} carries over unchanged to the internal setting and can be checked in every context.
\end{proof}

We can also extend this symmetric monoidal structure to a symmetric monoidal structure on $\intsmon(\internalcatofcats)$.

\begin{corollary}\label{cor:E_monoidal_cats_are_symmetric_monoidal}
    The $\vartopos$-category $\intsmon(\internalcatofcats)$ carries a symmetric monoidal structure such that the inclusion
    \[
    \intsmon\into\intsmon(\internalcatofcats)
    \]
    is a morphism of commutative algebras in $\intprl$.
\end{corollary}
\begin{proof}
    Recall that we showed in the proof of \Cref{prop:cmon_is_cat_module_in_prl} that
    \[
    \intsmon(\internalcatofcats)\simeq\intsmon\otimes\internalcatofcats,
    \]
    with respect to the multilinear tensor product. The inclusion is identified with the tensor product of the unit $\universe\into\internalcatofcats$ with $\intsmon$.
\end{proof}

\begin{notation}
    Diverging from Ben-Moshe's notation, we will denote the tensor product in $\intsmon(\internalcatofcats)$ by $\otimes$. If we have several context-free subuniverses, then we might denote the different tensor products by $\otimes^{\subuniverse}$.

    We also have an internal Hom functor, which we denote by $\intfun^{\otimes}$ or $\intfun^{\subuniverse-\otimes}$.
\end{notation}

\begin{remark}
    Note that if $\subuniverse$ is inductible, then $\intfun^{\subuniverse-\otimes}$ agrees with the self-enrichment described in \Cref{cor:cmon_is_self_enriched_if_inductible}, as both of them lift the $\internalcatofcats$-enrichment coming from the $\internalcatofcats$-module structure.
\end{remark}

\subsubsection*{Monoidal localization and change of subuniverse}

Next, we will show that the $\vartopos$-subcategory of $\subuniverse[I]$-cocartesian $\subuniverse$-monoidal $\vartopos$-categories is a monoidal localization of $\intsmon(\internalcatofcats)$. Recall that our current goal is to show that it is a smashing localization.

\begin{proposition}\label{prop:I_cocart_loc_is_sym_monoidal}
    The $\vartopos$-category $\intmon^{\subuniverse,\subuniverse[I]-\sqcup}(\internalcatofcats)$ carries a symmetric monoidal structure such that the localization
    \[
    L_{\subuniverse[I]-\sqcup}:\intsmon(\internalcatofcats)\onto\intmon^{\subuniverse,\subuniverse[I]-\sqcup}(\internalcatofcats)
    \]
    is a morphism of commutative algebras in $\intprl$.
\end{proposition}
\begin{proof}
    By the same argument as in \cite[Lemma 4.20]{Ben_Moshe_ambi_K}, it is enough to show that a morphism $F:\varintcat{C}\to\varintcat{D}$ is an $L_{\subuniverse[I]-\sqcup}$-equivalence if and only if, for any $\subuniverse[I]$-cocartesian $\subuniverse$-monoidal $\vartopos$-category $\varintcat{T}$, the morphism
    \[
    \intfun^{\subuniverse-\otimes}(\varintcat{D},\varintcat{T})\to\intfun^{\subuniverse-\otimes}(\varintcat{C},\varintcat{T})
    \]
    is an equivalence.

    Note that since the forgetful functor $\intsmon(\internalcatofcats)\to\internalcatofcats\into\universe_{\Delta}$ is conservative, a $\vartopos$-category $\varintcat{T}$ is internally local with respect to $F:\varintcat{C}\to\varintcat{D}$ if and only if, for every $n$, $\varintcat{T}^{\Delta^n}$ is local with respect to $F:\varintcat{C}\to\varintcat{D}$.

    Thus, the result follows by observing that an $\subuniverse$-monoidal $\vartopos$-category $\varintcat{T}$ is $\subuniverse[I]$-cocartesian if and only if $\varintcat{T}^{\Delta^n}$ is $\subuniverse[I]$-cocartesian for every $n$.
\end{proof}

We now want to understand the interaction between the $\subuniverse$-monoidal structures and the $\subuniverse[I]$-monoidal structures.

\begin{proposition}\label{prop:rel_free_is_sym_mon}
    Let $\subuniverse[I]\subset\subuniverse$ be an inductible context-free subuniverse. Then the functor
    \[
    F_{\subuniverse[I]}^{\subuniverse}:\intsmon[I](\internalcatofcats)\to\intsmon(\internalcatofcats),
    \]
    which is left adjoint to the forgetful functor $U^{\subuniverse}_{\subuniverse[I]}:\intsmon(\internalcatofcats)\to\intsmon[I](\internalcatofcats)$, can be promoted to a morphism of commutative algebras in $\intprl$.
\end{proposition}
\begin{proof}
    The functor $\intspan(\subuniverse[I]^{\full},\subuniverse[I])\into\spanEfE$ is symmetric monoidal since the inclusion $\subuniverse[I]\into\subuniverse^{\full}$ preserves products. Thus, we get that $\intsmon[I]\to\intsmon$ is a morphism of commutative algebras in $\intprl$.

    The morphism $F_{\subuniverse[I]}^{\subuniverse}:\intsmon[I](\internalcatofcats)\to\intsmon(\internalcatofcats)$ is obtained by tensoring $\intsmon[I]\to\intsmon$ with $\internalcatofcats$.
\end{proof}

In fact, we can describe $\intmon^{\subuniverse,\subuniverse[I]-\sqcup}(\internalcatofcats)$ as a tensor product.

\begin{proposition}\label{prop:pushout_in_prl_I_cocart}
    The following is a pushout in $\intalg(\intprl)$:
    \[
    \begin{tikzcd}
        \intsmon[I](\internalcatofcats)\ar[d]\ar[r]
        &\internalcatofcats^{\subuniverse[I]-\sqcup}\ar[d]\\
        \intsmon(\internalcatofcats)\ar[r]
        &\intmon^{\subuniverse,\subuniverse[I]-\sqcup}(\internalcatofcats).
    \end{tikzcd}
    \]
\end{proposition}
\begin{proof}
    Note that the horizontal morphisms are monoidal Bousfield localizations. Hence, to show that the square is a pushout in $\intalg(\intprl)$, it suffices to show that it is a pushout in $\intprl$. This was shown in \Cref{prop:pushout_in_prl_defining_E_mon_I_cocart}.
\end{proof}

\subsubsection*{Algebras and the localized unit}

\begin{corollary}
    Let $\varintcat{C}$ be an $\subuniverse$-monoidal $\vartopos$-category. Then $\intalg^{\subuniverse[I]}(\varintcat{C})$ carries a canonical $\subuniverse$-monoidal structure extending the $\subuniverse[I]$-monoidal structure that was discussed before. Moreover, we have a canonical isomorphism of $\subuniverse$-monoidal $\vartopos$-categories:
    \[
    \intalg^{\subuniverse[I]}(\varintcat{C})\simeq\intfun^{\subuniverse-\otimes}(F_{\subuniverse[I]}^{\subuniverse}(\subuniverse[I]),\varintcat{C}).
    \]
\end{corollary}
\begin{proof}
    This follows since
    \begin{align*}
        U^{\subuniverse}_{\subuniverse[I]}(\intfun^{\subuniverse-\otimes}(F_{\subuniverse[I]}^{\subuniverse}(\subuniverse[I]),\varintcat{C}))
        &\simeq\intfun^{\subuniverse[I]-\otimes}(\subuniverse[I],U^{\subuniverse}_{\subuniverse[I]}\varintcat{C})\\
        &\simeq\intalg^{\subuniverse[I]}(U^{\subuniverse}_{\subuniverse[I]}\varintcat{C}).
    \end{align*}
\end{proof}

Moreover, by the same proof, the forgetful functor $\intalg^{\subuniverse[I]}(\varintcat{C})\to\varintcat{C}$ is a natural morphism of $\subuniverse$-monoidal $\vartopos$-categories which is internally represented by $\subuniverse^{\simeq}\to F_{\subuniverse[I]}^{\subuniverse}(\subuniverse[I])$.

\begin{corollary}
    The fully faithful right adjoint
    \[
    \intmon^{\subuniverse,\subuniverse[I]-\sqcup}(\internalcatofcats)\into\intsmon(\internalcatofcats)
    \]
    has a further right adjoint, which is given by $\intalg^{\subuniverse[I]}$.
\end{corollary}
\begin{proof}
    Note that $U^{\subuniverse}_{\subuniverse[I]}:\intsmon(\internalcatofcats)\to\intsmon[I](\internalcatofcats)$ is conservative. Thus, the $\vartopos$-category $\intmon^{\subuniverse,\subuniverse[I]-\sqcup}(\internalcatofcats)$ is exactly the full $\vartopos$-subcategory of objects for which the natural morphism of endofunctors
    \[
    \intalg^{\subuniverse[I]}\to\id
    \]
    is an isomorphism. Hence this is the counit of the adjunction.
\end{proof}

\begin{corollary}
    The unit of $\intmon^{\subuniverse,\subuniverse[I]-\sqcup}(\internalcatofcats)$ is $F_{\subuniverse[I]}^{\subuniverse}\subuniverse[I]$.
\end{corollary}
\begin{proof}
    The Beck--Chevalley transformation of \Cref{prop:pushout_in_prl_I_cocart} gives an $\subuniverse$-monoidal functor $F_{\subuniverse[I]}^{\subuniverse}\subuniverse[I]\to L_{\subuniverse[I]-\sqcup}F_{\subuniverse[I]}^{\subuniverse}\subuniverse[I]^{\simeq}$. The target is the unit of $\intmon^{\subuniverse,\subuniverse[I]-\sqcup}(\internalcatofcats)$.

    This functor identifies the target with the $L_{\subuniverse[I]-\sqcup}$-localization of the source. Thus, we must show that the source is $L_{\subuniverse[I]-\sqcup}$-local.

    The composition
    \[
    \intsmon(\internalcatofcats)\xto{\intalg^{\subuniverse[I]}}\intmon^{\subuniverse,\subuniverse[I]-\sqcup}(\internalcatofcats)\into\intsmon(\internalcatofcats)
    \]
    is given by $\intfun^{\subuniverse-\otimes}(F_{\subuniverse[I]}^{\subuniverse}\subuniverse[I],-)$. Thus, its left adjoint is given by $F_{\subuniverse[I]}^{\subuniverse}\subuniverse[I]\otimes-$. Applying this to the unit gives that $F_{\subuniverse[I]}^{\subuniverse}\subuniverse[I]$ is $\subuniverse[I]$-cocartesian.
\end{proof}

\subsubsection*{Smashing localizations}

We can now prove that the localization $L_{\subuniverse[I]-\sqcup}$ is smashing.

\begin{theorem}\label{thm:cocart_is_a_mode}
    The $\vartopos$-category $L_{\subuniverse[I]-\sqcup}(\subuniverse^{\simeq})$ is an idempotent algebra in $\intsmon(\internalcatofcats)$.

    Let $\varintcat{C}$ be an $\subuniverse$-monoidal $\vartopos$-category. Then there is a canonical isomorphism
    \[
    L_{\subuniverse[I]-\sqcup}\varintcat{C}\simeq\varintcat{C}\otimes L_{\subuniverse[I]-\sqcup}(\subuniverse^{\simeq}).
    \]
\end{theorem}
\begin{proof}
    We show the first statement. Observe that the arrow $\intalg^{\subuniverse[I]}\circ\intalg^{\subuniverse[I]}\to\intalg^{\subuniverse[I]}$ is an equivalence of endofunctors of $\intsmon(\internalcatofcats)$. Passing to left adjoints, we get an isomorphism of endofunctors of $\intsmon(\internalcatofcats)$
    \[
    L_{\subuniverse[I]-\sqcup}(\subuniverse^{\simeq})\otimes L_{\subuniverse[I]-\sqcup}(\subuniverse^{\simeq})\otimes-\simeq L_{\subuniverse[I]-\sqcup}(\subuniverse^{\simeq})\otimes-,
    \]
    establishing the result.

    We show the second statement. By \Cref{prop:I_cocart_loc_is_sym_monoidal}, the morphism of endofunctors of $\intsmon(\internalcatofcats)$, $\id\to-\otimes L_{\subuniverse[I]-\sqcup}(\subuniverse^{\simeq})$, is an $L_{\subuniverse[I]-\sqcup}$-equivalence. We want to show that the target is always $L_{\subuniverse[I]-\sqcup}$-local. By the same argument as in \Cref{cor:Icocart_iff_forgetful_colocal}, it is enough to show that for any $\subuniverse$-monoidal $\vartopos$-categories $\varintcat{C,T}$, the morphism
    \[
    \begin{aligned}
        &\intfun^{\subuniverse-\otimes}(\varintcat{C}\otimes L_{\subuniverse[I]-\sqcup}(\subuniverse^{\simeq}),\intalg^{\subuniverse[I]}(\varintcat{T}))\\
        &\qquad\to\intfun^{\subuniverse-\otimes}(\varintcat{C}\otimes L_{\subuniverse[I]-\sqcup}(\subuniverse^{\simeq}),\varintcat{T})
    \end{aligned}
    \]
    is an isomorphism. This follows from the first statement.
\end{proof}

More generally:

\begin{corollary}\label{cor:P_I_ambi_is_a_smashing_localization}
    We have an equivalence
    \[
    L_{(\subuniverse[P],\subuniverse[I])-\oplus}(\subuniverse^{\simeq})\simeq L_{\subuniverse[P]-\sqcap}(\subuniverse^{\simeq})\otimes L_{\subuniverse[I]-\sqcup}(\subuniverse^{\simeq})
    \]
    in $\intsmon(\internalcatofcats)$. Therefore, $L_{(\subuniverse[P],\subuniverse[I])-\oplus}(\subuniverse^{\simeq})$ is an idempotent algebra in $\intsmon(\internalcatofcats)$.

    Let $\varintcat{C}$ be an $\subuniverse$-monoidal $\vartopos$-category. Then there is a canonical isomorphism
    \[
    L_{(\subuniverse[P,I])-\oplus}\varintcat{C}\simeq\varintcat{C}\otimes L_{(\subuniverse[P],\subuniverse[I])-\oplus}(\subuniverse^{\simeq}).
    \]
\end{corollary}
\begin{proof}
    The equivalence $L_{(\subuniverse[P],\subuniverse[I])-\oplus}(\subuniverse^{\simeq})\simeq L_{\subuniverse[P]-\sqcap}(\subuniverse^{\simeq})\otimes L_{\subuniverse[I]-\sqcup}(\subuniverse^{\simeq})$ follows since both objects have the same universal property in $\intmon^{\subuniverse,(\subuniverse[P,I])-\oplus}(\internalcatofcats)$. Idempotent algebras are closed under tensor products, and thus $L_{(\subuniverse[P],\subuniverse[I])-\oplus}(\subuniverse^{\simeq})$ is an idempotent algebra.

    We now show the equivalence $L_{(\subuniverse[P,I])-\oplus}\varintcat{C}\simeq\varintcat{C}\otimes L_{(\subuniverse[P],\subuniverse[I])-\oplus}(\subuniverse^{\simeq})$. Indeed, we have equivalences:
    \begin{align*}
        L_{(\subuniverse[P,I])-\oplus}\varintcat{C}
        &\simeq L_{\subuniverse[P]-\sqcap}L_{\subuniverse[I]-\sqcup}\varintcat{C}\\
        &\simeq L_{\subuniverse[P]-\sqcap}(\subuniverse^{\simeq})\otimes L_{\subuniverse[I]-\sqcup}(\subuniverse^{\simeq})\otimes\varintcat{C}\\
        &\simeq L_{(\subuniverse[P],\subuniverse[I])-\oplus}(\subuniverse^{\simeq})\otimes\varintcat{C}.
    \end{align*}
\end{proof}

\newpage

\section{Some internal \texorpdfstring{$2$}{2}-category theory}\label{sec:2_cats}

The goal of the next section, \Cref{sec:proof_of_C}, is to give a concrete description of $L_{(\subuniverse[I],\subuniverse[P])-\oplus}(\subuniverse^{\full}[A]^{\simeq})$ and use it to deduce \Cref{main_theorem:free_E_monoidal_semiadditive}. In this section, we develop the categorical tools needed for that construction. These results also have independent interest, though their role in the proof is best understood through their application in the next section. Readers primarily interested in the structure of the argument may find it helpful to begin there and return to the constructions here as needed.

In \Cref{subsec:internal_two_categories} we introduce $\vartopos$-$2$-categories, lax natural transformations, and cocartesian fibrations. The main result used in \Cref{sec:proof_of_C} is that a lax natural transformation induces a functor between unstraightenings.

In \Cref{subsec:internal_span_universal_property} we establish the internal universal property of the $\vartopos$-$2$-categories $\intspanhalf(\varintcat{C},\varintcat{C}_R)$ (\Cref{prop:univ_prop_of_span_half}) as the initial recipients of a right adjointable functor. 
Using the universal property, for an $\subuniverse^{\full}$-complete $\vartopos$-category $\varintcat{C}$, we construct the functor $\mu:\varintcat{C}^{\times}\to\varintcat{C}$ from the operad associated with its cartesian structure back to $\varintcat{C}$. The functor $\mu$ is given informally by taking the limit of a family of objects of $\varintcat{C}$ indexed by an object of $\subuniverse$.

In \Cref{subsec:iterated_internal_spans} we construct iterated span $\vartopos$-$2$-categories $\intspantwo(\varintcat{C},\varintcat{C}_R)$. Using this $\vartopos$-$2$-category, we describe the operad associated with the cartesian $\subuniverse$-monoidal $\vartopos$-category $\spanEfE$ (\Cref{prop:pullback_describing_span_EfE_times_using_span_2}). This framework is used in the next subsection to describe $\mu$ explicitly.

In \Cref{subsec:span_operad_description} we identify $\spanEfE^\times$ with an explicit $\vartopos$-category of spans of arrows (\Cref{prop:span_operad_as_spans_of_arrows}). Under this identification, $\mu$ is the source functor (\Cref{prop:description_of_mu_for_span}). We use this description to compare $\mu$ with the source functor appearing in the envelope construction from \Cref{prop:source_target_identity_adjunctions}.

In \Cref{subsec:envelope_as_restriction} we equip $\varintcat{C}^{\otimes}$ with its cartesian $\subuniverse$-monoidal structure and identify its restriction to $\subuniverse$ with the monoidal envelope (\Cref{thm:C_otimes_res_to_E_is_envelope}). We also construct the monoidal structure on the unstraightening of a lax $\subuniverse$-monoidal functor (\Cref{prop:lax_E_monoidal_unst}). The comparison between $\varintcat{C}^\otimes|_{\subuniverse}$ and $\intenv(\varintcat{C}^{\otimes})$ then shows that the $\subuniverse$ monoidal structure on $\varintcat{C}$ coming from its identification with the unstraightenning of the lax $\subuniverse$ monoidal functor
\[\explicitset{*}\into\spanEfE\xto{\varintcat{C}} \catofcats\]
identifies with the original structure (\Cref{ex:lax_unstraightening_recovers_monoidal_structure}).

In \Cref{subsec:E_monoidal_two_categories} we develop the monoidal versions of the unstraightening constructions in the context of $\vartopos$-$2$-categories. \Cref{prop:monoidal_two_unstraightening} equips the unstraightening of an $\subuniverse$-monoidal functor of $\vartopos$-$2$-categories with an $\subuniverse$-monoidal structure. \Cref{prop:unst_of_lax_E_monoidal_lax_natural_transformation} shows that unstraightening a lax monoidal lax natural transformation gives a lax monoidal functor.
We expect that the constructions of \Cref{prop:lax_E_monoidal_unst} and the constructions in this subsection will commonly generalize to a theory of $\subuniverse$-$2$-operads, which we do not develop in this paper.

Finally, in \Cref{subsec:E_monoidal_span_categories} we construct a candidate $\spanEPI[A]$ for the free $(\subuniverse[P],\subuniverse[I])$ ambidextrous $\subuniverse$-monoidal category on $A$. We will later show (\Cref{thm:generalized_main_C}) that, under certain conditions, $\spanEPI[A]$ is indeed free.
We also establish the $\subuniverse$-monoidal universal property of span $\vartopos$-$2$-categories $\intspanhalf(\varintcat{C},\varintcat{C}_R)$.

\subsection{\texorpdfstring{$\vartopos$-$2$-}{B-2-}categories.}\label{subsec:internal_two_categories}

In this subsection, we introduce $\vartopos$-$2$-categories (\Cref{def:internal_two_categories}) and develop the constructions needed to unstraighten lax natural transformations. We first discuss the Gray tensor product and its internal versions (\Cref{prop:gray_module_structures,prop:internal_gray_tensor_product}), followed by the description of lax natural transformations using the oplax arrow category (\Cref{eq:lax_transformations_as_oplax_arrows}). We then introduce $1$-cocartesian fibrations between $\vartopos$-$2$-categories (\Cref{def:internal_one_cocartesian_fibration}). We identify the universal cocartesian fibration of $\vartopos$-categories with the underlying $\vartopos$-category of an oplax slice (\Cref{prop:univ_internal_2_cocart_fib}), which we use to define unstraightening for functors of $\vartopos$-$2$-categories (\Cref{def:internal_two_unstraightening}). Finally, we construct the functor between unstraightenings induced by a lax natural transformation, whose restriction to each fibre agrees with the corresponding component of the transformation (\Cref{cor:extra_functoriality_of_unst}).

\subsubsection*{A zoo of categories of categories}

\begin{definition}\label{def:internal_two_categories}
    We define a $\vartopos$-$2$-category $\varinttwocat{C}$ as a sheaf 
    \[
    \varinttwocat{C}:\vartopos^{\op}\to \catoftwocats
    \]
    on $\vartopos$ taking values in the $(\infty,1)$-category of $(\infty,2)$-categories.

    We define the $\infty$-category of $\vartopos$-$2$-categories as the Lurie tensor product \[\catoftwocats(\vartopos):=\vartopos\otimes \catoftwocats \simeq \Fun^R(\vartopos^{\op}, \catoftwocats). \]

    Using the construction $\universe\otimes \catoftwocats$ from \cite[Construction A.1]{MWcocomplete}, we define the large $\vartopos$-category of $\vartopos$-$2$-categories $\internalcatoftwocats\simeq \catoftwocats\otimes \universe$, or more explicitly \[\internalcatoftwocats: \vartopos^{\op}\to \catoflargecats: A\mapsto \catoftwocats(\vartopos_{/A}).\]

    Of course, we also have large versions of these objects, denoted by $\widehat{\catoftwocats}(\vartopos)$ and $\widehat{\internalcatoftwocats}$, respectively. 
\end{definition}

Note that we choose the fonts to denote these categories according to their types, rather than according to the types of their objects. Thus $\internalcatoftwocats$ is denoted by a non-bold sans-serif font when considered as the $\vartopos$-category of $\vartopos$-$2$-categories.

\begin{example}
    Note that the $\infty$-category of $\vartopos$-categories $\catofcats(\vartopos)$ is also defined as the Lurie tensor product $\vartopos\otimes \catofcats$ and thus it is a $\catofcats$-module in $\prl$. In particular, it is the underlying $\infty$-category of a $(\infty,2)$-category $\twocatofcats(\vartopos)$. We define the large $\vartopos$-$2$-category of $\vartopos$-categories 
    \[\internaltwocatofcats:\vartopos^{\op}\to \widehat{\catoftwocats}:A\mapsto \twocatofcats(\vartopos_{/A}).\]
\end{example}

\begin{example}
    Since the $\infty$-category of $\vartopos$-$2$-categories $\catoftwocats(\vartopos)$ is defined as the Lurie tensor product $\vartopos\otimes \catoftwocats$, and since $\catoftwocats$ is a $\catofcats$-module in $\prl$, $\catoftwocats(\vartopos)$ is a $\catofcats$-module in $\prl$. In particular, it is the underlying $\infty$-category of a $(\infty,2)$-category $\twocatoftwocats(\vartopos)$. We define the large $\vartopos$-$2$-category of $\vartopos$-$2$-categories 
    \[\internaltwocatoftwocats:\vartopos^{\op}\to \widehat{\catoftwocats}:A\mapsto \twocatoftwocats(\vartopos_{/A}).\]
\end{example}

\begin{definition}
    Let $\varinttwocat{C}$ be a $\vartopos$-$2$-category. We define a \emph{$\vartopos$-$2$-subcategory} to be a morphism $\varinttwocat{C}_0\to \varinttwocat{C}$  such that in every context $A\in \vartopos$, the $2$-functor $\varinttwocat{C}_0(A)\to \varinttwocat{C}(A)$ is an inclusion of a $(\infty,2)$-subcategory.

    We say that a $\vartopos$-$2$-subcategory is full if in every context $A\in \vartopos$ the $(\infty,2)$-subcategory $\varinttwocat{C}_0(A)\to \varinttwocat{C}(A)$ is full in the sense that it is an equivalence on mapping $\infty$-categories.

    We say that a $\vartopos$-$2$-subcategory is locally full if in every context $A\in \vartopos$ the $(\infty,2)$-subcategory $\varinttwocat{C}_0(A)\to \varinttwocat{C}(A)$ is locally full in the sense that it is fully faithful on mapping $\infty$-categories.
\end{definition}

\begin{example}
    We have natural inclusions of full $\vartopos$-$2$-subcategories \[\internaltwocatofcats\subset \internaltwocatoftwocats, \quad \widehat{\internaltwocatofcats}\subset \widehat{\internaltwocatoftwocats}.\]
    We also have natural inclusions of $\vartopos$-$2$-subcategories
    \[\internalcatofcats\subset \internaltwocatofcats,\quad \internalcatoftwocats\subset \internaltwocatoftwocats.\]

    We will not denote any of these inclusions in the text, as we view them as natural inclusions.
\end{example}


\begin{proposition}
    The natural inclusion 
    \[\internalcatofcats\subset \internalcatoftwocats\]
    admits a right adjoint $(-)^{\simeq2}:\internalcatoftwocats\to \internalcatofcats$. The counit $\varinttwocat{C}^{\simeq2}\into \varinttwocat{C}$ is the embedding of the $\vartopos$-$2$-subcategory spanned by invertible $2$-morphisms.
\end{proposition}

\begin{proof}
    This follows since the inclusion $\catofcats\into \catoftwocats$ is a morphism in $\prl$ and so it induces an adjunction after applying $-\otimes \universe$.
\end{proof}

\subsubsection*{The Gray tensor product and lax natural transformations}

The $\infty$-category $\catoftwocats$ is an associative algebra in $\prl$ with the underlying monoidal structure called the \emph{Gray} tensor product, which we denote \[\gray:\catoftwocats\otimes \catoftwocats\to \catoftwocats.\]
This tensor product was defined and studied in \cite{campion2023modelindependentgraytensorproduct}.

The defining adjunction of the Gray tensor product expresses its relation to the $(\infty,2)$-category of \emph{lax} natural transformations. Let $\vartwocat{I},\vartwocat{J},\vartwocat{C}$ be three $(\infty,2)$-categories. Then there is a canonical adjunction
\[\map_{\catoftwocats}(\vartwocat{I}\gray\vartwocat{J},\vartwocat{C})\simeq \map_{\catoftwocats}(\vartwocat{I},\funtwocat^{\lax}(\vartwocat{J},\vartwocat{C})).\]
The $(\infty,2)$-category $\funtwocat^{\lax}(\vartwocat{J},\vartwocat{C})$ is the $(\infty,2)$-category of $2$-functors $\vartwocat{J}\to \vartwocat{C}$, lax natural transformations and modifications between them. 

\begin{proposition}\label{prop:gray_module_structures}
    The $\infty$-category $\catoftwocats(\vartopos)$ carries a left and a right module structure over $\catoftwocats$, equipped with the monoidal structure given by the Gray tensor product $\gray$. Thus $\catoftwocats(\vartopos)$ is tensored and cotensored under $\catoftwocats$ in two different ways.
\end{proposition}

\begin{proof}
    This follows immediately from the isomorphism $\catoftwocats(\vartopos)\simeq \catoftwocats\otimes \vartopos$ in $\prl$.
\end{proof}

Let us spell out this structure explicitly:

We have the action functors
\[\gray:\catoftwocats\times \catoftwocats(\vartopos)\to \catoftwocats(\vartopos)\]
and 
\[\gray:\catoftwocats(\vartopos)\times \catoftwocats\to \catoftwocats(\vartopos).\]
Let $\varinttwocat{C}\in \catoftwocats(\vartopos)$, and let $\vartwocat{I}\in \catoftwocats$. Then \[\varinttwocat{C}:\vartopos^{\op}\to \catoftwocats\] is a sheaf on $\vartopos$ valued in $(\infty,2)$-categories. Then we get that $\vartwocat{I}\gray\varinttwocat{C}$ and $\varinttwocat{C}\gray \vartwocat{I}$ are the sheafifications of the presheaves given by the compositions
\[
\vartopos^{\op}\xto{\varinttwocat{C}} \catoftwocats\xto{\vartwocat{I}\gray - } \catoftwocats
\]
and 
\[
\vartopos^{\op}\xto{\varinttwocat{C}} \catoftwocats\xto{-\gray \vartwocat{I} } \catoftwocats
\]
respectively.

The cotensoring structures are given by the right adjoints. Let $\vartwocat{I}\in \catoftwocats$. Then we get two adjunctions:
\[
-\gray \vartwocat{I}: \catoftwocats(\vartopos) \fromto \catoftwocats(\vartopos): \inttwofun^{\lax}(\vartwocat{I}, - )
\]
and 
\[
\vartwocat{I}\gray - : \catoftwocats(\vartopos) \fromto \catoftwocats(\vartopos): \inttwofun^{\oplax}(\vartwocat{I}, - ).
\]

The right adjoints $\inttwofun^{\lax}(\vartwocat{I}, - )$ and $\inttwofun^{\oplax}(\vartwocat{I}, - )$ are even easier to define. Let $\varinttwocat{C}\in \catoftwocats(\vartopos)$. Then $\inttwofun^{\lax}(\vartwocat{I}, \varinttwocat{C} )$  is the composition (without the need for sheafification)
\[\vartopos^{\op}\xto{\varinttwocat{C}} \catoftwocats\xto{\funtwocat^{\lax}(\vartwocat{I}, - )} \catoftwocats \]
 while 
 $\inttwofun^{\oplax}(\vartwocat{I}, \varinttwocat{C} )$  is the composition (without the need for sheafification)
\[\vartopos^{\op}\xto{\varinttwocat{C}} \catoftwocats\xto{\funtwocat^{\oplax}(\vartwocat{I}, - )} \catoftwocats. \]

In other words, we have a canonical equivalence of $(\infty,2)$-categories

\[\inttwofun^{\oplax}(\vartwocat{I}, \varinttwocat{C} )(A) = \funtwocat^{\oplax}(\vartwocat{I}, \varinttwocat{C}(A)). \]

\begin{proposition}\label{prop:internal_gray_tensor_product}
    $\catoftwocats(\vartopos)$ carries a further structure of an associative algebra in $\prl$.
\end{proposition}
\begin{proof}
    This follows since the functor \[-\otimes \vartopos: \prl\to \operatorname{Mod}_{\vartopos}(\prl)\]
    is symmetric monoidal, and its right adjoint, which is the forgetful functor $\operatorname{Mod}_{\vartopos}(\prl)\to \prl$, is lax symmetric monoidal.
\end{proof}

We will denote the  monoidal structure by $\gray$ and the internal homs by $\inttwofun^{\lax},\inttwofun^{\oplax}$ as before.

The formulas for the Gray tensor products of $\vartopos$-$2$-categories are more complicated. We will not really need much of this more general structure, so we will not explicate them further.

\begin{remark}
    Let $\vartwocat{I}$ be some $(\infty,2)$-category. Then we can view $\vartwocat{I}$ as the constant sheaf on $\vartopos$ with value $\vartwocat{I}$. In this case we have that the two meanings of \[\inttwofun^{\lax}(\vartwocat{I},-),\] as the internal hom in $\catoftwocats(\vartopos)$ and the one given by the structure of $\catoftwocats(\vartopos)$ as cotensored over $\catoftwocats$ coincide. 
\end{remark}

\subsubsection*{Lax natural transformations and oplax arrows}

Of particular interest to us in this work is the case where $\vartwocat{I} = \Delta^1$. Let $\varinttwocat{C}$ be a $\vartopos$-$2$-category, and let $\varintcat{T}$ be some $\vartopos$-category.

Let us explicate the objects and morphisms in $\inttwofun^{\lax}(\varintcat{T},\varinttwocat{C})$. The objects in the context $A$ are functors of $\vartopos$-$2$-categories  \[F:A\times\varintcat{T}\to \varinttwocat{C}\] which are the same as functors 
\[F:A\times\varintcat{T}\to \varinttwocat{C}^{\simeq2}.\]

Morphisms $\alpha:F\to G$ in $\inttwofun^{\lax}(\varintcat{T},\varinttwocat{C})$  between $F,G:A\times\varintcat{T}\to \varinttwocat{C}$ are given by the following data. For any object $x:B\to \varintcat{T}$, in the context $B$ where $B\in \vartopos_{/A}$, we have a $1$-morphism $\alpha_x:F(x)\to G(x)$ in $\varinttwocat{C}$. For any morphism $f:x\to y$ in the context $B$, we have a $2$-morphism in $\varinttwocat{C}(B)$:
\[\begin{tikzcd}
	{F(x)} & {G(x)} \\
	{F(y)} & {G(y)}
	\arrow["{\alpha_x}", from=1-1, to=1-2]
	\arrow["{F(f)}"', from=1-1, to=2-1]
	\arrow["{\alpha_f}"', Rightarrow, nfold, from=1-2, to=2-1]
	\arrow["{G(f)}", from=1-2, to=2-2]
	\arrow["{\alpha_y}"', from=2-1, to=2-2]
\end{tikzcd}\]
These are accompanied by higher coherence data for compositions in $\varintcat{T}$.

By the adjunctions described above, we have canonical equivalences of $\vartopos$-$2$-categories: 
\begin{equation}\label{eq:lax_transformations_as_oplax_arrows}
\begin{aligned}
    \inttwofun(\Delta^1,\inttwofun^{\lax}(\varintcat{T},\varinttwocat{C}))
    &\simeq \inttwofun(\Delta^1\gray \varintcat{T},\varinttwocat{C})\\
    &\simeq \inttwofun(\varintcat{T},\inttwofun^{\oplax}(\Delta^1, \varinttwocat{C})).
\end{aligned}
\end{equation}

\subsubsection*{Cocartesian fibrations}

\begin{definition}
    Let $p:\vartwocat{C}\to \vartwocat{D}$ be a functor of $(\infty,2)$-categories. Let $f:x\to y$ be a $1$-morphism in $\vartwocat{C}$. We say that $f$ is $p$-cocartesian if the following square is a pullback of $\infty$-categories.
    \[
        \begin{tikzcd}
        \Fun_{\vartwocat C}(y,z)
          \arrow[r, "{\circ f}"]
          \arrow[d, "p"']
        &
        \Fun_{\vartwocat C}(x,z)
          \arrow[d, "p"]
        \\
        \Fun_{\vartwocat D}(p(y),p(z))
          \arrow[r, "{\circ p(f)}"']
        &
        \Fun_{\vartwocat D}(p(x),p(z)).
        \end{tikzcd}
    \]

    In this case we say that $f$ is a cocartesian lift of $p(f):p(x)\to p(y)$ at $x$.
\end{definition}

\begin{definition}\label{def:one_cocartesian_fibration}
    Let $p:\vartwocat{C}\to \vartwocat{D}$ be a functor of $(\infty,2)$-categories. We will say that $p$ is a $1$-cocartesian fibration if the following conditions hold:
    \begin{enumerate}
        \item $p$ admits cocartesian lifts.
        \item For any $x,y\in \vartwocat{C}$, the functor \[\Fun_{\vartwocat{C}}(x,y)\to \Fun_{\vartwocat{D}}(p(x),p(y))\] is a right fibration.
    \end{enumerate}

    A morphism of $1$-cocartesian fibrations from $p:\vartwocat{C}\to\vartwocat{D}$ to $q:\vartwocat{X}\to\vartwocat{Y}$ is a commutative square in $\catoftwocats$:
    \[\begin{tikzcd}
        \vartwocat{C}\ar[r]\ar[d, "p"]& \vartwocat{X}\ar[d, "q"]\\
        \vartwocat{D}\ar[r] & \vartwocat{Y}
    \end{tikzcd}\]
    such that the upper horizontal functor sends $p$-cocartesian edges to $q$-cocartesian edges. 
\end{definition}

We denote by $\twococart_{\vartwocat{C}}$ the $(\infty,2)$-category of $1$-cocartesian fibrations over $\vartwocat{C}$ with morphisms of cocartesian fibrations over $\vartwocat{C}$ as morphisms. We denote by $\twococart$ the $(\infty,2)$-category of cocartesian fibrations and morphisms of cocartesian fibrations.

We have a corresponding (un)straightening equivalence.

\begin{theorem}(\cite[Theorem 3.2]{CLRUniversalityOfUnferling})
    Let $\vartwocat{C}\in \catoftwocats$. Then there is an equivalence of $(\infty,2)$-categories 
    \[\twococart_{\vartwocat{C}}\simeq \funtwocat(\vartwocat{C},\twocatofcats).\]
\end{theorem}

\begin{example}\label{example:oplax_slice}
    Let $\vartwocat{C}$ be an $(\infty,2)$-category. Let $x\in \vartwocat{C}$ be an object. Let $\vartwocat{C}_{x//}$ denote the following pullback:
    \[\begin{tikzcd}
        \vartwocat{C}_{x//}\ar[d]\ar[r]\pullbackdr
        &\funtwocat^{\oplax}(\Delta^1,\vartwocat{C})\ar[d,"s"]\\
        * \ar[r] & \vartwocat{C}.
    \end{tikzcd}\]
    Then the composition 
    \[\vartwocat{C}_{x//}\to \funtwocat^{\oplax}(\Delta^1,\vartwocat{C})\xto{t} \vartwocat{C}\]
    is the cocartesian fibration classifying the $2$-functor 
    \[\Fun_{\vartwocat{C}}(x,-):\vartwocat{C}\to \twocatofcats.\]
\end{example}

\begin{definition}\label{def:internal_one_cocartesian_fibration}
    We say that a functor of $\vartopos$-$2$-categories $p:\varinttwocat{C}\to \varinttwocat{D}$ is a $1$-cocartesian fibration if in every context $A\in \vartopos$, the $2$-functor 
    \[p(A):\varinttwocat{C}(A)\to \varinttwocat{D}(A)\]
    is a $1$-cocartesian fibration and for every change of context $s:B\to A$ the square 
    \[\begin{tikzcd}
        \varinttwocat{C}(A)\ar[r,"s^*"]\ar[d,"p"] & \varinttwocat{C}(B)\ar[d,"p"]\\
        \varinttwocat{D}(A)\ar[r,"s^*"] &\varinttwocat{D}(B).
    \end{tikzcd}\]
    is a morphism of $1$-cocartesian fibrations of $(\infty,2)$-categories. 
\end{definition}

\begin{example}
    Since \[\inttwofun^{\oplax}(\Delta^1,\varinttwocat{C})(A)\simeq \funtwocat^{\oplax}(\Delta^1,\varinttwocat{C}(A))\] is computed contextwise, and pullbacks of sheaves agree with pullbacks of presheaves, we get that if $\varinttwocat{C}$ is a $\vartopos$-$2$-category and $x:A\to \varinttwocat{C}$ is an object defined in the context $A$, we may define a $\vartopos_{/A}$-$2$-category $\varinttwocat{C}_{x//}$ by the pullback in $\catoftwocats(\vartopos)$:
    \[
    \begin{tikzcd}
        \varinttwocat{C}_{x//}\ar[d]\ar[r]\pullbackdr
        &\inttwofun^{\oplax}(\Delta^1,\varinttwocat{C})\ar[d,"s"]\\
        *\ar[r,"x"] & \varinttwocat{C}.
    \end{tikzcd}
    \]
    Then the composition $\varinttwocat{C}_{x//}\to \inttwofun^{\oplax}(\Delta^1,\varinttwocat{C}) \xto{t} \varinttwocat{C}$ is a $1$-cocartesian fibration of $\vartopos$-$2$-categories.
\end{example}

\begin{proposition}
    Let the following be a pullback in $\catoftwocats(\vartopos)$. If $p$ is a $1$-cocartesian fibration of $\vartopos$-$2$-categories, then so is $p'$.
    \[\begin{tikzcd}
        \varinttwocat{X}\ar[r]\ar[d,"p'"']\pullbackdr
        &\varinttwocat{Y}\ar[d,"p"]\\
        \varinttwocat{Z}\ar[r] &\varinttwocat{W}.
    \end{tikzcd}\]
\end{proposition}
\begin{proof}
    This follows since pullbacks of sheaves are computed as pullbacks of presheaves.
\end{proof}

\begin{lemma}
    Let $p:\varinttwocat{Y}\to \varinttwocat{W}$ be a $1$-cocartesian fibration of $\vartopos$-$2$-categories. If $\varinttwocat{W}$ is a $\vartopos$-category, i.e. all its $2$-arrows are isomorphisms, then so is $\varinttwocat{Y}$. Moreover, $p$ is then a cocartesian fibration of $\vartopos$-categories.
\end{lemma}
\begin{proof}
    $\varinttwocat{Y}$ is a $\vartopos$-category since a right fibration over a groupoid is a groupoid. 

    $p$ is a cocartesian fibration by the first condition in the definition of $1$-cocartesian fibrations of $\vartopos$-$2$-categories.
\end{proof}

\begin{lemma}
    Let $p:\varinttwocat{Y}\to \varinttwocat{W}$ be a $1$-cocartesian fibration of $\vartopos$-$2$-categories. Then the following square is a pullback square:
        \[\begin{tikzcd}
        \varinttwocat{Y}^{\simeq2}\ar[r]\ar[d,"p'"']\pullbackdr
        &\varinttwocat{Y}\ar[d,"p"]\\
        \varinttwocat{W}^{\simeq2}\ar[r] &\varinttwocat{W}.
    \end{tikzcd}\]
    In particular, $p'$ is a cocartesian fibration of $\vartopos$-categories.
\end{lemma}
\begin{proof}
    We must show that a $2$-morphism $\alpha:f\to g:X\to Y$ in the context $A$ in $\varinttwocat{Y}$ is an isomorphism if and only if $p(\alpha)$ is an isomorphism. This follows from $p$ being a homwise right fibration.
\end{proof}

\subsubsection*{The universal cocartesian fibration}

\begin{theorem}\label{prop:univ_internal_2_cocart_fib}
    Let $\internaltwocatofcats_{*//}$ be the oplax slice for the terminal $\vartopos$-category $*\in \internaltwocatofcats(*)$. Then \[(\internaltwocatofcats_{*//})^{\simeq_2}\to \internaltwocatofcats^{\simeq_2}\simeq \internalcatofcats\] is the cocartesian fibration of $\vartopos$-categories classifying the identity functor $\id:\internalcatofcats\to \internalcatofcats$.
\end{theorem}
\begin{proof}
    Let $\intunst(\id)\to \internalcatofcats$ be the unstraightening of the identity.  \cite[Proposition 6.4.9]{MInternalStraightenning} states that in each context $A$ the cocartesian fibration of $\infty$-categories $\intunst(\id)(A)\to \internalcatofcats(A)$ classifies the functor \[\Gamma:\internalcatofcats(A)\simeq \catofcats(\vartopos_{/A})\to \catofcats.\]

    In every context $A$, the $2$-functor 
    \[\Gamma:\twocatofcats(\vartopos_{/A})\to \twocatofcats\] is $2$-representable by the terminal $\vartopos_{/A}$-category $*\in \twocatofcats(\vartopos_{/A})$ in the sense that there is a canonical equivalence of functors $\Fun_{\twocatofcats(\vartopos_{/A})}(*,-)\simeq\Gamma:\twocatofcats(\vartopos_{/A})\to \twocatofcats$. Thus, in every context $A$ the functor \[(\internaltwocatofcats_{*//})(A)\to \internaltwocatofcats(A)\]
    classifies the functor $\Gamma$. 
    
    Therefore, in every context $A$, $\intunst(\id)(A)\to \internalcatofcats(A)$ agrees with $(\internaltwocatofcats_{*//})^{\simeq_2}(A)\to \internaltwocatofcats^{\simeq_2}(A)$. This identification can be  made functorial in $A$ by functoriality of the $2$-Yoneda embedding \cite{BenMosheEnrichedYoneda}, thus proving the claim.
\end{proof}

This motivates the following definition.

\begin{definition}\label{def:internal_two_unstraightening}
    Let 
    \[F:\varinttwocat{X}\to \internaltwocatofcats\]
    be a functor of $\vartopos$-$2$-categories. Then we define the $1$-cocartesian unstraightening of $F$ as the left vertical morphism in the following pullback.
    \[\begin{tikzcd}
    \inttwounst(F)\ar[r]\ar[d]\pullbackdr & \internaltwocatofcats_{*//}\ar[d]\\
    \varinttwocat{X}\ar[r] & \internaltwocatofcats.
\end{tikzcd}\]
\end{definition}

\subsubsection*{Unstraightening lax natural transformations}

We have the following theorem for $(\infty,2)$-categories.

\begin{theorem}\label{thm:lax_unstraightening_external}
    Let $\vartwocat{C}$ be an $(\infty,2)$-category. Then there is an equivalence of $(\infty,2)$-categories 
    \[\funtwocat^{\lax}(\vartwocat{C}, \twocatofcats) \simeq \twocatoftwocats^{\cocart}_{/\vartwocat{C}}\]
    where the right-hand side is the full $(\infty,2)$-subcategory of $\twocatoftwocats_{/\vartwocat{C}}$ spanned by $1$-cocartesian fibrations.
\end{theorem}
\begin{proof}
    This is a restriction of \cite[Theorem A]{AbellanGagnaHaugseng2025} to functors landing in $\twocatofcats\subset \twocatoftwocats$ and to $1$-cocartesian fibrations of $(\infty,2)$-categories.
\end{proof}

We expect there to be a version of this theorem for $\vartopos$-$2$-categories, but we do not need the full functoriality of unstraightening for lax natural transformations. We only need a single lax natural transformation to give rise to a functor between the unstraightenings. We will first discuss this for the universal lax natural transformation.

Let $\funtwocat^{\oplax}(\Delta^1,\twocatofcats)\xto{t}\twocatofcats$ and $\funtwocat^{\oplax}(\Delta^1,\twocatofcats)\xto{s}\twocatofcats$ denote the target and source functors coming from pulling back along the functors $*\to \Delta^1$. Then the evaluation functor (the counit)
\[
\Delta^1 \gray \funtwocat^{\oplax}(\Delta^1,\twocatofcats)\to \twocatofcats
\]
classifies a canonical lax natural transformation \[\alpha_{\univ}:s\to t:\funtwocat^{\oplax}(\Delta^1,\twocatofcats)\to\twocatofcats.\] 

Indeed, for any $(\infty,2)$-category $\vartwocat{C}$, giving a functor 
\[\vartwocat{C}\xto{T} \funtwocat^{\oplax}(\Delta^1,\twocatofcats)\] amounts to giving a functor \[\Delta^1\to \funtwocat^{\lax}(\vartwocat{C},\twocatofcats), \] that is, it is the same as giving  two functors $F,G:\vartwocat{C}\to \twocatofcats$ and a lax natural transformation $\alpha:F\to G$. Then $F,G,\alpha$ can be reconstructed as $s\circ T, t\circ T, \alpha_{\univ}\circ T$, respectively.

By \Cref{thm:lax_unstraightening_external}, we have a canonical commutative triangle in $\widehat{\twocatoftwocats}$
\[\begin{tikzcd}
    \twounst(s)\ar[rr, "\twounst(\alpha_{\univ})"]\ar[dr] && \twounst(t)\ar[dl]\\
    &\funtwocat^{\oplax}(\Delta^1, \twocatofcats).
\end{tikzcd}\]

We will next give better descriptions of $\twounst(s),\twounst(t)$ and $\twounst(\alpha_{\univ})$.

Let \[\funtwocat^{\oplax,0-*}(\Delta^2, \twocatofcats)\subset \funtwocat^{\oplax}(\Delta^2, \twocatofcats)\]
and let 
\[\funtwocat^{\oplax,0-*}(\Lambda^2_2, \twocatofcats)\subset \funtwocat^{\oplax}(\Lambda^2_2, \twocatofcats)\]
denote the full $(\infty,2)$-subcategories spanned by functors that send $0\in \Delta^2$ (respectively $0\in \Lambda^2_2$) to a terminal object in $\twocatofcats$.

That is, an object of $\funtwocat^{\oplax,0-*}(\Delta^2, \twocatofcats)$ is a commutative diagram
\[\begin{tikzcd}
    *\ar[rr]\ar[dr]&&C\ar[dl]\\
    &D.
\end{tikzcd}\]
And morphisms are oplax natural transformations. Less formally, objects in $\funtwocat^{\oplax,0-*}(\Delta^2, \twocatofcats)$ are functors $F:C\to D$ with a choice of an object $x\in C$ in the source. Morphisms between $F_0:C_0\to D_0, x_0\in C_0$ and $F_1:C_1\to D_1, x_1\in C_1$ are given by functors $X:C_0\to C_1$, $Y:D_0\to D_1$, a natural transformation $\alpha: Y\circ F_0\to F_1\circ X$ and a morphism $X(x_0)\to x_1$.

In the other case, an object of $\funtwocat^{\oplax,0-*}(\Lambda^2_2, \twocatofcats)$ is a cospan
\[\begin{tikzcd}
    *\ar[dr]&&C\ar[dl]\\
    &D.
\end{tikzcd}\]
And morphisms are oplax natural transformations. Less formally, objects in $\funtwocat^{\oplax,0-*}(\Lambda^2_2, \twocatofcats)$ are functors $F:C\to D$ with a choice of an object $x\in D$ in the target. Morphisms between $F_0:C_0\to D_0, x_0\in D_0$ and $F_1:C_1\to D_1, x_1\in D_1$ are given by functors $X:C_0\to C_1$, $Y:D_0\to D_1$, a natural transformation $\alpha: Y\circ F_0\to F_1\circ X$ and a morphism $Y(x_0)\to x_1$.

We never actually use the following theorem. We only include it as a motivation for the definition in the internal case:

\begin{theorem}\label{thm:desc_of_the_unst_of_universal_lax_natural_transformation}
    We have a commutative square in $\widehat{\twocatoftwocats}_{/ \funtwocat^{\oplax}(\Delta^1, \twocatofcats)}$:
    \[\begin{tikzcd}
        \twounst(s)\ar[r,"\sim"]\ar[d,"\twounst(\alpha_{\univ})"'] & \funtwocat^{\oplax,0-*}(\Delta^2, \twocatofcats)\ar[d,"\Lambda^2_2\into\Delta^2"]\\
        \twounst(t)\ar[r,"\sim"] & \funtwocat^{\oplax,0-*}(\Lambda^2_2, \twocatofcats),
    \end{tikzcd}\]
    where the horizontal functors are equivalences. 
\end{theorem}
\begin{proof}
    We prove the upper horizontal equivalence.
    
    Observe the following pasting of squares in $\widehat{\catoftwocats}$, which we claim are cartesian.
    \[\begin{tikzcd}
        \funtwocat^{\oplax,0-*}(\Delta^2, \twocatofcats)\ar[r]\ar[d]\pullbackdr
        &\funtwocat^{\oplax}(\Delta^2, \twocatofcats)\ar[r,"\explicitset{1,2}"]\ar[d,"\explicitset{0,1}"]\pullbackdr
        & \funtwocat^{\oplax}(\Delta^1, \twocatofcats)\ar[d,"s"]\\
        \twocatofcats_{*//}\ar[r]\ar[d]\pullbackdr
        &\funtwocat^{\oplax}(\Delta^1, \twocatofcats)\ar[r,"t"]\ar[d,"s"]
        &\twocatofcats\\
        \explicitset{*}\ar[r]
        &\twocatofcats.
    \end{tikzcd}\]
    The upper right square follows from the pushout in $\catofcats\subset \catoftwocats$: 
    \[\begin{tikzcd}
        \Delta^0\ar[r,"1"]\ar[d,"0"'] & \Delta^1\ar[d]\\
        \Delta^1 \ar[r] & \Delta^2.	\arrow["\lrcorner"{anchor=center, pos=0.125, rotate=180}, draw=none, from=2-2, to=1-1]
    \end{tikzcd}\]
    The lower square is a pullback by the definition of $\twocatofcats_{*//}$, and the upper left square is a pullback by the pasting lemma applied to the vertical pasting of squares on the left.

    Recall that the middle horizontal composition is the universal unstraightening. Thus the upper horizontal composition is the unstraightening of $s$ as we claimed before.

    The proof of the lower horizontal equivalence is similar and is left as an exercise to the reader. 

    We are left to show that the square in the theorem commutes.  
    The functoriality of $\twounst(\alpha_{\univ})$ is given by mapping an object to the target of the cocartesian lift of the arrow it sits over in $\funtwocat^{\oplax}(\Delta^1,\twocatofcats)$.
    This sends $x\in C$ to the image of $x$ in $D$. 
\end{proof}

\subsubsection*{The internal construction}

We will now lift this to the setting of internal category theory.

As in the $(\infty,2)$-categorical setting, there is a universal lax natural transformation of functors of $\vartopos$-$2$-categories \[\alpha_{\univ}:s\to t:\inttwofun^{\oplax}(\Delta^1,\internaltwocatofcats)\to \internaltwocatofcats.\]

\begin{definition}
    Let \[\inttwofun^{\oplax,0-*}(\Delta^2, \internaltwocatofcats)\subset \inttwofun^{\oplax}(\Delta^2, \internaltwocatofcats)\]
    and let 
    \[\inttwofun^{\oplax,0-*}(\Lambda^2_2, \internaltwocatofcats)\subset \inttwofun^{\oplax}(\Lambda^2_2, \internaltwocatofcats)\]
    denote the full $\vartopos$-$2$-subcategories spanned by functors that in every context $A$ send $0\in \Delta^2$ (respectively $0\in \Lambda^2_2$) to the terminal object in $\internaltwocatofcats(A)$.
\end{definition}

\begin{proposition}\label{prop:extra_functoriality_of_unst_univ_case}
    $\inttwofun^{\oplax,0-*}(\Delta^2, \internaltwocatofcats)$ and $\inttwofun^{\oplax,0-*}(\Lambda^2_2, \internaltwocatofcats)$ are the unstraightenings of the functors $s$ and $t$, respectively.
\end{proposition}

\begin{proof}
    The proof of \Cref{thm:desc_of_the_unst_of_universal_lax_natural_transformation} applies here in every context and is functorial under change of context since $\inttwofun^{\oplax}$ is applied contextwise. 
\end{proof}

\begin{corollary}\label{cor:extra_functoriality_of_unst}
    Let \[\alpha:F\to G:\varinttwocat{X}\to \internaltwocatofcats\] be a lax natural transformation. Then there exists a commutative triangle 
    \[\begin{tikzcd}
        \inttwounst(F)\ar[rr,"\inttwounst(\alpha)"]\ar[dr] && \inttwounst(G)\ar[dl] \\
        &\varinttwocat{X}
    \end{tikzcd}\]
    such that for any object $x:A\to \varinttwocat{X}$ in the context $A$, the restriction of $\inttwounst(\alpha)$ to the fiber over $x$ agrees with the component $F(x)\to G(x)$  of $\alpha$ at $x$ in $ \internaltwocatofcats(A)$.
\end{corollary}
\begin{proof}
    \Cref{prop:extra_functoriality_of_unst_univ_case} proves this for the universal case and gives a functor $\inttwounst(s)\to \inttwounst(t)$. Let us prove this for the general case. 
    
    The data of $\alpha:F\to G$ are equivalent to the data of a functor
    \[\varinttwocat{X}\to \inttwofun^{\oplax}(\Delta^1,\internaltwocatofcats).\]
    
    We have that \[\inttwounst(F)\simeq \varinttwocat{X}\times_{\inttwofun^{\oplax}(\Delta^1,\internaltwocatoftwocats)}\inttwounst(s)\] and that
    \[\inttwounst(G)\simeq \varinttwocat{X}\times_{\inttwofun^{\oplax}(\Delta^1,\internaltwocatoftwocats)}\inttwounst(t).\] Thus the result follows from the universal case.
\end{proof}

\subsection{Universal property of span \texorpdfstring{$\vartopos$-$2$-}{B-2-}categories.}\label{subsec:internal_span_universal_property}

In this subsection, we construct span $\vartopos$-$2$-categories $\intspanhalf(\varintcat{C},\varintcat{C}_R)$  (\Cref{def:internal_span_two_category}) and establish their universal property in terms of right adjointable functors (\Cref{prop:univ_prop_of_span_half}). We prove the $\subuniverse^{\full}$-completeness and continuity properties of the universal cocartesian fibration and its subcategory of cocartesian arrows (\Cref{thm:univ_unst_is_monoidal}). 

Finally, for an $\subuniverse^{\full}$-complete $\vartopos$-category $\varintcat{C}$, we construct a natural functor $\mu:\varintcat{C}^\times\to\varintcat{C}$ from the operad associated with its cartesian $\subuniverse$-monoidal structure, sending an indexed family to its limit (\Cref{eq:construction_of_mu}).

\subsubsection*{The internal universal property of spans}

Recall \Cref{def:internal_ad_triplet} for the definition of an adequate triplet of $\vartopos$-categories. We denote by \[\adtrip(\vartopos)\simeq \adtrip\otimes \vartopos\] the $\infty$-category of $\vartopos$ adequate triplets and \[\spanpairs(\vartopos)\subset \adtrip(\vartopos)\] its full $\infty$-subcategory spanned by $\vartopos$ span pairs. 

We have a functor
\[\spanhalf:\spanpairs\to \catoftwocats,\] sending a span pair $(C,C')$ to the $(\infty,2)$-category $\spanhalf(C,C')$. The $(\infty,2)$-category $\spanhalf(C,C')$ has underlying $\infty$-category $\Span(C,C')$. Its $2$-morphisms between spans are given by diagrams where morphisms in $C'$ are denoted by tailed arrows:
\[\begin{tikzcd}
    &f\ar[dl]\ar[dr,tail]\\
    x&&y\\
    &g\ar[ul]\ar[ur,tail]\ar[uu]
\end{tikzcd}\]
viewed as a morphism from $f$ to $g$.

We have a universal property of $\spanhalf(C,C')$ in $\catoftwocats$. More precisely, for every $(\infty,2)$-category $\vartwocat{D}$, we have an equivalence of $(\infty,2)$-categories 
\[\funtwocat^{C'\operatorname{-radj}}(C^{\op}, \vartwocat{D})\simeq \funtwocat(\spanhalf(C,C'),\vartwocat{D}).\]

\begin{definition}\label{def:internal_span_two_category}
    Let $(\varintcat{C},\varintcat{C'})\in \spanpairs(\vartopos)$. Then the following composition \[\intspanhalf(\varintcat{C},\varintcat{C'}): \vartopos^{\op}\xto{(\varintcat{C},\varintcat{C'})}\spanpairs\xto{\spanhalf}\catoftwocats\]
    is a $\vartopos$-$2$-category which we call the internal span $\vartopos$-$2$-category attached to $(\varintcat{C},\varintcat{C'})$.

    This construction defines a functor 
    \[\spanpairs(\vartopos)\to \catoftwocats(\vartopos).\]
\end{definition}

We will now prove the needed universal property.

Let $(\varintcat{C},\varintcat{C'})\in \spanpairs(\vartopos)$, and let $F:\varintcat{C}^{\op}\to \varinttwocat{D}$ be a functor of $\vartopos$-$2$-categories. Then we say that $F$ is $\varintcat{C}'$-right adjointable if, in every context $A\in \vartopos$, the functor \[\varintcat{C}^{\op}(A)\to \varinttwocat{D}(A)\]
is $\varintcat{C}'(A)$-right adjointable. 

We denote by $\Fun_{\twocatoftwocats(\vartopos)}^{\varintcat{C}'\operatorname{-radj}}(\varintcat{C}^{\op},\varinttwocat{D})$ the $\infty$-subcategory of $\Fun_{\twocatoftwocats(\vartopos)}(\varintcat{C}^{\op},\varinttwocat{D})$ spanned by $\varintcat{C'}$-adjointable functors and adjointable natural transformations. 

\begin{proposition}\label{prop:univ_prop_of_span_half}
    We have an equivalence of $\infty$-categories
    \[\Fun_{\twocatoftwocats(\vartopos)}^{\varintcat{C}'\operatorname{-radj}}(\varintcat{C}^{\op},\varinttwocat{D})\simeq \Fun_{\twocatoftwocats(\vartopos)}(\intspanhalf(\varintcat{C},\varintcat{C}'),\varinttwocat{D}).\]
\end{proposition}

\begin{proof}
    This follows from applying the universal property pointwise.
\end{proof}

From this follow the variants available by taking opposites or $2$-opposites. We denote them as follows:

\begin{align*}
    \intspanhalf(\varintcat{C},\varintcat{C}')^{\op} = \intspanhalf(\varintcat{C},\varintcat{C}',\varintcat{C})
\end{align*}
This admits the universal property
 \[\Fun_{\twocatoftwocats(\vartopos)}^{\varintcat{C}'\operatorname{-ladj}}(\varintcat{C},\varinttwocat{D})\simeq \Fun_{\twocatoftwocats(\vartopos)}(\intspanhalf(\varintcat{C},\varintcat{C}',\varintcat{C}),\varinttwocat{D}).\]

Applying $(-)^{\co}$ changes the directions of only the adjunctions. For example:
\[\Fun_{\twocatoftwocats(\vartopos)}^{\varintcat{C}'\operatorname{-ladj}}(\varintcat{C}^{\op},\varinttwocat{D})\simeq \Fun_{\twocatoftwocats(\vartopos)}(\intspanhalf(\varintcat{C},\varintcat{C}')^{\co},\varinttwocat{D}).\]

\subsubsection*{Cartesian monoids and \texorpdfstring{$2$}{2}-completeness}

Let $\varinttwocat{D}$ be a $\vartopos$-$2$-category and assume that $\varinttwocat{D}^{\simeq_2}$ is $\varintcat{\subuniverse^{\full}}$-complete. Recall that we have an equivalence of $\vartopos$-categories
\[\intfun^{\subuniverse-\sqcap}((\subuniverse^{\full})^{\op}, \varinttwocat{D}^{\simeq_2})\simeq \varinttwocat{D}^{\simeq_2}.\]

Let $x:*\to \varinttwocat{D}$ be an object. We say that $x$ admits a structure of a cartesian monoid in $\varinttwocat{D}$ if the corresponding functor 
\[(\subuniverse^{\full})^{\op}\to \varinttwocat{D}\]
is $\subuniverse$-right adjointable. In this case we get a canonical functor 
\[\intspanhalf(\subuniverse^{\full},\subuniverse)\to \varinttwocat{D}.\]
The functor given by applying $(-)^{\simeq2}$ classifies a structure of $x$ as an $\subuniverse$-monoid in $\varinttwocat{D}^{\simeq_2}$.

\begin{example}
    A $\vartopos$-category $\varintcat{D}$ is $\subuniverse^{\full}$-complete iff $\varintcat{D}$ admits a structure of a cartesian $\subuniverse$-monoid in $\internaltwocatofcats$.
\end{example}

This motivates the following definition:

\begin{definition}\label{def:E_two_complete}
    We say that a $\vartopos$-$2$-category $\varinttwocat{D}$ is $\subuniverse^{\full}$-$2$-complete if $\varinttwocat{D}$ admits a structure of a cartesian $\subuniverse$-monoid in $\internaltwocatoftwocats$.
\end{definition}

We do not define $2$-limits in $\vartopos$-$2$-categories in general, but we expect that this is equivalent to $\varinttwocat{D}$ admitting $2$-limits of shapes indexed by $\vartopos$-groupoids in $\subuniverse$.

The following proposition gives a more convenient condition for when $\varinttwocat{D}$ is $\subuniverse^{\full}$-$2$-complete.

\begin{proposition}\label{prop:E_two_complete_criterion}
    Let $\varinttwocat{D}$ be a $\vartopos$-$2$-category. Then the following conditions are equivalent:
    \begin{enumerate}
        \item $\varinttwocat{D}$ is $\subuniverse^{\full}$-$2$-complete.
        \item The functor \[\varinttwocat{D}:\vartopos^{\op}\to \twocatoftwocats\] is $\vartopos[E]$-right adjointable, where $\vartopos[E]$ is the local class corresponding to the context-free subuniverse $\subuniverse$.
    \end{enumerate}
\end{proposition}

\begin{proof}
    Suppose $\varinttwocat{D}$ is $\subuniverse^{\full}$-$2$-complete. 

    That is, we assume that the functor $\subuniverse^{\full}(A)^{\op}\to \twocatoftwocats(\vartopos_{/A})$ sending $[s:B\to A]\mapsto s_*s^*\pi_A^*(\varinttwocat{D})$ is $\subuniverse(A)$-right adjointable. 
    
    We first prove that $\varinttwocat{D}:\vartopos^{\op}\to \twocatoftwocats$ sends morphisms in $\vartopos[E]^{\op}$ to left adjoints. Let $s:B\to A \in \vartopos[E]$. We must show that $s^*:\varinttwocat{D}(A)\to \varinttwocat{D}(B)$ is a left adjoint.
    
    Then the morphism $\pi_A^*(\varinttwocat{D})\to s_*s^*\pi_A^*(\varinttwocat{D})$ is a left adjoint in $\twocatoftwocats(\vartopos_{/A})$. 
    
    Recall that $\Gamma:\twocatoftwocats(\vartopos_{/A})\to \twocatoftwocats$ is a $2$-functor. It sends $\pi_A^*(\varinttwocat{D})$ to $\varinttwocat{D}(A)$ and it sends $ s_*s^*\pi_A^*(\varinttwocat{D})$ to $\varinttwocat{D}(B)$.
    
    The resulting morphism $s^*:\varinttwocat{D}(A)\to \varinttwocat{D}(B)$ is a left adjoint as we stated. We next must show the adjointability for pullback squares, but this follows from the fact that adjunctions in $\twocatoftwocats(\vartopos_{/A})$ require adjointability of squares.
    
    The proof of the other implication is analogous.
\end{proof}

\begin{example}
    The $\vartopos$-$2$-category $\internaltwocatofcats\in\widehat{\internaltwocatoftwocats}$ is $\subuniverse^{\full}$-$2$-complete.
\end{example}
\begin{proof}
    The functor $\internaltwocatofcats:\vartopos^{\op}\to \widehat{\twocatoftwocats}$ is given on underlying $\infty$-categories as the composition
    \[
    \internaltwocatofcats:\vartopos^{\op}\xto{A\mapsto \vartopos_{/A}}\prl \xto{-\otimes \catofcats}\operatorname{Mod}_{\catofcats}(\prl)\to \widehat{\twocatoftwocats}.
    \]
    It sends all morphisms in $\vartopos^{\op}$ to left adjoints by definition. The Beck-Chevalley conditions can also be verified directly.
\end{proof}

\begin{example}
    Let $\varinttwocat{D}$ be $\subuniverse^{\full}$-$2$-complete, and let $\vartwocat{I}\in \catoftwocats$. Then $\inttwofun^{\oplax}(\vartwocat{I},\varinttwocat{D})$ is $\subuniverse^{\full}$-$2$-complete.
\end{example}
\begin{proof}
    The functor $\inttwofun^{\oplax}(\vartwocat{I},\varinttwocat{D})$ is given as the composition
    \[\vartopos^{\op}\xto{\varinttwocat{D}}\twocatoftwocats\xto{\funtwocat^{\oplax}(\vartwocat{I},-)}\twocatoftwocats.\]
    The result follows since $\funtwocat^{\oplax}(\vartwocat{I},-)$ is a $2$-functor. 
\end{proof}

\subsubsection*{\texorpdfstring{$2$}{2}-continuity and limits}

\begin{definition}\label{def:E_two_continuous}
    Let $F:\varinttwocat{C}\to \varinttwocat{D}$ be a $2$-functor between $\subuniverse^{\full}$-$2$-complete $\vartopos$-$2$-categories. 
    Recall that \[\intfun^{\subuniverse-\sqcap}((\subuniverse^{\full})^{\op}, \internaltwocatoftwocats^{\simeq_2})\simeq \internaltwocatoftwocats^{\simeq_2}.\]
    
    We say that $F$ is $\subuniverse^{\full}$-$2$-continuous if the corresponding natural transformation  
    \[\begin{tikzcd}
    	(\subuniverse^{\full})^{\op} && \internaltwocatoftwocats
    	\arrow[""{name=0, anchor=center, inner sep=0}, shift right=2, curve={height=18pt}, from=1-1, to=1-3]
    	\arrow[""{name=1, anchor=center, inner sep=0}, shift left=2, curve={height=-12pt}, from=1-1, to=1-3]
    	\arrow[between={0.2}{0.8}, Rightarrow, nfold, from=1, to=0]
    \end{tikzcd}\]
    corresponding to $F$ is $\subuniverse$-right adjointable.
\end{definition}

\begin{proposition}\label{prop:E_two_continuous_criterion}
    Let $F:\varinttwocat{C}\to \varinttwocat{D}$ be a $2$-functor between $\subuniverse^{\full}$-$2$-complete $\vartopos$-$2$-categories. Then the following are equivalent:
    \begin{enumerate}
        \item $F$ is  $\subuniverse^{\full}$-$2$-continuous.
        \item The natural transformation $F$ of the type
        \[\begin{tikzcd}
        	\vartopos^{\op} && \twocatoftwocats
        	\arrow[""{name=0, anchor=center, inner sep=0}, shift right=2, curve={height=18pt}, from=1-1, to=1-3]
        	\arrow[""{name=1, anchor=center, inner sep=0}, shift left=2, curve={height=-12pt}, from=1-1, to=1-3]
        	\arrow[between={0.2}{0.8}, Rightarrow, nfold, from=1, to=0]
        \end{tikzcd}\]
        is $\vartopos[E]$-right adjointable.
    \end{enumerate}
\end{proposition}
\begin{proof}
    The same proof as before works.
\end{proof}

We let $\twocatoftwocats(\vartopos)^{\subuniverse-2\sqcap}\subset \twocatoftwocats(\vartopos)$ denote the locally full $(\infty,2)$-category spanned by $\subuniverse^{\full}$-$2$-complete $\vartopos$-$2$-categories and $\subuniverse^{\full}$-$2$-continuous functors. 

\begin{proposition}
    The inclusion $(\twocatoftwocats(\vartopos)^{\subuniverse-2\sqcap})^{\simeq2}\subset \twocatoftwocats(\vartopos)^{\simeq2}$ preserves all limits.
\end{proposition}
\begin{proof}
    The $\infty$-category $(\twocatoftwocats(\vartopos)^{\subuniverse-2\sqcap})^{\simeq2}$ is the full $\infty$-subcategory of \[\Fun(\spanhalf(\vartopos,\vartopos[E]), \twocatoftwocats)\] spanned by those functors that restrict to sheaves on $\vartopos^{\op}$. The inclusion in the proposition above agrees with restriction along $\vartopos^{\op}\into \spanhalf(\vartopos,\vartopos[E])$ and thus preserves all limits.
\end{proof}

\begin{proposition}
    Let $\vartwocat{I}\to \vartwocat{J}$ be a $2$-functor. Let $\varinttwocat{D}$ be an $\subuniverse^{\full}$-$2$-complete $\vartopos$-$2$-category. Then the functor \[\inttwofun^{\oplax}(\vartwocat{J},\varinttwocat{D})\to \inttwofun^{\oplax}(\vartwocat{I},\varinttwocat{D}) \]
    is $\subuniverse^{\full}$-$2$-continuous.
\end{proposition}

\begin{proof}
    This arrow agrees with the whiskering 
    \[\begin{tikzcd}
        \vartopos^{\op}\ar[r]&\twocatoftwocats && \twocatoftwocats
        \arrow[""{name=0, anchor=center, inner sep=0}, "{\funtwocat^{\oplax}(\vartwocat{I},-)}"', shift right=2, curve={height=18pt}, from=1-2, to=1-4]
        \arrow[""{name=1, anchor=center, inner sep=0}, "{\funtwocat^{\oplax}(\vartwocat{J},-)}", shift left=2, curve={height=-12pt}, from=1-2, to=1-4]
        \arrow[between={0.2}{0.8}, Rightarrow, nfold, from=1, to=0]
    \end{tikzcd}\]
    which is right adjointable since whiskering with a natural transformation of $2$-functors preserves adjointability.
\end{proof}

\begin{example}
    Let $\varinttwocat{D}$ be $\subuniverse^{\full}$-$2$-complete. Assume $\varinttwocat{D}$ has a terminal object $*$. Then $\varinttwocat{D}_{*//}$ is $\subuniverse^{\full}$-$2$-complete.
\end{example}
\begin{proof}
    The following is a pullback in $\twocatoftwocats(\vartopos)^{\subuniverse-2\sqcap}$.
    \[
    \begin{tikzcd}
        \varinttwocat{D}_{*//}\ar[d]\ar[r]\pullbackdr
        &\inttwofun^{\oplax}(\Delta^1,\varinttwocat{D})\ar[d,"s"]\\
        *\ar[r,"*"] & \varinttwocat{D}.
    \end{tikzcd}
    \]
\end{proof}

\subsubsection*{Completeness and continuity of universal unstraightening}

This has applications for the universal cocartesian fibration.

Let \[\internalcatofcats_{*//}:=(\internaltwocatofcats_{*//})^{\simeq2}.\]
Then the functor $\internalcatofcats_{*//}\to \internalcatofcats$ is the universal cocartesian fibration of $\vartopos$-categories. That is, it is the cocartesian fibration of $\vartopos$-categories classifying the identity functor of $\internalcatofcats$. Also note that the composite $\internalcatofcats_{*/}\to\internalcatofcats_{*//}\to\internalcatofcats$ is the left fibration classifying the core functor $(-)^{\simeq}:\internalcatofcats\to \universe$. Thus $\internalcatofcats_{*/}\subset\internalcatofcats_{*//}$ is the wide $\vartopos$-subcategory spanned by cocartesian edges.

\begin{theorem}\label{thm:univ_unst_is_monoidal}
    The three $\vartopos$-categories $\internalcatofcats_{*/},\internalcatofcats_{*//},\internalcatofcats$ are $
    \subuniverse^{\full}$-complete and both functors in the composition \[\internalcatofcats_{*/}\to \internalcatofcats_{*//} \to \internalcatofcats \] are $\subuniverse^{\full}$-continuous.
\end{theorem}
\begin{proof}
    Consider the following diagram:
        \[
    \begin{tikzcd}
        \internalcatofcats_{*/}\ar[d]\ar[r]\pullbackdr
        &\intfun(\Delta^1,\internalcatofcats)\ar[d,hook]\ar[r]
        &\internalcatofcats\ar[d,equal]\\
        \internalcatofcats_{*//}\ar[d]\ar[r]\pullbackdr
        &\inttwofun^{\oplax}(\Delta^1,\internaltwocatofcats)^{\simeq2}\ar[d,"s"]\ar[r]
        &\internalcatofcats\\
        *\ar[r,"*"] & \internalcatofcats.
    \end{tikzcd}
    \]
    It is enough to show that all of it is in $\internalcatofcats^{\subuniverse-\sqcap}$.

    The lower left square is in $\internalcatofcats^{\subuniverse-\sqcap}$ by the example above. 
    
    Therefore, to show that the upper left square is in $\internalcatofcats^{\subuniverse-\sqcap}$ it suffices to show that the middle upper vertical morphism is in  $\internalcatofcats^{\subuniverse-\sqcap}$. This follows since $\intfun(\Delta^1,\internalcatofcats) \simeq \inttwofun^{\oplax}(\Delta^1,\internaltwocatofcats^{\simeq2})^{\simeq2}$ and the inclusion $\internaltwocatofcats^{\simeq2}\into \internaltwocatofcats$ is $\subuniverse^{\full}$-continuous. A similar argument explains why the upper right square is in $\internalcatofcats^{\subuniverse-\sqcap}$.
\end{proof}

\begin{theorem}\label{thm:univ_unst_of_lax_natural_transformation_is_E_monoidal}
    The universal unstraightening of the universal lax natural transformation  \[\inttwounst(\alpha_{\univ}):\inttwounst(s)\to \inttwounst(t)\]
    described in \Cref{prop:extra_functoriality_of_unst_univ_case} is an $\subuniverse^{\full}$-$2$-continuous functor between $\subuniverse^{\full}$-$2$-complete $\vartopos$-$2$-categories. 
\end{theorem}

\begin{proof}
    The pullbacks defining both $\inttwounst(s), \inttwounst(t)$ are in $\twocatoftwocats(\vartopos)^{\subuniverse-2\sqcap}$ and the morphism between them as described in \Cref{thm:desc_of_the_unst_of_universal_lax_natural_transformation} is a base change of a morphism in $\twocatoftwocats(\vartopos)^{\subuniverse-2\sqcap}$.
\end{proof}

\subsubsection*{The natural functor \texorpdfstring{$\mu$}{mu}}\label{subsubsec:const_of_mu}

Let $\varintcat{C}$ be an $\subuniverse^{\full}$-complete $\vartopos$-category. Let $\varintcat{C}^\times$ denote the operad associated with the $\subuniverse$-monoidal $\vartopos$-category defined by $\varintcat{C}$ via the unfurling construction. Then our next goal is to describe a natural functor of $\vartopos$-categories
\[\mu:\varintcat{C}^\times \to \varintcat{C}\]
described as follows: 

Let $A\in \vartopos$ and let $\pi_B:B\to A\in \subuniverse(A)$. Then $B$ defines an object in $\spanEfE(A)$. An object in $\varintcat{C}^\times(A)$ lying over $B$ is an object $x\in \varintcat{C}^B(A) = \varintcat{C}(B)$. In the context $A$, the functor we are after will send $x$ to the object $(\pi_B)_* x \in \varintcat{C}(A)$.

We describe the action of $\mu$ on morphisms. A morphism in $\varintcat{C}^\times(A)$ from $x\in \varintcat{C}^B(A)$ lying over $[\pi_B:B\to A]$  to $y$ lying over $[\pi_C:C\to A]$ is given by a span
\[
\begin{tikzcd}
    &F\ar[dl,"p"']\ar[dr,"q"]\\ B&&C,
\end{tikzcd}
\]
and a morphism 
\[ q_*p^* x\xto{f} y \quad \in \quad\varintcat{C}(C).\]
Then $\mu$ sends this morphism to the composition
\[(\pi_B)_*x \to (\pi_F)_*p^* x \simeq (\pi_C)_*q_*p^*x\xto{f} (\pi_C)_* y. \]
The first arrow is the Beck–Chevalley morphism induced by the square
\[\begin{tikzcd}
    F\ar[r]\ar[d]&B\ar[d]\\A\ar[r,equal]&A.
\end{tikzcd}\]

Let $\intspanhalf(\subuniverse^{\full},\subuniverse)_{//*}$ be the $\vartopos$-$2$-category defined by the following pullback:

\[\begin{tikzcd}
    \intspanhalf(\subuniverse^{\full},\subuniverse)_{//*}\ar[r]\ar[d]\pullbackdr
    &\inttwofun^{\oplax}(\Delta^1, \intspanhalf(\subuniverse^{\full},\subuniverse))\ar[d,"t"]\\
    \explicitset{*}\ar[r,hook] & \intspanhalf(\subuniverse^{\full},\subuniverse).
\end{tikzcd}\]

An object of $\intspanhalf(\subuniverse^{\full},\subuniverse)_{//*}(A)$  is a span in $\subuniverse^{\full}(A)$:
\[B\from B' \to A.\]
A morphism is a diagram of the form:

\[\begin{tikzcd}
	{B} && {F} && C \\
	&&& {F\times_CC'} \\
	{B'} &&&& C' \\
	& {B'} \\
	{A} && {A} && {A}.
	\arrow[from=1-3, to=1-1]
	\arrow[tail,from=1-3, to=1-5]
	\arrow[from=2-4, to=1-3]
	\arrow["\lrcorner"{anchor=center, pos=0.125, rotate=90}, draw=none, from=2-4, to=1-5]
	\arrow[tail, from=2-4, to=3-5]
	\arrow[from=3-1, to=1-1]
	\arrow[from=3-1, to=5-1]
	\arrow[from=3-5, to=1-5]
	\arrow[tail,from=3-5, to=5-5]
    \arrow[from=2-4, to=4-2]
	\arrow[from=4-2, to=3-1]
	\arrow["\lrcorner"{anchor=center, pos=0.125, rotate=-90}, draw=none, from=4-2, to=5-1]
	\arrow[tail,from=4-2, to=5-3]
	\arrow[equal,from=5-3, to=5-1]
	\arrow[equal,from=5-3, to=5-5]
\end{tikzcd}\]

Let $\varintcat{S}\subset(\intspanhalf(\subuniverse^{\full},\subuniverse)_{//*})^{\simeq2}$ denote the full $\vartopos$-subcategory spanned by those objects of the form $B=B\to *$. 

\begin{proposition}\label{prop:E_cart_S_is_equivalent_to_span}
    The composition
    \[\varintcat{S}\subset\intspanhalf(\subuniverse^{\full},\subuniverse)_{//*} \to \inttwofun^{\oplax}(\Delta^1, \intspanhalf(\subuniverse^{\full},\subuniverse))\xto{s}\intspanhalf(\subuniverse^{\full},\subuniverse)\]
    is an equivalence between $\varintcat{S}$ and $\spanEfE\simeq \intspanhalf(\subuniverse^{\full},\subuniverse)^{\simeq_2}$.
\end{proposition}
\begin{proof}
    We will show that the composition in the proposition defines an equivalence on the $\vartopos$ anima of objects and on the $\vartopos$ anima of morphisms, which is enough. 

    The fact that this is an equivalence on objects is obvious. The fact that it is an equivalence on morphisms comes from the fact that morphisms in $\varintcat{S}$ are of the form
    \[\begin{tikzcd}
    {B} && {F} && C \\
    &&& {F} \\
    {B} &&&& C \\
    & {B} \\
    {A} && {A} && {A}.
    \arrow[from=1-3, to=1-1]
    \arrow[tail,from=1-3, to=1-5]
    \arrow[equal,from=2-4, to=1-3]
    \arrow["\lrcorner"{anchor=center, pos=0.125, rotate=90}, draw=none, from=2-4, to=1-5]
    \arrow[tail, from=2-4, to=3-5]
    \arrow[equal,from=3-1, to=1-1]
    \arrow[from=3-1, to=5-1]
    \arrow[equal,from=3-5, to=1-5]
    \arrow[tail,from=3-5, to=5-5]
    \arrow[from=2-4, to=4-2]
    \arrow[equal,from=4-2, to=3-1]
    \arrow["\lrcorner"{anchor=center, pos=0.125, rotate=-90}, draw=none, from=4-2, to=5-1]
    \arrow[tail,from=4-2, to=5-3]
    \arrow[equal,from=5-3, to=5-1]
    \arrow[equal,from=5-3, to=5-5]
    \end{tikzcd}\]
    and are visibly determined uniquely by the first row.
\end{proof}

Thus we get a lax natural transformation $\nu:\iota\Rightarrow *: \spanEfE\to \intspanhalf(\subuniverse^{\full},\subuniverse)$ between the inclusion $\iota:\spanEfE\into \intspanhalf(\subuniverse^{\full},\subuniverse)$ and the constant functor on $*$. This functor is given by the composition
\[\spanEfE\simeq \varintcat{S} \subset \intspanhalf(\subuniverse^{\full},\subuniverse)_{//*} \to \inttwofun^{\oplax}(\Delta^1, \intspanhalf(\subuniverse^{\full},\subuniverse)).\]

Let $\varintcat{C}\in \internalcatofcats^{\subuniverse-\sqcap}$. Then by \Cref{prop:univ_prop_of_span_half} and since the functor $(\subuniverse^{\full})^{\op}\to \internaltwocatofcats$ is $\subuniverse$-right adjointable, we get a functor
\[F:\intspanhalf(\subuniverse^{\full},\subuniverse)\to \internaltwocatofcats.\]
extending the functor classifying the cartesian $\subuniverse$-monoidal structure on $\varintcat{C}$.

We can unstraighten $F\nu$ and get a natural functor 
\begin{equation}\label{eq:construction_of_mu}
\mu:\varintcat{C}^\times\to \varintcat{C}\times\spanEfE\to \varintcat{C}.
\end{equation}

The formula above for $\mu$ is given since $\mu$ is computed by cocartesian lifts.

\subsection{Iterated span \texorpdfstring{$\vartopos$-$2$-}{B-2-}categories.}\label{subsec:iterated_internal_spans}

In this subsection, we review the construction of iterated span $(\infty,2)$-categories whose $2$-morphisms are spans between spans (\Cref{prop:construction_of_iterated_spans}). Applying it contextwise to $(\subuniverse^{\full},\subuniverse)$ gives an iterated span $\vartopos$-$2$-category (\Cref{eq:internal_iterated_span_construction}). Using this $\vartopos$-$2$-category we identify the operad associated with the cartesian $\subuniverse$-monoidal structure on $\spanEfE$ (\Cref{prop:pullback_describing_span_EfE_times_using_span_2}). This provides the model used in the next subsection to describe the natural functor $\mu$ explicitly (\Cref{prop:description_of_mu_for_span}).

\subsubsection*{Iterated spans of a span pair}

Let $C$ be an ordinary (i.e. non-internal) $\infty$-category with finite products. 
Let $(C,C_R)$ be a span pair. That is, $C_R\subset C$ is a wide $\infty$-subcategory such that pullbacks of morphisms in $C_R$ along arbitrary morphisms in $C$ exist and are again in $C_R$.\footnote{Note that we do not assume $C$ has pullbacks.} We begin this subsection by constructing an $(\infty,2)$-category $\spantwo(C,C_R)^{C}_{C_R}$ with maximal $\infty$-subcategory $\Span(C,C_R)$ and with $2$-morphisms given by a span between spans
\[\begin{tikzcd}
    &F\ar[ld]\ar[rd,tail]\\
    A
    &X\ar[u]\ar[d,tail]
    &B\\
    &G\ar[lu]\ar[ru,tail]
\end{tikzcd}\]
with tailed arrows denoting morphisms in $C_R$.

We wish to first convince the reader that this $(\infty,2)$-category should exist, namely that the necessary pullbacks exist in $C$ when we take the needed compositions. The pullbacks needed for composition of $1$-morphisms and vertical composition of $2$-morphisms exist since $(C,C_R)$ is a span pair. 

The whiskering
\[\begin{tikzcd}
    &H\ar[ld]\ar[rd,tail]&&
    F\ar[ld]\ar[rd,tail]\\
    C&&A
    &X\ar[u]\ar[d,tail]
    &B\\
    &&&G\ar[lu]\ar[ru,tail]
\end{tikzcd}\]
exists since the functor $H\times_A - $ is defined on any arrow to $A$ since the morphism $H\to A$ is in $C_R$.

To form the whiskering 
\[\begin{tikzcd}
    &F\ar[ld]\ar[rd,tail]
    && H\ar[ld]\ar[rd,tail]\\
    A
    &X\ar[u]\ar[d,tail]
    &B && C\\
    &G\ar[lu]\ar[ru,tail]
\end{tikzcd}\]
we observe that the arrows from all of $F,G,X$ to $B$ are in $C_R$, so the pullbacks $F\times_BH, G\times_BH$ and $X\times_BH$ required to perform the needed whiskering exist in $C$.

\begin{proposition}\label{prop:construction_of_iterated_spans}
    There exists a pullback-preserving fully faithful inclusion  $C\subset C'$, where $C'$ has all pullbacks.

    Furthermore, restricting $\spantwo(C')$ to objects in $C$, $1$-morphisms in $\Span(C,C_R)$ and $2$-morphisms as described above defines a  $(\infty,2)$-subcategory \[\spantwo(C,C_R)^{C}_{C_R}\subset \spantwo(C').\]
\end{proposition}
\begin{proof}
    For the existence of $C'$, take, for example, the Yoneda embedding.

    For the second part, the construction of $\spantwo(C,E)_{P,I}$ described in \cite[Construction 4.12]{CLLuniv} works in this situation as well.
\end{proof}

We also have the $2$-opposite $(\infty,2)$-category
\[
\spantwo(C,C_R)^{C_R}_{C} := (\spantwo(C,C_R)^{C}_{C_R})^{\operatorname{co}}.
\]

This construction is functorial in span pairs:

\begin{proposition}\label{construction_of_span_2_as_a_functor_on_span_pairs}
    The construction 
    \[(C,C_R)\mapsto \spantwo(C,C_R)^{C}_{C_R}\]
    assembles into a limit-preserving functor
    \[\spanpairs\to \catoftwocats.\]
\end{proposition}

\begin{proof}
    We first explain the functoriality. Take $C'=\presh(C)$. A morphism of span pairs $F:(C,C_R)\to(D,D_R)$ induces by left Kan extension a functor
    \[
    F_!:\presh(C)\to\presh(D)
    \]
    satisfying $F_!\circ \yo_C\simeq \yo_D\circ F$, where $\yo_C$ and $\yo_D$ are the Yoneda embeddings. Although $F_!$ need not preserve all pullbacks, it preserves the pullback squares among representables arising from pullbacks along morphisms in $C_R$, since $F$ preserves these pullbacks and the Yoneda embeddings preserve existing limits. Thus, applying $F_!$ to the diagrams defining the restricted iterated span $(\infty,2)$-categories preserves the prescribed markings and all pullbacks used in composition, giving a $2$-functor
    \[
    \spantwo(C,C_R)^C_{C_R}\to\spantwo(D,D_R)^D_{D_R}.
    \]
    The functoriality of left Kan extension makes this construction functorial in span pairs.

    To check that the functor preserves limits, it is enough to check that it preserves limits on spaces of objects, $1$-morphisms and $2$-morphisms.

    Recall that limits in $\spanpairs$ are computed in the $\infty$-category $\operatorname{Pairs}$ of pairs of an $\infty$-category and a wide $\infty$-subcategory. 
    The functor sending a span pair $(C,C_R)$ to the space of $2$-morphisms of $\spantwo(C,C_R)^{C}_{C_R}$ can be described as the composition 
    \[\spanpairs\into \operatorname{Pairs}\xto{\map((X,X'),-)}\catofanima\]
    where $(X,X')$ is the pair of $(1,1)$-categories in which $X$ is the following diagram and $X'$ is the wide $\infty$-subcategory of tailed arrows: 
    \[\begin{tikzcd}
    	& \bullet & \\
    	\bullet & \bullet & \bullet \\
    	& \bullet.
    	\arrow[from=1-2, to=2-1]
    	\arrow[tail, from=1-2, to=2-3]
    	\arrow[from=2-2, to=1-2]
    	\arrow[tail, from=2-2, to=3-2]
    	\arrow[from=3-2, to=2-1]
    	\arrow[tail, from=3-2, to=2-3]
    \end{tikzcd}\]
    Thus it preserves limits.

    The proofs that the functors of $1$-morphisms and of objects are limit-preserving are similar.
\end{proof}

\subsubsection*{Internal iterated spans and the cartesian operad}

We denote the composition as follows, where the second functor is the functor of \Cref{construction_of_span_2_as_a_functor_on_span_pairs}:
\begin{equation}\label{eq:internal_iterated_span_construction}
\intspantwo(\subuniverse^{\full},\subuniverse)^{\subuniverse^{\full}}_{\subuniverse}:\vartopos^{\op}\xto{(\subuniverse^{\full},\subuniverse)} \spanpairs\to \catoftwocats.
\end{equation}

Then $\intspantwo(\subuniverse^{\full},\subuniverse)^{\subuniverse^{\full}}_{\subuniverse}$ is a well-defined $\vartopos$-$2$-category.

\begin{proposition}\label{prop:pullback_describing_span_EfE_times_using_span_2}
    We have a pullback of $\vartopos$-$2$-categories:
    \[
    \begin{tikzcd}
        \spanEfE^\times\ar[r]\ar[d]\pullbackdr
        & (\intspantwo(\subuniverse^{\full},\subuniverse)^{\subuniverse^{\full}}_{\subuniverse})_{*//}\ar[d]\\
        \spanEfE\ar[r]
        &\intspantwo(\subuniverse^{\full},\subuniverse)^{\subuniverse^{\full}}_{\subuniverse}.
    \end{tikzcd}
    \]
\end{proposition}
\begin{proof}
    Consider the composition of pullbacks:
    \[
    \begin{tikzcd}
        \spanEfE^\times\ar[r]\ar[d]\pullbackdr
        & \bm{\varintcat{P}}\ar[r]\ar[d]\pullbackdr
        &(\intspantwo(\subuniverse^{\full},\subuniverse)^{\subuniverse^{\full}}_{\subuniverse})_{*//}\ar[d]\\
        \spanEfE\ar[r]
        &\intspanhalf(\subuniverse^{\full},\subuniverse)^{\subuniverse^{\full}}\ar[r]&\intspantwo(\subuniverse^{\full},\subuniverse)^{\subuniverse^{\full}}_{\subuniverse}.
    \end{tikzcd}
    \]
    The middle vertical map classifies, in the context $A$, the $2$-functor
    \[
    \intspanhalf(\subuniverse^{\full},\subuniverse)^{\subuniverse^{\full}}(A) = \spanhalf(\subuniverse^{\full}(A),\subuniverse(A))\to \twocatofcats,
    \]
    sending $B\to A$ to \[\Fun_{\intspantwo(\subuniverse^{\full},\subuniverse)^{\subuniverse^{\full}}_{\subuniverse}(A)}(*,B)\simeq\spanEfE^B(A)\simeq \spanEfE(B).\]
    
    This functor is the unique $2$-functor on $\spanhalf(\subuniverse^{\full}(A),\subuniverse(A))$ which arises from the unfurling construction and thus its restriction to $\spanEfE$ agrees with the cocartesian fibration classifying the unfurling construction. 
    
    Finally, by definition, the unfurling construction is classified by the cocartesian fibration $\spanEfE^\times\to \spanEfE$, proving the claim.
\end{proof}

\subsection{Another formula for \texorpdfstring{$\spanEfE^\times$}{the operad of Span(E full,E)}.}\label{subsec:span_operad_description}

Given a fully faithful embedding $\subuniverse^{\full}\hookrightarrow\varintcat{D}$ preserving the terminal object and existing pullbacks, the goal of this subsection is to provide a commutative diagram
\begin{equation}\label{eq:span_operad_into_spans_of_arrows}
\begin{tikzcd}
    \spanEfE^{\times}\ar[r, hook]\ar[d]
    &\intspan(\varintcat{D}^{\Delta^1})\ar[d,"t"]\\
    \spanEfE\ar[r,hook] &\intspan(\varintcat{D}).
\end{tikzcd}
\end{equation}
In this description, the natural functor $\mu$ is the source functor, sending an arrow $X'\to X$ to $X'$ (\Cref{prop:description_of_mu_for_span}).

\subsubsection*{Spans of arrows and oplax slices}

Fix a fully faithful embedding of $\vartopos$-categories $i:\subuniverse^{\full}\subset \varintcat{D}$ such that $\varintcat{D}$ has a terminal object and all pullbacks and $i$ preserves the terminal object and all pullbacks which exist in $\subuniverse^{\full}$.
For example, take $\varintcat{D} = \intpresh(\subuniverse^{\full})$ with the Yoneda embedding.

\begin{proposition}
    The functor 
    \[\intspantwo(\varintcat{D})_{*//} \to \intspantwo(\varintcat{D})\]
    is a $1$-cocartesian fibration of $\vartopos$-$2$-categories.
\end{proposition}
\begin{proof}
    This follows contextwise from \Cref{example:oplax_slice}.
\end{proof}

Since we do not verify the unstraightening equivalence in this work, we cannot say that the cocartesian fibration above classifies a functor of $\vartopos$-$2$-categories to $\internaltwocatofcats$, but if we restrict it to underlying $\vartopos$-categories we can say that:

\begin{proposition}
    We have a pullback
    \[\begin{tikzcd}
        (\intspantwo(\varintcat{D})_{*//})^{\simeq2}\ar[r]\ar[d]\pullbackdr
        &\intspantwo(\varintcat{D})_{*//}\ar[d]\\
        \intspan(\varintcat{D})\ar[r]
        &\intspantwo(\varintcat{D}).
    \end{tikzcd}\]
    Thus the left-hand arrow is the cocartesian fibration classifying the functor 
    \[\intspan(\varintcat{D})\to \internalcatofcats\]
    sending $X\in \intspan(\varintcat{D})(A)$ to \[\intfun_{\intspantwo(\varintcat{D})}(*,X).\]
\end{proposition}

\begin{proof}
    This follows by the same argument appearing in the proof of \Cref{prop:univ_internal_2_cocart_fib}.
\end{proof}

Our next goal is to identify another $1$-cocartesian fibration $\intspantwo(\varintcat{D}^{\Delta^1})^{\varintcat{X}}_{\varintcat{Y}}\to \intspantwo(\varintcat{D})$ and show that the underlying $\vartopos$-categories classify the same functors.

\begin{definition}
    Let $\varintcat{Y},\varintcat{X}\subset \varintcat{D}^{\Delta^1}$ denote the following wide $\vartopos$-subcategories. Let $A\in \vartopos$. Then an arrow of $\varintcat{D}^{\Delta^1}(A)$ is a commutative square 
    \[\begin{tikzcd}
        X'\ar[r]\ar[d,"x"] &Y'\ar[d,"y"]\\X\ar[r] & Y
    \end{tikzcd}\]
    viewed as a morphism from the left arrow $x$ to the right arrow $y$.
    
    We let morphisms in $\varintcat{Y}$ be those morphisms which are pullback squares. We let morphisms in $\varintcat{X}$ be those morphisms in which the source component $X'\to Y'$ is an isomorphism.
\end{definition}

\begin{proposition}\label{prop:arrow_markings_stable_under_base_change}
    The wide $\vartopos$-subcategories $\varintcat{Y},\varintcat{X}\subset \varintcat{D}^{\Delta^1}$ are stable under base change in $\varintcat{D}^{\Delta^1}$. 
\end{proposition}
\begin{proof}
    For $\varintcat{X}$, this follows since isomorphisms are stable under base change in $\varintcat{D}$ and for $\varintcat{Y}$, this follows by the pasting lemma.
\end{proof}

Thus we get a well-defined $\vartopos$-$2$-subcategory  $\intspantwo(\varintcat{D}^{\Delta^1})^{\varintcat{X}}_{\varintcat{Y}}\subset \intspantwo(\varintcat{D}^{\Delta^1})$ defined as follows.

We allow all objects and all morphisms. A $2$-arrow in $\intspantwo(\varintcat{D}^{\Delta^1})(A)$ is an iterated span diagram in $\varintcat{D}^{\Delta^1}(A)$:
\[\begin{tikzcd}
    &f\ar[ld]\ar[rd]\\
    x
    &k\ar[u]\ar[d]
    &y\\
    &g.\ar[lu]\ar[ru]
\end{tikzcd}\]
It is a $2$-morphism in $\intspantwo(\varintcat{D}^{\Delta^1})^{\varintcat{X}}_{\varintcat{Y}}$ if the upward-pointing arrow is in $\varintcat{X}$ and the downward-pointing arrow is in $\varintcat{Y}$. \Cref{prop:arrow_markings_stable_under_base_change} implies that these allowed $2$-morphisms are closed under composition and whiskering and thus they form a $\vartopos$-$2$-subcategory. 

\begin{proposition}
    The functor 
    \[t: \intspantwo(\varintcat{D}^{\Delta^1})^{\varintcat{X}}_{\varintcat{Y}}\to \intspantwo(\varintcat{D})\]
    sending an arrow $x:X'\to X$ to its target $X$ is a $1$-cocartesian fibration of $\vartopos$-$2$-categories.
\end{proposition}
\begin{proof}
    Let $A\in \vartopos$. We first want to show that 
    \[t: \intspantwo(\varintcat{D}^{\Delta^1})^{\varintcat{X}}_{\varintcat{Y}}(A)\to \intspantwo(\varintcat{D})(A)\]
    is a $1$-cocartesian fibration.

    Indeed, for an object $f:X'\to X$ and an arrow in $\intspantwo(\varintcat{D})(A)$ 
    \[X\from F\to Y,\]
    the cocartesian lifts on underlying $\infty$-categories are of the form
    \[\begin{tikzcd}
    	X' & F' & F' \\
    	X & F & Y.
    	\arrow[from=1-1, to=2-1]
    	\arrow[from=1-2, to=1-1]
    	\arrow[from=1-2, to=1-3,equal]
    	\arrow["\lrcorner"{anchor=center, pos=0.125, rotate=-90}, draw=none, from=1-2, to=2-1]
    	\arrow[from=1-2, to=2-2]
    	\arrow[from=1-3, to=2-3]
    	\arrow[from=2-2, to=2-1]
    	\arrow[from=2-2, to=2-3]
    \end{tikzcd}\]

    By \cite[Lemma 3.3]{CLRUniversalityOfUnferling}, if we show that $t$ above is a homwise right fibration, then we get that $t$ is a $1$-cocartesian fibration. This follows since cartesian $2$-morphisms (which are given by opposites of the span of arrows directly above) are the only allowed $2$-morphisms in $\intspantwo(\varintcat{D}^{\Delta^1})^{\varintcat{X}}_{\varintcat{Y}}(A)$.

    We are left to show that change of context preserves cocartesian lifts. This follows since it preserves pullbacks and isomorphisms.
\end{proof}

These two cocartesian fibrations are equivalent:

\begin{proposition}\label{prop:description_of_oplax_slice}
    We have an equivalence of cocartesian fibrations of $\vartopos$-$2$-categories 
    \[\intspantwo(\varintcat{D})_{*//} \simeq \intspantwo(\varintcat{D}^{\Delta^1})^{\varintcat{X}}_{\varintcat{Y}}.\]
\end{proposition}
\begin{proof}
    We first show this equivalence in a context $A$.

    Let 
    \[T:\intspantwo(\varintcat{D})(A)\to\twocatofcats\] be the $2$-functor classified by $\intspantwo(\varintcat{D}^{\Delta^1})^{\varintcat{X}}_{\varintcat{Y}}$. The $2$-functor classified by $\intspantwo(\varintcat{D})_{*//}$ is $\yo(*)$, the $2$-Yoneda embedding of $*$. Then, by the $2$-Yoneda lemma, we have a natural transformation $\yo(*) \to T$  determined by the object $*=*$ in $T(*)$.

    In the context $A$, we have that 
    \[T(B) = \Span(\varintcat{D}(A)_{/B}).\]
    The component at $B$ of this natural transformation is the usual equivalence: 
    \[\Fun_{\intspantwo(\varintcat{D})(A)}(*,B)\simeq \Span(\varintcat{D}(A)_{/B}). \]

    By the $2$-categorical unstraightening theorem, we get an equivalence $\intspantwo(\varintcat{D}^{\Delta^1})^{\varintcat{X}}_{\varintcat{Y}}(A)\simeq \intspantwo(\varintcat{D})_{*//}(A)$.

    By naturality of Yoneda, this equivalence is natural as well.
\end{proof}

\begin{corollary}\label{cor:square_with_span_of_arrows_and_span_2_exists_for_presh}
    We have a commutative square 
    \begin{equation}\label{eq:diag_of_BC_comp}\begin{tikzcd}
        \intspan(\varintcat{D}^{\Delta^1})\ar[r]\ar[d]
        & \intspantwo(\varintcat{D})_{*//}\ar[d]\\
        \intspan(\varintcat{D})\ar[r]
        &\intspantwo(\varintcat{D}).
    \end{tikzcd}\end{equation}
\end{corollary}
\begin{proof}
    By \Cref{prop:description_of_oplax_slice}, the left-hand functor is the core of the right-hand functor.
\end{proof}

\subsubsection*{The cartesian operad as spans of arrows}

 We have the chosen embedding $\subuniverse^{\full}\into \varintcat{D}$. 
This inclusion preserves the terminal object and all pullbacks that exist in $\subuniverse^{\full}$. 
Thus we may define the following subsheaves of $\intspantwo(\varintcat{D})$ viewed as a $\vartopos$-sheaf of $(\infty,2)$-categories 
\[\intspantwo(\subuniverse^{\full},\subuniverse)^{\subuniverse^{\full}}_{\subuniverse}\subset \intspantwo(\varintcat{D})\]
\[\intspanhalf(\subuniverse^{\full},\subuniverse)^{\subuniverse^{\full}}\subset \intspanhalf(\varintcat{D})^{\varintcat{D}}.\]

We consider the following $\vartopos$-subcategory of $\intspan(\varintcat{D}^{\Delta^1})$:

\begin{definition}\label{def:span_of_arrows_E}
    Let $\varintcat{Sp}_{\subuniverse}^{(1)}\subset \intspan(\varintcat{D}^{\Delta^1})$ denote the following $\vartopos$-subcategory. In the context $A$, the allowed objects are morphisms \[Y'\rightarrowtail Y\]
    where both $Y$ and $Y'$ are in $\subuniverse^{\full}$ and the morphism is in $\subuniverse$. The allowed morphisms are spans
    \[\begin{tikzcd}
        Y'\ar[tail,d]
        &X'\ar[d]\ar[tail,r]\ar[l]
        &Z'\ar[tail,d]\\
        Y
        &X\ar[l]\ar[tail,r]
        &Z
    \end{tikzcd}\]
    where $X,X'$ are in $\subuniverse^{\full}$ and the tailed morphisms are required to be in $\subuniverse$. This defines a $\vartopos$-subcategory since  $(\subuniverse^{\full},\subuniverse)$ is a span pair, so the required pullbacks exist in $\subuniverse^{\full}$ and are preserved by the Yoneda embedding. 
\end{definition}

\begin{remark}
    The $\vartopos$-category $\varintcat{Sp}_{\subuniverse}^{(1)}$ need not be a span $\vartopos$-category of any sub-adequate triplet of $(\varintcat{D}^{\Delta^1},\varintcat{D}^{\Delta^1},\varintcat{D}^{\Delta^1})$. Indeed, the arrows in the roof of the spans are more general objects of $\intspan(\varintcat{D}^{\Delta^1})$ than the objects of  $\varintcat{Sp}_{\subuniverse}^{(1)}$ so that there is no appropriate factorization system of left-pointing and right-pointing maps.
\end{remark}

\begin{corollary}\label{cor:square_with_span_of_arrows_and_span_2_exists}
    The commutative diagram of \Cref{cor:square_with_span_of_arrows_and_span_2_exists_for_presh} restricts to a pullback square: 
    \begin{equation} \label{eq:diag_of_BC_comp_restricted}\begin{tikzcd}
        \varintcat{Sp}_{\subuniverse}^{(1)}\ar[r]\ar[d]
        & (\intspantwo(\subuniverse^{\full},\subuniverse)^{\subuniverse^{\full}}_{\subuniverse})_{*//}\ar[d]\\
        \spanEfE\ar[r]
        &\intspantwo(\subuniverse^{\full},\subuniverse)^{\subuniverse^{\full}}_{\subuniverse}.
    \end{tikzcd}\end{equation}
\end{corollary}
\begin{proof}
    To check that the square exists, we need to check that each object, each morphism and each $2$-morphism in each of the $\vartopos$-$2$-categories is sent by each arrow to the $\vartopos$-$2$-subcategory in the target. We will only show that the upper morphism is well-defined\footnote{In fact, all other checks are subsumed in this one.}.

    Let
    \begin{equation}\label{eq:diag_of_morph_in_sp_1}\begin{tikzcd}
        Y'\ar[tail,d]
        &X'\ar[d]\ar[tail,r]\ar[l]
        &Z'\ar[tail,d]\\
        Y
        &X\ar[l]\ar[tail,r]
        &Z
    \end{tikzcd}\end{equation}
    be a morphism in $\varintcat{Sp}_{\subuniverse}^{(1)}$. The upper arrow in \Cref{eq:diag_of_BC_comp} sends this morphism to the arrow
    \begin{equation}\label{eq:diag_of_morph_inlax_slice_of_sp_2}\begin{tikzcd}
    	{*} && {Y'} && Y \\
    	&&& {Y'\times_YX} \\
    	{*} && X' && X \\
    	& {Z'} \\
    	{*} && {Z'} && Z.
    	\arrow[from=1-3, to=1-1]
    	\arrow[tail, from=1-3, to=1-5]
    	\arrow[from=2-4, to=1-3]
    	\arrow["\lrcorner"{anchor=center, pos=0.125, rotate=90}, draw=none, from=2-4, to=1-5]
    	\arrow[from=2-4, to=3-5]
    	\arrow[equal, from=3-1, to=1-1]
    	\arrow[equal, from=3-1, to=5-1]
    	\arrow[from=3-3, to=2-4]
    	\arrow[tail, from=3-3, to=4-2]
    	\arrow[from=3-5, to=1-5]
    	\arrow[tail, from=3-5, to=5-5]
    	\arrow[from=4-2, to=3-1]
    	\arrow["\lrcorner"{anchor=center, pos=0.125, rotate=-90}, draw=none, from=4-2, to=5-1]
    	\arrow[equal, from=4-2, to=5-3]
    	\arrow[from=5-3, to=5-1]
    	\arrow[tail, from=5-3, to=5-5]
    \end{tikzcd}\end{equation}

    We now check  that the  square \Cref{eq:diag_of_BC_comp_restricted} is a pullback square. Since $\varintcat{Sp}_{\subuniverse}^{(1)}$ is a sheaf of $\infty$-categories, it is enough to check this contextwise on the anima of objects and morphisms. The statement about objects is trivial. 

    Note that in \Cref{eq:diag_of_BC_comp_restricted} the lower horizontal functor is the inclusion of the maximal $\vartopos$-subcategory. Thus we only need to show in every context $A$ that the upper morphism gives an equivalence on the anima of $1$-morphisms.

    The anima of $1$-morphisms in $\varintcat{Sp}_{\subuniverse}^{(1)}(A)$ is equivalent to the anima of diagrams as in \Cref{eq:diag_of_morph_in_sp_1}. The anima of $1$-morphisms in $(\intspantwo(\subuniverse^{\full},\subuniverse)^{\subuniverse^{\full}}_{\subuniverse})_{*//}(A)$ is equivalent to the anima of diagrams as in \Cref{eq:diag_of_morph_inlax_slice_of_sp_2}, that is, the anima of diagrams of the shape in \Cref{eq:diag_of_morph_inlax_slice_of_sp_2} in $\subuniverse^{\full}(A)$ such that the morphisms denoted by $=$ are isomorphisms, the objects denoted by $*$ are terminal and the marked squares in the corners are pullbacks.
    The morphism between these animae described above is clearly an equivalence.
\end{proof}

This gives us the desired description of $\spanEfE^\times$.

\begin{proposition}\label{prop:span_operad_as_spans_of_arrows}
    There is an equivalence of cocartesian fibrations
    \[\begin{tikzcd}
        \varintcat{Sp}_{\subuniverse}^{(1)}\ar[r,"\sim"]\ar[d]
        &\spanEfE^{\times}\ar[d]\\
        \spanEfE\ar[r, equal] & \spanEfE.
    \end{tikzcd}\]
\end{proposition}
\begin{proof}
    By \Cref{eq:diag_of_BC_comp_restricted}, it is enough to show that $\spanEfE^\times$ sits in the following pullback:
    \[\begin{tikzcd}
        \spanEfE^\times \ar[r]\ar[d]
        & (\intspantwo(\subuniverse^{\full},\subuniverse)^{\subuniverse^{\full}}_{\subuniverse})_{*//}\ar[d]\\
        \spanEfE\ar[r]
        &\intspantwo(\subuniverse^{\full},\subuniverse)^{\subuniverse^{\full}}_{\subuniverse}.
    \end{tikzcd}\]

    This was shown in \Cref{prop:pullback_describing_span_EfE_times_using_span_2}.
\end{proof}

\subsubsection*{The source functor and \texorpdfstring{$\mu$}{mu}}

Recall the functor $\mu:\spanEfE^\times\to \spanEfE$ defined in \Cref{eq:construction_of_mu}. 

\begin{proposition}\label{prop:description_of_mu_for_span}
    The composition 
    \[\varintcat{Sp}_{\subuniverse}^{(1)}\simeq \spanEfE^\times\xto{\mu} \spanEfE\]
    is equivalent to the source functor 
    \[s: \varintcat{Sp}_{\subuniverse}^{(1)}\to\spanEfE\]
    sending $(X'\rightarrowtail X)\mapsto X'$.
\end{proposition}
\begin{proof}
    Both $s$ and $\mu$ were constructed by sending an arrow to its source. The equivalence above is immediate from those constructions.
\end{proof}

\begin{corollary}\label{cor:s_and_mu_compatibility_diagram}
    There is a commutative diagram 
    \[\begin{tikzcd}
        &\subuniverse^\otimes\ar[dd,hook]\ar[dr,"s"]\\
        \spanEfE\ar[ru,"r"]\ar[dr,"\tau"']
        &&\spanEfE\\
        &\spanEfE^\times\ar[ru,"\mu"']
    \end{tikzcd}\]
\end{corollary}
\begin{proof}
    Since $\subuniverse\subset \spanEfE$ is an $\subuniverse$-monoidal $\vartopos$-subcategory, we have a functor $\subuniverse^\otimes \into \spanEfE^\times$ which is an inclusion of a $\vartopos$-subcategory.

    In \Cref{cor:description_of_the_env_of_triv} we showed that morphisms in $\subuniverse^\otimes$ are given by diagrams of the form
    \[\begin{tikzcd}
        Y'\ar[tail,d]
        &X'\ar[tail,d]\ar[tail,r]\ar[l]\arrow["\lrcorner"{anchor=center, pos=0.125, rotate=-90}, draw=none, dl]
        &Z'\ar[tail,d]\\
        Y
        &X\ar[l]\ar[tail, r]
        &Z.
    \end{tikzcd}\]
    Morphisms in $\spanEfE^\times$ are given by diagrams as in \Cref{eq:diag_of_morph_in_sp_1}, and the inclusion sends a diagram above to itself. Thus we get the commutativity of the right triangle from \Cref{prop:description_of_mu_for_span}.
    
     The commutativity of the left triangle was the definition of $\tau$.
\end{proof}

\subsection{Another formula for the \texorpdfstring{$\subuniverse$}{E}-monoidal envelope.}\label{subsec:envelope_as_restriction}

In this subsection, we give another description of the $\subuniverse$-monoidal envelope of the operad associated with an $\subuniverse$-monoidal $\vartopos$-category $\varintcat{C}$. We show that $\varintcat{C}^\otimes$ carries a cartesian $\subuniverse$-monoidal structure and that the equivalence
\[
\intenv(\varintcat{C}^\otimes)\simeq \varintcat{C}^\otimes\times_{\spanEfE}\subuniverse
\]
is an equivalence of $\subuniverse$-monoidal $\vartopos$-categories (\Cref{thm:C_otimes_res_to_E_is_envelope}). Thus the envelope identifies with an $\subuniverse$-monoidal $\vartopos$-subcategory of $\varintcat{C}^\otimes$.

We also show that unstraightening a lax $\subuniverse$-monoidal functor $F:\varintcat{C}\to\internalcatofcats$ gives a canonical $\subuniverse$-monoidal structure on $\intunst(F)$, with $\subuniverse$-monoidal projection to $\varintcat{C}$ (\Cref{prop:lax_E_monoidal_unst}). In \Cref{ex:lax_unstraightening_recovers_monoidal_structure}, we recover the original $\subuniverse$-monoidal structure on $\varintcat{C}$ by applying this construction to the functor $*\to\internalcatofcats$ selecting $\varintcat{C}$.

\subsubsection*{Monoidal structures on unstraightenings}

\begin{corollary}\label{cor:associated_operad_total_category_complete}
    Let $F:\spanEfE \to \internalcatofcats$ be an $\subuniverse$-monoidal $\vartopos$-category with $\varintcat{C}= F(*)$. Then $\varintcat{C}^\otimes := \intunst(F)$ is $\subuniverse^{\full}$-complete and the classifying cocartesian fibration $\varintcat{C}^\otimes\to \spanEfE$ is $\subuniverse^{\full}$-continuous. Moreover, the $\vartopos$-subcategory of $\varintcat{C}^\otimes$ spanned by cocartesian edges is closed under $\subuniverse^{\full}$-limits.
\end{corollary}
\begin{proof}
    The pullback
    \[\begin{tikzcd}
        \varintcat{C}^\otimes\ar[r]\ar[d]\pullbackdr & \internalcatofcats_{*//}\ar[d]\\
        \spanEfE\ar[r] & \internalcatofcats
    \end{tikzcd}\]
    is a pullback of $\subuniverse^{\full}$-complete $\vartopos$-categories along $\subuniverse^{\full}$-continuous functors. Thus it can be computed in $\internalcatofcats$ since the functor $\internalcatofcats^{\subuniverse-\sqcap}\to \internalcatofcats$ is continuous. 
\end{proof}

Another corollary is that given an $\subuniverse$-monoidal functor to $\internalcatofcats$, the unstraightening carries a natural $\subuniverse$-monoidal structure:

\begin{corollary}\label{cor:unst_of_monoidal_functor_is_monoidal}
    Let $\varintcat{J}$ be an $\subuniverse$-monoidal $\vartopos$-category and let $F:\varintcat{J}\to \internalcatofcats$ be an $\subuniverse$-monoidal functor where the target is equipped with the cartesian $\subuniverse$-monoidal structure. Then $\intunst(F)$ carries a natural $\subuniverse$-monoidal structure such that the unstraightening functor $\intunst(F)\to \varintcat{J}$ is $\subuniverse$-monoidal. Moreover, the $\vartopos$-subcategory of $\intunst(F)$ spanned by cocartesian edges is an $\subuniverse$-monoidal $\vartopos$-subcategory.
\end{corollary}
\begin{proof}
    The forgetful functor $\intmon^{\subuniverse}(\internalcatofcats)\to \internalcatofcats$ preserves pullbacks, and since $\internalcatofcats_{*//}\to \internalcatofcats$ preserves $\subuniverse$-limits, its base change to $\varintcat{J}$ along $F$ is $\subuniverse$-monoidal.
\end{proof}

\subsubsection*{The canonical lax monoidal functor}

Let $\varintcat{C}$ be an $\subuniverse$-monoidal $\vartopos$-category. We have an equivalence of $\vartopos$-categories:
\[
\intenv(\varintcat{C}^\otimes)\simeq \varintcat{C}^\otimes \times_{\spanEfE} \subuniverse.
\]
Our next goal is to verify that this is a pullback of $\subuniverse$-monoidal $\vartopos$-categories.

Let $\varintcat{C}$ be a $\subuniverse^{\full}$-complete $\vartopos$-category. Let $\varintcat{C}^\times$ denote the operad associated with the $\subuniverse$-monoidal $\vartopos$-category defined by $\varintcat{C}$ via the unfurling construction. Recall from \Cref{eq:construction_of_mu} the natural functor of $\vartopos$-categories
\[\mu:\varintcat{C}^\times \to \varintcat{C}.\]

\begin{lemma}\label{lem:mu_is_equifibed_for_univ_unst}
    Let $\internalcatofcats_{*//}\to \internalcatofcats$ be the universal cocartesian fibration. Then the commutative naturality square
    \[
    \begin{tikzcd}
        (\internalcatofcats_{*//})^\times\ar[r,"\mu"]\ar[d]\pullbackdr
        &\internalcatofcats_{*//}\ar[d]\\
        \internalcatofcats^\times\ar[r,"\mu"']
        &\internalcatofcats
    \end{tikzcd}
    \]
    is cartesian.
\end{lemma}
\begin{proof}
    To show that a square of $\vartopos$-categories is cartesian, it is enough to check that for each $A\in \vartopos$, the functors $\intmap(A, -)$ and $\intmap(A\times \Delta^1, - )$ send it to cartesian squares in $\vartopos$. 

    Let $A\in \vartopos$. A functor $A\to \internalcatofcats_{*//}$ classifies a $\vartopos_{/A}$-category $\varintcat{D}$ and an object $x:A\to \varintcat{D}$. A functor $A\to \internalcatofcats^\times$ classifies an object $[\pi_B:B\to A]\in \spanEfE(A)$ and a $\vartopos_{/B}$-category $\varintcat{C}$. Thus, an object in the pullback classifies a $\vartopos_{/A}$-category $\varintcat{D}$, a $\vartopos_{/B}$-category $\varintcat{C}$, an isomorphism $(\pi_B)_*\varintcat{C}\simeq \varintcat{D}$ and an object $x:A\to \varintcat{D}$. The data of an object $x$ correspond uniquely to the data of an object $y:B\to \varintcat{C}$. Thus we get that mapping from $A$ to the pullback amounts to giving $[\pi_B:B\to A]$, $\varintcat{C}$, and $y:B\to \varintcat{C}$. This is exactly the data of an object in $\internalcatofcats_{*//}^\times(A)$.

    Next, we must show that $\intmap(A\times \Delta^1, - )$ sends our square to a cartesian square in $\vartopos$. This is achieved by a similar argument.
\end{proof}

Our next goal is to define a canonical lax $\subuniverse$-monoidal functor $\varintcat{C}\to \varintcat{C}^\otimes$ for every $\subuniverse$-monoidal $\vartopos$-category $\varintcat{C}$, where the target is given the $\subuniverse$-monoidal structure coming from the fact that $\varintcat{C}^\otimes$ is $\subuniverse^{\full}$-complete. 
That is, we want to define a canonical functor of $\subuniverse$-operads\footnote{Notice that $\varintcat{C}^{\otimes}$ has both the structure of an $\subuniverse$-operad and the structure of a cartesian $\subuniverse$-monoidal $\vartopos$-category.} \[\varintcat{C}^\otimes\to (\varintcat{C}^\otimes)^\times.\]
We first construct this for the specific example of the terminal $\subuniverse$-monoidal $\vartopos$-category $\varintcat{C} = *$.

\begin{proposition}
    The inclusion
    \[\explicitset{*}\to \spanEfE\]
    can be given the structure of a lax $\subuniverse$-monoidal functor
    \[\tau:\spanEfE\to \spanEfE^\times.\]
\end{proposition}
\begin{proof}
    We have a natural inclusion of an $\subuniverse$-monoidal $\vartopos$-subcategory
    \[\subuniverse\into \spanEfE.\]
    Thus, it is enough to give $*\into \subuniverse$ the structure of a lax $\subuniverse$-monoidal functor. Recall the morphism of $\subuniverse$-operads
    \[
    r:\spanEfE\into \subuniverse^\otimes.
    \]
    This morphism classifies a lax monoidal structure for the functor $\explicitset{*}\into \subuniverse$.

    Hence we get that the composition of morphisms of $\subuniverse$-operads
    \[
    \spanEfE\xto{r} \subuniverse^\otimes\to \spanEfE^\times
    \]
    gives the desired result. 
\end{proof}

Recall the following diagram from \Cref{cor:s_and_mu_compatibility_diagram}:

\begin{proposition}[\Cref{cor:s_and_mu_compatibility_diagram}]\label{prop:s_and_mu_compatibility_diagram}
    There is a commutative diagram 
    \[\begin{tikzcd}
        &\subuniverse^\otimes\ar[dd,hook, "i"]\ar[dr,"s"]\\
        \spanEfE\ar[ru,"r"]\ar[dr,"\tau"']
        &&\spanEfE\\
        &\spanEfE^\times\ar[ru,"\mu"']
    \end{tikzcd}\]
\end{proposition}

\begin{corollary}
    The composition 
    \[\spanEfE\xto{\tau} \spanEfE^\times\xto{\mu}\spanEfE \]
    is homotopic to the identity.
\end{corollary}
\begin{proof}
    Consider the following isomorphism:
    \[\id \simeq s\circ r \simeq \mu\circ i \circ r \simeq \mu\circ \tau.\]
\end{proof}


We now define a canonical lax $\subuniverse$-monoidal functor $\varintcat{C}\to \varintcat{C}^\otimes$, that is, a functor $\tau_{\varintcat{C}}:\varintcat{C}^\otimes\to (\varintcat{C}^\otimes)^\times$ of $\subuniverse$-operads for every $\subuniverse$-monoidal $\vartopos$-category $\varintcat{C}$.

\begin{proposition}
    There exists a canonical pullback diagram 
    \[\begin{tikzcd}
        \varintcat{C}^\otimes \ar[r,"\tau_{\varintcat{C}}"]\ar[d]\pullbackdr
        &(\varintcat{C}^\otimes)^\times\ar[d]\\
        \spanEfE\ar[r,"\tau"]
        &\spanEfE^\times,
    \end{tikzcd}\]
    natural in $\varintcat{C}$, and the upper functor $\tau_{\varintcat{C}}$ is a morphism of $\subuniverse$-operads. 
\end{proposition}
\begin{proof}
    Let \[F:\spanEfE\to\internalcatofcats\] denote the functor defining the $\subuniverse$-monoidal structure of $\varintcat{C}$.

    Consider the following pasting of pullbacks of $\vartopos$-categories:
    \[\begin{tikzcd}
        \varintcat{P}\ar[r]\ar[d]\pullbackdr
        &(\varintcat{C}^\otimes)^\times \ar[r]\ar[d]\pullbackdr
        &\internalcatofcats_{*//}^\times \ar[r,"\mu"]\ar[d]\pullbackdr
        &\internalcatofcats_{*//}\ar[d]\\
        \spanEfE\ar[r,"\tau"]
        &\spanEfE^\times\ar[r,"F^\times"]
        &\internalcatofcats^\times\ar[r,"\mu"]
        &\internalcatofcats.
    \end{tikzcd}\]
    The right square is the pullback square described in \Cref{lem:mu_is_equifibed_for_univ_unst}.
    The middle square is obtained by applying $(-)^\times$ to the pullback square defining $\varintcat{C}^\otimes$. The left square is defined as the pullback as written. 
    
    We claim that $\varintcat{P}\simeq \varintcat{C}^\otimes$ and so the promised pullback is the left square. 
    
    To check this, it is enough to verify that the lower composition is homotopic to $F$. Indeed, this follows since 
    \[\mu\circ F^\times \circ \tau \simeq F\circ \mu \circ \tau \simeq F\circ \id \simeq F.\]

    The morphism $\tau_{\varintcat{C}}$ is a morphism of operads since $\tau$ is.

    Lastly, we need to check that $\tau_{\varintcat{C}}$ is natural, but this follows since all the ingredients in its construction are functorial.
\end{proof}

\subsubsection*{The envelope as a monoidal subcategory}

As a corollary, by the envelope--associated operad adjunction, we get a canonical $\subuniverse$-monoidal functor
\[\iota_{\varintcat{C}}:\intenv(\varintcat{C}^\otimes)\to \varintcat{C}^\otimes.\]

We can now finish our discussion by identifying $\intenv(\varintcat{C}^\otimes)$ with an $\subuniverse$-monoidal $\vartopos$-subcategory of $\varintcat{C}^\otimes$.

\begin{theorem}\label{thm:C_otimes_res_to_E_is_envelope}
    The functor $\iota_{\varintcat{C}}$ is an inclusion of an $\subuniverse$-monoidal $\vartopos$-subcategory, participating in a pullback of $\subuniverse$-monoidal $\vartopos$-categories:
    \[\begin{tikzcd}
        \intenv(\varintcat{C}^\otimes)\ar[r,"\iota_{\varintcat{C}}"]\ar[d]\pullbackdr
        &\varintcat{C}^\otimes\ar[d]\\
        \subuniverse\ar[r]
        &\spanEfE.
    \end{tikzcd}\]
\end{theorem}
\begin{proof}
    Any $\subuniverse$-monoidal functor which is an inclusion of a $\vartopos$-subcategory is canonically an $\subuniverse$-monoidal equivalence with the $\subuniverse$-monoidal $\vartopos$-subcategory spanned by its image. Thus the existence of a pullback square as in the statement of the theorem will imply the first statement.

    The square in the statement of the theorem is the naturality square of $\iota$ applied to the canonical $\subuniverse$-monoidal functor $\varintcat{C}\to *$. Since the functor 
    \[(-)^\otimes:\intsmon(\internalcatofcats)\to \intoperads\] is a conservative right adjoint, it suffices to show that the square is cartesian after applying $(-)^\otimes$.
    
    Consider the following diagram:
    \[\begin{tikzcd}
        \intenv(\varintcat{C}^\otimes)^\otimes && (\varintcat{C}^\otimes)^\times && \varintcat{C}^\otimes & \\
        &&& \internalcatofcats_{*//}^\times && \internalcatofcats_{*//} \\
        \subuniverse^\otimes && \spanEfE^\times && \spanEfE \\
        &&& \internalcatofcats^\times && \internalcatofcats.
        \arrow[from=1-1, to=1-3]
        \arrow[from=1-1, to=3-1]
        \arrow[from=1-3, to=1-5]
        \arrow[from=1-3, to=2-4]
        \arrow[from=1-3, to=3-3]
        \arrow[from=1-5, to=2-6]
        \arrow[from=1-5, to=3-5]
        \arrow[from=3-3, to=3-5]
        \arrow[from=2-4, to=2-6, crossing over]
        \arrow[from=2-4, to=4-4, crossing over]
        \arrow[from=2-6, to=4-6]
        \arrow[from=3-1, to=3-3]
        \arrow[from=3-3, to=4-4]
        \arrow[from=3-5, to=4-6]
        \arrow[from=4-4, to=4-6]
    \end{tikzcd}\]
    We wish to show first that the back square of the cube is cartesian. This follows since, as we showed before, the front face is cartesian as well as the right and left faces.

    The outer rectangle at the back is the cartesian square defining $\intenv(\varintcat{C}^\otimes)^\otimes$. Thus, by the pasting lemma, the left square is cartesian, and so is the square in the statement of the theorem. 
\end{proof}

\subsubsection*{Lax monoidal unstraightening}

\begin{proposition}\label{prop:lax_E_monoidal_unst}
    Let 
    \[F:\varintcat{C}\to \internalcatofcats\]
    be a lax $\subuniverse$-monoidal functor. Then the unstraightening $\intunst(F)\to \varintcat{C}$ is canonically an $\subuniverse$-monoidal functor of $\subuniverse$-monoidal $\vartopos$-categories.
\end{proposition}
\begin{proof}
    Let us consider the following pullback of $\subuniverse$-operads
    \[\begin{tikzcd}
    \varintcat{P}\ar[r]\ar[d]\pullbackdr
    &(\internalcatofcats_{*//})^\times\ar[d]\\
    \varintcat{C}^{\otimes}\ar[r]&\internalcatofcats^\times.    
    \end{tikzcd}\]
    
    The right-hand functor is a cocartesian fibration, and so is the left-hand functor. Thus, the composition  $\varintcat{P}\to \varintcat{C}^{\otimes}\to \spanEfE$ is a composition of cocartesian fibrations and the arrow  $\varintcat{P}\to \varintcat{C}^{\otimes}$ preserves cocartesian lifts.

    Thus $\varintcat{P}$ classifies an $\subuniverse$-monoidal structure on $\intunst(F)$ and the functor $\varintcat{P}\to \varintcat{C}^{\otimes}$ classifies an $\subuniverse$-monoidal functor.
\end{proof}

\begin{example}\label{ex:lax_unstraightening_recovers_monoidal_structure}
    Let $\varintcat{C}$ be an $\subuniverse$-monoidal $\vartopos$-category. Then consider the following composition
    \[*\into \subuniverse\into \spanEfE \to \internalcatofcats.\]
    This composition chooses the $\vartopos$-category $\varintcat{C}$. The composition above has a canonical structure of a lax $\subuniverse$-monoidal functor. We get that the unstraightening is $\varintcat{C}$ with its original $\subuniverse$-monoidal structure. 
\end{example}
\begin{proof}
    Let us consider the following pasting of pullbacks of operads:
    \[\begin{tikzcd}
        \varintcat{P}\ar[r]\ar[d]\pullbackdr
        &\intenv(\varintcat{C}^{\otimes})^{\otimes}\ar[r]\ar[d]\pullbackdr
        &(\internalcatofcats_{*//})^{\times}\ar[d]\\
        \spanEfE \ar[r,"r"] 
        &\subuniverse^\otimes \ar[r]
        &\internalcatofcats^\times.
    \end{tikzcd}\]
    The bottom composition is the composition of morphisms of operads classifying the lax $\subuniverse$-monoidal functor described as the composition in the example above.

    The right square is a pullback by \Cref{thm:C_otimes_res_to_E_is_envelope}. The left square computes the same pullback as the left square appearing in \Cref{prop:eta_is_a_morphism_of_preoperads}.

    Thus, $\varintcat{P}$ is canonically isomorphic to $\varintcat{C}^\otimes$, proving the claim.
\end{proof}

\subsection{\texorpdfstring{$\subuniverse$}{E}-monoidal \texorpdfstring{$\vartopos$-$2$-}{B-2-}categories.}\label{subsec:E_monoidal_two_categories}

In this subsection, we define $\subuniverse$-monoidal $\vartopos$-$2$-categories (\Cref{def:E_monoidal_two_categories}) and equip every $\subuniverse^{\full}$-$2$-complete $\vartopos$-$2$-category with a canonical cartesian $\subuniverse$-monoidal structure (\Cref{ex:cartesian_monoidal_two_category}). We show that unstraightening an $\subuniverse$-monoidal functor to $\internaltwocatofcats$ produces an $\subuniverse$-monoidal $\vartopos$-$2$-category, and that unstraightening an $\subuniverse$-monoidal lax natural transformation gives an $\subuniverse$-monoidal functor (\Cref{prop:monoidal_two_unstraightening,prop:monoidal_unstraightening_of_lax_transformation}).

We then consider lax $\subuniverse$-monoidal functors whose source is an $\subuniverse$-monoidal $\vartopos$-category. We define lax $\subuniverse$-monoidal lax natural transformations between such functors (\Cref{def:lax_monoidal_lax_natural_transformation}) and show that their unstraightenings are lax $\subuniverse$-monoidal functors (\Cref{prop:unst_of_lax_E_monoidal_lax_natural_transformation}).

\subsubsection*{Monoidal \texorpdfstring{$\vartopos$-$2$-}{B-2-}categories and cartesian structures}

\begin{definition}\label{def:E_monoidal_two_categories}
    We define $\subuniverse$-monoidal $\vartopos$-$2$-categories as $\subuniverse$-monoids in $\internalcatoftwocats$.
    We define the $\vartopos$-category of $\subuniverse$-monoidal $\vartopos$-$2$-categories as  $\intsmon(\internalcatoftwocats)$.

    The $\vartopos$-category $\intsmon(\internalcatoftwocats)$ has a canonical structure of a $\vartopos$-$2$-category $\intsmontwo(\internaltwocatoftwocats)$ since \[\inttwofun(\spanEfE,\internaltwocatoftwocats)\]
    is a $\vartopos$-$2$-category and $\intsmon(\internalcatoftwocats)$ is a full $\vartopos$-subcategory of its core.  

    We denote by $\twocatoftwocats^{\subuniverse-\otimes}(\vartopos)$ the $(\infty,2)$-category of global sections of this $\vartopos$-$2$-categorical refinement of $\intsmon(\internalcatoftwocats)$.
\end{definition}

\begin{example}\label{ex:cartesian_monoidal_two_category}
    Let $\varinttwocat{X}$ be an $\subuniverse^{\full}$-$2$-complete $\vartopos$-$2$-category. Then $\varinttwocat{X}$ has a canonical structure of an $\subuniverse$-monoidal $\vartopos$-$2$-category $\intcart_{\otimes}(\varinttwocat{X})$, which we call the cartesian structure. 
    
    The construction is as follows. $\varinttwocat{X}$ determines a unique $\subuniverse^{\full}$-continuous functor
    \[ (\subuniverse^{\full})^{\op}\to \internaltwocatoftwocats \]
    which is $\subuniverse$-right adjointable.
    Thus we get a canonical functor of $\vartopos$-$2$-categories
    \[\intspanhalf(\subuniverse^{\full},\subuniverse)\to \internaltwocatoftwocats.\]
    We restrict this functor to a functor 
    \[\intspan(\subuniverse^{\full},\subuniverse)\to \internalcatoftwocats.\]
    Since this functor extends an $\subuniverse^{\full}$-continuous functor from $(\subuniverse^{\full})^{\op}$, it determines an $\subuniverse$-monoid in $\internalcatoftwocats$ by \Cref{prop:limits_in_span_E_are_computed_in_E_op}.

    This construction gives rise to a functor of $\infty$-categories 
    \[\intcart_{\otimes}:\catoftwocats(\vartopos)^{\subuniverse-2\sqcap}\to \cmon^{\subuniverse}(\internalcatoftwocats).\]
\end{example}

As usual, we will sometimes suppress the notation $\intcart_\otimes$.

\begin{proposition}
    The functor $\intcart_{\otimes}$ preserves all limits and is conservative. 
\end{proposition}

\begin{proof}
    We have a commutative triangle:
    \[\begin{tikzcd}
        \catoftwocats(\vartopos)^{\subuniverse-2\sqcap} \ar[dr]\ar[rr]
        && \cmon^{\subuniverse}(\internalcatoftwocats)\ar[dl]\\
        &\catoftwocats(\vartopos).
    \end{tikzcd}\]
    The downward-pointing morphisms are both conservative and preserve limits. Thus the horizontal morphism is also conservative and preserves limits. 
\end{proof}

\subsubsection*{Monoidal unstraightening of lax natural transformations}

\begin{proposition}\label{prop:monoidal_two_unstraightening}
    Let $F:\varinttwocat{C}\to \intcart_\otimes(\internaltwocatofcats)$ be an $\subuniverse$-monoidal functor. Then $\inttwounst(F)$ carries a natural $\subuniverse$-monoidal structure.

    Moreover, if $F$ is given by applying $\intcart_\otimes$ to a $\subuniverse^{\full}$-$2$-continuous functor, then $\inttwounst(F)$ is $\subuniverse^{\full}$-$2$-complete and the $\subuniverse$-monoidal structure agrees with the cartesian one.
\end{proposition}
\begin{proof}
    Recall that the following pullback defines $\inttwounst(F)$ as a $\vartopos$-$2$-category:
    \[\begin{tikzcd}
    \inttwounst(F)\ar[r]\ar[d]\pullbackdr & \internaltwocatofcats_{*//}\ar[d]\\
    \varinttwocat{C}\ar[r] & \internaltwocatofcats.
    \end{tikzcd}\]
    If $F$ is $\subuniverse$-monoidal, then this pullback can be taken in $\cmon^{\subuniverse}(\internalcatoftwocats)$. If $F$ is $\subuniverse^{\full}$-$2$-continuous, then the pullback can be taken in $\catoftwocats(\vartopos)^{\subuniverse-2\sqcap}$.
\end{proof}

Our next goal is to define $\subuniverse$-monoidal lax natural transformations $\alpha:F\to G:\varinttwocat{C}\to \internaltwocatofcats$ and show that these induce $\subuniverse$-monoidal structures on the functors $\inttwounst(F)\to \inttwounst(G)$ discussed in \Cref{cor:extra_functoriality_of_unst}.

\begin{proposition}
    Let $\varinttwocat{C}$ be an $\subuniverse$-monoidal $\vartopos$-$2$-category. Then $\inttwofun^{\oplax}(\Delta^1, \varinttwocat{C})$ carries a canonical $\subuniverse$-monoidal structure such that both \[s,t:\inttwofun^{\oplax}(\Delta^1, \varinttwocat{C})\to \varinttwocat{C}\]
    are $\subuniverse$-monoidal functors.
\end{proposition}

\begin{proof}
    This follows since 
    \[\internalcatoftwocats\xto{\inttwofun^{\oplax}(\Delta^1, -)}\internalcatoftwocats\] is $\subuniverse^{\full}$-continuous and thus induces a functor  
    \[\intsmon(\internalcatoftwocats)\xto{\inttwofun^{\oplax}(\Delta^1, -)}\intsmon(\internalcatoftwocats).\]
\end{proof}

\begin{definition}
    Let $F,G:\varinttwocat{C}\to \varinttwocat{D}$ be two $\subuniverse$-monoidal functors of $\vartopos$-$2$-categories. We define an $\subuniverse$-monoidal lax natural transformation $\alpha:F\to G$ as an $\subuniverse$-monoidal functor  \[T:\varinttwocat{C}\to \inttwofun^{\oplax}(\Delta^1,\varinttwocat{D})\] together with identifications of $\subuniverse$-monoidal functors $s\circ T \simeq F,  t\circ T\simeq G$. 
\end{definition}

\begin{proposition}\label{prop:monoidal_unstraightening_of_lax_transformation}
    Let $F,G:\varinttwocat{C}\to \internaltwocatofcats$ be two $\subuniverse$-monoidal functors, and let $\alpha:F\to G$ be an $\subuniverse$-monoidal lax natural transformation. Then the functor $\inttwounst(\alpha):\inttwounst(F)\to \inttwounst(G)$ coming from \Cref{cor:extra_functoriality_of_unst} is canonically an $\subuniverse$-monoidal functor.
\end{proposition}
\begin{proof}
    It is enough to show this for the universal case, where it was proven in \Cref{thm:univ_unst_of_lax_natural_transformation_is_E_monoidal}.
\end{proof}

\subsubsection*{Lax monoidal lax natural transformations}

Recall that we defined in \Cref{def:lax_E_monoidal_functors} lax $\subuniverse$-monoidal  functors of $\subuniverse$-monoidal $\vartopos$-categories $F:\varintcat{C}\to \varintcat{D}$ as morphisms of $\subuniverse$-operads $\varintcat{C}^{\otimes}\to \varintcat{D}^{\otimes}$.

We wish to define lax $\subuniverse$-monoidal natural transformations between lax $\subuniverse$-monoidal functors. We will do so under the assumption that the source is a $\vartopos$-category to avoid discussing $2$ $\subuniverse$-operads. 

\begin{definition}\label{def:lax_monoidal_lax_natural_transformation}
    Let $\varintcat{C}$ be an $\subuniverse$-monoidal $\vartopos$-category and let $\varinttwocat{D}$ be an $\subuniverse$-monoidal $\vartopos$-$2$-category.
    Let $F,G:\varintcat{C}\to \varinttwocat{D}^{\simeq2}$ be two lax $\subuniverse$-monoidal functors of $\vartopos$-$2$-categories. We define a lax $\subuniverse$-monoidal lax natural transformation $\alpha:F\to G$ as a morphism of $\subuniverse$-operads  \[T:\varintcat{C}^{\otimes}\to (\inttwofun^{\oplax}(\Delta^1,\varinttwocat{D})^{\simeq2})^{\otimes}\] together with identifications of lax $\subuniverse$-monoidal functors $s\circ T \simeq F,  t\circ T\simeq G$.
\end{definition}

Recall that in \Cref{prop:lax_E_monoidal_unst}, we showed that if $F:\varintcat{C}\to \internalcatofcats$ is lax $\subuniverse$-monoidal, then $\intunst(F)$ carries a canonical $\subuniverse$-monoidal structure. 

\begin{proposition}\label{prop:unst_of_lax_E_monoidal_lax_natural_transformation}
    Let $F,G:\varintcat{C}\to \internaltwocatofcats$ be two lax $\subuniverse$-monoidal functors, and let $\alpha:F\to G$ be a lax $\subuniverse$-monoidal lax natural transformation. Then the functor $\intunst(\alpha):\intunst(F)\to \intunst(G)$ coming from \Cref{cor:extra_functoriality_of_unst} is a lax $\subuniverse$-monoidal functor.
\end{proposition}
\begin{proof}
    Consider the following pasting in $\intoperads^{\subuniverse}$:
    \[
    \begin{tikzcd}
        \intunst(F)^{\otimes}\ar[r]\ar[d]\pullbackdr
        &(\inttwounst(s)^{\simeq2})^{\otimes}\ar[d]\\
        \intunst(G)^{\otimes}\ar[r]\ar[d]\pullbackdr
        &(\inttwounst(t)^{\simeq2})^{\otimes}\ar[d]\\
        \varintcat{C}^{\otimes} \ar[r]
        &(\inttwofun^{\oplax}(\Delta^1,\internaltwocatofcats)^{\simeq2})^{\otimes}.
    \end{tikzcd}
    \]
    The upper left vertical functor encodes a lax $\subuniverse$-monoidal functor $\intunst(\alpha):\intunst(F)\to \intunst(G)$.
\end{proof}

\subsection{\texorpdfstring{$\subuniverse$}{E}-monoidal span categories.}\label{subsec:E_monoidal_span_categories}

In this subsection, we equip the span $\vartopos$-category associated with an $\subuniverse$-monoidal adequate triplet with a natural $\subuniverse$-monoidal structure (\Cref{prop:E_monoidal_span_cats}). For inductible context-free subuniverses $\subuniverse[I],\subuniverse[P]\subset\subuniverse$, we then define the $\subuniverse$-monoidal $\vartopos$-category $\spanEPI[A]$ (\Cref{def:span_EPI_A,prop:span_EPI_A_is_monoidal}) and show that it is $(\subuniverse[P,I])$-ambidextrous (\Cref{prop:span_EPI_A_is_ambidextrous}). Under an additional truncation hypothesis, we will identify this category with $L_{\oplus}\subuniverse^{\simeq}[A]$ in \Cref{sec:proof_of_C}.

We also equip span $\vartopos$-$2$-categories $\intspanhalf(\varintcat{C},\varintcat{C}_R)$ with $\subuniverse$-monoidal structures and establish their universal property in terms of right adjointable $\subuniverse$-monoidal functors (\Cref{prop:E_monoidal_span_two_cats,prop:univ_prop_of_E_monoidal_span_half}). This universal property is an $\subuniverse$ monoidal version of \Cref{prop:univ_prop_of_span_half}.

\subsubsection*{Monoidal subcategories of \texorpdfstring{$\subuniverse^{\full}[A]$}{E full[A]}}

\begin{definition}
    Let $\subuniverse\subset \universe$ be a context-free subuniverse and let $\subuniverse[I]\subset\subuniverse$ be another context-free subuniverse. 
    Let $A\in \vartopos$. Recall from \Cref{def:free_on_A_with_E_colimits} that $\subuniverse[E]^{\full}[A]\subset\intpresh(A)$ is the free $\subuniverse$-cocomplete $\vartopos$-category on $A$. We endow it, using \Cref{prop:categories_with_E_colimits_have_E_monoidal_structure}, with an $\subuniverse$-monoidal structure. Recall that in the context $B$, a morphism in $\subuniverse[E]^{\full}[A](B)^{\Delta^1}$ consists of the following data: a morphism $D\to C$  with commuting arbitrary morphisms to $A$ and commuting morphisms in $\subuniverse$ to $B$ (see \Cref{fig:morphism_in_E_A}).
    
    Let $\subuniverse_{\subuniverse[I]}[A]\subset \subuniverse[E]^{\full}[A]$ denote the wide $\vartopos$-subcategory with morphisms in $\subuniverse[I]$. That is, $D\to C \in \subuniverse[E]^{\full}[A](B)^{\Delta^1}$ over $A$ is in $\subuniverse_{\subuniverse[I]}[A](B)$ if and only if $(D\to C)\in \subuniverse[I](C)$.
\end{definition}

\begin{figure}[ht]
    \centering
\[\begin{tikzcd}
	B & \\
	D & C \\
	A
	\arrow[tail, from=2-1, to=1-1]
	\arrow["{\shortmid\shortmid}"{marking}, from=2-1, to=2-2]
	\arrow[from=2-1, to=3-1]
	\arrow[tail, from=2-2, to=1-1]
	\arrow[from=2-2, to=3-1]
\end{tikzcd}\]
    \caption{The tailed morphisms are in $\subuniverse$. The marked arrow is in $\subuniverse[I]$}
    \label{fig:morphism_in_E_A}
\end{figure}

\begin{proposition}
    $\subuniverse_{\subuniverse[I]}[A]\subset \subuniverse[E]^{\full}[A]$ is an $\subuniverse$-monoidal $\vartopos$-subcategory.
\end{proposition}
\begin{proof}
    The $\subuniverse$-monoidal structure of $\subuniverse[E]^{\full}[A]$ is given by change of context. The question of whether $D\to C$ is in $\subuniverse[I]$ is independent of the context.
\end{proof}

\begin{proposition}\label{prop:E_I_is_cocart}
    Assume $\subuniverse[I]$ is inductible. Then
    $\subuniverse_{\subuniverse[I]}[A]$ is an $\subuniverse[I]$-cocartesian $\subuniverse$-monoidal $\vartopos$-category.
\end{proposition}
\begin{proof}
    It is straightforward to verify that $\subuniverse_{\subuniverse[I]}[A]\subset \subuniverse[E]^{\full}[A]$ preserves $\subuniverse[I]$-colimits, and thus the criterion of \Cref{cor:critiria_for_amby} is automatic.
\end{proof}

\subsubsection*{Monoidal structures on span categories}

\begin{definition}
    Let $\varintcat{C}$ be an $\subuniverse$-monoidal $\vartopos$-category. Let $(\varintcat{C},\varintcat{C}_L,\varintcat{C}_R)$ be an adequate triplet of $\vartopos$-categories. We say that $(\varintcat{C},\varintcat{C}_L,\varintcat{C}_R)$ is an $\subuniverse$-monoidal adequate triplet if $\varintcat{C}_L,\varintcat{C}_R\subset \varintcat{C}$ are $\subuniverse$-monoidal $\vartopos$-subcategories, and if for each $p:C\to B\in \subuniverse(A)$, the functor of $\vartopos_{/A}$-categories $p_\otimes:\varintcat{C}^C\to \varintcat{C}^B$ induces a morphism of adequate triplets
    \[(\varintcat{C}^C,\varintcat{C}_L^C,\varintcat{C}_R^C)\to (\varintcat{C}^B,\varintcat{C}_L^B,\varintcat{C}_R^B).\]
\end{definition}

\begin{proposition}\label{prop:E_monoidal_span_cats}
    Let $(\varintcat{C},\varintcat{C}_L,\varintcat{C}_R)$ be an $\subuniverse$-monoidal adequate triplet. Then $\intspan(\varintcat{C},\varintcat{C}_L,\varintcat{C}_R)$ carries a natural $\subuniverse$-monoidal structure. 
\end{proposition}

\begin{proof}
    By the proof of \Cref{prop:monoidal_subcategories_are_monoidal}, we get that $(\varintcat{C},\varintcat{C}_L,\varintcat{C}_R)$ determines an $\subuniverse^{\full}$-continuous functor \[\spanEfE\to \intadtrip\] which we can post-compose with $\intspan$.
\end{proof}

\begin{proposition}\label{prop:E_monoidal_span_subcats}
    In the situation of \Cref{prop:E_monoidal_span_cats}, if $\varintcat{C}_L'\subset \varintcat{C}_L,\varintcat{C}_R'\subset\varintcat{C}_R$ are wide $\vartopos$-subcategories such that $(\varintcat{C},\varintcat{C}_L',\varintcat{C}_R')$ is another $\subuniverse$-monoidal adequate triplet, then $\intspan(\varintcat{C},\varintcat{C}_L',\varintcat{C}_R')\subset \intspan(\varintcat{C},\varintcat{C}_L,\varintcat{C}_R)$ is an $\subuniverse$-monoidal $\vartopos$-subcategory in the sense that it satisfies the universal property of \Cref{prop:monoidal_subcategories_are_monoidal}.
\end{proposition}

\begin{proof}
    This is immediate from the construction.
\end{proof}

    

\subsubsection*{The ambidextrous span category}

\begin{definition}\label{def:span_EPI_A}
    We denote \[\spanEPI[A]:=\intspan(\subuniverse[E]^{\full}[A],\subuniverse[E]_{\subuniverse[P]}[A],\subuniverse[E]_{\subuniverse[I]}[A]).\]
\end{definition}

\begin{proposition}\label{prop:span_EPI_A_is_monoidal}
    Let $\subuniverse[I],\subuniverse[P]\subset \subuniverse$ be two inductible context-free subuniverses. Then $\spanEPI[A]$
    is an $\subuniverse$-monoidal $\vartopos$-category.
\end{proposition}
\begin{proof}
    This is a direct application of \Cref{prop:E_monoidal_span_subcats}.
\end{proof}

\begin{proposition}\label{prop:span_EPI_A_is_ambidextrous}
    $\spanEPI[A]$ is $(\subuniverse[P,I])$-ambidextrous.
\end{proposition}
\begin{proof}
    $\spanEPI[A]$ is $\subuniverse[I]$-cocomplete and $\subuniverse[P]$-complete by the same argument as in \Cref{prop:limits_in_span_E}. Moreover, the inclusion $\subuniverse_{\subuniverse[P]}[A]^{\op}\into \spanEPI[A]$ is $\subuniverse[P]^{\full}$-continuous and the inclusion $\subuniverse_{\subuniverse[I]}[A]\into \spanEPI[A]$ is $\subuniverse[I]^{\full}$-cocontinuous.
    
    Thus the criterion of \Cref{cor:critiria_for_amby} applies, since it can be checked in $\subuniverse_{\subuniverse[P]}[A]^{\op}$ and $\subuniverse_{\subuniverse[I]}[A]$.
\end{proof}

\subsubsection*{The monoidal universal property}

We now turn to proving an $\subuniverse$-monoidal version of the universal property of span $\vartopos$-$2$-categories described in \Cref{prop:univ_prop_of_span_half}.

\begin{proposition}\label{prop:E_monoidal_span_two_cats}
    Let $(\varintcat{C},\varintcat{C}')$ be an $\subuniverse$-monoidal span pair. Then $\intspanhalf(\varintcat{C},\varintcat{C}')$ carries a natural $\subuniverse$-monoidal structure. 
\end{proposition}
\begin{proof}
    The proof of \Cref{prop:E_monoidal_span_cats} works in this case as well.
\end{proof}

Let $(\varintcat{C},\varintcat{C'})$ be an $\subuniverse$-monoidal span pair, and let $F:\varintcat{C}^{\op}\to \varinttwocat{D}$ be an $\subuniverse$-monoidal functor of $\vartopos$-$2$-categories. Then we say that $F$ is $\varintcat{C}'$-right adjointable if, in every context $A\in \vartopos$, for every $B\in \spanEfE(A)$, the functor of $\vartopos_{/A}$-$2$-categories \[(\pi_A^*\varintcat{C}^{\op})^{B}\to \pi_A^{*}\varinttwocat{D}^{B}\]
is $\varintcat{C}'$-right adjointable. 

We denote by $\Fun_{\twocatoftwocats^{\subuniverse-\otimes}(\vartopos)}^{\otimes,\varintcat{C}'\operatorname{-radj}}(\varintcat{C}^{\op},\varinttwocat{D})$ the $\infty$-subcategory of $\Fun_{\twocatoftwocats^{\subuniverse-\otimes}(\vartopos)}(\varintcat{C}^{\op},\varinttwocat{D})$ spanned by $\varintcat{C'}$-adjointable $\subuniverse$-monoidal functors and adjointable natural transformations. 

\begin{proposition}\label{prop:univ_prop_of_E_monoidal_span_half}
    We have an equivalence of $\infty$-categories
    \[\Fun_{\twocatoftwocats^{\subuniverse-\otimes}(\vartopos)}^{\otimes,\varintcat{C}'\operatorname{-radj}}(\varintcat{C}^{\op},\varinttwocat{D})\simeq \Fun_{\twocatoftwocats^{\subuniverse-\otimes}(\vartopos)}^{\otimes}(\intspanhalf(\varintcat{C},\varintcat{C}'),\varinttwocat{D}).\]
\end{proposition}

\begin{proof}
    This follows from applying the universal property of \Cref{prop:univ_prop_of_span_half} pointwise.
\end{proof}

\section{Proof of \texorpdfstring{\Cref{main_theorem:free_E_monoidal_semiadditive}}{Theorem C}}\label{sec:proof_of_C}

Let $\vartopos$ be a topos. Let $\subuniverse[I],\subuniverse[P]\subset \subuniverse$ be context-free subuniverses. Assume that $\subuniverse[I],\subuniverse[P]$ are both inductible. 

We make the following additional assumption.
\begin{assumption}\label{assumption:global_trunc_bound}
    There exists an integer $n \geq -2$ such that every morphism in $\subuniverse[I],\subuniverse[P]$ in any context is $n$-truncated.\footnote{See \Cref{remark:colimit_of_n_truncated} for a discussion of whether this assumption is necessary for \Cref{main_theorem:free_E_monoidal_semiadditive}.}
\end{assumption}

In this section, we will prove under the assumption above a version of \Cref{main_theorem:free_E_monoidal_semiadditive} for an arbitrary topos $\vartopos$ (\Cref{thm:generalized_main_C}).

We will now explain the main steps:

Let $\varintcat{C}$ be a $(\subuniverse[P,I])$-ambidextrous $\subuniverse$-monoidal $\vartopos$-category (\Cref{def:P_I_ambi_E_monoidal_cats}) and let $x:A\to \varintcat{C}$ be an object of $\varintcat{C}$ in the context $A$. Recall $\spanEPI[A]$ from \Cref{def:span_EPI_A}, \Cref{prop:span_EPI_A_is_monoidal}, and \Cref{prop:span_EPI_A_is_ambidextrous}. We will construct a natural $\subuniverse$-monoidal functor \[\Xi_x:\spanEPI[A]\to \varintcat{C}.\]

Recall that $\varintcat{C}^\otimes$ denotes the $\subuniverse$-operad associated with $\varintcat{C}$ (\Cref{prop:otimes_sends_monoidal_to_operads}). We equip $\varintcat{C}^\otimes$ with its cartesian $\subuniverse$-monoidal structure (\Cref{cor:associated_operad_total_category_complete}). Given an $\subuniverse$-monoidal $\vartopos$-subcategory $\varintcat{J}\subset \spanEfE$, we denote by $\varintcat{C}^{\otimes}|_{\varintcat{J}}$ the pullback of $\subuniverse$-monoidal $\vartopos$-categories
\[\begin{tikzcd}
    \varintcat{C}^{\otimes}|_{\varintcat{J}}\ar[r]\ar[d]\pullbackdr
    &\varintcat{C}^\otimes\ar[d]\\
    \varintcat{J}\ar[r] & \spanEfE.
\end{tikzcd}\]

There are two main steps to the construction:

\begin{steps}
    \item \label{step:construction_of_the_object_functor} 
    We will construct in \Cref{subsec:Phi_x} an $\subuniverse$-monoidal functor \[\Phi_x: \spanEPI[A]\to \varintcat{C}^{\otimes}|_{\spanEPI}.\]
    This functor will not depend on the $\subuniverse$-monoidal structure of $\varintcat{C}$, only on the fact that $\varintcat{C}$ is $\subuniverse[I]$-cocomplete.
    \item \label{step:construction_of_the_tensoring_functor} We will construct in \Cref{subsec:Psi} an $\subuniverse$-monoidal functor \[\Psi:\varintcat{C}^{\otimes}|_{\spanEPI}\to \varintcat{C}.\] This functor will not depend on $x$, but only on the $\subuniverse$-monoidal structure of $\varintcat{C}$.
\end{steps}

Then we will define \[\Xi_x:=\Psi \circ \Phi_x.\]
We will finish the section by proving in \Cref{subsec:universality_of_Xi} that the construction $(x,\varintcat{C})\mapsto \Xi_x$ identifies $\spanEPI[A]$ with $L_{(\subuniverse[P,I])-\oplus}\subuniverse[E][A]^{\simeq}$ (\Cref{thm:generalized_main_C}).

\subsection{Construction of \texorpdfstring{$\Phi_x$}{Phi}}\label{subsec:Phi_x}
We now detail the construction of \[\Phi_x: \spanEPI[A]\to \varintcat{C}^{\otimes}|_{\spanEPI}\] described in \Cref{step:construction_of_the_object_functor}. 

The construction has two steps. In the first, we construct a $\vartopos$-$2$-category whose underlying $\vartopos$-category is
\[\varintcat{C}^{\otimes}|_{\intspan(\subuniverse^{\full},\subuniverse_{\subuniverse[I]})}\]
(\Cref{prop:core_of_unst_F_is_associated_operad}).

In the second, we use the universal property of the functor
\[(\subuniverse^{\full}[A])^{\op}\to \intspanhalf(\subuniverse^{\full}[A],\subuniverse_{\subuniverse[I]}[A])^{\co}\]
to construct $\Phi_x$, using the left adjointable variant of \Cref{prop:univ_prop_of_E_monoidal_span_half}.

\subsubsection*{Step one: the \texorpdfstring{$\vartopos$-$2$-}{B-2-}categorical refinement}

Let $\varintcat{C}$ be a $(\subuniverse[P,I])$-ambidextrous $\subuniverse$-monoidal $\vartopos$-category.

Let 
\[ 
F_0: (\subuniverse^{\full})^{\op}\to \internalcatofcats
\]
be the $\subuniverse^{\full}$-continuous functor corresponding to $\varintcat{C}$ under the equivalence 
\[
\intfun^{\subuniverse-\sqcap}((\subuniverse^{\full})^{\op}, \internalcatofcats)\simeq \internalcatofcats
\]
which follows from \Cref{prop:S_is_free_with_S_colmitis}.

The composition
\[(\subuniverse^{\full})^{\op}\xto{F_0} \internalcatofcats\into \internaltwocatofcats\]
is $\subuniverse_{\subuniverse[I]}$-left adjointable. Thus, we get a functor of $\vartopos$-$2$-categories
\[F:\intspanhalf(\subuniverse^{\full}, \subuniverse_{\subuniverse[I]})^{\co}\to\internaltwocatofcats.\]

\begin{proposition}
    The functor $F$ is an $\subuniverse$-monoidal functor.
\end{proposition}
\begin{proof}
    The functor $F_0$ is $\subuniverse^{\full}$-continuous by assumption. Thus it is $\subuniverse$-monoidal for the cartesian structure. By \Cref{prop:univ_prop_of_E_monoidal_span_half}, the functor $F$ is $\subuniverse$-monoidal as well. 
\end{proof}

We get an $\subuniverse$-monoidal $\vartopos$-$2$-category $\inttwounst(F)$ by \Cref{prop:monoidal_two_unstraightening}.

\begin{proposition}\label{prop:core_of_unst_F_is_associated_operad}
    We have an $\subuniverse$-monoidal equivalence \[\inttwounst(F)^{\simeq2}\simeq \varintcat{C}^{\otimes}|_{\intspan(\subuniverse^{\full},\subuniverse_{\subuniverse[I]})}.\]
\end{proposition}
\begin{proof}
    The inclusion of 
    \[\intspan(\subuniverse^{\full},\subuniverse_{\subuniverse[I]})\into \intspanhalf(\subuniverse^{\full}, \subuniverse_{\subuniverse[I]})^{\co}\]
    agrees with the inclusion of the underlying $\vartopos$-category and is $\subuniverse$-monoidal itself.

    Thus, the result follows from a general statement: the underlying $\infty$-category of a $1$-cocartesian fibration $p:\vartwocat{P}\to \vartwocat{X}$ between $(\infty,2)$-categories is the cocartesian fibration classifying the restriction of the functor $\vartwocat{X}\to \twocatofcats$ to $\vartwocat{X}^{\simeq2}$.

    We now need to compare the two monoidal functors
    \[\intspan(\subuniverse^{\full},\subuniverse_{\subuniverse[I]})\to \internalcatofcats:\]
    the one coming from $F$ and the one coming from the restriction of the $\subuniverse$-monoidal structure of $\varintcat{C}$.
    
    The restrictions to $(\subuniverse^{\full})^{\op}$ agree since they agree on the image $\varintcat{C}$ of $*$ and they both classify $\subuniverse^{\full}$-continuous functors 
    \[(\subuniverse^{\full})^{\op}\to \internalcatofcats.\]

    By \Cref{assumption:global_trunc_bound}, every morphism in $\subuniverse_{\subuniverse[I]}$ is truncated.\footnote{This is the only place in the construction where we use \Cref{assumption:global_trunc_bound} about $\subuniverse[I]$.}
    
    Then the two monoidal functors agree by \Cref{thm:internal_uniqueness_of_unfurling}.
\end{proof}

\subsubsection*{Step two: extending \texorpdfstring{$x$ to $\Phi_x$}{x to Phi}}

We need a general lemma about $1$-cocartesian fibrations between $(\infty,2)$-categories, which we plan to apply to $\inttwounst(F)$.

\begin{lemma}\label{lem:lifting_of_adjunction_by_cocart_fibrations}
    Let 
    \[p:\vartwocat{P}\to \vartwocat{X}\]
    be a $1$-cocartesian fibration between $(\infty,2)$-categories. Then a cocartesian $1$-morphism $f:x\to y\in \vartwocat{P}$ admits a left adjoint if and only if $p(f)$ admits a left adjoint.

    Moreover, a commutative square 
    \[\begin{tikzcd}
        x\ar[r,"f"]\ar[d] &y\ar[d]\\
        x'\ar[r,"f'"] & y'
    \end{tikzcd}\]
    of cocartesian edges is horizontally left adjointable if and only if the square 
    \[\begin{tikzcd}
        p(x)\ar[r,"p(f)"]\ar[d] &p(y)\ar[d]\\
        p(x')\ar[r,"p(f')"] & p(y')
    \end{tikzcd}\]
    is horizontally left adjointable.
\end{lemma}
\begin{proof}
    Let $r':x'\to y' \in \vartwocat{P}$. If $r'$ has a left adjoint, then so does $r:=p(r')$ since $p$ is a $2$-functor and thus preserves adjunctions.

    Let $r':x'\to y'\in \vartwocat{P}$ be a cocartesian edge. Suppose that $[r:x\to y ]= [p(r'):p(x')\to p(y')]$ has a left adjoint $l$. Then we want to show that $r'$ has a left adjoint.

    Let $\epsilon':e\to \id_{x'}$ be the $2$-cartesian lift of the counit $\epsilon:lr\to \id_x$. Then $p(e) = lr$ factors through $r$ and thus, since $r'$ is cocartesian, $e$ factors as a composition 
    \[x'\xto{r'} y' \xto{l'} x'.\]
    
    Let $\eta':f'\to r'l'$ be likewise a $2$-cartesian lift of the unit $\eta:\id_{y}\to rl$. By the zig-zag identity for $l\dashv r$, both $(r'*\epsilon')\circ(\eta'*r'):f'\circ r'\to r'$ and the identity $2$-morphism $r' \to r'$ are cartesian lifts of the identity $2$-morphism of $r$. Thus $f'\circ r' = r'$. By the universal property of the cocartesian edge $r'$, we get that $f' = \id_{y'}$ as both are sources of a $2$-cartesian lift of the same morphism. 

    We claim that $\eta',\epsilon'$ are unit and counit of an adjunction $l'\dashv r'$.
    The zig-zag identities follow from functoriality of cartesian lifts.

    The statement about the adjointability of squares is similar. The Beck-Chevalley property is preserved by any functor. The Beck-Chevalley property is reflected since $p$ induces right fibrations on mapping categories and right fibrations are conservative. 
\end{proof}

The $\vartopos$-category $\varintcat{C}^{\otimes}|_{\intspan(\subuniverse^{\full},\subuniverse_{\subuniverse[I]})}$ is $\subuniverse^{\full}$-complete for the same reason that $\varintcat{C}^\otimes$ is $\subuniverse^{\full}$-complete (\Cref{cor:associated_operad_total_category_complete}), since $\intspan(\subuniverse^{\full},\subuniverse_{\subuniverse[I]})$ is $\subuniverse^{\full}$-complete by the same argument as for $\spanEfE$ (\Cref{prop:limits_in_span_E}).

The functor $x:A \to \varintcat{C}$ gives rise to the composition $A\to\varintcat{C}\into \varintcat{C}^{\otimes}|_{\intspan(\subuniverse^{\full},\subuniverse_{\subuniverse[I]})}$, where the second map is the inclusion of the fiber lying over $*\in {\intspan(\subuniverse^{\full},\subuniverse_{\subuniverse[I]})}$. This induces an $\subuniverse^{\full}$-continuous functor \[x':\subuniverse[E]^{\full}[A]^{\op}\to \varintcat{C}^{\otimes}|_{\intspan(\subuniverse^{\full},\subuniverse_{\subuniverse[I]})}\] by (the dual of) \Cref{prop:free_on_A_with_E_colimits}.

The functor $x'$
lands in the wide $\vartopos$-subcategory spanned by cocartesian arrows in $\varintcat{C}^{\otimes}|_{\intspan(\subuniverse^{\full},\subuniverse_{\subuniverse[I]})}$. Indeed, this is true for the composition $x\colon A\to \varintcat{C}\to \varintcat{C}^{\otimes}|_{\intspan(\subuniverse^{\full},\subuniverse_{\subuniverse[I]})}$ since $A$ is a $\vartopos$-groupoid and follows for the functor $x'$ from $\subuniverse^{\full}[A]^{\op}$ by \Cref{thm:univ_unst_is_monoidal,cor:associated_operad_total_category_complete}, which imply that the wide $\vartopos$-subcategory spanned by cocartesian lifts is closed under $\subuniverse^{\full}$-limits. 

\begin{proposition}
    The composition
    \[
    \subuniverse^{\full}[A]^{\op}\xto{x'} \varintcat{C}^{\otimes}|_{\intspan(\subuniverse^{\full},\subuniverse_{\subuniverse[I]})} \into \inttwounst(F)
    \]
    is an $\subuniverse$-monoidal $\subuniverse_{\subuniverse[I]}[A]$-left adjointable functor.
\end{proposition}
\begin{proof}
    The composition above is $\subuniverse$-monoidal by \Cref{prop:core_of_unst_F_is_associated_operad}.

    By \Cref{lem:lifting_of_adjunction_by_cocart_fibrations}, to check that the composition is $\subuniverse_{\subuniverse[I]}[A]$-left adjointable, it is enough to check that the composition
    \[
    \subuniverse^{\full}[A]^{\op}\xto{x'} \varintcat{C}^{\otimes}|_{\intspan(\subuniverse^{\full},\subuniverse_{\subuniverse[I]})} \into \inttwounst(F) \to \intspanhalf(\subuniverse^{\full}, \subuniverse_{\subuniverse[I]})^{\co}
    \]
    is $\subuniverse_{\subuniverse[I]}[A]$-left adjointable.

    Note that the projection $\varintcat{C}^{\otimes}|_{\intspan(\subuniverse^{\full},\subuniverse_{\subuniverse[I]})}\to \intspan(\subuniverse^{\full}, \subuniverse_{\subuniverse[I]})$ is $\subuniverse^{\full}$-continuous. Thus, by construction of $x'$, this composition is the $\subuniverse^{\full}$-continuous functor $\subuniverse^{\full}[A]^{\op}\to \intspan(\subuniverse^{\full}, \subuniverse_{\subuniverse[I]})$ corresponding to the constant functor $A\to \intspan(\subuniverse^{\full}, \subuniverse_{\subuniverse[I]})$  on $*$.

    Thus the functor $\subuniverse^{\full}[A]^{\op}\to \intspanhalf(\subuniverse^{\full}, \subuniverse_{\subuniverse[I]})^{\co}$ factors as 
    \[\subuniverse^{\full}[A]^{\op}\to (\subuniverse^{\full})^\op \into \intspanhalf(\subuniverse^{\full}, \subuniverse_{\subuniverse[I]})^{\co}.\]

    Note that the functor \[\subuniverse^{\full}[A]\to \subuniverse^{\full}\] preserves pullbacks and sends $\subuniverse_{\subuniverse[I]}[A]$ to $\subuniverse_{\subuniverse[I]}$. Thus we get that the composition $\subuniverse^{\full}[A]^{\op}\to (\subuniverse^{\full})^\op \into \intspanhalf(\subuniverse^{\full}, \subuniverse_{\subuniverse[I]})^{\co}$ is $\subuniverse_{\subuniverse[I]}[A]$-left adjointable.
\end{proof}

Thus, by \Cref{prop:univ_prop_of_E_monoidal_span_half}, we get an $\subuniverse$-monoidal functor of $\vartopos$-$2$-categories
\[\intspanhalf(\subuniverse^{\full}[A], \subuniverse_{\subuniverse[I]}[A])^{\co}\to \inttwounst(F).\]

Restricting to underlying $\vartopos$-categories, we get an $\subuniverse$-monoidal functor 
\[\intspan(\subuniverse^{\full}[A], \subuniverse_{\subuniverse[I]}[A])\to \varintcat{C}^{\otimes}|_{\intspan(\subuniverse^{\full},\subuniverse_{\subuniverse[I]})}.\]

By construction, the composition 
\[\intspan(\subuniverse^{\full}[A], \subuniverse_{\subuniverse[I]}[A])\to \varintcat{C}^{\otimes}|_{\intspan(\subuniverse^{\full},\subuniverse_{\subuniverse[I]})}\to \intspan(\subuniverse^{\full}, \subuniverse_{\subuniverse[I]})\]
forgets the map to $A$ and so the restriction to $\spanEPI[A]$ lands in $\spanEPI$.

Thus, further restricting to $\spanEPI[A]$, we get an $\subuniverse$-monoidal functor by \Cref{prop:E_monoidal_span_subcats}
\[\Phi_x: \spanEPI[A]\to \varintcat{C}^{\otimes}|_{\spanEPI}.\]

\subsection{Construction of \texorpdfstring{$\Psi$}{Psi}}\label{subsec:Psi}

We are left to construct the $\subuniverse$-monoidal functor 
\[\Psi:\varintcat{C}^{\otimes}|_{\spanEPI}\to \varintcat{C}\] 
as detailed in \Cref{step:construction_of_the_tensoring_functor}.

Let us first consider the case where $\varintcat{C}$ is $\subuniverse^{\full}$-complete and its monoidal structure is cartesian. In this case, we have constructed in \Cref{eq:construction_of_mu} a functor 
\[\mu:\varintcat{C}^{\otimes}\to \varintcat{C}\]
which is $\subuniverse^{\full}$-continuous.\footnote{We did not show that $\mu$ is $\subuniverse^{\full}$-continuous nor do we require this fact, but it is true.} In this case, $\Psi$ will be the restriction of $\mu$:
\[\varintcat{C}^{\otimes}|_{\spanEPI}\to \varintcat{C}^{\otimes} \xto{\mu} \varintcat{C}.\]

We proceed to construct $\Psi$ in general, using the $2$-categorical methods developed in \Cref{sec:2_cats}.

\subsubsection*{Reduction to a lax natural transformation}

Let 
\[F:\spanEfE\to \internalcatofcats\]
be the functor encoding the $\subuniverse$-monoidal structure on $\varintcat{C}$. 

Recall that we showed in \Cref{prop:lax_E_monoidal_unst} that given a lax $\subuniverse$-monoidal functor \[H:\varintcat{T}\to \internalcatofcats,\] the unstraightening $\intunst(H)$ is an $\subuniverse$-monoidal $\vartopos$-category. In \Cref{ex:lax_unstraightening_recovers_monoidal_structure}, we also showed that the $\subuniverse$-monoidal structure on the unstraightening of the composition
\[
\explicitset{*}\into \spanEfE \xto{F} \internalcatofcats
\]
is again $\varintcat{C}$ with its original $\subuniverse$-monoidal structure.

Let us denote 
\[F':\spanEPI\into \spanEfE\xto{F} \internalcatofcats. \]
Then $\intunst(F')\simeq \varintcat{C}^{\otimes}|_{\spanEPI}$ as an $\subuniverse$-monoidal $\vartopos$-category. 

Let us denote the composition \[G:\spanEPI\into \spanEfE\to \explicitset{*}\into \spanEfE \xto{F} \internalcatofcats.\] Then we have a canonical $\subuniverse$-monoidal equivalence $\intunst(G)\simeq \spanEPI\times \varintcat{C}$.
Our strategy to construct $\Psi$ is to construct a lax $\subuniverse$-monoidal lax natural transformation
\[\alpha:F'\Rightarrow G:\spanEPI\to \internalcatofcats.\]
We will then use \Cref{prop:unst_of_lax_E_monoidal_lax_natural_transformation}, which will show that \[\intunst(\alpha):\intunst(F')\to \intunst(G)\]
is lax $\subuniverse$-monoidal.

Lastly, we will show that $\intunst(\alpha)$ is $\subuniverse$-monoidal.

\begin{proposition}
    Suppose $\varintcat{C}$ is a $\subuniverse[P]$-cartesian $\subuniverse$-monoidal $\vartopos$-category. Then the composition of $\subuniverse$-monoidal functors
    \[\subuniverse\into \spanEfE \to \internaltwocatofcats\] is $\subuniverse_{\subuniverse[P]}$-left adjointable.
\end{proposition}
\begin{proof}
    Let $[p:B\to A]\in \subuniverse_{\subuniverse[P]}(A)$.
    The existence of the left adjoints $p^*$ to \[p_*\simeq p_\otimes: \varintcat{C}^{B}\to \varintcat{C}^A\] is immediate from the definition. The adjointability follows from the fact that the $\subuniverse$-monoidal structure of $\varintcat{C}$ is indexed by $\spanEfE$ and the declared base-change maps in that functoriality agree with the Beck-Chevalley transformations. 
\end{proof}

\begin{corollary}\label{cor:extending_the_alg_structure_of_C_to_backwards_facing_morphisms}
    The $\subuniverse$-monoidal composition 
    \[\subuniverse\into \spanEfE \to \internaltwocatofcats\]
    extends to an $\subuniverse$-monoidal functor of $\vartopos$-$2$-categories:
    \[\intspanhalf(\subuniverse,\subuniverse_{\subuniverse[P]},\subuniverse)\to \internaltwocatofcats.\]
\end{corollary}
\begin{proof}
    This is a direct application of \Cref{prop:univ_prop_of_E_monoidal_span_half}.
\end{proof}

\subsubsection*{The monoidal structure on the oplax slice}

Recall that we wish to construct a lax $\subuniverse$-monoidal lax natural transformation $\alpha:F'\to G:\spanEPI\to \internaltwocatofcats$. By \Cref{cor:extending_the_alg_structure_of_C_to_backwards_facing_morphisms}, it is enough to construct a lax $\subuniverse$-monoidal lax natural transformation $\alpha':F''\to G':\spanEPI\to \intspanhalf(\subuniverse,\subuniverse_{\subuniverse[P]},\subuniverse)$. 
Here $F''$ is the inclusion and $G'$ is the constant functor at $*$. Such a transformation is encoded by a functor
\[
\spanEPI \to \inttwofun^{\oplax}(\Delta^1, \intspanhalf(\subuniverse,\subuniverse_{\subuniverse[P]},\subuniverse)).
\]
We then compose with the functor induced by \Cref{cor:extending_the_alg_structure_of_C_to_backwards_facing_morphisms}:
\[
\inttwofun^{\oplax}(\Delta^1, \intspanhalf(\subuniverse,\subuniverse_{\subuniverse[P]},\subuniverse)) \to \inttwofun^{\oplax}(\Delta^1, \internaltwocatofcats).
\]

Consider the following pullback in $\twocatoftwocats(\vartopos)$:
\[\begin{tikzcd}
    \intspanhalf(\subuniverse,\subuniverse_{\subuniverse[P]},\subuniverse)_{//*}\ar[d]\ar[r]\pullbackdr
    &\inttwofun^{\oplax}(\Delta^1, \intspanhalf(\subuniverse,\subuniverse_{\subuniverse[P]},\subuniverse))\ar[d,"t"]\\
    \explicitset{*}\ar[r]
    & \intspanhalf(\subuniverse,\subuniverse_{\subuniverse[P]},\subuniverse).
\end{tikzcd}\]

We may apply $(-)^{\simeq2}$ to the diagram above:

\[\begin{tikzcd}
    (\intspanhalf(\subuniverse,\subuniverse_{\subuniverse[P]},\subuniverse)_{//*})^{\simeq 2}\ar[d]\ar[r]\pullbackdr
    &\inttwofun^{\oplax}(\Delta^1, \intspanhalf(\subuniverse,\subuniverse_{\subuniverse[P]},\subuniverse))^{\simeq2}\ar[d,"t"]\\
    \explicitset{*}\ar[r]
    & \intspan(\subuniverse,\subuniverse_{\subuniverse[P]},\subuniverse).
\end{tikzcd}\]

The functor $t$ is $\subuniverse$-monoidal. The lower horizontal arrow is lax $\subuniverse$-monoidal.

\begin{proposition}
    The upper morphism in the pullback above is lax $\subuniverse$-monoidal. 
\end{proposition}
\begin{proof}
    Consider the following pullback in $\intoperads^{\subuniverse}$:
    \[\begin{tikzcd}
        \varintcat{P}\ar[d]\ar[r]\pullbackdr
        &(\inttwofun^{\oplax}(\Delta^1, \intspanhalf(\subuniverse,\subuniverse_{\subuniverse[P]},\subuniverse))^{\simeq2})^{\otimes}\ar[d,"t^{\otimes}"]\\
        \spanEfE\ar[r]
        & \intspan(\subuniverse,\subuniverse_{\subuniverse[P]},\subuniverse)^{\otimes},
    \end{tikzcd}\]
    where the right and bottom arrows classify the lax $\subuniverse$-monoidal functors on the right and bottom of the previous square.

    The functor $t$ in the previous square is a cocartesian fibration. Therefore, the right functor $t^\otimes$ of the diagram directly above is a cocartesian fibration. Therefore, the left functor in the diagram above is a cocartesian fibration. Thus, the operad $\varintcat{P}$ is associated to an $\subuniverse$-monoidal structure on $(\intspanhalf(\subuniverse,\subuniverse_{\subuniverse[P]},\subuniverse)_{//*})^{\simeq 2}$ and the upper morphism is a morphism of operads encoding the promised lax $\subuniverse$-monoidal structure.
\end{proof}

Let us describe the $\subuniverse$-monoidal structure of $(\intspanhalf(\subuniverse,\subuniverse_{\subuniverse[P]},\subuniverse)_{//*})^{\simeq 2}$.

An object of $(\intspanhalf(\subuniverse,\subuniverse_{\subuniverse[P]},\subuniverse)_{//*})^{\simeq 2}(A)$ is a span in $(\subuniverse,\subuniverse_{\subuniverse[P]},\subuniverse)(A)$:
\[B\from B' \to A.\]
A morphism is a diagram of the form

\[\begin{tikzcd}
	{B} && {F} && C \\
	&&& {F\times_CC'} \\
	{B'} &&&& C' \\
	& {B'} \\
	{A} && {A} && {A}.
	\arrow[from=1-3, to=1-1]
	\arrow[tail,from=1-3, to=1-5]
	\arrow[from=2-4, to=1-3]
	\arrow["\lrcorner"{anchor=center, pos=0.125, rotate=90}, draw=none, from=2-4, to=1-5]
	\arrow[tail, from=2-4, to=3-5]
	\arrow[from=3-1, to=1-1]
	\arrow[from=3-1, to=5-1]
	\arrow[from=3-5, to=1-5]
	\arrow[tail,from=3-5, to=5-5]
    \arrow[from=2-4, to=4-2]
	\arrow[from=4-2, to=3-1]
	\arrow["\lrcorner"{anchor=center, pos=0.125, rotate=-90}, draw=none, from=4-2, to=5-1]
	\arrow[tail,from=4-2, to=5-3]
	\arrow[equal,from=5-3, to=5-1]
	\arrow[equal,from=5-3, to=5-5]
\end{tikzcd}\]
Let $s:A'\to A\in \subuniverse(A)$. Then $s_\otimes$ takes the object $B\from B' \to A'$ to the object 
\[B\from B'\to A'\xto{s} A.\]
The action on morphisms is likewise described.

\begin{proposition}\label{prop:oplax_slice_source_is_monoidal}
    The functor 
    \[(\intspanhalf(\subuniverse,\subuniverse_{\subuniverse[P]},\subuniverse)_{//*})^{\simeq 2}\xto{s} \intspan(\subuniverse,\subuniverse_{\subuniverse[P]},\subuniverse)\]
    is $\subuniverse$-monoidal.
\end{proposition}
\begin{proof}
    We first show that the functor $s$ above is lax $\subuniverse$-monoidal.
    This functor is the composition of a lax $\subuniverse$-monoidal functor with an $\subuniverse$-monoidal functor:
    \[(\intspanhalf(\subuniverse,\subuniverse_{\subuniverse[P]},\subuniverse)_{//*})^{\simeq 2}\xto{a} \inttwofun^{\oplax}(\Delta^1, \intspanhalf(\subuniverse,\subuniverse_{\subuniverse[P]},\subuniverse))^{\simeq2}\xto{s} \intspan(\subuniverse,\subuniverse_{\subuniverse[P]},\subuniverse).\]

    Next, we need to check that the composition is actually $\subuniverse$-monoidal. Let 
    \[x = (B\from B' \to A')\]
    be an object in $(\intspanhalf(\subuniverse,\subuniverse_{\subuniverse[P]},\subuniverse)_{//*})^{\simeq 2}(A')$. Let $p:A'\to A\in \subuniverse(A)$. We must check that $p_\otimes(s(x)) \to  s(p_\otimes (x))$ is an equivalence. 
    
    We first write the comparison map for $p_\otimes(a(x)) \to  a(p_\otimes (x))$ in 
    $\inttwofun^{\oplax}(\Delta^1, \intspanhalf(\subuniverse,\subuniverse_{\subuniverse[P]},\subuniverse))^{\simeq2}(A)$:
    \[\begin{tikzcd}
    	{B} && {B} && B \\
    	&&& {B'} \\
    	{B'} &&&& B' \\
    	& {B'} \\
    	{A'} && {A'} && {A}.
    	\arrow[equal,from=1-3, to=1-1]
    	\arrow[equal,from=1-3, to=1-5]
    	\arrow[from=2-4, to=1-3]
    	\arrow["\lrcorner"{anchor=center, pos=0.125, rotate=90}, draw=none, from=2-4, to=1-5]
    	\arrow[equal,from=2-4, to=3-5]
    	\arrow[from=3-1, to=1-1]
    	\arrow[from=3-1, to=5-1]
    	\arrow[from=3-5, to=1-5]
    	\arrow[from=3-5, to=5-5]
        \arrow[equal,from=2-4, to=4-2]
    	\arrow[equal,from=4-2, to=3-1]
    	\arrow["\lrcorner"{anchor=center, pos=0.125, rotate=-90}, draw=none, from=4-2, to=5-1]
    	\arrow[from=4-2, to=5-3]
    	\arrow[equal,from=5-3, to=5-1]
    	\arrow[from=5-3, to=5-5]
    \end{tikzcd}\]
    This is sent by $s$ to the span in the upper edge, which is clearly an isomorphism. 
\end{proof}

We let $\varintcat{S}\subset (\intspanhalf(\subuniverse,\subuniverse_{\subuniverse[P]},\subuniverse)_{//*})^{\simeq 2}$ denote the full $\vartopos$-subcategory spanned by objects where the left-facing morphism is an isomorphism (this $\vartopos$-category has a similar description to the category appearing in \Cref{prop:E_cart_S_is_equivalent_to_span}).

\begin{proposition}\label{prop:oplax_slice_S_is_monoidal}
    $\varintcat{S}$ is an $\subuniverse$-monoidal $\vartopos$-subcategory.
\end{proposition}
\begin{proof}
    We must show that for any $s:B\to A\in \subuniverse(A)$, the functor \[s_\otimes: (\intspanhalf(\subuniverse,\subuniverse_{\subuniverse[P]},\subuniverse)_{//*})^{\simeq 2}(B)\to (\intspanhalf(\subuniverse,\subuniverse_{\subuniverse[P]},\subuniverse)_{//*})^{\simeq 2}(A)\] sends objects in $\varintcat{S}(B)$ to objects in $\varintcat{S}(A)$.
    
    This is clear by the description of objects and of the $\subuniverse$-monoidal structure above.
\end{proof}

\begin{proposition}
    The composition 
    \[\varintcat{S}\subset (\intspanhalf(\subuniverse,\subuniverse_{\subuniverse[P]},\subuniverse)_{//*})^{\simeq 2}
    \xto{s} \intspan(\subuniverse,\subuniverse_{\subuniverse[P]},\subuniverse)\]
    is an equivalence of $\subuniverse$-monoidal $\vartopos$-categories.
\end{proposition}
\begin{proof}
    The composition is $\subuniverse$-monoidal by \Cref{prop:oplax_slice_source_is_monoidal,prop:oplax_slice_S_is_monoidal}. Thus it is enough to check that it is an equivalence of $\vartopos$-categories.
    
    To do that, the same proof as in \Cref{prop:E_cart_S_is_equivalent_to_span} works in this case. 
\end{proof}

\subsubsection*{Construction of \texorpdfstring{$\alpha$}{alpha}}

Thus, we have constructed an $\subuniverse$-monoidal inclusion of a $\vartopos$-subcategory: 
\[
\intspan(\subuniverse,\subuniverse_{\subuniverse[P]},\subuniverse)\into (\intspanhalf(\subuniverse,\subuniverse_{\subuniverse[P]},\subuniverse)_{//*})^{\simeq 2}.
\]

We consider the following composition of lax $\subuniverse$-monoidal functors:

\begin{align*}
    T:\intspan(\subuniverse,\subuniverse_{\subuniverse[P]},\subuniverse)
    &\into (\intspanhalf(\subuniverse,\subuniverse_{\subuniverse[P]},\subuniverse)_{//*})^{\simeq 2}\\
    &\to \inttwofun^{\oplax}(\Delta^1, \intspanhalf(\subuniverse,\subuniverse_{\subuniverse[P]},\subuniverse))^{\simeq2}\\
    &\to \inttwofun^{\oplax}(\Delta^1, \internaltwocatofcats)^{\simeq2}
\end{align*}

\begin{proposition}
    The lax $\subuniverse$-monoidal functor 
    \[s\circ T:\intspan(\subuniverse,\subuniverse_{\subuniverse[P]},\subuniverse)\to \internalcatofcats\]
    is equivalent to the $\subuniverse$-monoidal functor
    \[\widetilde{F}:\intspan(\subuniverse,\subuniverse_{\subuniverse[P]},\subuniverse)\into \spanEfE\xto{F}\internalcatofcats.\]
\end{proposition}
\begin{proof}
    Consider the following diagram:
    \[\begin{tikzcd}
        \inttwofun^{\oplax}(\Delta^1, \intspanhalf(\subuniverse,\subuniverse_{\subuniverse[P]},\subuniverse))^{\simeq2}& \inttwofun^{\oplax}(\Delta^1, \internaltwocatofcats)^{\simeq2}\\
        \intspan(\subuniverse,\subuniverse_{\subuniverse[P]},\subuniverse) & \internalcatofcats
        \arrow[from=1-1, to=1-2]
        \arrow["s", from=1-1, to=2-1]
        \arrow["s", from=1-2, to=2-2]
        \arrow[shift left=3, dashed, from=2-1, to=1-1]
        \arrow[from=2-1, to=2-2]
    \end{tikzcd}\]
    The solid arrows form a canonical commutative square of $\subuniverse$-monoidal $\vartopos$-categories. The left dashed arrow is the lax $\subuniverse$-monoidal section to $s$ factoring through $(\intspanhalf(\subuniverse,\subuniverse_{\subuniverse[P]},\subuniverse)_{//*})^{\simeq 2}$ constructed above. Its composite with $s$ is $\subuniverse$-monoidal and is equivalent to the identity as an $\subuniverse$-monoidal functor. 
    
    The functor $T$ is by definition the composition of the upper morphism with the section.

    Thus $s\circ T$ is equivalent to the bottom arrow in this diagram.

    The bottom arrow in this diagram is by definition the underlying $\vartopos$-functor of the $\vartopos$-$2$-functor defined in \Cref{cor:extending_the_alg_structure_of_C_to_backwards_facing_morphisms}.

    By \Cref{thm:internal_uniqueness_of_unfurling}, this functor agrees with the restriction of $F:\spanEfE\to \internalcatofcats$ because their restrictions to $\subuniverse$ agree by definition. Indeed, $\subuniverse_{\subuniverse[P]}$ is left cancellable because $\subuniverse[P]$ is inductible (\Cref{prop:from_local_class_to_subuniverse}), and every morphism in $\subuniverse_{\subuniverse[P]}$ is truncated by \Cref{assumption:global_trunc_bound}.
    \footnote{This is the only place in the construction where we use that $\subuniverse[P]$ satisfies \Cref{assumption:global_trunc_bound}.}
\end{proof}

\begin{proposition}
    The lax $\subuniverse$-monoidal functor 
    \[t\circ T:\intspan(\subuniverse,\subuniverse_{\subuniverse[P]},\subuniverse)\to \internalcatofcats\]
    is equivalent to the lax $\subuniverse$-monoidal functor
    \[\widetilde{G}:\intspan(\subuniverse,\subuniverse_{\subuniverse[P]},\subuniverse)\to \explicitset{*}\into \spanEfE\xto{F}\internalcatofcats.\]
\end{proposition}
\begin{proof}
    Consider the following diagram of $\subuniverse$-monoidal $\vartopos$-categories and lax $\subuniverse$-monoidal functors:
    \[\scalebox{0.75}{
    \begin{tikzcd}[ampersand replacement=\&]
        \intspan(\subuniverse,\subuniverse_{\subuniverse[P]},\subuniverse)\ar[d,equal]\ar[r]
        \&(\intspanhalf(\subuniverse,\subuniverse_{\subuniverse[P]},\subuniverse)_{//*})^{\simeq 2}\ar[r]\ar[d,"t"]\pullbackdr
        \&\inttwofun^{\oplax}(\Delta^1, \intspanhalf(\subuniverse,\subuniverse_{\subuniverse[P]},\subuniverse))^{\simeq2}\ar[r]\ar[d,"t"]
        \& \inttwofun^{\oplax}(\Delta^1, \internaltwocatofcats)^{\simeq2}\ar[d,"t"]\\
        \intspan(\subuniverse,\subuniverse_{\subuniverse[P]},\subuniverse)\ar[r]
        \&*\ar[r]
        \&\intspan(\subuniverse,\subuniverse_{\subuniverse[P]},\subuniverse)\ar[r]
        \&\internalcatofcats.
    \end{tikzcd}
    }\]
    The left square is commutative since $*$ is the terminal $\subuniverse$-monoidal $\vartopos$-category. The middle square is commutative since it is the pullback defining the $\subuniverse$-monoidal $\vartopos$-category $(\intspanhalf(\subuniverse,\subuniverse_{\subuniverse[P]},\subuniverse)_{//*})^{\simeq 2}$. The right square is a naturality square for $t$.
    
    The lower composition is $\widetilde{G}$ by definition. The upper composition followed by the rightmost arrow $t$ is $t\circ T$ by definition.
\end{proof}

Thus $T$ encodes a lax $\subuniverse$-monoidal lax natural transformation \[
    \widetilde{\alpha}:\widetilde{F}\Rightarrow\widetilde{G}:\intspan(\subuniverse,\subuniverse_{\subuniverse[P]},\subuniverse)\to \internaltwocatofcats.
\]
Restricting it to $\spanEPI$, we get the promised $\alpha:F'\Rightarrow G$, since $\widetilde{F}|_{\spanEPI}=F'$ and $\widetilde{G}|_{\spanEPI}=G$.

\subsubsection*{Monoidality of \texorpdfstring{$\Psi$}{Psi}}

Using \Cref{prop:unst_of_lax_E_monoidal_lax_natural_transformation}, we get a lax $\subuniverse$-monoidal functor
\[\Psi:\varintcat{C}^{\otimes}|_{\spanEPI}\simeq\intunst(F')\xto{\intunst(\alpha)} \intunst(G)\simeq \varintcat{C}\times \spanEPI\to \varintcat{C}.\]

Our last goal is to show that this functor is indeed $\subuniverse$-monoidal. To do so, we need to verify the following:

Let $p:B\to A\in \subuniverse(A)$. Then, using the universal property of cocartesian edges, we get a lax square in $\twocatofcats$:
\[\begin{tikzcd}
	\varintcat{C}^{\otimes}|_{\spanEPI}(B) & \varintcat{C}(B) \\
	\varintcat{C}^{\otimes}|_{\spanEPI}(A) & \varintcat{C}(A).
	\arrow[from=1-1, to=1-2,"\Psi"]
	\arrow[from=1-1, to=2-1,"p_\otimes"]
	\arrow[Rightarrow, nfold, from=1-2, to=2-1,"\psi"]
	\arrow[from=1-2, to=2-2,"p_\otimes"]
	\arrow[from=2-1, to=2-2,"\Psi"]
\end{tikzcd}\]
The functor $\Psi$ is $\subuniverse$-monoidal if and only if $\psi$ is an equivalence for every $p:B\to A$.

\begin{proposition}
    $\psi$ is an equivalence for every $p:B\to A$. 
\end{proposition}
\begin{proof}
    An object of $\varintcat{C}^{\otimes}|_{\spanEPI}(B)$ is an object $q:C\to B\in \spanEPI(B)$, that is, an object $q:C\to B\in \subuniverse(B)$, and an object $x\in \varintcat{C}(C)$.

    The functor $\Psi:\varintcat{C}^{\otimes}|_{\spanEPI}(B)\to \varintcat{C}(B)$ by definition applies the component of $\alpha$ at $C$ in the context $B$. Thus, it sends $x$ to $q_\otimes(x)$.

    Thus, the composition of the upper morphism and the right morphism takes $x$ to $p_\otimes q_\otimes x$.

    The left morphism takes $(x,q:C\to B)\in \varintcat{C}^{\otimes}|_{\spanEPI}(B)$ to $(x,p\circ q:C\to A)\in \varintcat{C}^{\otimes}|_{\spanEPI}(A)$. Then $\Psi$ takes $(x,p\circ q:C\to A)$ to $(p\circ q)_\otimes x$.

    The natural transformation $\psi$ can be seen to be the natural isomorphism $(p\circ q)_\otimes x\simeq p_\otimes q_\otimes x$.
\end{proof}

\subsection{Universality of \texorpdfstring{$\Xi_x$}{Xi}.}\label{subsec:universality_of_Xi}

We now prove that $\Xi_x$ has the required universal property. More precisely, the span $\vartopos$-category $\spanEPI[A]$ of \Cref{def:span_EPI_A} is the free $(\subuniverse[P,I])$-ambidextrous $\subuniverse$-monoidal $\vartopos$-category on $A$ (\Cref{thm:generalized_main_C}). We first describe the universal $A$ object in $\spanEPI[A]$ and the action of $\Xi_x$ on objects and morphisms.

\subsubsection*{The universal object and the functor \texorpdfstring{$\Xi_x$}{Xi}}

Recall from \Cref{def:span_EPI_A} that an object of $\spanEPI[A](B)$ is the same as an object of $\subuniverse[E][A](B)$. That is, an object $C\in \vartopos$ with an arbitrary morphism $\pi_C:C\to A$ and a morphism $p_C:C\to B$ in $\subuniverse$. Thus, an object of $\spanEPI[A](B)$ is a span
\begin{equation}\label{diagram:spanEPI_object}\begin{tikzcd}
    &C\ar[dl, "\pi_C"']\ar[dr, "p_C"]\\
    A&&B
\end{tikzcd}\end{equation}
with $[p_C]\in \subuniverse(B)$. A morphism $C\to D$ in $\spanEPI[A](B)$ is an iterated span 
\[\begin{tikzcd}
     & C\ar[dl, "\pi_C"']\ar[dr, "p_C"]\\
    A & F\ar[u,"p"']\ar[d,"i"] & B\\
     & D\ar[ul, "\pi_D"]\ar[ur, "p_D"']\\
\end{tikzcd}\]
with $p\in \subuniverse[P](C)$ and $i\in \subuniverse[I](D)$.

Let $A\in \vartopos$, and let $x_{\univ}^A:A\to \spanEPI[A]$ be the following functor. 
We view $A$ as a $\vartopos$-groupoid, that is, in the context $B$, we have $A(B):= \operatorname{map}_{\vartopos}(B,A)$. We define $x_{\univ}^A:A\to \spanEPI[A]$ to send $[B\to A]\in A(B)$ to 
\[\begin{tikzcd}
    &B\ar[dl]\ar[dr, equal]\\
    A&&B.
\end{tikzcd}\]

Let $\varintcat{C}$ be a $(\subuniverse[P,I])$-ambidextrous $\subuniverse$-monoidal $\vartopos$-category and let $x:A\to \varintcat{C}$ be an object of $\varintcat{C}$ in the context $A$. In \Cref{subsec:Phi_x,subsec:Psi}, we constructed a natural $\subuniverse$-monoidal functor \[\Xi_x:=\Psi\circ\Phi_x:\spanEPI[A]\to \varintcat{C}.\]

Note that for $x:A\to \varintcat{C}$, we have a commutative triangle
\[\begin{tikzcd}
    A\ar[dr, "x"']\ar[rr,"x_{\univ}^A"] && \spanEPI[A]\ar[dl, "\Xi_x"]\\
    &\varintcat{C}.
\end{tikzcd}\]

Let us describe the action of the functor $\Xi_x:\spanEPI[A]\to \varintcat{C}$ on objects and morphisms.

For every object $C\in \spanEPI[A](B)$ as described in \Cref{diagram:spanEPI_object}, we have $(p_C)_{\otimes}\pi_C^*(x_{\univ}^A)\simeq C$.
Since $\Xi_x$ is $\subuniverse$-monoidal, we have a canonical isomorphism $\Xi_x(C)\simeq (p_C)_{\otimes}\pi_C^*(x)$.

A morphism in $\subuniverse_{\subuniverse[I]}[A](B)\subset \spanEPI[A](B)$ from $C$ to $D$ is a morphism $i:C\to D \in \subuniverse[I](D)$ over $A\times B$. 
Note that $i$ is a component of the counit of the adjunction \[i_{\otimes}:\spanEPI[A](C)\fromto \spanEPI[A](D):i^*.\]
So $\Xi_x$ sends $i$ to the composition \[(p_C)_{\otimes}\pi_C^*(x)\simeq (p_D)_{\otimes}i_{\otimes} i^*\pi_D^*(x) \simeq (p_D)_{\otimes}i_{\sharp}i^*\pi_D^*(x)\xto{\epsilon} (p_D)_{\otimes}\pi_D^*(x).\]

A morphism in $\subuniverse_{\subuniverse[P]}[A](B)\subset \spanEPI[A](B)^{\op}$ from $C$ to $D$ is a morphism $f:C\to D \in \subuniverse[P](D)$ over $A\times B$. 
Note that $f$ is a component of the unit of the adjunction \[f^*:\spanEPI[A](D)\fromto \spanEPI[A](C):f_{\otimes}.\]
Then, the opposite morphism of $f$ is sent by $\Xi_x$ to the unit
\[(p_D)_{\otimes}\pi_D^*(x)\xto{\eta} (p_D)_{\otimes}f_{*}f^*\pi_D^*(x) \simeq (p_D)_{\otimes}f_{\otimes}f^*\pi_D^*(x)\simeq (p_C)_{\otimes}\pi_C^*(x).\]



    
    

From this discussion, we get the following:

\begin{corollary} \label{cor:xi_x_univ_is_id}
    The functor $x_{\univ}^A:A\to \spanEPI[A]$ gives rise to an endofunctor \[\Xi_{x_{\univ}^A}:\spanEPI[A]\to \spanEPI[A].\]
    This functor is an equivalence.
\end{corollary}
\begin{proof} 
    $\Xi_{x_{\univ}^A}$ is an $\subuniverse$-monoidal functor under the canonical $\subuniverse$-monoidal functor from $\subuniverse[E][A]^{\simeq}$, that is, we have a commutative diagram:
    \[
    \begin{tikzcd}
        &\subuniverse[E][A]^{\simeq}\ar[dl]\ar[dr]\\
        \spanEPI[A]\ar[rr, "\Xi_{x_{\univ}^A}"'] && \spanEPI[A].
    \end{tikzcd}
    \]
    Thus $\Xi_{x_{\univ}^A}$ is homotopic to the identity on groupoid cores.

    The anima of morphisms $C\to D$ in $\spanEPI[A](B)$ is the anima of iterated spans
    \[\begin{tikzcd}
         & C\ar[dl, "\pi_C"']\ar[dr, "p_C"]\\
        A & F\ar[u,"p"']\ar[d,"i"] & B\\
         & D\ar[ul, "\pi_D"]\ar[ur, "p_D"']\\
    \end{tikzcd}\]
    with $p\in \subuniverse[P](C)$ and $i\in \subuniverse[I](D)$.

    The morphisms associated with $p$ and $i$ come from the unit and counit, respectively, of the corresponding adjunctions determined by the $(\subuniverse[P,I])$-ambidextrous $\subuniverse$-monoidal structure on $\spanEPI[A]$. Thus, $\Xi_{x_{\univ}^A}$ acts as the identity on this anima.

    We have shown that $\Xi_{x_{\univ}^A}$ restricts to the identity on groupoid cores and that it is homotopic to the identity on mapping spaces. Thus $\Xi_{x_{\univ}^A}$ is an equivalence. 
\end{proof}

\begin{remark}
    It will follow later from \Cref{thm:generalized_main_C} that this endofunctor is homotopic to the identity, as suggested by the proof.
\end{remark}

\subsubsection*{The universal property}

We next observe that the construction of $\Xi_x$ is functorial.

\begin{remark}
    The construction \[[x:A\to \varintcat{C}]\mapsto[\Xi_x:\spanEPI[A]\to \varintcat{C}]\] is functorial in the sense that it lifts to a functor \[\Xi:\intmon^{\subuniverse,(\subuniverse[P,I])-\oplus}(\internalcatofcats)\times_{\internalcatofcats}\internalcatofcats_{A/} \to \intmon^{\subuniverse,(\subuniverse[P,I])-\oplus}(\internalcatofcats)_{\spanEPI[A]/}.\] Moreover, this functor fits into a commutative triangle:
    \[
    \begin{tikzcd}
        \intmon^{\subuniverse,(\subuniverse[P,I])-\oplus}(\internalcatofcats)\times_{\internalcatofcats}\internalcatofcats_{A/}\ar[rr]\ar[rd]
        && \intmon^{\subuniverse,(\subuniverse[P,I])-\oplus}(\internalcatofcats)_{\spanEPI[A]/}\ar[dl]\\
        &\internalcatofcats.
    \end{tikzcd}
    \]

    Indeed, it is enough to verify that \[\Phi_x:\spanEPI[A]\to \varintcat{C}^{\otimes}|_{\spanEPI}\] is functorial, which is evident from the construction of $\Phi_x$.
\end{remark}

We are ready to prove the main theorem of this section. Note that \Cref{main_theorem:free_E_monoidal_semiadditive} is a particular case of the following.
\begin{theorem}\label{thm:generalized_main_C}
    Let $\subuniverse,\subuniverse[P],\subuniverse[I]$ satisfy the standing hypotheses of \Cref{sec:proof_of_C}, including \Cref{assumption:global_trunc_bound}.
    For every $A\in\vartopos$, we have an equivalence of $\subuniverse$-monoidal $\vartopos$-categories \[L_{\oplus}\subuniverse^{\simeq}[A]\simeq \spanEPI[A].\]
\end{theorem}
\begin{proof}

    Applying the construction of $\Xi_x$ to the tautological object $A\to L_{\oplus}\subuniverse^{\simeq}[A]$ gives an $\subuniverse$-monoidal functor $\spanEPI[A]\to L_{\oplus}\subuniverse^{\simeq}[A]$.

    We have a canonical $\subuniverse$-monoidal functor $\subuniverse^{\simeq}[A]\to \spanEPI[A]$ from the universal property described in \Cref{prop:free_E_monoid_on_A} applied to $x_{\univ}^A$. By adjunction, we get an $\subuniverse$-monoidal functor $L_{\oplus}\subuniverse^{\simeq}[A]\to \spanEPI[A]$.

    The composition 
        \[
        L_{\oplus}\subuniverse^{\simeq}[A]\to \spanEPI[A]\to L_{\oplus}\subuniverse^{\simeq}[A]
        \]
    is canonically homotopic to the identity by the universal property of $L_{\oplus}\subuniverse^{\simeq}[A]$ (\Cref{cor:existance_of_free_ambi_monoidal_category}), which combines the free-monoid universal property with the localization adjunction. 

    The composition \[\spanEPI[A]\to L_{\oplus}\subuniverse^{\simeq}[A] \to \spanEPI[A]\] is canonically homotopic to $\Xi_{x_{\univ}^A}$ by the naturality of $\Xi$. Thus it is an equivalence by \Cref{cor:xi_x_univ_is_id}.
\end{proof}

\subsubsection*{The truncation assumption}

We end this section with a few remarks about \Cref{assumption:global_trunc_bound}.

\begin{remark}\label{remark:colimit_of_n_truncated}
    It is natural to ask whether \Cref{assumption:global_trunc_bound} is actually needed. Indeed, we used it to show that the $\vartopos$-categories $\subuniverse_{\subuniverse[I]}$ and $\subuniverse_{\subuniverse[P]}$ are truncated. 
    
    Without the assumption, $\subuniverse_{\subuniverse[I]}$ and $\subuniverse_{\subuniverse[P]}$ are not necessarily truncated. Indeed, take $\vartopos = \catofanima$, take $\subuniverse[I]$ to consist of morphisms with $\pi$-finite homotopy fibres, take $\subuniverse[P]$ to consist of isomorphisms, and take $\subuniverse$ to consist of morphisms whose homotopy fibres admit countable CW structures. Then in $\subuniverse_{\subuniverse[I]}(*)$ we have the morphism
    \[\coprod_{n\in\nats}B^nC_2 \to \nats.\]
    This morphism is locally truncated on its target and is in $\subuniverse[I]$ in the context of its target. However, this morphism is \emph{not} locally truncated in the context $*$.

    Nevertheless, we do not know if \Cref{assumption:global_trunc_bound} is needed for the correctness of \Cref{main_theorem:free_E_monoidal_semiadditive}. 
\end{remark}

\section{Proofs of the main theorems}\label{sec:proof_of_main}
In this section, we deduce \Cref{main_theorem:free_adjointed_2_cat} from \Cref{main_theorem:free_E_monoidal_semiadditive} and prove its lax symmetric monoidal version, \Cref{main_theorem:lax_monoidal_free_adjointed_2_cat}. Recall that, by \Cref{prop:reduction_of_main_theorem}, it suffices to prove that the functor
\[
b\mapsto\Span(\varcat{E}[a](b),\varcat{P},\varcat{I})
\]
is the free $(\varcat{P},\varcat{I})$-adjointed functor to $\catofcats$ on $\id_a$ (\Cref{thm:red_of_main}).

In \Cref{subsec:2_FF_are_just_E_monoidal}, we identify $2$-functor formalisms on $(\varcat{C},\varcat{E})$ with $\subuniverse$-monoidal $\vartopos$-categories for $\vartopos=\presh(\varcat{C})$ (\Cref{thm:equiv_between_2ff_and_E_monoidal}). Under this identification, $(\varcat{P},\varcat{I})$-adjointedness corresponds to $(\subuniverse[P,I])$-ambidexterity (\Cref{prop:adjointed_is_just_ambi}). We apply \Cref{thm:generalized_main_C} to establish the required universal property under a uniform truncation bound, and then remove this additional assumption (\Cref{thm:red_of_main_globally_trunc,prop:main_thm_follows_from_truncated_case}).

In \Cref{subsec:lax_monoidal}, we establish the universal property with respect to lax symmetric monoidal functors (\Cref{thm:lax_monoidal_free_adjointed_2_cat}), proving \Cref{main_theorem:lax_monoidal_free_adjointed_2_cat} and deducing Mann's conjecture (\Cref{cor:Mann_conj}). When $\varcat{C}$ carries its cartesian symmetric monoidal structure, we also describe adjointed $2$-functor formalisms as modules over an idempotent algebra, and adjointed $3$-functor formalisms as commutative algebras under it (\Cref{prop:adjointed_formalisms_as_modules}).

\subsection{The category of \texorpdfstring{$2$}{2}-functor formalisms.}\label{subsec:2_FF_are_just_E_monoidal}

In this subsection, we identify $2$-functor formalisms on a span pair $(\varcat{C},\varcat{E})$ with $\subuniverse$-monoidal $\vartopos$-categories for $\vartopos=\presh(\varcat{C})$, where $\subuniverse$ corresponds to morphisms representable in $\varcat{E}$ (\Cref{thm:equiv_between_2ff_and_E_monoidal}). Under this identification, $(\varcat{P},\varcat{I})$-adjointedness corresponds to $(\subuniverse[P,I])$-ambidexterity (\Cref{prop:adjointed_is_just_ambi}). We then apply \Cref{thm:generalized_main_C} to complete the proof of \Cref{main_theorem:free_adjointed_2_cat}, first under a uniform truncation bound and then without this additional assumption (\Cref{thm:red_of_main_globally_trunc,prop:main_thm_follows_from_truncated_case}).

\subsubsection*{Extension to the presheaf topos.}

Let $(\varcat{C},\varcat{E})$ be a span pair. A $2$-functor formalism on $(\varcat{C},\varcat{E})$ is a functor \[F:\Span(\varcat{C},\varcat{E})\to \catofcats.\]

We let $\vartopos := \presh(\varcat{C})$. Let $\vartopos[E] \subset \vartopos$ denote the wide sub-$\infty$-category of morphisms which are representable in $\varcat{E}$, that is, morphisms $X\to Y$ in $\vartopos$ such that for any morphism $C\to Y$ where $C$ is representable, the pullback $X\times_YC\to C$ is in $\varcat{E}$. By definition of representable morphisms, $(\vartopos,\vartopos[E])$ is a span pair.

Let $\vartopos[D]$ be a presentable $\infty$-category. Denote by \[\Fun^{\vartopos-R}(\Span(\vartopos,\vartopos[E]),\vartopos[D])\subset \Fun(\Span(\vartopos,\vartopos[E]),\vartopos[D]) \] the full sub-$\infty$-category spanned by functors $\Span(\vartopos,\vartopos[E])\to\vartopos[D]$ such that the composition \[\vartopos^{\op}\to \Span(\vartopos,\vartopos[E])\to\vartopos[D] \] preserves all limits.

Note that there is a canonical inclusion $\Span(\varcat{C},\varcat{E})\subset \Span(\vartopos,\vartopos[E])$.

\begin{lemma}\label{lem:extension_of_formalisms_to_presheaves}
    The inclusion $\Span(\varcat{C},\varcat{E})\subset \Span(\vartopos,\vartopos[E])$ induces an equivalence of $\infty$-categories:
    \[
    \Fun^{\vartopos-R}(\Span(\vartopos,\vartopos[E]), \vartopos[D]) \simeq \Fun(\Span(\varcat{C},\varcat{E}), \vartopos[D]).
    \]
\end{lemma}
\begin{proof}
    Note that $\Span(\varcat{C},\varcat{E})\subset \Span(\vartopos,\vartopos[E])$ is fully faithful. Thus, by \cite[Proposition 4.3.2.17]{HTT}, right Kan extension along it is fully faithful. It remains to show that its essential image consists precisely of those functors whose restriction to $\vartopos^{\op}$ preserves limits. This follows from a very similar argument to \cite[Proposition A.5.16]{Mann6Functors}. 
\end{proof}

\subsubsection*{Identification with internal monoidal categories.}

We now consider $\vartopos[E]$ as a local class of morphisms in the $\infty$-topos $\vartopos$. This local class corresponds to a context-free subuniverse $\subuniverse\subset \universe_{\vartopos}$ by \Cref{prop:from_local_class_to_subuniverse}. 

\begin{proposition}\label{prop:2FF_are_E_mon}
    There is an equivalence of $\infty$-categories \[\Fun^{\vartopos-R}(\Span(\vartopos,\vartopos[E]), \catofcats)\simeq \intsmon(\internalcatofcats)(*).\]
\end{proposition}

\begin{proof}
    The equivalence of \Cref{prop:extra_back_functoriality_for_E_monoidal_cats} induces an equivalence of $\infty$-categories on global sections
    \[
    \intsmon(\internalcatofcats)(*) \simeq \intfun^{\universe-\sqcap}(\intspan(\universe,\universe_{\subuniverse}), \internalcatofcats)(*).
    \]
    The right-hand side is a full sub-$\infty$-category of the $\infty$-category of $\vartopos$-functors from $\intspan(\universe,\universe_{\subuniverse})$ to $\internalcatofcats$. 
    
    We get a functor 
    \[\theta:\intfun^{\universe-\sqcap}(\intspan(\universe,\universe_{\subuniverse}), \internalcatofcats)(*)\to \Fun(\Span(\vartopos,\vartopos[E]), \catofcats)\]
    as follows. We evaluate a $\vartopos$-functor $F:\intspan(\universe,\universe_{\subuniverse})\to \internalcatofcats$ at the terminal object $*$ to get a functor $\Span(\vartopos,\vartopos[E])\to \catofcats(\vartopos)$ and compose it with the global sections functor $\catofcats(\vartopos)\to \catofcats$.

    We show this functor lands in $\Fun^{\vartopos-R}(\Span(\vartopos,\vartopos[E]),\catofcats)$. Since $F$ preserves $\universe$-limits, the restriction of $\theta(F)$ to $\vartopos^{\op}$ reconstructs the underlying $\vartopos$-category of $F$. This functor is continuous by definition. 

    There is an inverse of $\theta$ which we will now explain how to construct. 
    
    Let $X\in \vartopos$. Denote by $\Span(\vartopos_{/X},\vartopos[E])$ the span $\infty$-category of $\vartopos_{/X}$ with right-pointing maps that are in $\vartopos[E]$. There is a functor $(\vartopos_{/X})^{\op}\times \Span(\vartopos_{/X},\vartopos[E])\to \Span(\vartopos,\vartopos[E])$ taking $A\to X \in (\vartopos_{/X})^{\op}$ and $B\to X \in \Span(\vartopos_{/X},\vartopos[E])$ to the pullback $A\times_X B$. It is not hard to see that this is functorial in $X\in \vartopos^{\op}$ in the following sense:

    Given $G:\Span(\vartopos,\vartopos[E])\to \catofcats$ in $\Fun^{\vartopos-R}(\Span(\vartopos,\vartopos[E]),\catofcats)$, we get, functorially in $X\in \vartopos^{\op}$, a functor \[\intspan(\universe,\universe_{\subuniverse})(X)\simeq  \Span(\vartopos_{/X},\vartopos[E])\to \Fun^R((\vartopos_{/X})^{\op}, \catofcats)\simeq \internalcatofcats(X).\]

    Checking that this constructs an inverse to $\theta$ is straightforward.
\end{proof}

Overall, we get the following theorem saying that $2$-functor formalisms correspond uniquely to globally defined $\subuniverse$-monoidal $\vartopos$-categories. 

\begin{theorem}\label{thm:equiv_between_2ff_and_E_monoidal}
    There is an equivalence of $\infty$-categories 
    \[\Fun(\Span(\varcat{C},\varcat{E}), \catofcats)\simeq \intsmon(\internalcatofcats)(*).\]
\end{theorem}
\begin{proof}
    This is an immediate corollary of \Cref{lem:extension_of_formalisms_to_presheaves,prop:2FF_are_E_mon}.
\end{proof}

Let $\varcat{P},\varcat{I}\subset\varcat{E}$ satisfy the hypotheses of \Cref{main_theorem:free_adjointed_2_cat}. Let $\subuniverse[P],\subuniverse[I]\subset\subuniverse$ denote the subuniverses corresponding to the local classes of morphisms representable in $\varcat{P}$ and $\varcat{I}$, respectively.

\begin{proposition}\label{prop:adjointed_is_just_ambi}
    The equivalence of \Cref{thm:equiv_between_2ff_and_E_monoidal} restricts to an equivalence \[\Fun^{(\varcat{P},\varcat{I})\operatorname{-adjointed}}(\Span(\varcat{C},\varcat{E}), \catofcats)\simeq \intmon^{\subuniverse,(\subuniverse[P],\subuniverse[I])\text{-}\oplus}(\internalcatofcats)(*).\]
\end{proposition}
\begin{proof}
    Unraveling the definitions, we can observe that the conditions are defined by the same inductive procedure.
\end{proof}

\subsubsection*{Proof of \texorpdfstring{\Cref{main_theorem:free_adjointed_2_cat}}{Theorem A}.}

We are now ready to prove \Cref{thm:red_of_main} in the case where there exists $n\geq -2$ such that every morphism in $\varcat{P}$ and $\varcat{I}$ is $n$-truncated. 

\begin{theorem}\label{thm:red_of_main_globally_trunc}
    Let $\varcat{C}$ be an $\infty$-category, and let $\varcat{E},\varcat{I},\varcat{P}\subset\varcat{C}$ be wide sub-$\infty$-categories that are stable under base change in $\varcat{C}$. Assume that $\varcat{I},\varcat{P}\subset\varcat{E}$ are left cancellable and that every morphism in $\varcat{I}$ and $\varcat{P}$ is $n$-truncated for some fixed $n\geq -2$. Let $a\in \varcat{C}$. Let $F\colon\Span(\varcat{C},\varcat{E})\to\catofcats$ denote the functor given by
    \[
    F(b):=\Span(\varcat{E}[a](b),\varcat{P},\varcat{I}),
    \]
    obtained by composing the functor $b\mapsto(\varcat{E}[a](b),\varcat{P},\varcat{I})$ with $\Span$. Then $F$ is the free $(\varcat{P},\varcat{I})$-adjointed functor on $\id_a\in F(a)$.
\end{theorem}

\begin{proof}
    Let $A=\yo(a)\in\vartopos$ be the representable presheaf associated to $a$. By \Cref{thm:generalized_main_C,prop:adjointed_is_just_ambi}, it remains to verify that $\spanEPI[A]$ corresponds to the functor $F\colon\Span(\varcat{C},\varcat{E})\to\catofcats$ given by $F(b):=\Span(\varcat{E}[a](b),\varcat{P},\varcat{I})$. This follows by direct inspection, in the same way as in the proof of \cite[Theorem 4.20.]{CLLuniv}.
\end{proof}

Thus, \Cref{main_theorem:free_adjointed_2_cat} holds if every morphism in $\varcat{I}$ and $\varcat{P}$ is $n$-truncated for some fixed $n$.

\begin{proposition}\label{prop:main_thm_follows_from_truncated_case}
    Assume \Cref{main_theorem:free_adjointed_2_cat} in the case where every morphism in $\varcat{I}$ and $\varcat{P}$ is $n$-truncated for some fixed $n$. Then \Cref{main_theorem:free_adjointed_2_cat}   follows.
\end{proposition}
\begin{proof}
    Let $\varcat{P}_n\subset\varcat{P},\varcat{I}_n\subset\varcat{I}$ denote the wide sub-$\infty$-categories of morphisms which are $n$-truncated.
    
    By comparison of universal properties, it is enough to show that the following is a sequential colimit of $(\infty,2)$-categories:
    \begin{align*}
    \Span(\varcat{C},\varcat{E})&=\spantwo(\varcat{C},\varcat{E})^{\varcat{P}_{-2}}_{\varcat{I}_{-2}}\\
    &\to  \spantwo(\varcat{C},\varcat{E})^{\varcat{P}_{-1}}_{\varcat{I}_{-1}}\\
    &\to \spantwo(\varcat{C},\varcat{E})^{\varcat{P}_{0}}_{\varcat{I}_{0}} \to
    \hdots \\
    &\to \spantwo(\varcat{C},\varcat{E})^{\varcat{P}_{n}}_{\varcat{I}_{n}}
    \to
    \hdots \\
    &\to \spantwo(\varcat{C},\varcat{E})^{\varcat{P}}_{\varcat{I}}.
    \end{align*}

    This is a colimit diagram on the level of bisimplicial anima. 
\end{proof}

Thus we have established \Cref{main_theorem:free_adjointed_2_cat}.




\subsection{Universality with respect to lax monoidal functors.}\label{subsec:lax_monoidal}

In this subsection, we prove \Cref{main_theorem:lax_monoidal_free_adjointed_2_cat} by establishing the universal property of $\spantwo(\varcat{C},\varcat{E})^{\varcat{P}}_{\varcat{I}}$ with respect to lax symmetric monoidal functors (\Cref{thm:lax_monoidal_free_adjointed_2_cat}). This also proves Mann's conjecture (\Cref{cor:Mann_conj}).

When $\varcat{C}$ carries its cartesian symmetric monoidal structure, we show that the equivalence of \Cref{thm:equiv_between_2ff_and_E_monoidal} is compatible with Day convolution (\Cref{prop:formalisms_equivalence_is_symmetric_monoidal}). We then identify $(\varcat{P},\varcat{I})$-adjointed $2$-functor formalisms with modules over the universal adjointed formalism $h_*$, and their lax symmetric monoidal counterparts with commutative algebras under $h_*$ (\Cref{prop:adjointed_formalisms_as_modules}). We conclude by observing that, in this presheaf setting, \Cref{main_theorem:free_E_monoidal_semiadditive} holds without the additional uniform truncation bound (\Cref{remark:presheaf_case_without_uniform_truncation}).

\subsubsection*{Symmetric monoidal span categories.}

Let $\vartwocat{C}$ be a symmetric monoidal $(\infty,2)$-category.
Then the symmetric monoidal structure on $\vartwocat{C}$ is classified by a finite-product-preserving functor 
\[
F_{\vartwocat{C}}:\Span(\Fin)\rightarrow \catoftwocats.
\]
We denote by $\vartwocat{C}^\otimes$ the cocartesian unstraightening of $F_{\vartwocat{C}}$ (see \cite[Definition 5.1. and subsequent discussions]{CLLuniv} for the definitions and basic results on $2$-cocartesian fibrations). 

Let $\operatorname{SpanQuad}$ denote the following sub-$\infty$-category of the $\infty$-category of diagrams in $\catofcats$ of the shape 
\[\begin{tikzcd}
    P\ar[d]\\ E\ar[r] & C\\
    I.\ar[u]
\end{tikzcd}\]
Objects of $\operatorname{SpanQuad}$ are those diagrams where $E,P,I$ are wide sub-$\infty$-categories of $C$ that are stable under base change in $C$, with $P,I\subset E$ left cancellable and every morphism in $P$ and $I$ truncated. Morphisms of $\operatorname{SpanQuad}$ are those morphisms that preserve pullbacks along all of $I,P,E$. We call an object $(C,E,P,I)\in \operatorname{SpanQuad}$ a span quadruple.

We have a natural transformation between finite-product-preserving functors 
\[\begin{tikzcd}
	\operatorname{SpanQuad} && \catoftwocats.
	\arrow[""{name=0, anchor=center, inner sep=0}, "{(C,E,P,I)\mapsto \spantwo(C,E)^P_I}"', shift right=2, from=1-1, to=1-3]
	\arrow[""{name=1, anchor=center, inner sep=0}, "{(C,E,P,I)\mapsto \Span(C,E)}", shift left=2, from=1-1, to=1-3]
	\arrow[between={0.2}{0.8}, Rightarrow, nfold, from=1, to=0]
\end{tikzcd}\]

Suppose that $C$ is a symmetric monoidal $\infty$-category and that $(C,E,P,I)\in \operatorname{SpanQuad}$ is a span quadruple. Suppose further that $E,P,I\subset C$ are symmetric monoidal sub-$\infty$-categories of $C$ and that the tensor product $C^2\to C$ induces a morphism of span quadruples $(C^2,E^2,P^2,I^2)\to (C,E,P,I)$. Let $F_C:\Span(\Fin)\to\catofcats$ classify the symmetric monoidal structure on $C$. Then $F_C$ lifts to a finite-product-preserving functor 
\[\widetilde{F}:\Span(\Fin)\to \operatorname{SpanQuad}.\]

Under these assumptions, $\Span(C,E)$ is a symmetric monoidal $\infty$-category and $\spantwo(C,E)^P_I$ is a symmetric monoidal $(\infty,2)$-category. The inclusion
\[\Span(C,E)\into \spantwo(C,E)^P_I\]
is canonically symmetric monoidal. 

\begin{definition}
    Let \[\wt{C}\to \Span(\Fin)^{\op}\] denote the cartesian unstraightening of the functor 
    \[\Span(\Fin)\xto{F_C}\catofcats\]
    above. Let $\wt{E},\wt{P},\wt{I}\subset\wt{C}$ denote the wide sub-$\infty$-categories spanned by morphisms over isomorphisms in $\Span(\Fin)^{\op}$ which are in $E,P,I$, respectively.
\end{definition}

\begin{lemma}
    $\wt{P},\wt{I}\subset \wt{E}\subset \wt{C}$ are wide sub-$\infty$-categories that are stable under base change in $\wt{C}$. Moreover, $\wt{P},\wt{I}$ are left cancellable, and every morphism in $\wt{P}$ and $\wt{I}$ is truncated.
\end{lemma}
\begin{proof}
    This follows by the same argument as in \cite[Lemma 5.5.]{CLLuniv}.
\end{proof}

\begin{proposition}\label{prop:span_commutes_with_tensor}
    We have an equivalence of $(\infty,2)$-categories 
    \[
    \spantwo(\wt{C},\wt{E})^{\wt{P}}_{\wt{I}}\simeq (\spantwo(C,E)^P_I)^\otimes.
    \]
\end{proposition}
\begin{proof}
    First, we know that $\Span(\wt{C},\wt{E}) \simeq \Span(C,E)^\otimes$ by \cite[Theorem 3.9]{haugseng2023twovariablefibrationsfactorisationsystems}. Thus, we get a functor 
    \[T:\Span(\wt{C},\wt{E})\to (\spantwo(C,E)^P_I)^\otimes.\]

    The rest of the proof follows the proof of \cite[Theorem 5.10.]{CLLuniv} closely. Namely, one shows that the functor $T$ satisfies the requirements of \Cref{main_theorem:free_adjointed_2_cat} with respect to the quadruple $(\wt{C},\wt{E},\wt{P},
    \wt{I})$, gets a functor
    \[
        T':\spantwo(\wt{C},\wt{E})^{\wt{P}}_{\wt{I}}\to (\spantwo(C,E)^P_I)^\otimes
    \]
    and shows that $T'$ is an equivalence.
\end{proof}

\subsubsection*{Lax symmetric monoidal universality.}

We want to use this proposition in order to prove that $\spantwo(C,E)^P_I$ has a universal property with respect to lax symmetric monoidal functors. We first need the following lemma:

\begin{lemma}\label{lemma:characterization_of_P_tensor_I_tenso_adjointed}
    Let $\vartwocat{K}$ be a symmetric monoidal $(\infty,2)$-category. 
    Let
    \[f:\Span(C,E)\to \vartwocat{K}\]
    be a lax symmetric monoidal functor. Then $f$ is $(P,I)$-adjointed if and only if the composition
    \[
    g:\Span(\wt{C},\wt{E})\simeq \Span(C,E)^\otimes\xto{f^\otimes} \vartwocat{K}^\otimes
    \]
    is $(\wt{P},\wt{I})$-adjointed.
\end{lemma}
\begin{proof}
    Note that $\wt{P},\wt{I}\subset \Span(C,E)^\otimes\to \Span(\Fin)$ live over the maximal subgroupoid $\Fin^\simeq\subset \Span(\Fin)$. Thus, the lemma reduces to the following obvious claim:

    $f$ is $(P,I)$-adjointed if and only if, for any finite set $A$, $f^A:\Span(C,E)^A\to \vartwocat{K}^A$ is $(P^A,I^A)$-adjointed.
\end{proof}

\begin{remark}
    In \cite{CLLuniv}, there is a requirement for the compatibility of the symmetric monoidal structure of $C$ with the adjointability condition (see \cite[Lemma 5.9.]{CLLuniv}). As the proof above suggests, in our case this additional condition is not needed.
\end{remark}

We are now ready to prove \Cref{main_theorem:lax_monoidal_free_adjointed_2_cat}.

\begin{theorem}\label{thm:lax_monoidal_free_adjointed_2_cat}
    Let $(C,E,P,I)$ be as above and let $\vartwocat{K}$ be a symmetric monoidal $(\infty,2)$-category. 
    We have a fully faithful embedding 
    \[\Fun^{\lax-\otimes}(\spantwo(C,E)^P_I,\vartwocat{K})\into \Fun^{\lax-\otimes}(\Span(C,E),\vartwocat{K}) 
    \]
    and the essential image is spanned by the lax symmetric monoidal functors whose underlying functors are $(P,I)$-adjointed.
\end{theorem}
\begin{proof}
    By \Cref{prop:span_commutes_with_tensor} and \Cref{main_theorem:free_adjointed_2_cat}, we have a fully faithful embedding
    \[\Fun_{/\Span(\Fin)}((\spantwo(C,E)^P_I)^\otimes,\vartwocat{K}^\otimes)\into \Fun_{/\Span(\Fin)}(\Span(C,E)^\otimes,\vartwocat{K}^\otimes) \]
    with essential image those functors that are $(\wt{P},\wt{I})$-adjointed.
    Restricting to functors that preserve $\Fin^{\op}$-cocartesian lifts, we get a fully faithful embedding
    \[\Fun^{\lax-\otimes}(\spantwo(C,E)^P_I,\vartwocat{K})\into \Fun^{\lax-\otimes}(\Span(C,E),\vartwocat{K}). 
    \]

    The characterization of the essential image follows from \Cref{lemma:characterization_of_P_tensor_I_tenso_adjointed}.
\end{proof}

Mann's conjecture is an immediate consequence:
\begin{corollary}\label{cor:Mann_conj}
    Let $(C,E)$ be a geometric context and let 
    \[D:\Span(C,E)\to \catofcats\]
    be a three-functor formalism. Let $P,I\subset E$ denote the wide sub-$\infty$-categories of $D$-proper and $D$-\'etale morphisms in $C$. Then $D$ extends uniquely to a lax symmetric monoidal $2$-functor
    \[D:\spantwo(C,E)^P_I\to \twocatofcats.\]
\end{corollary}

\begin{proof}
    The fact that $P,I$ are stable under base change in $C$ and left cancellable, and that every morphism in $P$ and $I$ is truncated, was proved by Mann and Heyer (see \cite[Remark 4.6.5.]{just_6ff}). Thus this is a direct application of \Cref{thm:lax_monoidal_free_adjointed_2_cat}.
\end{proof}

\subsubsection*{Adjointed formalisms as modules.}

For the rest of the subsection, we analyze the situation further when $\vartwocat{K}=\twocatofcats$ and the symmetric monoidal structure on $C$ is cartesian.

Let $(C,E)$ be a span pair such that $C$ has finite products. Set $\vartopos:=\presh(C)$ and let $\mathcal{E}\subset\vartopos$ be the wide sub-$\infty$-category of morphisms which are representable by morphisms in $E$. Let $\subuniverse\subset \universe_{\vartopos}$ be the context-free subuniverse in $\vartopos$ corresponding to the local class $\mathcal{E}$.  

\begin{proposition}\label{prop:formalisms_equivalence_is_symmetric_monoidal}
    The equivalence of $\infty$-categories
    \[\Fun(\Span(\varcat{C},\varcat{E}), \catofcats)\simeq \intsmon(\internalcatofcats)(*)\]
    established in \Cref{thm:equiv_between_2ff_and_E_monoidal} lifts to a symmetric monoidal equivalence where the left-hand side is equipped with Day convolution and the right-hand side is equipped with the symmetric monoidal structure described in \Cref{cor:E_monoidal_cats_are_symmetric_monoidal}.
\end{proposition}
\begin{proof}
    The equivalence
    \[\Fun(\Span(\varcat{C},\varcat{E}), \catofcats)\simeq \intsmon(\internalcatofcats)(*)\]
    is given by the universal property of $\Fun(\Span(\varcat{C},\varcat{E}), \catofcats)$ in the $\infty$-category of $\catofcats$-modules in $\prl$ applied to the functor 
    \[i:\Span(\varcat{C},\varcat{E})^{\op}\into \intsmon(\internalcatofcats)(*).\]
    The functor $i$ is clearly symmetric monoidal, since it identifies with a symmetric monoidal sub-$\infty$-category of the $\vartopos$-Yoneda image of 
    \[\intfun^{\universe-\sqcap}(\intspan(\universe,\universe_{\subuniverse}), \internalcatofcats)(*)\simeq \intsmon(\internalcatofcats)(*).\]
\end{proof}

Recall that in \Cref{cor:P_I_ambi_is_a_smashing_localization}, we showed that $L_{(\subuniverse[P,I])-\oplus}(\subuniverse^{\simeq})$ is an idempotent algebra in $\intsmon(\internalcatofcats)(*)$ and that there is an equivalence 
\[
\intmon^{\subuniverse,(\subuniverse[P],\subuniverse[I])\text{-}\oplus}(\internalcatofcats)(*) \simeq\operatorname{Mod}_{L_{(\subuniverse[P,I])-\oplus}(\subuniverse^{\simeq})}(\intsmon(\internalcatofcats)(*)).
\]
We can transfer this to the $2$-functor formalisms side:

Since $C$ has finite products, it in particular has a terminal object $*\in C$.
Let \[h_*:\Span(C,E)\to \catofcats\] be the universal $(P,I)$-adjointed functor on $*$, that is, for $x\in C$,
\[h_*(x):= \Fun_{\spantwo(C,E)^P_I}(*,x).\]

Since $h_*$ is the unit of the symmetric monoidal $\infty$-category $\Fun^{(P,I)\operatorname{-adjointed}}(\Span(\varcat{C},\varcat{E}), \catofcats)$, it corresponds to $L_{(\subuniverse[P,I])-\oplus}(\subuniverse^{\simeq})$ which is the unit in $\operatorname{Mod}_{L_{(\subuniverse[P,I])-\oplus}(\subuniverse^{\simeq})}(\intsmon(\internalcatofcats)(*))$.

We get the following:
\begin{proposition}\label{prop:adjointed_formalisms_as_modules}
    There are equivalences of $\infty$-categories
    \[\Fun^{(P,I)\operatorname{-adjointed}}(\Span(\varcat{C},\varcat{E}), \catofcats)\simeq \operatorname{Mod}_{h_*}(\Fun(\Span(\varcat{C},\varcat{E}), \catofcats))\]
    and
    \[\Fun^{\lax-\otimes,(P,I)\operatorname{-adjointed}}(\Span(\varcat{C},\varcat{E}), \catofcats)\simeq \operatorname{CAlg}_{h_*}(\Fun(\Span(\varcat{C},\varcat{E}), \catofcats)).\]
\end{proposition}
\begin{proof}
    The first equivalence follows since $h_*$ corresponds to $L_{(\subuniverse[P,I])-\oplus}(\subuniverse^{\simeq})$ which is an idempotent algebra in $\intsmon(\internalcatofcats)(*)$ together with the identification of \Cref{prop:adjointed_is_just_ambi}.

    The second equivalence follows from the first together with the identification of commutative algebras for Day convolution with lax symmetric monoidal functors, using \Cref{thm:lax_monoidal_free_adjointed_2_cat}.
\end{proof}

\begin{remark}\label{remark:presheaf_case_without_uniform_truncation}
    On the other hand, direct comparison shows that $h_*$ corresponds to $\spanEPI\in \intsmon(\internalcatofcats)(*)$. Thus, in this case, \Cref{main_theorem:free_E_monoidal_semiadditive} holds without the additional uniform truncation bound of \Cref{assumption:global_trunc_bound}.
\end{remark}

\printbibliography

\bigskip
\begingroup
\small
\noindent
Shachar Carmeli:
\href{mailto:shachar.carmeli@weizmann.ac.il}
{\texttt{shachar.carmeli@weizmann.ac.il}}\\
Guy Kapon:
\href{mailto:guy.kapon@weizmann.ac.il}
{\texttt{guy.kapon@weizmann.ac.il}}\\
Noam Nissan:
\href{mailto:noam.zimhoni@weizmann.ac.il}
{\texttt{noam.zimhoni@weizmann.ac.il}}
\par
\endgroup

\end{document}